%% file: embook.tex
\documentclass{amsbook}

\usepackage{amsmath, amssymb, amsthm, amscd, graphics}
\newtheorem{theorem}{Theorem}[chapter]
\newtheorem{lemma}[theorem]{Lemma}
\newtheorem{cor}{Corollary}[theorem]

\newtheorem*{thm}{Theorem}
\newtheorem*{lem}{Lemma}
\newtheorem*{ex}{Example}
\newtheorem*{cor*}{Corollary}

\renewcommand{\geq}{\geqslant}
\renewcommand{\leq}{\leqslant}

\theoremstyle{definition}

\newtheorem*{defn}{Definition}

\theoremstyle{remark}

\numberwithin{section}{chapter}
\numberwithin{equation}{chapter}

\begin{document}
\frontmatter
\title{Locally flat embeddings of 3-manifolds in $S^4$}

\author{J.A.Hillman}
\address{The University of Sydney}
\email{jonathan.hillman@sydney.edu.au}

\maketitle

\setcounter{page}{4}

\tableofcontents

\include{epref}

\mainmatter

\include{e1}

\include{e2}
\include{e3}

\include{e4}

\include{e5}

\include{e6}

\include{e7}

\include{e8}

\backmatter

\include{eappA}

\include{eappB}

\include{eappC}

\include{ebib}

\include{eind}

\end{document}

%% file: epref.tex
\chapter*{Preface}

Every closed 3-manifold may be smoothly embedded in $S^5$.
The question of which 3-manifolds embed in $S^4$ 
depends markedly on the interpretation of ``embedding".
Freedman showed that every integral homology 3-sphere 
bounds a contractible TOP 4-manifold, 
and so embeds as a TOP locally flat submanifold of $S^4$. 
This is strikingly simpler than in the case of smooth (or PL locally flat) embeddings,
where the corresponding question remains open, and seems quite delicate.
The exemplary uniqueness result is the TOP Schoenflies Theorem,
which shows that all locally flat embeddings of $S^{n-1}$ in $S^n$
are topologically equivalent.
In every dimension $n\not=4$ the smooth analogue also holds.
There is a remarkable contrast between the simplicity of the proofs in the
TOP case and the as yet unresolved status of the 4-dimensional DIFF 
Schoenflies Conjecture.

We have chosen therefore to focus on TOP locally flat embeddings,
rather than smooth embeddings.
Many of the examples in this book are constructed by ambient surgery on
``bipartite links" in the equatorial $S^3$, 
and so are smooth, 
but most of the algebraic constraints on embeddings that we use 
also obstruct locally flat embeddings into homology 4-spheres.

It seems unlikely that there will ever be a reduction of the question of which 
closed orientable 3-manifolds embed in $S^4$ to the determination of familiar invariants.
However if we restrict the scope of the question to 3-manifolds with
virtually solvable fundamental groups or which are Seifert fibred then much is known. 
Manifolds of the first type have one of the geometries
$\mathbb{S}^3$, $\mathbb{S}^2\times\mathbb{E}^1$, 
$\mathbb{E}^3$, $\mathbb{N}il^3$ and $\mathbb{S}ol^3$.
Just 13 manifolds of this type embed, 
as was shown in joint work with J. S. Crisp.
Seifert fibred manifolds either have one of the first four 
of these geometries or are $\mathbb{H}^2\times\mathbb{E}^1$- 
or $\widetilde{\mathbb{SL}}$-manifolds.
The strong results on smooth embeddings of Seifert fibred 3-manifolds
due to A. Donald and to A. Issa and D. McCoy suggest plausible outcomes
for the Seifert fibred cases.

The main themes of this book are the concentration on well-understood classes 
of 3-manifolds, 
the study of the complementary regions, 
and the use of simple invariants to show that most 3-manifolds which embed 
in $S^4$ do so in many distinct ways.
We use 4-dimensional TOP surgery at various points, 
but  on the other hand we make no use of gauge-theoretic ideas.

The first two chapters are introductory.
Chapter 1 gives the notation and terminology from algebra and topology 
that we shall use.
The key invariants are the Euler characteristic and fundamental groups 
of the complementary regions.
In Chapter 2 we summarize the basic facts about these invariants,
which go back to the first paper on the subject, by W. Hantzsche in 1938,
enhanced much later by A. Kawauchi and S. Kojima.
We also describe the construction of embeddings by 0-framed surgery 
on bipartedly slice links.
This construction was introduced  by P. Gilmer and C. Livingston,
and was used by W. B. R. Lickorish and F. Quinn (independently),
to realize pairs of groups with balanced presentations and isomorphic abelianization.
We then define the notion of 2-knot surgery (first used by Livingston),
by means of which it may be shown that most 3-manifolds which embed do so in infinitely many ways.

In Chapter 3 we consider 3-manifolds which are Seifert fibred over orientable base orbifolds, and apply the $G$-Index Theorem (for finite cyclic $G$) to show that if 
$M$ is an $\mathbb{H}^2\times\mathbb{E}^1$-manifold and
the base orbifold has genus 0 and all cone points of odd order then $M$ embeds in
$S^4$ if and only if the Seifert data is skew-symmetric.
We state here some of the results of Donald, Issa and McCoy.
In Chapter 4 we consider 3-manifolds which are Seifert fibred over non-orientable bases
and give strong bounds for the Euler numbers of manifolds with given base orbifold
which can embed. (The argument here also uses the $G$-Index Theorem,
albeit in a different manner, and for $G=\mathbb{Z}/2\mathbb{Z}$ only.)
In chapter 5 we settle the case of 3-manifolds with virtually solvable group.
We then uses 4-dimensional topological surgery to show that if $M$ is a 3-manifold such that $\pi_1(M)$ is an extension of a torsion-free solvable group $G$ by a perfect normal subgroup then $M$ embeds if and only if $G\cong\pi_1(P)$ for some 3-manifold $P$ which embeds.

In the remaining three chapters we consider the complementary regions in some detail.
If $M=\#^r(S^2\times{S^1})$ then $M$ has different embeddings with 
complementary regions having Euler characteristics of all possible values
allowed by the elementary considerations of Chapter 2. 
We use Massey products to show that this is not the case for $M$ 
a Seifert fibred 3-manifold with $\beta_1(M)>1$.

In order to use 4-dimensional surgery arguments we must restrict 
the possible fundamental groups of the complementary regions.
Under our present understanding of the Disc Embedding Theorem,
these groups should be in the class of groups generated from groups 
with sub-exponential growth by increasing unions and extensions.
In Chapters 7 and 8 we assume that the complementary regions
have abelian or nilpotent fundamental groups (respectively).
The possible groups are known in the abelian case and severely restricted 
in the nilpotent case. 
If the group is also torsion-free then there is scope for identifying the
homotopy types of the complementary regions and applying surgery.

There are three appendices.
In Appendix A we analyze the linking pairings of orientable 
3-manifolds which are Seifert fibred over orientable base orbifolds.
In Appendix B we give some results on nilpotent groups with balanced presentations, and Appendix C is a short list of questions about embeddings of 3-manifolds in $S^4$.

The book is based on a number of my papers and one paper written 
jointly with John Crisp.
These have only been cited when the presentation here 
is different from the earlier version.
Allowing for consolidation of common material into the earlier chapters
and some reorganization, Chapter 3 corresponds to \cite{Hi09},
Chapter 4 and the first half of Chapter 5 to \cite{CH98}, 
the second half of Chapter 5 to \cite{Hi96}, 
and the remaining chapters to \cite{Hi17},
\cite{Hi20a}, and \cite{Hi20b}, 
respectively.
The appendices are based on \cite{Hi11} and \cite{Hi22a}.

I would like to thank John for the collaboration which is recorded
in Chapter 4 and the first part of Chapter 5.
In particular he saw how to extend unpublished work of mine on $S^1$-bundles 
to the Seifert fibred case, 
and then worked out how to handle the $\mathbb{S}ol^3$-manifolds 
which do not fibre over $S^1$.
I would also like to thank Ray Lickorish, 
Ben Burton and Ryan Budney, for permission to include their proofs of Theorems \ref{LQ} , \ref{Kawabeta1} below, and Andy Putnam, for allowing me to make use of his account 
of the proofs of the Generalised Schoenflies Theorem (Theorem \ref{GST} below).
Finally, 
I would like to thank the referee for their suggestions for improving 
the exposition of the material in this book.

\bigskip
Jonathan Hillman

\medskip 
7 August 2024

%% file: e1.tex
\chapter{Preliminaries}

This chapter is intended as a compendium of notation and terminology 
for the algebra and topology used in the subsequent chapters.
Our general approach shall be to try to prove all assertions that are specifically 
about our topic, but to cite standard references for other supporting material.
In particular, 
we invoke the $G$-Signature Theorem in the proofs of two key results 
(Theorems \ref{Hi09-thm4.1} and \ref{CH-thm1.1}),
and we use 4-dimensional TOP surgery to provide homeomorphisms 
(in Chapters 5--7).
As these results and techniques may not be familiar to all 3-manifold topologists,
we give condensed accounts in Sections 1.8 and 1.9 below.

%We shall begin by summarizing some of the notation and terminology that we
%shall use, first for groups and pairings, and then for topology.

\section{Groups}

Let $G$ be a group.
Our commutator convention is that if $g,h\in{G}$ then $[g,h]=ghg^{-1}h^{-1}$.
We shall let $\zeta{G}$, $G'=[G,G]$ and $G^{ab}=G/G'$ denote the centre,
commutator subgroup and abelianization of $G$,  respectively.
The lower central series is defined by $\gamma_1G=G$ 
and $\gamma_{n+1}G=[G,\gamma_nG]$ for all $n\geq1$.
Similarly, the rational lower central series is given by
letting $\gamma_{1}^\mathbb{Q}G=G$ and $\gamma_{k+1}^\mathbb{Q}G$ 
be the preimage  in $G$ of the torsion subgroup of $G/[G,\gamma_k^\mathbb{Q}G]$.
Then $G/\gamma_k^\mathbb{Q}G$ is a torsion free nilpotent group, and 
$\{\gamma_k^\mathbb{Q}G\}_{k\geq1}$ is the most rapidly descending 
central series of subgroups of $G$ with this property. 

We may also let $G''=[G',G']$ and $I(G)=\gamma_{[2]}^\mathbb{Q}G$
denote the second derived subgroup and  isolator subgroup of $G$, respectively.
(Thus $G/I(G)$ is the maximal torsion-free abelian quotient of $G$.)
Each of the subgroups defined thus far is characteristic in $G$, 
and so is normal in any overgroup $H$ containing $G$ as a normal subgroup.

A group $G$ is {\it restrained\/} if it has no non-cyclic free subgroup.
It is {\it virtually solvable\/} if it has a normal subgroup of finite index 
which is solvable.
The {\it Hirsch length} $h(G)$ of such a group $G$ is the sum of the ranks of the abelian sections of a composition series for $G$.
Amenable groups and groups which are ``good" in the sense of \cite{FQ} are restrained,
but the most important examples for us are virtually solvable, 
and in fact abelian or nilpotent.

Let $F(r)$ be the free group of rank $r$.
Let $D_\infty=\mathbb{Z}/2\mathbb{Z}*\mathbb{Z}/2\mathbb{Z}$ be the infinite dihedral group.
Let $BS(1,m)$ be the Baumslag-Solitar group with presentation
$\langle{t,a}\mid{tat^{-1}=a^m}\rangle$, for $m\in\mathbb{Z}\setminus\{0\}$.
Then $BS(1,1)\cong\mathbb{Z}^2$, while $BS(1,-1)\cong\pi_1Kb$
is the Klein bottle group.
 
Our principal reference for group theory is \cite{Rob}.

\newpage
\section{Homological group theory}

The Universal Coefficient Theorems give exact sequences
\[
0\to{F}\otimes{H_2(G;\mathbb{Z})}\to{H_2(G;F)}\to{Tor(F,G^{ab})}\to0
\]
\[
\mathrm{and}\quad\quad0\to{Ext(G^{ab},F)}\to{H^2(G;F)}\to{Hom(G^{ab},F)}\to0,
\] 
for any group $G$ and field $F$, since $H_1(G;\mathbb{Z})=G^{ab}$.

A finite presentation for a group is {\it balanced\/} if it has
the same number of relations as of generators.
The group then has deficiency $\geq0$.
(We may always add trivial relators to a presentation with positive
deficiency  to get one which is balanced.)
A finitely generated group $G$ is {\it homologically balanced\/} if
$\beta_2(G;R)\leq\beta_1(G;R)$, for any coefficient ring $R$.

If $G$ is finitely presentable then $H_i(G;R)$
is finitely generated for $i\leq2$ and all simple coefficients $R$.
It follows easily that $\beta_i(G;\mathbb{Q})=\beta_i(G;\mathbb{F}_p)$,
for $i\leq2$ and almost all primes $p$. 
Hence $\beta_2(G;\mathbb{Q})=\beta_1(G;\mathbb{Q})$
if and only if $\beta_2(G;\mathbb{F}_p)=\beta_1(G;\mathbb{F}_p)$,
for almost all primes $p$. 
If $A$ is a finitely generated abelian group and $F=\mathbb{F}_p$ then 
$Tor(\mathbb{F}_p,A)\cong{_pA}=\mathrm{Ker}(p.id_A)$.
If $A$ is finite then $A/pA$ and $\mathrm{Ker}(p.id_A)$
have the same dimension.
Hence if $G$ is finite then
$\beta_2(G;\mathbb{F}_p)\geq\beta_1(G;\mathbb{F}_p)$,
for any prime $p$, and $G$ is homologically balanced if and only if 
$H_2(G;\mathbb{Z})=0$.

If $G$ has a balanced presentation then it is homologically balanced, 
and if $G$ is also finite then it must have trivial multiplicator: 
$H_2(G;\mathbb{Z})=0$.
(These assertions follow most easily from consideration of the
homology of the 2-complex associated to a balanced presentation
for the group.)
In general there may be a gap between homological 
necessary conditions and combinatorial sufficient conditions.

Restrained groups have deficiency $\leq1$ \cite{BP78}.
If $G$ is restrained and $def(G)=1$ then $G$ 
is an ascending HNN extension, 
and so the first $L^2$-Betti number $\beta^{(2)}_1(G)=0$.
Hence $c.d.G\leq2$, by \cite[Theorem 2.5]{FMGK}. 
The argument is homological, 
and so it suffices that the augmentation ideal in $\mathbb{Z}[G]$ 
have a presentation of deficiency 1 as a $\mathbb{Z}[G]$-module.

If $A$ is abelian then $H_2(A;\mathbb{Z})=A\wedge{A}$.
A finite abelian group $A$ is homologically balanced 
if and only if it is cyclic,
for if $A/pA$ is not cyclic for some prime $p$ then 
$(A/pA)\wedge(A/pA)\not=0$, and so
$\beta_2(A;\mathbb{F}_p)>\beta_1(A;\mathbb{F}_p)$,
by the Universal Coefficient exact sequences of \S1.
Finite cyclic groups clearly have balanced presentations.

Our principal references for cohomological group theory are \cite{Bie}
and \cite{Br}.

\section{Commutative algebra}

Let $\Lambda=\mathbb{Z}[\mathbb{Z}]=\mathbb{Z}[t,t^{-1}]$ 
and $R\Lambda=R\otimes_\mathbb{Z}\Lambda$, for any coefficient ring $R$.
If $F$ is a field then $F\Lambda$ is a PID.

If $R$ is a PID and $L$ is a finitely generated $R$-torsion module
then $L$ has a square presentation matrix, 
by the Structure Theorem for modules over PIDs.
The {\it order} of $L$ is the principal ideal generated by the determinant
of any such presentation matrix for $L$.
We shall let $\Delta_0(L)$ denote a generator of the  order of $L$.
(This is well-defined up to multiplication by units of $R$.)
If 
\[
0\to{A}\to{C}\to{B}\to0
\]
is a short exact sequence of such $R$-modules then 
$\Delta_0(C)=\Delta_0(A)\Delta_0(B)$ (up to units).
If $f:A\to{B}$ is a homomorphism then
$\Delta_0(\mathrm{Im}(f))=\Delta_0(\mathrm{Im}(\mathrm{Ext_R}(f)))$,
where $\mathrm{Ext_R}(f):\mathrm{Ext_R}(B,R)\to\mathrm{Ext_R}(A,R)$
is the induced homomorphism.

\section{Bilinear pairings}

A {\it linking pairing\/} on a finite abelian group $N$ is a symmetric
bilinear function $\ell:N\times N\to\mathbb{Q/Z}$ which is nonsingular
in the sense that $\mathrm{ad}(\ell):n\mapsto\ell(-,n)$ 
defines an isomorphism from $N$ to $Hom(N,\mathbb{Q/Z})$.
If $L$ is a subgroup of $N$ then $\mathrm{ad}(\ell)$ induces an isomorphism
$L^\perp=\{ t\in N\mid \ell(t,l)=0~\forall{l\in{L}}\}\cong
{Hom(N/L,\mathbb{Q/Z})}$, which is non-canonically isomorphic to $N/L$. 
It is metabolic if there is a subgroup $P$ with $P=P^\perp$,
split \cite{KK80} if also $P$ is a direct summand 
and {\it hyperbolic\/} if $N$ is the direct sum of two such subgroups.
If $\ell$ is split then $N/P$ is (non-canonically) isomorphic to $P$,
and so $N$ is a direct double.
A linking pairing $\ell$ is {\it even\/} if 
$2^{k-1}\ell(x,x)\in\mathbb{Z}$ for all $x\in{N}$ such that $2^kx=0$.
Hyperbolic pairings are even.
We shall say that $\ell$ is {\it odd} if it is not even.

Every linking pairing splits uniquely as the orthogonal sum (over primes $p$) 
of its restrictions to the $p$-primary subgroups of $N$.
If $w=\frac{p}q\in\mathbb{Q}^\times$ (where $(p,q)=1$)
let $\ell_w$ be the pairing on $\mathbb{Z}/q\mathbb{Z}$ given by 
$\ell_w(m,n)=[mnw]\in\mathbb{Q}/\mathbb{Z}$.
Then $\ell_w\cong\ell_{w'}$ if and only if $w'=n^2w$ 
for some integer $n$ with $(n,q)=1$.
When $q=p^k$ is a power of an odd prime there are just 
two isomorphism classes of such pairings.
Thus every linking pairing on an abelian group of odd order
is an orthogonal sum of pairings on cyclic groups.
However, if $q=2^k$ then $\ell_w\cong\ell_{w'}$ 
if and only if $2^kw'\equiv2^kw$ {\it mod} $(2^k,8)$.
In this case there are also indecomposable pairings $E_0^k$ and $E_1^k$ on the groups 
$(\mathbb{Z}/2^k\mathbb{Z})^2$,
with matrices
$\left(\smallmatrix0&2^{-k}\\2^{-k}&0\endsmallmatrix\right)$,
for $k\geq1$,
and 
$\left(\smallmatrix2^{1-k}&2^{-k}\\2^{-k}&2^{1-k}\endsmallmatrix\right)$,
for $k\geq2$, respectively. 
The set of all such pairings is a semigroup with respect to
orthogonal direct sum, and its structure has been completely determined
\cite{KK80,Wa64}.

It is often convenient to study linking pairings via matrices.  
Let $N\cong(\mathbb{Z}/p^k\mathbb{Z})^\rho$, with basis $e_1,\dots,e_\rho$,
and let $\ell$ be a linking pairing on $N$. 
Let be the $\rho\times\rho$ matrix with $(i,j)$ entry $p^k\ell(e_i,e_j)$,
considered as an element of $\mathbb{Z}/p^k\mathbb{Z}$.
Then $L\in\mathrm{GL}(\rho,\mathbb{Z}/p^k\mathbb{Z})$, 
since $\ell$ is non-singular.
The {\it rank} of $\ell$ is $rk(\ell)=\mathrm{dim}_{\mathbb{F}_p}N/pN=\rho$.
If $p$ is odd then a linking pairing $\ell$ 
on a free $\mathbb{Z}/p^k\mathbb{Z}$-module $N$ is determined up to isomorphism 
by $rk(\ell)$ and the image $d(\ell)$ of $\mathrm{det}(L)$ 
in $\mathbb{F}_p^\times/(\mathbb{F}_p^\times)^2=\mathbb{Z}/2\mathbb{Z}$.
(This is independent of the choice of basis for $N$.)
In particular, $\ell$ is hyperbolic if and only if
$\rho=rk(\ell)$ is even and $d(\ell)=[(-1)^{\frac{\rho}2}]$.
If $p=2$ and $k\geq3$ then $\ell$ is determined by the image of $L$ in $GL(\rho,\mathbb{Z}/8\mathbb{Z})$; if moreover $\ell$ is even and $k\geq2$ then $\rho$ is even and $\ell$ is determined by the image of $L$ in
$GL(\rho,\mathbb{Z}/4\mathbb{Z})$ \cite{De05,KK80,Wa64}.

\section{Algebraic topology}

If $W$ is a cell-complex its universal cover $\widetilde{W}$ 
has an induced cellular structure.
Let $\Gamma=\mathbb{Z}[\pi_1W]$, 
and let $C_*=C_*(W;\mathbb{Z}[\pi_1W])$ be 
the chain complex of $\widetilde{W}$, 
considered as a complex of free left $\Gamma$-modules.
Then $H_i(W;\Gamma)=H_i(C_*)$ is $H_i(\widetilde{W})$, 
with the natural $\Gamma$-module structure,
for all $i$.
The {\it equivariant cohomology\/} of $\widetilde{W}$ is defined 
in terms of the  cochain complex $C^*=Hom_\Gamma(C_*,\Gamma)$,
which is naturally a complex of right modules.
Let $\overline{C}^q$ be the left $\Gamma$-module obtained via 
the canonical anti-involution of $\Gamma$,
defined by $g\mapsto{g}^{-1}$ for all $g\in\pi_1W$,
and let $H^j(W;\Gamma)=H^j(\overline{C}^*)$.
We use similar notation for pairs of spaces.
If $W$ is a 4-manifold with boundary then equivariant 
Poincar\'e-Lefshetz duality gives isomorphisms
$H_i(W;\Gamma)\cong{H^{4-i}}(W,\partial{W};\Gamma)$
and $H^j(W;\Gamma)\cong{H_{4-j}}(W,\partial{W};\Gamma)$, for all $i,j\leq4$.

The {\it Wang sequences\/} for homology and cohomology (with coefficients $R$)
associated to an infinite cyclic covering are the long exact sequences corresponding 
to the coefficient module sequence
\[
0\to{R}\Lambda\to{R}\Lambda\to{R}\to0.
\]
In Chapter 8 we shall use such sequences in the context of group (co)homology,
when a group $G$ has a normal subgroup $K$ such that $G/K\cong\mathbb{Z}$.

Let $X$ and $Y$ be connected cell complexes, 
and let $R$ be a $\mathbb{Z}[\pi_1X]$-module.
A map $f:Y\to{X}$ is a {\it $R$-homology isomorphism\/} if $\pi_1f$ is an epimorphism and $f$ induces isomorphisms on homology with local coefficients (induced from) $R$.
A cobordism $W$ with $\partial{W}=M_1\sqcup{M_2}$ is an 
{\it $H$-cobordism over $R$} (or an $R$-homology cobordism)
if the inclusions of the $M_i$ into $W$ are $R$-homology isomorphisms, 
for $i=1,2$.

Let $\mathbb{F} $ be a prime field ($\mathbb{Q}$ or $\mathbb{F}_p$,
where $p$ is a prime).
If $G$ is a group the Massey product structures for classes in $H^1(G;\mathbb{F})$ are closely related to the rational and 
$p$-lower central series of $G$. 
We shall only use the following special case.
Let $a,b,c$ be classes in $H^1(G;R)$ represented by 1-cycles $\alpha$, 
$\beta$ and $\gamma$.
If $a\cup{b}=b\cup{c}=0$ then the Massey triple product$\langle{a,b,c}\rangle$
in $H^2(G;R)$ is the class represented by the 2-cycle $e_{\alpha,\beta}\gamma+\alpha{e_{\beta,\gamma}}$,
where $e_{\alpha,\beta}$ and $e_{\beta,\gamma}$ are 1-chains
such that $\partial^1(e_{\alpha,\beta})=\alpha\beta$ and
$\partial^1(e_{\beta,\gamma})=\beta\gamma$ as 2-cocycles. 
(See \cite{Ma68,Dw75} and also Chapter 12 of \cite{AIL}).

We shall not write coefficients $\mathbb{Z}$ for integral homology,
cohomology or Betti numbers.

\section{Manifolds}

Let $T$ be the torus, $T_g=\#^gT$ the closed orientable surface 
of genus $g\geq0$, $Kb$ the Klein bottle and $\#^c\mathbb{RP}^2$ 
the closed non-orientable surface with $c\geq1$ cross-caps.
In Chapter 6 we shall also use the notation $P_\ell=S^1\cup_\ell{e^2}$
for the pseudo-projective plane with fundamental group $\mathbb{Z}/\ell\mathbb{Z}$.
If $B$ is a 2-orbifold then $|B|$ is the underlying surface.
A {\it $2$-handlebody\/} is a (smooth) 4-manifold constructed 
by adding 1- and 2-handles to $D^4$.

If $M$ is an $n$-manifold $M_o=\overline{M\setminus{D^n}}$ is the bounded manifold obtained by deleting a small open $n$-disc.

An embedding $j$ of an $m$-manifold $M$ in an $(m+k)$-manifold $N$ is 
{\it locally flat\/} if each point $x\in{M}$ has a neighbourhood 
$U\subset{M}$ such that $j(U)$ has a product neighbourhood $V\cong{U}\times(-1,1)$ in $N$. 
(Note that $PL$ embeddings of 3-manifolds into 4-manifolds are 
alwaysy locally flat, in the stronger $PL$ sense, 
since the $PL$ Schoenflies Theorem holds for $PL$ embeddings of $S^{n-1}$ in $S^n$
when $n\leq3$. We shall not need this observation below.)

An $n$-sphere $S^n$ is a {\it twisted double} if $S^n\cong{W\cup_\theta{W}}$, 
where $W$ is a compact $n$-manifold with connected boundary and 
$\theta$ is a self-homeomorphism of $\partial{W}$.

If $L$ is a link in $S^3$ with open regular neighbourhood $n(L)$
then $X(L)=S^3\setminus{n(L)}$ is the link exterior,
$\pi{L}=\pi_1X(L)$ is the link group,
and  $M(L)$ is the 3-manifold obtained by 0-framed surgery on $L$.
The order of $H_1(M(K);\Lambda)$ is generated by the Alexander polynomial
of $K$.

An $m$-component link $L$ in $S^3$ is {\it slice\/} if it bounds a set 
of $m$ disjoint 2-discs properly embedded in $D^4$.
(It is smoothly slice if the discs are also smoothly embedded.)
It is a {\it ribbon link\/} if there is a map $R:m{D^2}\to{S^3}$ 
which is locally an embedding,
and whose only singularities are transverse double points, 
the double point sets being a disjoint union of intervals, 
and such that $R|_{m\partial{D^2}}$ is an embedding with image $L$.
The map $R$ may be homotoped $rel~\partial$ to a proper embedding in $D^4$,
and so every ribbon link is slice.

A knot $K$ in $S^3$ is {\it algebraically slice\/} if its Blanchfield pairing 
is neutral.
It is {\it homotopically ribbon\/} if it bounds a 2-disc $D\subset{D^4}$ 
with an open product neighbourhood $n(D)$ such that the inclusion of
$M(K)=\partial{D^4\setminus{n(D)}}$ into $D^4\setminus{n(D)}$
induces an epimorphism on fundamental groups \cite{CG}.
Ribbon knots are homotopically ribbon.
It is {\it doubly slice\/} if it is the transverse intersection of the equatorial $S^3$
in $S^4$ with an unknotted embedding of $S^2$ in $S^4$.

We shall say that $M$ is a {\it homology handle\/} if
$H_1(M)\cong\mathbb{Z}$, 
so $M$ has the homology of $S^2\times{S^1}$.

\begin{lemma}
\label{homhandlemap}
Let $M$ is a homology handle.
Then there is a $\mathbb{Z}$-homology isomorphism $h:M\to{S^2\times{S^1}}$.
\end{lemma}

\begin{proof}
Since $K(\mathbb{Z},1)\simeq{S^1}$ and $K(\mathbb{Z},2)=\mathbb{CP}^\infty$ 
may be constructed by adding cells of dimension $\geq4$ to $S^2$,
there are maps $f:M\to{S^1}$ and $g:M\to{S^2}$ 
such that $f^*\iota_1$ generates $H^1(M)$ and 
$g^*\iota_2$ generates $H^2(M)$. 
Let $h=(g,f):M\to{S^2\times{S^1}}$. 
Then $h$ induces an isomorphism of cohomology rings.
In particular, it is a $\mathbb{Z}$-homology isomorphism.
\end{proof}

Our principal references for knots and links are the books \cite{Rol} and \cite{AIL}, 
and for the Kirby calculus \cite{GS}.
The notation for knots and links is as in the tables in \cite{Rol},
augmented by the symbol $U$ for the unknot. 

\section{Duality pairings for 3-manifolds}

Let $M$ be a closed connected orientable 3-manifold 
with fundamental group $\pi$,
and let $\beta=\beta_1(M;\mathbb{Q})$.
Let $\tau_M$ be the torsion subgroup of $H_1(M)$.
Then Poincar\'e duality determines 
a linking pairing $\ell_M:\tau_M\times\tau_M\to\mathbb{Q}/\mathbb{Z}$.
This is symmetric, bilinear, and nonsingular in the sense that 
the adjoint function $\widetilde{\ell_M}:m\mapsto\ell(-,m)$ defines
an isomorphism from $\tau_M$ to $Hom(\tau_M,\mathbb{Q}/\mathbb{Z})$.
(Open 3-manifolds also have well-defined torsion linking pairings, 
but these may be singular.)

The  linking pairing $\ell_M$ may be 
described as follows.
Let $w$, $z$ be disjoint 1-cycles representing elements of $\tau_M$ and 
suppose that $mz=\partial C$ for some 2-chain $C$ which is transverse 
to $w$ and some nonzero $m\in\mathbb{Z}$.
Then $\ell_M ([w],[z])=(w\bullet{C})/m\in\mathbb{Q/Z}$.
There is a dual formulation, in terms of cohomology.
Let $\beta_{\mathbb{Q}/\mathbb{Z}}:H^1(M;\mathbb{Q}/\mathbb{Z})\to
{H^2(M)}$ be the Bockstein homomorphism associated
with the coefficient sequence
\[0\to\mathbb{Z}\to\mathbb{Q}\to\mathbb{Q}/\mathbb{Z}\to0,\]
and let $D:H_1(M)\to{H^2(M)}$ 
be the Poincar\'e duality isomorphism.
Then $\ell_M$ may be given by the equation
\[\ell_M(w,z)=(D(w)\cup\beta_{\mathbb{Q}/\mathbb{Z}}^{-1}D(z))([M])
\in\mathbb{Q}/\mathbb{Z}.\]

There are analogous pairings on covering spaces of $M$.
In particular, if $\phi:\pi_1M\to\mathbb{Z}$ is an epimorphism with 
associated covering space $M_\phi$, and $t$ is a generator $t$ 
for the covering group then the homology modules $H_1(M_\phi;R)$
are finitely generated $R\Lambda$-modules.
Let $H_i(M;\mathbb{Q}\Lambda)=H_i(M_\phi;\mathbb{Q})$
and $H_i(M;\mathbb{Q}(t))=\mathbb{Q}(t)\otimes_\Lambda{H_i(M;\mathbb{Q}\Lambda)}$.
Let $B$ be the $\mathbb{Q}\Lambda$-torsion submodule of 
$H_1(M;\mathbb{Q}\Lambda)$.
Then equivariant Poincar\'e duality and the universal coefficient theorem 
together define a pairing
$b:B\times{B}\to\mathbb{Q}(t)/\mathbb{Q}[t,t^{-1}]$,
which is called the {\it Blanchfield pairing\/} associated to the covering.
This pairing is nonsingular and hermitian 
with respect to the involution sending $t$ to $t^{-1}$.
(See Chapter 2 of \cite{AIL}.)

When $\phi$ corresponds to a fibre bundle projection from $M$ to $S^1$
with fibre $F$ a closed surface $b_\phi$ is equivalent 
to the isometric structure given by the
intersection pairing $I_F$ on $H_1(F;\mathbb{Q})$, 
together with the isometric action of $\mathbb{Z}$.
(See Appendix A of \cite{Lit84}.)
Such a pairing is {\it neutral\/} if the underlying 
$\mathbb{Q}\Lambda$-torsion module has a submodule 
which is its own annihilator with respect to the pairing.

Two such pairings are {\it Witt-equivalent\/} if they become 
isomorphic after addition of suitable neutral pairings.
The {\it Witt group\/} of isometric structures on finite dimensional
$\mathbb{Q}$-vector spaces is the set
$W_+(\mathbb{Q}(t)/\mathbb{Q}\Lambda)$ of Witt equivalence classes
of such pairings, with the addition induced by direct sum of pairings.
(See \cite{Ne}.)

\section{Surgery}

We shall give a brief, very utilitarian ``black box" outline of the surgery exact sequence 
as it is used in Chapters 5--7 below. 
For more details see either the original source \cite{Wall} 
or the more recent \cite{Ran}.
The  extension of surgery methods to the 4-dimensional TOP case 
is covered in \cite{FQ} and \cite{BKKPR}. 
See also \cite{KT02}.

Let $W$ be a compact, connected orientable $n$-manifold, where $n\geq5$.
The {\it structure set} $\mathcal{S}_{TOP}(W,\partial{W})$ 
is the set of equivalence classes of simple homotopy equivalences 
$f:(P,\partial{P})\to(W,\partial{W})$, 
where $P$ is a compact $n$-manifold and $f|_{\partial{P}}$ is a homeomorphism,
and where two such maps $f_1$ and $f_2$ are equivalent if there is a 
homeomorphism $h:P_2\to{P_1}$ such that $f_2\simeq{f_1\circ{h}}$.
This is a pointed set with distinguished element $id_W$.
It is the central object of interest for surgery, and sits in a sequence
\begin{equation*}
\begin{CD}
L_{n+1}(\mathbb{Z}[G])@>\omega>>\mathcal{S}_{TOP}(W,\partial{W})
@>\nu_W>>\mathcal{N}(W,\partial{W})@>\sigma_n(W,\partial{W})>>L_n(\mathbb{Z}[G])
\end{CD}
\end{equation*}
where the {\it surgery obstruction groups} $L_{n+1}(\mathbb{Z}[G])$ and $L_n(\mathbb{Z}[G])$ are algebraically defined,
$L_{n+1}(\mathbb{Z}[G])$ acts on $\mathcal{S}_{TOP}(W,\partial{W})$
via $\omega$,
the set of {\it normal invariants} $\mathcal{N}(W,\partial{W})=[W,\partial{W};G/TOP,*]$ 
is a topologically defined abelian group and 
$\sigma_n=\sigma_n(W,\partial{W})$ is a homomorphism.
This sequence is exact in the sense that two elements of the structure set 
have the same image in $\mathcal{N}(W,\partial{W})$ if and only if 
they are in the same orbit of $\omega$, 
and the kernel of $\sigma_n$ is the image of the normal invariant map $\nu_W$.
(We use the ring-theoretic notation of Ranicki rather than the original 
group-theoretic notation of Wall for the surgery obstruction groups,
writing $L_n(\mathbb{Z}[G])$ rather than $L_n(G)$.
In particular, the groups $L_n(\mathbb{Z})$ are the obstruction groups for 1-connected surgery.
In the cases we consider all homotopy equivalences are simple,
and so we have written $L_n$ rather than $L_n^s$.)

In dimension 4 we must modify the definition.
The $s$-cobordism structure set $\mathcal{S}^s_{TOP}(W,\partial{W})$
is the set of equivalence classes of simple homotopy equivalences $f$ as above, 
where simple homotopy equivalences $f_1$ and $f_2$ are equivalent 
if there is an $s$-cobordism $Q$ {\it rel} $\partial$ between $P_1$ and $P_2$
and a map $F:(Q,\partial{Q})\to(W,\partial{W})$ which extends $f_1\sqcup{f_2}$.
There is again a sequence of pointed sets 
\[
\mathcal{S}^s_{TOP}(W,\partial{W})\to\mathcal{N}(W,\partial{W})\to{L_4(\mathbb{Z}[G])}.
\]
We may identify $\mathcal{N}(W,\partial{W})$ with
$H^2(W,\partial{W};\mathbb{F}_2)\oplus\mathbb{Z}$,
and $\sigma_4$ is trivial on the image of $\nu_W$ and
injective on the $\mathbb{Z}$ summand.
However, in general we do not know whether a 4-dimensional normal map 
with trivial surgery obstruction must be normally cobordant to 
a homotopy equivalence.
In our applications we shall always have a homotopy equivalence in hand.
A more serious problem is that it is not clear how to define the action $\omega$ in general.
{\it Ad hoc\/} arguments apply in some of the cases of interest to us.
(See \cite[Chapter 6.2]{FMGK}.)

If $\pi_1W$ is virtually solvable (or, more generally, 
is ``good" in the sense of \cite{FQ})
then the $s$-cobordism theorem holds, and so $\mathcal{S}^s_{TOP}(W,\partial{W})=\mathcal{S}_{TOP}(W,\partial{W})$, 
and we again have a 4-term surgery exact sequence as above.
In this case the order of the structure set is bounded 
by the order of $H^2(W,\partial{W};\mathbb{F}_2)\cong{H_2(W;\mathbb{F}_2)}$.

In dimension 3 we need a more drastic modification.
The structure set is now defined in terms of $\mathbb{Z}[G]$-homology equivalences, 
and the equivalence relation is coarsened to $\mathbb{Z}[G]$-homology cobordism \cite{KT02}. 
(We use this in \S5.5 below.)

Two obvious issues are the determination of the surgery obstruction groups $L_n(\mathbb{Z}[\pi])$ and the surgery obstruction homomorphisms $\sigma_n$.
The groups $L_n(\mathbb{Z}[\pi])$ are periodic in $n$, with period 4.
When $\pi$ is finite the calculation of these groups involves representation theory and algebraic number theory (as in \cite{Wa76}),
but when there is a finite $K(\pi,1)$ complex it is expected that
this groups are largely determined by the homology of $\pi$ with coefficients in
$L_0(\mathbb{Z})=\mathbb{Z}$ and $L_2(\mathbb{Z})=\mathbb{Z}/2\mathbb{Z}$.
(This is roughly the content of the Novikov and Farrell-Jones Conjectures.
See \cite{KL}.)
The fundamental groups that we shall consider in the second half of this book 
are all of the latter type,  
and the main difficulty that we shall meet in attempting to apply surgery
is in identifying the homotopy types of pairs $(W,M)$.

\section{The $G$-Signature Theorem}

The Index Theorem is one of the major achievements of the past century.
Like many great theorems,  there are a number of proofs, 
emphasizing different aspects and using a variety of techniques.
Fortunately we shall only need two very special cases, 
for a finite cyclic group $G$ acting on a closed orientable surface 
and for involutions of a closed orientable 4-manifold,
with ``good" fixed point set.
The $G$-Signature Theorem in these cases relates global invariants
based on signatures of bilinear pairings derived from cup-product 
to properties of the action of $G$ on the normal bundle of the fixed point set.

The intersection form $(~,~)$ on the middle dimensional homology of a $2k$-manifold $M$ induces a hermitean form $\phi$ on $H=H_k(M;\mathbb{C})=\mathbb{C}\otimes{H_k(M;\mathbb{Z})}$ by the formulae $\phi(\alpha{x},\beta{y})=\alpha\overline\beta(x,y)$ for $k$ even and  $\phi(\alpha{x},\beta{y})=i\alpha\overline\beta(x,y)$ for $k$ odd.
If $G$ is a finite group acting orientably on $M$ then there is a $G$-equivariant 
orthogonal decomposition $H=H^+\oplus{H^-}\oplus{H^0}$, 
where $\phi$ is positive definite on $H^+$, 
negative definite on $H^-$ and zero on $H^0$.
If $g\in{G}$ then we let
\[
\mathrm{sign}(g,M)=tr(g_*|_{H^+}-tr(g_*|_{H^-}).
\]
Clearly $|sign(g,M)|\leq\beta_k(M;\mathbb{C})$.

Suppose first that $G=\mathbb{Z}/\sigma\mathbb{Z}$ acts on a closed orientable 
surface $M$, with finite fixed point set $F$.
Then the $G$-Index Theorem gives
\[
\mathrm{sign}(g,M)=-i\Sigma_{P\in{F}}\cot(\theta_P),
\]
where $g$ rotates a disc neighbourhood of the fixed point $P$ through 
$\theta_P$ radians.

Suppose now that $g$ is an involution of a closed connected 4-manifold $M$
such that the fixed point set $F$ is a finite disjoint of surfaces, 
and $g$ acts linearly on neighbourhoods of points in $F$.
The $G$-Signature Theorem then gives
\[
\mathrm{sign}(g,M)=e(F),
\]
where $e(F)$ is the normal Euler number, 
which is just the Euler number of $\partial{N(F)}$,
considered as a circle bundle over $F$.

The $G$-Signature Theorem is due to Atiyah, Bott and Singer \cite{AB68,AS68}.
An account of the theorem for $G$ finite 
and $M$ of dimension 2 or 4 which is adapted to the needs and backgrounds 
of low-dimensional topologists may be found in \cite{Go86}.
The manifold and the action are assumed there to be smooth,
but smoothness is required for the proofs only in the neighbourhood of the fixed point set.

\section{Other tools}

We mention briefly several other notions which are perhaps not yet 
part of the standard background of geometric topologists, 
but which appear in minor roles in this book.
The Andrews-Curtis moves are used in Theorem \ref{LQ},
but this result is not used elsewhere in the book. 
See \cite{HAMS}.
In Theorem 6.9 we refer to the Bass Conjectures \cite{Ba76},
in order to show that certain projective modules are 0.
$L^2$-homology is invoked in \S6.9 and \S8.2.
We only need to know the $L^2$-Euler formula $\chi(X)=\Sigma(-1)^q\beta_q^{(2)}(X)$ and that $\beta_1^{(2)}(G)=0$
if $G$ is virtually solvable or has a finitely generated infinite normal subgroup of infinite index. See \cite{Lue}.
Profinite completion is used in Theorem \ref{nilp emb}.
See \cite{Ser}.

%% file: e2.tex
\chapter{Invariants and Constructions}

The key algebraic invariants for our purposes are the fundamental groups 
of the spaces involved, the torsion linking pairing on the 3-manifold and the
Euler characteristic of the complementary regions.
In this chapter we shall review the basic constraints on these invariants and
describe the construction by 0-framed surgery on bipartitedly slice links,
from which many of our examples derive.
We also give an application of this construction by W. B. R. Lickorish,
who showed that any two groups with balanced finite presentations 
and isomorphic abelianization could be realized as the fundamental groups
of complementary regions of some embedding of a 3-manifold in $S^4$.
In the final section we introduce 2-knot surgery.

\section{Codimension-1 embeddings}

An embedding of one topological space $A$ into another $B$ is an injective continuous map $f:A\to{B}$ which induces a homeomorphism from $A$ onto its image $f(A)$.
This definition is too broad for the purposes of geometric topology, 
as it allows for pathologies such as Alexander's horned sphere \cite{Al24}.
We shall impose a condition that holds for all smooth embeddings.

\begin{defn}
A {\it locally flat embedding\/} of an $n$-manifold $M$ in an $(n+k)$-manifold $N$
is an injective continuous map $j:M\to{N}$ such that for each $m\in{M}$ there is a neighbourhood $U$ of $j(m)$ in $N$ and a homeomorphism $h:U\to\mathbb{R}^{n+k}$
such that $h(U\cap{j(M)})=\mathbb{R}^n\times\{O_k\}$, 
where $O_k$ is the origin in $\mathbb{R}^k$. 
Two embeddings $j$ and $\tilde{j}$ of $M$ in $N$ are {\it equivalent\/} if there are 
self-homeomorphisms $\phi$ of $M$ and $\psi$ of $N$ such that $\psi{j}=\tilde{j}\phi$.
\end{defn}

Our focus shall be primarily on embeddings of closed 3-manifolds in $S^4$,
but we shall give a very brief overview of codimension-1 embeddings (i.e., $k=1$)
in spheres of other dimensions.

If a closed connected $n$-manifold $M$ embeds in $S^{n+1}$ then $M$ 
must be orientable and the complement must have two components.
These basic observations are consequences of Alexander duality, 
and hold even if the embedding is not locally flat.
In the smooth (or locally flat) case,  there is a more geometric argument, 
since it is then clear that there are arcs transverse to $M$ in $S^{n+1}$.
If $S^{n+1}\setminus{M}$ were connected there would be a simple closed curve 
in $S^{n+1}$ intersecting $M$ transversally in one point,
contradicting the fact that $S^{n+1}$ is simply connected.
It is easy to see that $S^{n+1}\setminus{M}$ has at most two components,
and so it has exactly two components.
Thus $M$ has a product neighbourhood $M\times(-1,1)$, and so must be orientable.
We shall call the closures of the components of $S^{n+1}\setminus{M}$ 
the {\it complementary regions\/} of the embedding, for brevity.

The only connected closed 1-manifold is the circle $S^1$, 
and embeddings of $S^1$ in $S^2$ are all equivalent, 
by the classical Schoenflies theorem. 
(The proof actually shows that every embedding in the broader sense is equivalent to the equatorial embedding. See \cite{Si05} for a modern account.)
The study of embeddings of more than one copy of $S^1$ 
reduces to a simple combinatorial analysis of the lattice of ovals
bounded by a finite family of disjoint simple closed curves in the plane.

Alexander showed that PL embeddings of $S^2$ in $S^3$ are standard,
while locally flat embeddings of $S^n$ into $S^{n+1}$ are all equivalent
(for each $n\geq1$),
by the generalized Schoenflies Theorem of Brown and Mazur.
(This is true also for smooth embeddings of $S^n$ into $S^{n+1}$,
if $n\geq4$, but remains an open problem when $n=3$.)

The following result of M. Brown is a basic consequence of the existence of normal bundles in differential topology, and plays a role in the generalized Schoenflies Theorem.
The present proof is due to R. Connelly \cite{Co71}.
An immediate consequence of this theorem is that every closed locally flat hypersurface $M\subset{S^n}$ has a product neighbourhood $M\times(-1,1)$.

\begin{theorem}[Brown]
\label{Brown}
Let $M$ be a compact manifold with non-empty boundary.
Then $\partial{M}$ has a collared neighbourhood $V\cong\partial{M}\times(0,1]$.
\end{theorem}

\begin{proof}
Let $N=M\cup_{\partial{M}}\partial{M}\times\mathbb{R}_+$,
where we identify $x\in\partial{M}$ with $(x,0)$ in 
$\partial{M}\times\mathbb{R}_+$.
We shall construct a homeomorphism of $N$ which pushes the closed subspace $M$ out onto $M\cup_{\partial{M}}\partial{M}\times[0,1]$.
Since $\partial{M}$ is compact it has a finite open cover $\mathcal{U}=\{U_1,\dots,U_n\}$
by subsets with product neighbourhoods.
Thus there are embeddings $h_i:U_i\times[-1,0]\to{M}$ such that
$h_i(u_i,0)=u_i$ for all $u_i\in{U_i}$ and $i\leq{n}$.
These embeddings  clearly extend to embeddings $H_i:U_i\times[-1,\infty)\to{N}$.

Since $\partial{M}$ is a compact Hausdorff space there is 
a continuous partition of unity $\{\lambda_i\}_{i\leq{n}}$ subordinate to $\mathcal{U}$.
For each $0\leq{a}\leq1$ let $f_a$ be the self-homeomorphism of $[-1,\infty)$ given by
$f_a(t)=t$ if $-1\leq{t}\leq-\frac12$, $f_1(t)=(1-2a)t+a$ if
$-\frac12\leq{t}\leq0$ and $f_a(t)=t+a$ if $t\geq0$.
Let $P_i:N\to{N}$ be the self-homeomorphism which maps $(u_i,t)$ to
$(u_i,f_{\lambda_i(u_i)}(t))$, 
and which is the identity outside $U_i\times[-1,\infty)$,
for all $u_i\in{U_i}$ and $i\leq{n}$.
Then the composite $P=P_1\circ{P_2}\circ\dots\circ{P_n}$ is a homeomorphism 
such that $P(M)=M\cup_{\partial{M}}\partial{M}\times[0,1]$.
Clearly $V=P^{-1}(\partial{M}\times(0,1])$ is as required.
\end{proof} 

When $n=2$ Alexander showed also that if $M$ is the torus $T$ then one 
of the complementary regions is $S^1\times{D^2}$, 
and so the classification of embeddings of $T$ reduces to knot theory.
If $\chi(M)<0$ then we can use pairwise connected sums to construct
embeddings which are ``knotted on both sides",
but there are embeddings of a more complicated nature. 
See \cite{BPW19}.

Alexander's torus theorem has an analogue due to I. R.  Aitchison:
if $M=S^2\times{S^1}$ then one of the complementary regions 
is $S^2\times{D^2}$, 
and so the classification of embeddings of $S^2\times{S^1}$ 
reduces to 2-knot theory. (See Theorem \ref{aitch} below.)

We shall assume henceforth that $n=3$. 
Unless otherwise stated,  
all 3-manifolds considered here shall be closed, connected and orientable,
and  we shall usually write ``homology sphere" instead of ``integral homology 3-sphere".

\section{Embedding 3-manifolds}

A 3-manifold $M$ embeds in $\mathbb{R}^4$ if and only if it embeds in $S^4$.
Since our arguments are largely homological it is often more natural to assume 
that $M$ embeds in a homology 4-sphere $\Sigma$.
However, we shall focus on the standard case.
An embedding $j:M\to{S^4}$ is {\it smoothable\/} if it is smooth with respect
to some smooth structure on $S^4$, 
equivalently, if each complementary region is a 4-dimensional handlebody
(i.e., may be constructed by attaching 1-, 2- and 3-handles to $D^4$).
Although the embeddings that we shall construct are 
usually smooth embeddings in the standard 4-sphere,
we wish to apply 4-dimensional topological surgery,
and so henceforth {\it embedding\/} shall mean TOP locally flat
embedding, unless otherwise qualified.
A {\it Poincar\'e embedding\/} of  $M$ into a homology $4$-sphere $\Sigma$
is a homotopy equivalence $X\cup_MY\simeq\Sigma$,
where $(X,M)$ and $(Y,M)$ are $PD_4$-pairs.

Let $j:M\to{S^4}$ be an embedding,  
with complementary regions $X$ and $Y$,
and let $j_X$ and $j_Y$ be the inclusions of $M$ into $X$ and $Y$, 
respectively.
Clearly $X$ and $Y$ are compact 4-manifolds,
with boundary $\partial{X}=\partial{Y}=M$.
We may assume $X$ and $Y$ are chosen so that $\chi(X)\leq\chi(Y)$.
(On occasion, 
we shall use $W$ for either of the complementary regions when the size of 
$\chi(W)$ is not relevant.)
Let $\pi=\pi_1M$, $\pi_X=\pi_1X$ and $\pi_Y=\pi_1Y$,
and let $j_{X*}=\pi_1j_X$ and  $j_{Y*}=\pi_1j_Y$.
One of the subsidiary themes of this book is the interaction between
$\chi(X)$ and $\pi_1X$.

The symbols $X$, $Y$, $\pi$, $\pi_X$ and $\pi_Y$ shall have the above interpretations 
throughout the book.  
We shall also abbreviate $\beta_1(M)$ as $\beta$, 
when the meaning is clear from the context.

We shall summarize the basic properties of the homology and cohomology of
the complementary regions in the next lemma.
One simple but important observation is that the natural homomorphism
$H_2(X)\to{H_2(X,M)}$ is 0, 
since it factors through $H_2(S^4)\to{H_2(S^4,Y)}$ and an excision isomorphism,
and similarly for $H_2(Y)\to{H_2(Y,M)}$.
Equivalently, the intersection pairings are trivial on $H_2(X)$ and $H_2(Y)$.
(See Theorem \ref{homhandle} below for one use of this observation.)

\begin{lemma}
\label{Hi17-lem2.1}
Suppose $M$ embeds in $S^4$,  with complementary regions $X$ and $Y$.
Then  
\begin{enumerate}
\item$\chi(X)+\chi(Y)=2$;
\item$H_i(M;R)\cong{H_i(X;R)}\oplus{H_i(Y;R)}$, for $i=1,2$,
while $H_i(X;R)=0$ for $i>2$,  
for any simple coefficients $R$ (and similarly for cohomology);
\item$H^1(X;R)\cong{H_2(Y;R)}$ and $H^2(X;R)\cong{H_1(Y;R)}$;
\item$\beta=\beta_1(M)=\beta_1(X)+\beta_2(X)$; 
\item$1-\beta\leq\chi(X)\leq\chi(Y)\leq1+\beta$ and $\chi(X)\equiv\chi(Y)\equiv1+\beta$ mod $(2)$; and
\item{}$\pi_X$ is homologically balanced.
\end{enumerate}
\end{lemma}

\begin{proof}
The first assertion is clear, 
since $2=\chi(S^4)=\chi(X)+\chi(Y)-\chi(M)$ and $\chi(M)=0$.
The Mayer-Vietoris sequence for $S^4=X\cup_MY$ with coefficients $R$ 
gives isomorphisms 
\[
H_i(M;R)\cong{H_i(X;R)\oplus{H_i(Y;R)}},
\]
for $i=1,2$, while $H_3(M;R)\cong{R}$ and $H_j(X;R)=H_j(Y;R)=0$ for $j>2$.
Moreover, $H_2(X;R)\cong{H^1(Y;R)}$,  by Poincar\'e-Lefshetz duality
and excision (or by Alexander duality).
Then $\beta=\beta_1(X)+\beta_2(X)$, 
so $\chi(X)=1+\beta-2\beta_1(X)$, where $0\leq\beta_1(X)\leq\beta$.

Since $H_2(\pi_X;F$ is a quotient of $H_2(X;F)$ for any field $F$,
by Hopf's Theorem, the final assertion follows from (2) and (3).
\end{proof}

In particular,
 $H^1(M;R)\cong{H^1}(X;R)\oplus{H^1(Y;R)}$
has a basis consisting of epimorphisms which extend on one side or the other.
If $\beta=0$ then $\chi(X)=\chi(Y)=1$, 
while if $\beta=1$ then $\chi(X)=0$ and $\chi(Y)=2$.

The cohomology ring $H^*(M)$ is determined by the
3-fold product 
\[
\mu_M:\wedge^3H^1(M)\to{H^3(M)}
\]
and Poincar\'e duality.
If we identify $H^3(M)$ with $\mathbb{Z}$ we may view $\mu_M$
as an element of $\wedge^3(H_1(M)/\tau_M)$.
Every finitely generated free abelian group $H$ and linear homomorphism
$\mu:\wedge^3H\to\mathbb{Z}$ is realized by some closed orientable 3-manifold \cite{Su75}.
(If $\beta\leq2$ then $\wedge^3\mathbb{Z}^\beta=0$, and so $\mu_M=0$.)

\begin{lemma}
\label{Hi17-lem2.2}
The cup product $3$-form $\mu_M$ is $0$ if and only if all cup products 
of classes in $H^1(M)$ are $0$.
Its restrictions to each of $\wedge^3H^1(X)$ 
and $\wedge^3H^1(Y)$ are $0$.
\end{lemma}

\begin{proof}
Poincar\'e duality implies immediately that $\mu_M=0$ if and only if 
all cup products from $\wedge^2H^1(M)$ to $H^2(M)$ are 0.
The second assertion is clear,  since $H^3(X)=H^3(Y)=0$.
\end{proof}

See \cite{Le83} for the parallel case of doubly sliced knots.

If $\mu_M\not=0$ then $H^1(X)$ and $H^1(Y)$
are non-trivial proper summands of $H^1(M)$.
However, if $\mu_M=0$ this lemma places no condition on these summands.

The 3-form $\mu_M$ is 0 if and only if
$\pi/\gamma_3^\mathbb{Q}\pi\cong{F(\beta)/\gamma_3^\mathbb{Q}F(\beta)}$ 
\cite{Su75}.
However, this is a rather weak condition.
The next lemma gives a stronger result.

\begin{lemma}
\label{Hi17-lem4.1}
If $H_1(Y)=0$ then $\pi/\gamma_k\pi\cong{F(\beta)/\gamma_kF(\beta)}$, 
for all $k\geq1$.
\end{lemma}

\begin{proof}
If $H_1(Y)=0$ then $H_2(X)=0$,
and $T$ must be 0, by the non-degeneracy of $\ell_M$, 
so $H_1(M)\cong{H_1(X)}\cong\mathbb{Z}^\beta$.
Let $f:\vee^\beta{S^1}\to{X}$ be any map such that $H_1(f)$ 
is an isomorphism.
Then $j_X$ and $f$ induce isomorphisms on all quotients of the 
lower central series, by Stallings' Theorem \cite{St65},
and so $\pi/\gamma_k\pi\cong{F(\beta)/\gamma_kF(\beta)}$, 
for all $k\geq1$.
\end{proof}

If $M$ is the result of surgery on a $\beta$-component slice link $L$
then it has an embedding with a 1-connected complementary region,
and so this lemma applies.
However there are such examples for which $\pi$ does not have any 
epimorphisms onto $F(\beta)$.
(See \cite[Figure 8.1]{AIL}.)

There are parallel results for the rational lower central series 
and the $p$-central series, for primes $p$,
with coefficients $\mathbb{Q}$ and $\mathbb{F}_p$, respectively.
In particular, if $\beta_1(Y)=0$ then 
$\pi/\gamma_k^\mathbb{Q}\pi\cong{F(\beta)/\gamma_k^\mathbb{Q}F(\beta)}$, 
for all $k\geq1$.

The diagram of fundamental groups determined by the inclusions of $M$ into $X$ and $Y$
and of $X$ and $Y$ into $S^4$ is a push-out diagram, by Van Kampen's Theorem.
\begin{equation*}
\begin{CD}
\pi@>j_{X*}>>\pi_X\\
@Vj_{Y*}VV @VVV\\
\pi_Y@>>>1.
\end{CD}
\end{equation*}
The use of this observation in the proof of Aitchison's Theorem 
(Theorem \ref{aitch} below) is 
formalized in the next lemma.

\begin{lemma}
\label{VKsplitmono}
If $j_{X*}$ is a split monomorphism then $\pi_Y=1$.
\end{lemma}

\begin{proof}
Let $\sigma:\pi_X\to\pi$ be a homomorphism such that
$\sigma\circ{j_{X*}}=id_\pi$, and let $f_X=j_{Y*}\circ\sigma$.
Then $f_X\circ{j_{X*}}=id_{\pi_Y}\circ{j_{Y*}}$ and so $id_{\pi_Y}$ factors through 
the pushout group 1.
Hence $\pi_Y=1$.
\end{proof}

\section{Hyperbolicity of the linking pairing}

In the early paper \cite{Ha38} W. Hantzsche observed that if $M$ embeds in $S^4$ then
the torsion subgroup of $H_1(M)$ is a direct double.
This follows easily from the Universal Coefficient Theorem and Alexander duality.
Let $\tau_M$, $\tau_X$ and $\tau_Y$ denote the torsion subgroups of $H_1(M)$,
$H_1(X)$ and $H_1(Y)$, respectively.
Then $\tau_M\cong\tau_X\oplus\tau_Y$, by the Mayer-Vietoris argument, 
and $\tau_X\cong{Ext(\tau_Y,\mathbb{Z})}=Hom(\tau_Y,\mathbb{Q}/\mathbb{Z})$, 
since $H_1(M)\cong{H^2(Y)}$.
Hence $\tau_X$ is non-canonically isomorphic to $\tau_Y$, 
and so $\tau_M\cong\tau_X\oplus\tau_X$.

Looking more closely at the role of duality,
A. Kawauchi and S. Kojima strengthened this result \cite[Lemma 6.1]{KK80}.

\begin{lemma}
[Kawauchi-Kojima]
\label{hyperbolic lp}
If a closed connected $3$-manifold $M$ embeds as a locally flat submanifold 
of $S^4$ then $\ell_M$ is hyperbolic.
\end{lemma}

\begin{proof}
Let $D_{X,M}:H_2(X,M)\to{H^2(X)}$ and 
$D_M:H_1(M)\to{H^2(M)}$
be the duality isomorphisms determined by an orientation for $(X,M)$,
and let $\delta_X:H_2(X,M)\to{H_1(M)}$ 
be the connecting homomorphism in the homology exact sequence for the pair.
Then $j_X^*D_{X,M}=D_M\delta_X$.

It follows from the Mayer-Vietoris sequence that
$H_i(M)$ maps onto $H_i(X)$, for $i=1$ or 2,
and hence that there is a short exact sequence
\begin{equation*}
\begin{CD}
0\to{H_2(X,M)}@>\delta_X>>{H_1(M)}@>j_{X*}>>{H_1(X)}\to0.
\end{CD}
\end{equation*}
Let  $\tau_{X,M}$ and $\tau_{Y,M}$ be the torsion subgroups 
of $H_2(X,M)$ and $H_2(Y,M)$,  respectively.
Then $D_{X,M}$ restricts to give an  isomorphism 
$\tau_{X,M}\cong{Ext(\tau_X,\mathbb{Z})}=Hom(\tau_X,\mathbb{Q}/\mathbb{Z})$.
(Similarly,  $D_M$ restricts to give $\tau_M\cong{Hom(\tau_M,\mathbb{Q}/\mathbb{Z})}$.)
Hence $|\tau_M|=|\tau_{X,M}||\tau_X|$, and so the sequence of torsion subgroups
\[
0\to\tau_{X,M}\to\tau_M\to\tau_X\to0
\]
is also exact.
(In particular,  $\tau_{X,M}$ and $\tau_{Y,M}$ are the kernels of
the induced homomorphisms from $\tau_M$ to $H_1(X)$ and $H_1(Y)$, 
respectively.)
Since
\[
\ell_M(\delta_Xa,\delta_Xb)=
D_M(\delta_Xb)(\delta_Xa)=j_X^*D_{X,M}(b)(\delta_Xa)=D_{X,M}(b)(j_{X*}\delta_Xa)=0,
\]
for all $a,b\in\tau_{X,M}$,
the image of $\tau_{X,M}$ under $\delta_X$ is self-annihilating.
Similarly,  $\delta_Y(\tau_{Y,M})$ is self-annihilating,
and clearly $\tau_M$ is the direct sum of these two subgroups.
\end{proof}

The following corollary was used to show that punctured lens spaces $L(2k,q)_o$ do not embed in $S^4$ \cite[Corollary 6.1]{KK80}. 
The proof here is taken from \cite{GL83}.
(Note that open 3-manifolds have well-defined torsion linking pairings, 
which may be singular.)

\begin{cor}
\label{KKcor}
Let $N$ be a connected $3$-manifold (possibly open) which is a locally flat submanifold of $S^4$.
If $x\in{H_1(N)}$ has order $2^k$ for some $k\geq1$ then $2^{k-1}\ell_N(x,x)=0$.
\end{cor}

\begin{proof}
We may assume that $x$ and a null-homology of $2^kx$ are supported in a compact 
bounded codimension-0 submanifold $Q\subset{N}$.
Let $M=\partial(Q\times{I})$ be the boundary of a regular neighbourhood of $Q$ in $S^4$.
Then  $M=DQ$ is the double of $Q$ along its boundary, 
and $\tau_M\cong{A\oplus{B}}$,
where $A$ and $B$ are self-annihilating with respect to $\ell_M$.
Let $x'$ denote the image of $x$ in $\tau_M$, and write $x'=x_A+x_B$, with $x_A\in{A}$
and $x_B\in{B}$.
Then $2^kx_A=2^kx_B=0$, and $\ell_N(x,x)=$
\[
\ell_Q(x,x)=\ell_M(x',x')= \ell_M(x_A,X_A)+2\ell_M(x_A,x_B)+\ell_M(x_B,x_B)=
2\ell_M(x_A,x_B).
\]
Hence $2^{k-1}\ell_N(x,x)=2^k\ell_M(x_A,x_B)=0$.
\end{proof}

Kawauchi had earlier given a related result for the Blanchfield pairing on $M$ 
corresponding to an epimorphism from $\pi_1M$ to $\mathbb{Z}$ which is 
the restriction of an epimorphism from $\pi_1X$ to $\mathbb{Z}$
\cite[Theorem 4.2]{Ka77}. 

\begin{thm}
[Kawauchi] 
Let $X$ be an orientable $4$-manifold with connected boundary $M$, 
and let $\phi:\pi_1X\to\mathbb{Z}$ be an epimorphism such that 
the restriction to the image of $\pi_1M$ is also an epimorphism.
Let $X_\phi$ and $M_\phi$ be the corresponding $\mathbb{Z}$-covering spaces,
and let $t$ be a generator of the covering group $Aut(X_\phi/X)\cong\mathbb{Z}$.
If $t^2-1$ acts invertibly on $H_1(M_\phi;\mathbb{Q})$
and $H_2(X_\phi,M_\phi;\mathbb{Q})$ is a $\mathbb{Q}[t,t^{-1}]$-torsion
module then the Blanchfield pairing on $H_1(M_\phi;\mathbb{Q})$ is neutral.
Hence the order ideal of $H_1(M;\mathbb{Q}\Lambda)=H_1(M_\phi;\mathbb{Q})$ 
has a generator of the form $\Delta_\phi(t)=f(t)f(t^{-1})$, 
for some $f(t)\in\mathbb{Q}[t,t^{-1}]$.
\qed
\end{thm}

We shall only use this result when $\beta=1$.
The epimorphism $\phi$ is then unique up to sign,
since $H^1(X)\cong{H^1(M)}\cong\mathbb{Z}$.
See Lemma \ref{Kawabeta1} for a proof of this case.

\section{Comparison of embeddings}

If two embeddings $j$ and $\tilde{j}$ are isotopic through locally flat embeddings 
then they are ambient isotopic, by the Isotopy Extension Theorem 
\cite[Theorem 2.20]{FNOP}.
Embeddings $j_0,j_1:M\to{S^4}$ are {\it $s$-concordant\/}
if they extend to an embedding of $M\times[0,1]$ in $S^4\times[0,1]$ 
whose complementary regions are $s$-cobordisms {\it rel\/} $\partial$.
We need this notion as it is not yet known whether 5-dimensional 
$s$-cobordisms are always products.

We shall say that an embedding has a group-theoretic property 
(e.g., abelian, nilpotent, $\dots$) 
if the groups $\pi_X$ and $\pi_Y$ have this property.
If $j:M\to{S^4}$ is an abelian embedding and all abelian embeddings 
of $M$ in $S^4$ are equivalent to $j$,
we shall say that $j$ is {\it essentially unique}.
 
\begin{lemma}
\label{Hcoblem}
If $M$ and $M'$ are $\mathbb{Z}$-homology cobordant then $M$ embeds 
in a homology $4$-sphere if and only if 
$M'$ embeds in a (possibly different) homology $4$-sphere.
\end{lemma}

\begin{proof}
Suppose that $M$ embeds in a homology 4-sphere $\Sigma$,
with complementary regions $X$ and $Y$, and $W$ is a 
$\mathbb{Z}$-homology cobordism with $\partial{W}=M\sqcup{M'}$.
Let $Z=W\cup_{M'}W$ be the union of two copies of $W$, 
identified along the boundary $M'$.
A Mayer-Vietoris argument shows that $\Sigma'=X\cup_MW'\cup_MY$ 
is a homology 4-sphere, and $M'$ is clearly a locally flat submanifold of $W'$.
\end{proof}
 
We shall use a variation of this construction,
taking into account the fundamental groups, 
in Theorem \ref{Hi96-thm1} below.

Let $L$ be the $(2,2k)$-torus link and $M=M(L)$.
Then $\tau_M=H_1(M)\cong(\mathbb{Z}/k\mathbb{Z})^2$, 
and $\ell_M$ is hyperbolic, since $M$ embeds in $S^4$.
Taking connected sums shows that every hyperbolic linking pairing is
realized by some 3-manifold which embeds in $S^4$.
Every closed orientable 3-manifold is $\mathbb{Z}$-homology cobordant 
to a 3-manifold which is Haken and hyperbolic \cite{Liv81,My83}.
It then follows from Lemma 2.7 that every hyperbolic linking pairing 
is realized by a hyperbolic $\mathbb{Q}$-homology sphere which embeds 
in a ($\mathbb{Z}$-)homology 4-sphere.
In Chapter 3 we shall show that every hyperbolic linking pairing on a finite abelian group of odd order is realized by a Seifert fibred 3-manifold which embeds in $S^4$.

\section{Some constructions of embeddings}

Any closed orientable 3-manifold $M$ may be obtained by integrally 
framed surgery on some $r$-component link $L$ in $S^3$, with $r\geq\beta$.
We may assume that the framings are even \cite{Kap}, and then
after adjoining copies of the 0-framed Hopf link $Ho=2^2_1$ 
(i.e., replacing $M$ by $M\#S^3\cong{M}$) 
we may modify $L$ so that it is 0-framed.
(If the component $L_i$ has framing $2k\not=0$ we adjoin $|k|$ 
disjoint copies of $Ho$ and band-sum $L_i$ to each of the $2k$ 
new components, with appropriately twisted bands.)
Let $M(L)$ be the closed 3-manifold obtained by 0-framed surgery on $L$. 

\begin{defn}
A link $L$ is {\it bipartedly trivial\/} (respectively, {\it ribbon} or {\it slice}) 
if it has a partition $L=L_+\cup{L_-}$ into two sublinks 
which are each trivial links (respectively, ribbon or slice links).
\end{defn}

The partition then determines an embedding $j_L:M\to{S^4}$,
given by ambient surgery on an equatorial $S^3$ in $S^4=D_+\cup{D_-}$.
We add 2-handles to these 4-balls along $L_+$ on one side 
and along $L_-$ on the other,  
using sets of disjoint slice discs to achieve the ambient surgery.
The complementary components have $\chi=1+2s-r$ and $1-2s+r$.
If $L_+$ and $L_-$ are smoothly slice then $j_L$ is smooth, 
and if they are trivial each complementary region 
has a natural Kirby-calculus presentation,
with 1-handles represented by dotting the components of one part of $L$
and 2-handles represented by the remaining components of $L$,
with framing $0$ \cite{GS}.
Hence it is homotopy equivalent to a finite 2-complex,
and its fundamental group has a presentation with generators corresponding to
the meridians of the dotted circles and relators corresponding to the remaining components.
We extend this Kirby-calculus notation to allow the deleted discs 
corresponding to the dotted components to be a set of slice discs.
(In our figures we shall distinguish the moieties by using thick and thin lines, with dots $\bullet$ on the thin lines.)

The notation $j_L$  is ambiguous, 
for if $L$ has more than one component it may have 
several different partitions leading to distinct embeddings.
Moreover we must choose a set of slice discs for each of $L_+$ and $L_-$.
If $L_-$ or $L_+$ is non-trivial, but is obviously a ribbon link, 
we shall use the discs obtained by desingularizing the ribbon discs,
and then finding a presentation for the fundamental group of the corresponding
complementary region is also straightforward.

If $L$ is itself a slice link then $\beta=r$ and there are embeddings of $M(L)$
realizing each value of $\chi(X)$ allowed by Lemma \ref{Hi17-lem2.1},
including one with a 1-connected complementary region.

F.  Quinn showed that any pair of connected finite 2-complexes $\{C,D\}$ 
which satisfy the conditions $H^2(C)\cong{H_1(D)}$ and $H^2(D)\cong{H_1(C)}$ 
deriving from Alexander duality (as in Lemma 2.2)
may be realized up to homotopy as the complementary regions of 
a smoothable embedding of some 3-manifold \cite[Corollary 1.5]{Qu01}.
(The case when $D$ and $D$ are acyclic was treated earlier
\cite{Cr88,Hu90}.
See also \cite{Li03} and \cite{Liv03}.)
Although we shall not use this result directly, 
it provides the right conceptual framework for some of our examples.
W. B. R. Lickorish found a similar result with more algebraic hypotheses
and a more explicit use of link presentations,
for the cases with $\chi(C)=\chi(D)$, 
and we shall outline his argument \cite{Li04}.
(The hypothesis in Theorem \ref{LQ} that the groups have balanced
presentations is equivalent to requiring that $\chi(C)=\chi(D)$ 
in Quinn's result.)
We shall give the topological part of the argument, 
and refer to \cite[Corollary 2.2]{Li04} for the proof of the next lemma.

\begin{lem}
Let $\mathcal{P}$ be a balanced finite presentation of a group $G$
and $B$ be a square presentation matrix for $G^{ab}$.
Then $\mathcal{P}$ is Andrews-Curtis equivalent to a presentation $\mathcal{Q}$ for which the associated presentation matrix for $G^{ab}$ is
$\left(\smallmatrix{B}&0\\0&{I_r}\endsmallmatrix\right)$ for some $r\geq0$.
\qed
\end{lem}

The proof is based on the correspondence between Andrews-Curtis moves 
on a group presentation \cite[Chapter 1]{HAMS}
and elementary matrix operations (plus block diagonal enlargements and their inverses) on the associated  presentation for the 
abelianization. 

\begin{theorem}
[Lickorish-Quinn]
\label{LQ}
Let $\mathcal{P}_1$ and $\mathcal{P}_2$ be balanced presentations of groups $G_1$ and $G_2$ such that $G_1^{ab}\cong{G_2^{ab}}$.
Then $S^4=X_1\cup{X_2}$ is the union of two codimension-$0$ 
submanifolds with $\pi_1X_1\cong{G_1}$, $\pi_1X_2\cong{G_2}$ and connected boundary 
$\partial{X_1}=\partial{X_2}=M$.
Each of $X_1$ and $X_2$ has a handle structure consisting of one $0$-handle,
$n$ $1$-handles and $n$ $2$-handles, 
with the associated presentations for $\pi_1X_1$ and $\pi_1X_2$ 
being Andrews-Curtis equivalent to $\mathcal{P}_1$ and $\mathcal{P}_2$,
respectively.
\end{theorem}

\begin{proof}
We shall begin by using Andrews-Curtis (AC) moves to modify the presentations.
Abelianizing 
$\mathcal{P}_1=\langle{a_1,\dots,a_n}\mid{r_1,\dots,r_n}\rangle$ determines 
an $n\times{n}$ integral matrix $C$ which is a presentation matrix for $G_1^{ab}$.
The transpose $C^{tr}$ is also a presentation matrix for $G_1^{ab}$,
by the Elementary Divisor Theorem.
We shall view $C^{tr}$ as a presentation matrix for $G_2^{ab}$.
Then there is a presentation $\mathcal{P}_2^{AC}=
\langle{\alpha_1,\dots,\alpha_{n+r}}\mid{\rho_1,\dots,\rho_{n+r}}\rangle$ 
for $G_2$ which is AC-equivalent to $\mathcal{P}_2$ and for which the corresponding presentation matrix for $G_2^{ab}$ is the block diagonal matrix
$D_2=\left(\smallmatrix{C^{tr}}&0\\0&{I_r}\endsmallmatrix\right)$
 \cite[Corollary 2.2]{Li04}.
Add $r$ new generators $a_{n+1},\dots,a_{n+r}$ and $r$ new relators
$r_{n+1}=a_{n+1},\dots,r_{n+r}=a_{n+r}$ to $\mathcal{P}_1$, 
to get an AC-equivalent presentation  for $G_1$ for which the
corresponding presentation matrix for $G_1^{ab}$ is $A=D_2^{tr}$.

Suppose that in each relator $r_i$ there are $n_+^{i,j}$ occurrences of the 
generator $a_j$ and $n_-^{i,j}$ occurrences of its inverse $a_j^{-1}$,
and that in each $\rho_j$ there are $\nu_+^{i,j}$ occurrences of $\alpha_i$
and $\nu_-^{i,j}$ occurrences of the symbol $\alpha_i^{-1}$.
Then $n_+^{i,j}-n_-^{i,j}=A_{i,j}=\nu_+^{i,j}-\nu_-^{i,j}$.
If $e(i,j)=n_+^{i,j}-\nu_+^{i,j}>0$ alter $\mathcal{P}_2^{AC}$ by changing $\rho_j$ to 
$\rho_j(\alpha_i\alpha_i^{-1})^{e(i,j)}$.
If $e(i,j)<0$ alter $\mathcal{P}_1$ by changing $r_i$ to $r_i(a_ja_j^{-1})^{-e(i,j)}$.
In this way we may assume that $n_+^{i,j}=\nu_+^{i,j}$ and hence also
$n_-^{i,j}=\nu_-^{i,j}$, for all $i, j$.

We now construct a link in $S^3$ as follows.
Let $D_1,\dots,D_{n+r}$ and $\Delta_1,\dots,\Delta_{n+r}$ be 
mutually disjoint oriented discs.
For each pair $(i,j)$  with $1\leq{i,j}\leq{n+r}$
let $H_+^{i,j}$ be set of $n_+^{i,j}$ copies of the positive Hopf link (of two ordered,
oriented components with linking number $+1$) and
$H_-^{i,j}$ a set of $n_-^{i,j}$ copies of the negative Hopf link.
Each of these Hopf links is to be in a (small) ball in which each of the two components 
bounds an oriented disc meeting the other component in one point.
These balls are to be all mutually disjoint and disjoint from the original discs.
Now join the boundary of $\Delta_i$ once to the first component of each link in
$\bigcup(H_+^{i,j}\cup{H_-^{i,j}}$ with (long thin) bands.
Do this in order around $\partial\Delta_i$ specified by the relator $r_i$.
When $a_j$ occurs in the relator connect to the first component of one of the links
in $H_+^{i,j}$,
and when $a_j^{-1}$ occurs in the relator connect to the first component of 
one of the links in $H_-^{i,j}$.
Similarly when $\alpha_i^{\pm1}$ occurs in $\rho_j$ connect to the first component 
of one of the links in $H_{\pm}^{i,j}$.
For an occurrence of $\alpha_i$ any unused second component of any Hopf link in 
$H_+^{i,j}$ may be selected and similarly for $\alpha_i^{-1}$,
it can easily be ensured that all the bands used are mutually disjoint 
and that they respect all orientations.
(However there are many ways of choosing the bands.)
The numbers of links in the $H_\pm^{i,j}$ are to be chosen so that each link in each 
$H_\pm^{i,j}$ has its first component banded to $\Delta_i$ and its second component banded to $D_j$.
This banding process changes the original discs to two new collections $D_1',\dots,D_{n+r}'$
and $\Delta_1',\dots,\Delta_{n+r}'$ , each of mutually disjoint discs, by adding to the 
original discs the bands and discs spanning the components of the Hopf links.
Let $L_+=\{\partial{D_i'}\}$ and $L_-=\{\partial\Delta_j'\}$ be the links 
given by the boundaries of these discs. 
Then $L=L_+\cup{L_-}$ is bipartedly trivial.

Let $X_\pm$  be the handlebodies determined by giving $L_\pm$ 
the 0-framing and dotting $L_\mp$.
The handle decomposition of $X_\pm$ gives rise to a presentation for $\pi_1X_\pm$ 
in the standard way.
We associate a generator $a_i$ to the 1-handle corresponding to $D_i'$ and
a relator $r_j$ to to 2-handle corresponding to $\Delta_j'$.
Then $r_j$ has an entry $a_i^{\pm1}$ for every signed point of $\partial\Delta_j'\cap{D_i'}$
taken in order along  $\partial\Delta_j'$.
The construction has been engineered so that $X_-$ gives rise to
$\mathcal{P}_1$ and $X_+$ gives rise to
$X_-$ and $\mathcal{P}_2^{AC}$.
\end{proof}

We note that the embeddings given by this theorem always have $\chi(X)=\chi(Y)$.

In particular, 
any two perfect groups with balanced presentations can be realized as 
$\pi_X$ and $\pi_Y$ for some embedding of a homology sphere in $S^4$.
C. Livingston has given examples in which $\pi_X$ is superperfect but has no balanced presentation \cite{Liv05}.

In the following lemma of A. Donald,  
the notion of doubly slice knot is extended to say that a link $L$ in $S^3$ 
is {\it  doubly slice\/} if it is the transverse intersection of
the equator in $S^4$ with an unknotted embedding of $S^2$.
(The embeddings studied in \cite{Do15} are all smooth, 
but the argument also applies here.)

\begin{lemma}
\cite{Do15}
Let $L$ be a doubly slice link in $S^3$, 
and let $B_n(L)$ be the $n$-fold cyclic branched cover of $S^3$, 
branched over $L$.
Then $B_n(L)$ embeds  in $S^4$.
\end{lemma}

\begin{proof}
Since $L$ is  doubly slice the pair $(S^3,L)$ sits inside $(S^4,U)$, 
where $U$ is an unknotted copy of $S^2$.
The branched covering of $S^3$,
branched over $L$, extends to a branched covering of $S^4$,
branched over $U$,
and $B_n(L)$ embeds  in the covering space,
as the preimage of $S^3$.
Cyclic branched coverings of $S^4$ branched over $U$ are again homeomorphic to $S^4$,
and so $B_n(L)$ embeds in $S^4$.
\end{proof}

\setlength{\unitlength}{1mm}
\begin{picture}(95,75)(-40,5)

\put(23,47){$b$}
\put(41,47){$-b$}

\qbezier(-8,60)(-8,75)(7,75)
\qbezier(-5,33)(-5,18)(10,18)

\qbezier(-8,60)(-8,58)(-6,58)
\put(-6,58){\line(1,0){2}}
\qbezier(-4,58)(-2,58)(-2,56)
\qbezier(-4,54)(-2,54)(-2,56)
\qbezier(-8,52)(-8,54)(-6,54)
\qbezier(-8,52)(-8,50)(-6,50)
\put(-6,50){\line(1,0){2}}

\qbezier(-4,50)(-2,50)(-2,48)
\qbezier(-4,46)(-2,46)(-2,48)
\qbezier(-8,44)(-8,46)(-6,46)
\qbezier(-8,44)(-8,42)(-6,42)
\put(-6,42){\line(1,0){2}}
\qbezier(-4,42)(-2,42)(-2,40)
\qbezier(-4,38)(-2,38)(-2,40)
\qbezier(-8,36)(-8,38)(-6,38)
\qbezier(-8,36)(-8,34)(-6,34)
\put(-6,34){\line(1,0){2}}

\put(-5,35){\line(0,1){6}}
\put(-5,43){\line(0,1){6}}
\put(-5,51){\line(0,1){6}}
\put(-5,59){\line(0,1){2}}

\qbezier(-5,59)(-5,65)(1,65)
\qbezier(1,65)(7,65)(7,59)
\put(6,58){\line(1,0){2}}

\qbezier(4,56)(4,58)(6,58)
\qbezier(4,56)(4,54)(6,54)

\qbezier(8,54)(10,54)(10,52)
\qbezier(8,50)(10,50)(10,52)

\qbezier(4,48)(4,50)(6,50)
\qbezier(4,48)(4,46)(6,46)
\put(6,42){\line(1,0){2}}

\qbezier(4,40)(4,42)(6,42)
\qbezier(4,40)(4,38)(6,38)

\qbezier(8,46)(10,46)(10,44)
\qbezier(8,42)(10,42)(10,44)
\put(6,50){\line(1,0){2}}

\qbezier(8,38)(10,38)(10,36)
\qbezier(8,34)(10,34)(10,36)
\put(-4,34){\line(1,0){12}}

\put(7,35){\line(0,1){6}}
\put(7,43){\line(0,1){6}}
\put(7,51){\line(0,1){6}}

\qbezier(8,58)(10,58)(10,60)
\qbezier(10,60)(10,65)(15,65)
\qbezier(15,65)(20,65)(20,60)

\put(7,75){\line(1,0){26}}
\put(9,18){\line(1,0){24}}

\qbezier(33,75)(48,75)(48,60)
\qbezier(33,18)(48,18)(48,33)
\put(48,33){\line(0,1){2}}

\put(17,35){\line(0,1){25}}
\put(30,35){\line(0,1){25}}
\put(17,35){\line(1,0){13}}
\put(17,60){\line(1,0){13}}

\put(37,35){\line(0,1){25}}
\put(50,35){\line(0,1){25}}
\put(37,35){\line(1,0){13}}
\put(37,60){\line(1,0){13}}

\qbezier(7,33)(7,28)(12,28)
\put(12,28){\line(1,0){3}}
\qbezier(15,28)(20,28)(20,33)
\put(20,33){\line(0,1){2}}

\put(27,33){\line(0,1){2}}
\qbezier(27,33)(27,28)(32,28)
\put(32,28){\line(1,0){3}}
\qbezier(35,28)(40,28)(40,33)
\put(40,33){\line(0,1){2}}

\qbezier(27,60)(27,65)(32,65)
\put(32,65){\line(1,0){3}}
\qbezier(35,65)(40,65)(40,60)

\put(-12,10){Figure 2.1.  The pretzel link $P(-7,7,b,-b)$}

\end{picture}

Since cyclic branched coverings of $S^4$ branched over a smoothly unknotted
2-sphere are diffeomorphic to $S^4$, 
the same argument shows that if $L$ is smoothly doubly slice then
$B_n(L)$ embeds smoothly in $S^4$.
Donald shows also that the pretzel links $P(a,-a,a)$,  $P(a,-a,a,-a\pm1)$ and
$P(a,-a,a,-a)$ are all doubly slice, for any $a\in\mathbb{Z}$,
while $P(a,-a,b,-b)$ is doubly slice if $a$ and $b$ are both odd.
In Figure 2.1 the letter $b$ represents $b$ half twists of a 2-strand tangle, 
with the minus sign representing reversal of the twist.

The ``1-surgical embedding" of R.  Budney and B. A. Burton gives smooth embeddings in smooth homotopy 4-spheres for homology spheres obtained by surgery on a link 
with all framings $\pm1$ \cite{BB22}.

\section{Connected sums}

We may form ambient connected sums of embeddings as follows.
Let $j:M\to\Sigma$ and $j':M'\to\Sigma'$ be two embeddings of 3-manifolds in homology 4-spheres,  with complementary regions $X,Y$ and $X',Y'$, respectively.
We shall assume that $\Sigma$ and $\Sigma'$ are oriented, 
and that $M$ and $M'$ have the induced orientation as the boundaries of $X$ and $X'$.
Choose small balls $D\subset\Sigma$ and $D'\subset\Sigma'$ such that
$D\cap{j(M)}$ and $D'\cap{j'(M')}$ are each the equatorial 3-discs,
and let $\Sigma_o=\overline{\Sigma\setminus{D}}$,
$\Sigma_o'=\overline{\Sigma'\setminus{D'}}$,
$M_o=\overline{M\setminus{D}}$ and $M'_o=\overline{M'\setminus{D'}}$.
Let $h^+:\partial{D}\to\partial{D'}$ and $h^-:\partial{D}\to\partial{D'}$ 
be orientation reversing homeomorphisms 
such that $h^+(\partial{D}\cap{X})=\partial{D'}\cap{X'}$
and $h^-(\partial{D}\cap{X})=\partial{D'}\cap{Y'}$.
Then $j|_{M_o}\cup{j'|{M'_o}}$ defines an embedding $j\#^+j'$
of $M\#{M'}$ in $\Sigma_o\cup_{h^+}\Sigma_o'$ and an embedding 
$j\#^-j'$ of $M\#-M'$ in $\Sigma_o\cup_{h^-}\Sigma_o'$.
(If $\Sigma=\Sigma'=S^4$ then the new ambient homology 4-spheres are again 
copies of $S^4$.)
The complementary regions are $X\natural{X'}$ and $Y\natural{Y'}$ for  $j\#^+j'$,
and $X\natural{Y'}$ and $Y\natural{X'}$ for  $j\#^-j'$.
Thus these two ambient connected sums are usually distinct.

A 3-manifold which is a proper connected sum may embed in $S^4$ 
even though its indecomposable summands do not.
The lens space $L=L(p,q)$ does not embed in $S^4$, 
since $H_1(L)$ is not a direct double. 
The connected sum $L\#L$ does not embed either, 
since its torsion linking pairing is not hyperbolic.
However if $p$ is odd then $L\#-L$ embeds as a regular neighbourhood 
of a Seifert hypersurface for the 2-twist spin of a suitable 2-bridge knot.
A. Donald has shown that the connected sums of lens spaces which 
embed smoothly in $S^4$ are just those which are connected sums of pairs 
$L(p_i,q_i)\#-L(p_i,q_i)$ with all $p_i$ odd \cite{Do15}.
(See also Corollary \ref{KKcor} above.)

\section{Bi-epic embeddings}

\begin{defn}
The embedding $j$ is {\it bi-epic\/} if each of the homomorphisms 
$j_{X*}=\pi_1j_X$ and  $j_{Y*}=\pi_1j_Y$ is an epimorphism.
\end{defn}

If $j$ is an embedding such that $\pi_X$ and $\pi_Y$ are nilpotent,
then it is bi-epic,
since $H_1(j_X)$ and $H_1(j_Y)$ are always epimorphisms.

\begin{lemma}
\label{0framedbiepic}
Let $L$ be a bipartedly ribbon link.
Then the embedding $j_L$ constructed using the ribbon discs for each of the sublinks
 is bi-epic.
\end{lemma}

\begin{proof}
In this case $\pi$,  $\pi_X$ and $\pi_Y$ are generated by images of the meridians of $L$.
\end{proof}

\begin{lemma} 
\label{minimal}
The homomorphisms $j_{X*}$ and $j_{Y*}$ are both epimorphisms if and only if 
$j_\Delta=(j_{X*},j_{Y*})$ is an epimorphism.
\end{lemma}

\begin{proof}
Let $K_X=\mathrm{Ker}(j_{X*})$ and $K_Y=\mathrm{Ker}(j_{Y*})$.
If $j_{X*}$ and $j_{Y*}$ are epimorphisms then they induce isomorphisms
$\pi/K_X\to\pi_X$ and $\pi/K_Y\to\pi_Y$.
Hence $\pi/K_XK_Y\cong\pi_X/j_{X*}(K_Y)$ and $\pi/K_XK_Y\cong\pi_Y/j_{Y*}(K_X)$.
Since  $\pi_1(X\cup_MY)=1$, these quotients must all be trivial.
If $g\in{K_X}$ and $h\in{K_Y}$ then $j_\Delta(gh)=(j_{X*}(h),j_{Y*}(g))$.
Hence $j_\Delta$ is an epimorphism.

Conversely, if $j_\Delta$ is an epimorphism then so are 
its components  $j_{X*}$ and $j_{Y*}$.
\end{proof}

The link $L$ obtained from the Borromean rings $Bo=6^3_2$ by replacing 
one component by its $(2,1)$-cable and another by its $(3,1)$-cable 
may be partitioned as the union of two trivial links in three ways. 
The resulting three embeddings of $M(L)$ in $S^4$ each have $Y\simeq{S^1}\vee2S^2$,
but the groups $\pi_X$ have presentations $\langle{a,b}\mid[a,b^2]^3\rangle$,
$\langle{a,c}\mid[a,c^3]^2\rangle$, and $\langle{b,c}\mid[b^2,c^3]\rangle$, 
respectively, and so are distinct.
In the first two cases $\pi$ has torsion, 
while in the third case $X$ is aspherical.
(None of these groups is  abelian.)
This example can obviously be generalized in various ways.
The homology sphere in  Figure 2.2 is another example;
the embedding determined by the link is bi-epic, but the 3-manifold also has
an embedding with both complementary regions contractible.
However the latter embedding may not derive from a
0-framed link representing the homology sphere.

\setlength{\unitlength}{1mm}
\begin{picture}(95,68)(-29.3,10)

\put(-5,69){$\vartriangleright$}
\put(-8,66){$x$}
\put(60,69){$\vartriangleright$}
\put(58,66){$y$}
\put(26,65){$\bullet$}
\put(28,67){$s$}
\put(26,50.6){$\bullet$}
\put(25.5,53){$r$}

\linethickness{1pt}
\put(-15,70){\line(1,0){22}}
\put(-20,30){\line(0,1){35}}
\qbezier(-20,65)(-20,70)(-15,70)
\put(-15,25){\line(1,0){22}}
\qbezier(-20,30)(-20,25)(-15,25)
\qbezier(7,25)(10,25)(10,28)
\qbezier(7,70)(10,70)(10,67)
\put(10,30){\line(0,1){6}}
\put(10,38){\line(0,1){6}}
\put(10,46){\line(0,1){6}}
\put(10,54){\line(0,1){11}}

\put(48,70){\line(1,0){22}}
\put(75,30){\line(0,1){35}}
\put(48,25){\line(1,0){22}}
\qbezier(70,70)(75,70)(75,65)
\qbezier(70,25)(75,25)(75,30)
\qbezier(45,67)(45,70)(48,70)
\qbezier(45,28)(45,25)(48,25)

\put(45,30){\line(0,1){8}}
\put(45,40){\line(0,1){4}}
\put(45,46){\line(0,1){4}}
\put(45,52){\line(0,1){3}}
\put(45,57){\line(0,1){3}}
\put(45,62){\line(0,1){3}}

\thinlines
\put(5,29){\line(1,0){45}}
\qbezier(5,29)(3.25,29)(3.25,30.75)
\qbezier(3.25,30.75)(3.25,32.5)(5,32.5)
\put(5,32.5){\line(1,0){4}}
\qbezier(15, 32.5)(17.25, 32.5)(17.25,34.75)
\qbezier(15,37)(17.25,37)(17.25,34.75)
\put(5,37){\line(1,0){10}}
\qbezier(5,37)(2.75,37)(2.75,39.25)
\qbezier(2.75,39.25)(2.75,41.5)(5,41.5)
\put(5, 41.5){\line(1,0){4}}
\put(5,45){\line(1,0){13}}
\qbezier(18,45)(20.5,45)(20.5,42.5)
\qbezier(5,45)(2.75,45)(2.75,46.25)
\qbezier(2.75,46.75)(2.75,48.5)(5,48.5)
\put(5,48.5){\line(1,0){4}}
\put(5,53){\line(1,0){10}}
\qbezier(5,53)(2.75,53)(2.75,55.25)
\qbezier(2.75, 55.25)(2.75,57.5)(5,57.5)
\put(5,57.5){\line(1,0){4}}

\put(5,53){\line(1,0){16}}
\qbezier(21,53)(23.5,53)(23.5,50.5)
\put(23.5,34.5){\line(0,1){16}}
\qbezier(23.5,34.5)(23.5,32)(26,32)
\put(26,32){\line(1,0){15}}

\put(5,61.5){\line(1,0){4}}
\qbezier(5,61.5)(2.75,61.5)(2.75,63.75)
\qbezier(2.75,63.75)(2.75,66)(5,66)
\put(5,66){\line(1,0){45}}

\put(11, 32.5){\line(1,0){4}}
\put(11,41.5){\line(1,0){11.5}}
\qbezier(24.6,41.5)(34.3,42)(44,42.5)
\put(11,48.5){\line(1,0){11.5}}
\put(24.5,48.5){\line(1,0){2}}
\put(28,48.5){\line(1,0){7}}
\put(37,48.5){\line(1,0){7}}

\put(11,57.5){\line(1,0){13}}
\qbezier(24,57.5)(27,57.5)(27,54.5)
\put(27,47.5){\line(0,1){7}}
\qbezier(27,47.5)(27,45)(29.5,45)
\put(29.5,45){\line(1,0){20.5}}

\put(11,61.5){\line(1,0){20}}

\put(46,63.5){\line(1,0){4}}
\qbezier(50,63.5)(51.25,63.5)(51.25,64.75)
\qbezier(50,66)(51.25,66)(51.25,64.75)

\put(46,58.5){\line(1,0){4}}
\qbezier(50,58.5)(51.25,58.5)(51.25,59.75)
\qbezier(50,61)(51.25,61)(51.25,59.75)
\put(46,53.5){\line(1,0){4}}
\qbezier(50,53.5)(51.25,53.5)(51.25,54.75)
\qbezier(50,56)(51.25,56)(51.25,54.75)

\put(40,51){\line(1,0){10}}
\qbezier(40,51)(38.75,51)(38.75,52.25)
\qbezier(38.75,52.25)(38.75,53.5)(40,53.5)
\put(40,53.5){\line(1,0){4}}
\qbezier(40,56)(38.75,56)(38.75,57.25)
\qbezier(38.75,57.25)(38.75,58.5)(40,58.5)
\put(40, 58.5){\line(1,0){4}}
\put(40,56){\line(1,0){10}}
\put(40,61){\line(1,0){10}}
\qbezier(40,61)(38.75,61)(38.75,62.25)
\qbezier(38.75,62.25)(38.75,63.5)(40,63.5)

\put(40,63.5){\line(1,0){4}}

\put(38.2,39){\line(1,0){11.8}}
\put(46,32){\line(1,0){4}}
\put(46, 36){\line(1,0){4}}
\qbezier(50,29)(51.5,29)(51.5,30.5)
\qbezier(50,32)(51.5,32)(51.5,30.5)
\qbezier(50,36)(51.5,36)(51.5,37.5)
\qbezier(50,39)(51.5,39)(51.5,37.5)
\qbezier(50,42.5)(51.25,42.5)(51.25,43.75)
\qbezier(50,45)(51.25,45)(51.25,43.75)

\qbezier(50,48.5)(51.25,48.5)(51.25,49.75)
\qbezier(50,51)(51.25,51)(51.25,49.75)

\put(46,42.5){\line(1,0){4}}
\put(46, 48.5){\line(1,0){4}}

\put(40,32){\line(1,0){4}}

\put(36,42.5){\line(0,1){1.5}}
\put(36,46){\line(0,1){10.5}}

\qbezier(36,41.2)(36,39)(38.2,39)
\qbezier(31,61.5)(36,61.5)(36,56.5)

%\qbezier(15,53)(25,44.2)(40,32)
\put(37,36){\line(1,0){7}}
\put(24.5,36){\line(1,0){12.5}}
\qbezier(20.5,38.5)(20.5,36)(23,36)
\put(20.5,38.5){\line(0,1){2}}

\put(-8.8,15){Figure 2.2 \quad{Both complements have group $I^*$}}
\end{picture}

If the embedding is nilpotent then $j_\Delta$ induces epimorphisms on all corresponding quotients of the lower central series, 
since $H_1(j_\Delta)$ is an isomorphism.
However these quotients are rarely isomorphic.

\begin{theorem}
\label{Hi17-thm6.2}
If $\pi/\gamma_3^\mathbb{Q}\pi\cong
(\pi_X)/\gamma_3^\mathbb{Q}\pi_X)\times(\pi_Y)/\gamma_3^\mathbb{Q}\pi_Y)$
then $\chi(X)=1-\beta$ or $3-\beta$.
\end{theorem}

\begin{proof}
We use the fact that if $G$ is a group then the kernel of cup product 
$\cup_G$ from $\wedge^2H^1(G;\mathbb{Q})$ to $H^2(G;\mathbb{Q})$ 
is isomorphic to $\gamma_q^\mathbb{Q}G/\gamma_3^\mathbb{Q}G$.
(See \cite[\S12.2]{AIL}.)
Hence the rank of $\gamma_2^\mathbb{Q}G/\gamma_3^\mathbb{Q}G$ lies between 
$\binom{\beta_1(G)}2-\beta_2(G)$ and $\binom{\beta_1(G)}2$.

If the 2-step quotients ($G/\gamma_3^\mathbb{Q}G$) are isomorphic then so are their commutator subgroups
$\gamma_2^\mathbb{Q}\pi/\gamma_3^\mathbb{Q}\pi\cong
(\gamma_2^\mathbb{Q}\pi_X/\gamma_3^\mathbb{Q}\pi_X)\times
(\gamma_2^\mathbb{Q}\pi_Y/\gamma_3^\mathbb{Q}\pi_Y)$.
Let $\gamma=\beta_1(X)$. 
Then $\gamma\geq\frac\beta2$, and $\beta_2(\pi)\leq\beta$, 
so the above bounds give
\[
\binom\beta2-\beta\leq\binom\gamma2+\binom{\beta-\gamma}2.
\]
This reduces to $\beta\geq\gamma(\beta-\gamma)$, and so
either $\gamma\geq\beta-1$ or $\beta=4$ and $\gamma=2$.
In the latter case, consideration of $\mu_M$ shows that the rank of 
$\gamma_3^\mathbb{Q}\pi/\gamma_3^\mathbb{Q}\pi$ 
is at least $3\not=\binom22+\binom22$, 
so this cannot occur.
Thus $\chi(X)=1+\beta-2\gamma\leq3-\beta$.
\end{proof}

If $j$ is any embedding with $\chi(X)=1-\beta$ 
then $H_2(j_\Delta;\mathbb{Q})$ is an epimorphism,
and so $j_\Delta$ induces isomorphisms on all quotients of
the rational lower central series.
(If $H_1(Y)=0$ then $\pi/\gamma_n\pi\cong\pi_X/\gamma_n\pi_X$, for all $n$.)

If $F$ is a closed orientable surface then the embedding $j$ of
$M\cong{F}\times{S^1}$ as the boundary of a regular neighbourhood 
of the standard unknotted embedding of $F$ in $S^4$ has 
$\chi(X)=3-\beta$ and $j_\Delta$ an isomorphism.

The cases when $j_\Delta$ is an isomorphism are quite rare.

\begin{lemma}
If $j_\Delta$ is an isomorphism then either $M\cong{F}\times{S^1}$ for some 
aspherical closed orientable surface $F$ or $M\cong\#^r(S^2\times{S^1})$ for some $r\geq0$.
\end{lemma}

\begin{proof}
If $\pi\cong\pi_X\times\pi_Y$ with $\pi_X$ infinite and $\pi_Y\not=1$ then
$M\cong{F}\times{S^1}$ for some aspherical closed orientable surface $F$ \cite{Ep61}.
If $\pi_Y=1$ then $j_{X*}$ is an isomorphism, 
and so $\pi$ is a free group \cite{Da94}.
Hence $M\cong\#^r(S^2\times{S^1})$ for some $r\geq1$.
Finally, if $\pi_X$ and $\pi_Y$ are both finite and have non-trivial abelianization
then their orders have a common prime factor $p$, and so $\pi$ has
$(\mathbb{Z}/p\mathbb{Z})^2$ as a subgroup, which is not possible.
We may also exclude $\pi_X\cong\pi_Y\cong{I^*}$, for a similar reason,
and so there remains only the case $\pi=1$, when $M=S^3=\#^0(S^2\times{S^1})$.
\end{proof}

These 3-manifolds do in fact have bi-epic embeddings with $j_\Delta$ an isomorphism.

If $\pi_X$ is a non-trivial proper direct factor of $\pi$ then 
$\pi\cong\pi_1F\times\mathbb{Z}$ for some closed orientable surface $F$,
and so $M\cong{F}\times{S^1}$.
In this case, either $F=S^2$ and $\pi_1X\cong\mathbb{Z}$ 
or $F$ is aspherical and $\pi_1X\cong\pi_1F$.

If $\pi_X$ is a free factor of $\pi$ then $M\cong{M_X}\#{M'}$, 
where $\pi_1M_X\cong\pi_X$. 
The degree-1 collapse of $M$ onto $M_X$ induces an isomorphism
$H_3(M)\cong{H_3(M_X)}$.
Since $M=\partial{X}$ the images of $[M]$ and hence of $[M_X]$ in $H_3(\pi_X)$ are 0. 
Hence $\pi_X$ is a free group (as in \cite{Da94}). 
In particular,
$\pi\cong\pi_X*\pi_Y$ only if $\pi$ is itself a free group,
and then $M\cong\#^\beta(S^2\times{S^1})$.

\section{Modifying the group}

We may modify embeddings by ``2-knot surgery" on a complementary region, 
as follows.
Let $N_\gamma$ be a regular neighbourhood in $X$ of a simple closed curve representing $\gamma\in\pi_X$.
Then $\overline{S^4\setminus{N_\gamma}}\cong{S^2\times{D^2}}$ 
contains $Y$ and $M$.
If $K$ is a 2-knot with exterior $E(K)$ then 
$\Sigma=\overline{S^4\setminus{N_\gamma}}\cup{E(K)}$ 
is a homotopy 4-sphere, and so is homeomorphic to $S^4$.
The complementary components to $M$ in $\Sigma$ are
$X_{\gamma,K}=\overline{X\setminus{N_\gamma}}\cup{E(K)}$ and $Y$.
This construction applies equally well to simple closed curves in $Y$.
We shall say that a 2-knot surgery is {\it proper\/} if $\gamma$ is essential in $X$ (or $Y$).

When $M=S^2\times{S^1}$ is embedded as the boundary of 
a regular neighbourhood of the trivial 2-knot, 
with $X=D^3\times{S^1}$ and $Y=S^2\times{D^2}$,
the core $S^2\times\{0\}\subset{Y_1}$ is $K$,
realized as a satellite of the  trivial knot.
This construction gives all possible embeddings of $S^2\times{S^1}$ 
in $S^4$ (up to equivalence), 
by the result of Aitchison cited above.

Let $t$ be the image of a meridian for $K$ in the knot group $\pi{K}=\pi_1E(K)$.
If $\gamma$ has infinite order in $\pi_X$ then 
$\pi_1X_{\gamma,K}$ is a free product with amalgamation $\pi_X*_\mathbb{Z}\pi{K}$; 
if it has finite order $c$ then 
$\pi_1X_{\gamma,K}\cong
\pi_X)*_{\mathbb{Z}/c\mathbb{Z}}(\pi{K}/\langle\langle{t^c}\rangle\rangle)$.
(Note that if $K=\tau_ck$ is a non-trivial twist spin then
$\pi{K}/\langle\langle{t^c}\rangle\rangle\cong\pi{K}'\rtimes
\mathbb{Z}/c\mathbb{Z}$.)

If $\gamma=1$ then any simple closed curve representing $\gamma$
is isotopic to one contained in a small ball,
since homotopy implies isotopy for curves in 4-manifolds.
Hence in this case 2-knot surgery does not change the topology of $X$.

It is well known that  a nilpotent group with cyclic abelianization is cyclic.
It follows that the natural projection of $\pi_1X_{\gamma,K}$ onto $\pi_X$ 
induces isomorphisms of corresponding quotients by terms of the lower central series.
Thus we cannot distinguish these groups by such quotients.
Nevertheless, we have the following result.

\begin{theorem}
If $\pi_X\not=1$ then there are infinitely many groups of the form $\pi_1X_{\gamma,K}$.
\end{theorem}

\begin{proof} 
Suppose first that $\pi_X$ is torsion-free and that $\gamma\not=1$.
If $\pi{K}\cong\mathbb{Z}/n\mathbb{Z}\rtimes\mathbb{Z}$ then
$\pi_1X_{\gamma,K}\cong\pi_X*_\mathbb{Z}\pi{K}$ is an extension 
of a torsion-free group by the free product of countably many copies 
of $\mathbb{Z}/n\mathbb{Z}$.
Since $\mathbb{Z}/n\mathbb{Z}\rtimes\mathbb{Z}$ 
is the group of the 2-twist spin of a 2-bridge knot, 
for every odd $n$, the result follows.

If $\pi_X$ has an element $\gamma$ of finite order $c>1$ then we use instead
Cappell-Shaneson 2-knots.
Let $a$ be an integer, and let $f_a(t)=t^3-at^2+(a-1)t-1$.
If $a>5$ the roots $\alpha,\beta$ and $\gamma$ of $f_a$ are real,
and we may assume that $\gamma<\beta<\alpha$.
Elementary estimates give the bounds
\[
\frac1a<\gamma<\frac12<\beta<1-\frac1a<a-2<\alpha<a.
\]
If $A\in{SL(3,\mathbb{Z})}$ is the companion matrix of $f_a$ then 
$\mathbb{Z}^3\rtimes_A\mathbb{Z}$ is the group of a ``Cappell-Shaneson" 2-knot  $K$.
The quotient $\mathbb{Z}^3/(A^c-I)\mathbb{Z}^3$ is a finite group of order
the resultant $Res(f_a(t),t^c-1)=(\alpha^c-1)(\beta^c-1)(\gamma^c-1)$,
where $\alpha,\beta$ and $\gamma$ are the roots of $f_a(t)$.
This simplifies to
\[
\alpha^p+\beta^p+\gamma^p-(\alpha\beta)^p-(\beta\gamma)^p-(\gamma\alpha)^p
=\alpha^p(1-\beta^p-\gamma^p)+\varepsilon,
\]
where $0<\varepsilon<2$.
It follows easily from our estimates that $|Res(f_a(t),t^c-1)|>a^{c-1}$, if $a>3c$.
Hence $\pi{K}/\langle\langle{t^c}\rangle\rangle$ is a finite group of order $>ca^{c-1}$.
We then use the fact that finitely presentable groups have an essentially unique representation as the fundamental group of a graph of groups, 
with all vertex groups finite or one ended \cite[Prop. IV.7.4]{DD}.)
Thus if $K$ and $L$ are two such 2-knots such that
$\pi{K}/\langle\langle{t^c}\rangle\rangle$ and
$\pi{L}/\langle\langle{t^c}\rangle\rangle$ are finite groups of different orders,
both greater than that of any of the finite vertex groups in such a representation of $\pi_X$
then $\pi_1X_{\gamma,K}\not\cong\pi_1X_{\gamma,L}$.
\end{proof}

The following observation can be construed as a minimality condition, 
since it shows that bi-epic embeddings cannot be obtained from other embeddings 
by non-trivial 2-knot surgery.

\begin{lemma}
\label{knot surg not biepic}
Let $M$ be a closed $3$-manifold with an embedding $j:M\to{S^4}$,
and let $J=j_{K,\gamma}$ be the embedding obtained from $j$
by a proper $2$-knot surgery using the $2$-knot $K$ 
and the loop $\gamma\in\pi_{X(j)}$.
Then $J$ is not bi-epic, and $\pi_{X(J)}$ is not restrained,
unless $\pi_{X(J)}$ is itself a restrained $2$-knot group, 
in which case $\beta_1(M;\mathbb{Q})=1$ or $2$.
\end{lemma}

\begin{proof}
Let $C\cong\mathbb{Z}/q\mathbb{Z}$ be the subgroup of $\pi_{X(j)}$
generated by $\gamma$, and let $t$ be a meridian for 
the knot group $\pi{K}$.
Then 
\[
\pi_{X(J)}\cong\pi_{X(j)}*_C\pi{K}/\langle\langle{t^q}\rangle\rangle.
\]
Since the 2-knot surgery is proper,
$\langle\langle{t^q}\rangle\rangle$ is a proper normal subgroup 
of $\pi{K}$.
Since the image of $\pi_1M$ lies in $\pi_{X(j)}$,
the embedding $J$ cannot be bi-epic.
Moreover $\pi_{X(J)}$ can only be restrained if $\pi_{X(j)}\cong\mathbb{Z}$,
in which case $\pi_{X(J)}\cong\pi{K}$ and
$\chi(X(j))=0$ or 1, and so $\beta_1(M;\mathbb{Q})=1$ or 2.
\end{proof}

\begin{cor}
Embeddings given by proper $2$-knot surgery on another embedding
are not of the form $j_L$ for any $0$-framed bipartedly ribbon link $L$.
\qed
\end{cor}

If $H_1(M)\not=0$ then $X$ is not simply-connected,
and so  we may use 2-knot surgery to construct infinitely many embeddings 
with one complementary region $Y$ 
and distinguishable by the fundamental groups of the other region.
However if $M$ is a homology sphere then $X$ and $Y$ are homology balls,
and it may not be easy to decide whether $\pi_X$ and $\pi_Y$ are non-trivial.
When $M=S^3$ the complementary regions are homeomorphic to the 4-ball $D^4$,
by the Brown-Mazur-Schoenflies Theorem.
If $\pi_1M\not=1$ is there an homology 4-ball $X$ with $M\cong\partial{X}$,
$\pi_X\not=1$ and the normal closure of the image of $\pi_1M$ in $\pi_X$ 
being the whole group?
If so, there is an embedding with one complementary region $X$ and the other 1-connected.

Perhaps the simplest non-trivial example of a smooth embedding of an homology
3-sphere with neither complementary region 1-connected is given by the link 
in Figure 2.2.
If we swap the 0-framings and the dots, we obtain a Kirby-calculus presentation for $Y$.
Since the loops $r$, $s$, $x$ and $y$ determine words $x^{-2}yxy$, $y^{-4}xyx$,
$srsr^{-2}$ and $s^{-4}rsr$, respectively,  $\pi_X$ and $\pi_Y$ have equivalent presentations, and $\pi_X\cong\pi_Y\cong{I^*}$,
the binary icosahedral group.

\section{Smooth embeddings}

In this section we shall comment briefly on some closely related problems: 
embeddings of punctured 3-manifolds in $S^4$ and
smooth embeddings in $S^4$ and in other 4-manifolds.
We refer to the discussions of Problems 3.20, 4.2 and 4.5 in \cite{Ki}
and the surveys \cite{APM24}, \cite{BB22} and \cite{Sa24}
for more details and further references.

As observed above,  most known constructions 
(other than those using surgery) give smooth embeddings.
On the other hand,  the arguments that we use to show that a 3-manifold 
does not embed in $S^4$ usually show that 
it does not even have a Poincar\'e embedding in $S^4$.

If $M$ embeds then so does the punctured manifold $M_o$. 
Conversely, if $M_o$ embeds then $M\#-M$ embeds as the boundary of a regular neighbourhood of $M_o$ in $S^4$.
If $M_o$ and $N_o$ each embed then so does $(M\#N)_o=M_o\natural{N_o}$;
the converse is clear.
Thus for the question of whether $M_o$ embeds we may assume $M$ irreducible.
No lens space $L(p,q)$ embeds (since $\tau_{L(p,q)}$ is cyclic),
but if $p$ is odd then $L(p,q)_o$ embeds smoothly as the fibre of a fibration 
of the complement of a twist spun 2-bridge knot \cite{Ze65},
and so $L(p,q)\#-L(p,q)$ embeds smoothly.
If $p$ is even then $L(p,q)_o$ does not embed \cite{Ep65}.
(More generally,
if $M_o$ embeds then the 2-primary summand of $\tau_M$ 
is a direct double \cite{Ka79}.)
Thus whether a closed manifold embeds does not reduce to the embeddability of its summands.

Much of the work on smooth embeddings has concentrated on 
$\mathbb{Q}$-homology 3-spheres,
in particular on homology spheres and connected sums of lens spaces.
Aarguments based on the intersection pairings of orientable 4-manifolds 
with given boundary have been prominent.
For instance, 
since $S^3/I^*$ bounds the $E_8$-plumbing \cite[page 153]{GS},
it has no smooth embedding, 
by Rochlin's Theorem on signatures of smooth $Spin$ 4-manifolds.
The Diagonalization Theorem of Donaldson (on smooth 4-manifolds with definite intersection pairing \cite{Do87}) has had striking consequences.
After early partial results \cite[Proposition 6.1]{KK80} and \cite[Theorem 3.4]{GL83},
A. Donald settled a long standing question by 
showing that the only connected sums 
of lens spaces which embed are sums of the type just described.

\begin{thm}
\cite[Theorem 1.1]{Do15}
Let $L=\#_{i=1}^nL(p_i,q_i)$ be a connected sum of lens spaces.
Then $L$ embeds smoothly in $S^4$ if and only if each $p_i$ is odd and 
there is a closed $3$-manifold $N$ such that $L\cong{N}\#-N$.
\qed
\end{thm}
 
Donald,  and A.Issa and D.McCoy have also applied the Diagonalization Theorem
to study the smooth embeddability of Seifert manifolds. 
We shall outline their results in the final sections of Chapters 3 and 4.

C.McDonald has revealed a further subtlety of the smooth case:
there are infinitely many homology spheres 
which embed smoothly in homology 4-spheres, 
but which have no smooth embedding in any smooth homotopy 4-sphere \cite{Mc22}.
This contrasts strongly with the TOP locally flat case, 
where Freedman has shown that every homology sphere embeds in $S^4$.
McDonald's argument also uses the Diagonalization Theorem.

Closed orientable 3-manifolds have trivial tangent bundles, and so are $Spin$ manifolds.
Each $Spin$ structure $\mathfrak{s}$ on $M$ determines a Rochlin invariant 
$\bar\mu(M,\mathfrak{s})\in\mathbb{Q}/2\mathbb{Z}$.
Similarly,  $\mathbb{Q}$-homology 3-spheres have $Spin^c$ structures, 
and Heegaard-Floer theory has been used to define $\mathbb{Q}$-valued
functions $d$ on the set of $Spin^c$ structures of such manifolds \cite{OS03}.
There are 149 manifolds in the census of 3-manifolds built from at most 11 tetrahedra 
and with hyperbolic linking form.
Of these, 41 are known to embed smoothly in $S^4$, 
and 4 more embed smoothly in some homotopy 4-sphere.
The Rochlin $\bar\mu$-invariant and Ostv\'ath-Sz\'abo $d$-invariant have been used 
to show that 67 of the remaining 104 cannot embed smoothly in $S^4$ \cite{BB22}.

The earliest work on embeddings of 3-manifolds in other 4-manifolds may be \cite{Ka88}.
Kawauchi considers embeddings of 3-manifolds $M$ with $H_1(M)$ infinite
and obtains bounds for signatures associated to infinite cyclic covers of $M$,
in terms of the Betti numbers of a 4-manifold $W$ with $M=\partial{W}$.

Beyond this,
little consideration has been given to 3-manifolds other than lens spaces in other 4-manifolds, according to \cite{APM24}.
Every lens space embeds in $\#^n\overline{\mathbb{CP}}^2$, 
for $n$ sufficiently large \cite{EL96}.
For each $n>0$ there are lens spaces that do not embed smoothly as a separating hypersurface into any smooth negative-definite 4-manifold $W$ with $\beta_2(W)=n$ \cite{AMP22}.
This again uses the Diagonalization Theorem.
However, if the lens space $L$ is amphicheiral then it embeds in a 4-manifold $W$ with
$\beta_1(W)=0$ and $\beta_1(W)=2$  \cite{APM24}.
A key idea is to characterize such lens spaces in terms of their canonical plumbing graph.

Embeddings of  homology spheres in homology 4-spheres leads naturally 
to the study of homology cobordism \cite{Sa24}.
There are parallel questions about $R$-homology cobordism of $R$-homology 3-spheres;
the cases $R=\mathbb{F}_2$ and $R=\mathbb{Q}$ are of greatest interest.

%% file: e3.tex
\chapter{3-manifolds with $S^1$-actions}

The class of Seifert manifolds is in many respects well-understood, 
and has a natural parametrization in terms of Seifert data, 
and so we might expect criteria for embedding in terms of such data.
We first review the notion of Seifert manifold and Seifert invariants.
In this chapter we shall consider orientable Seifert manifolds which
are Seifert fibred over orientable base orbifolds.
These may also be described as (orientable) 3-manifolds of $S^1$-action type.
(We shall consider 3-manifolds which are Seifert fibred over nonorientable 
base orbifolds in Chapter 4.)
Our goal is to show that skew-symmetry of the Seifert data is a necessary and sufficient condition for a Seifert manifold $M$ with generalized Euler number $\varepsilon(M)=0$ 
to embed in $S^4$.
We are partially successful.
In \S2 we show that if the Seifert data of $M$ is skew-symmetric and 
all the cone point orders are odd then $M$ embeds smoothly in $S^4$.
Our main result (in \S3) is that if the base orbifold is $S^2(\alpha_1,\dots,\alpha_r)$, 
where all the cone point orders $\alpha_i$ are odd, 
and  $\varepsilon(M)=0$ , 
then $M$ embeds in $S^4$ if and only if the Seifert data is skew-symmetric.
In \S4 we outline briefly some apparent difficulties in extending this result.
Motivated by the results of Hantzsche and of Kawauchi and Kojima described in \S2.3,
we determine the torsion subgroup $\tau_M$ (for $M$ of $S^1$-action type) in \S5, 
and show that if $M$ is Seifert fibred over $T_g(\alpha_1,\dots,\alpha_r)$ and $\tau_M$ is a nonzero direct double then $|\varepsilon(M)|$ is determined by the cone point orders 
$\{\alpha_1,\dots,\alpha_r\}$.
A related calculation for $M$ Seifert fibred over a nonorientable base orbifold plays an essential  role in Chapter 4.
In the final section we outline some recent work by A. Donald and by
A.  Issa and D. McCoy on smooth embeddings.

\section{Seifert manifolds}

In this book a {\it Seifert manifold} is a closed orientable 3-manifold $M$
with an orbifold $S^1$-fibration $p:M\to{B}$ over a 2-orbifold $B$.
Since $M$ is orientable, the base $B$ has finitely many cone point singularities,  
and no reflector curves.
If  the surface $|B|$ underlying the base orbifold $B$ is orientable and of genus $g$
then $B=T_g(\alpha_1,\dots,\alpha_r)$, where $\alpha_i$ is the order of the $i$th cone point.
(The cone point order is also known as the multiplicity of the associated exceptional fibre.)
The {\it Seifert data\/} for $M$ is then a finite string of ordered pairs
$S=((\alpha_1,\beta_1),\dots,(\alpha_r,\beta_r))$,
where $(\alpha_i,\beta_i)=1$, for all $1\leq{i}\leq{r}$,
and we shall write $M=M(g;S)$.
If the base is non-orientable, so that $|B|=\#^c\mathbb{RP}^2$ for some $c>0$,
we shall write $M=M(-c;S)$.
(We may refer to $S$ as a Seifert data {\it set\/}, but the multiplicities
of the pairs $(\alpha,\beta)$ are significant.)

Our notation is based on that of \cite{JN}. 
(In particular, we do not assume that $0<\beta_i<\alpha_i$.)
We shall allow also cone points of order $\alpha=1$.
Such points are nonsingular, 
but are useful in that they allow a uniform notation which includes 
the Seifert fibrations associated to $S^1$-bundles.
The reference \cite{JN} also considers a {\it generalized Seifert manifold\/},
in which fibres of type $(0,1)$ are allowed as well,
corresponding to fibres fixed pointwise under a circle action.
We use this definition, 
as a notational convenience only, 
in proving Lemma \ref{CH-lem3.1} and Theorem \ref{CH-thm1.1}.

If $p:E\to{F}$ is an $S^1$-bundle with base a closed surface $F$
and orientable total space $E$ then $\pi_1F$ acts on the fibre 
via $w=w_1(F)$, 
and such bundles are classified by an Euler class $e(p)$ in $H^2(F;\mathbb{Z}^w)\cong\mathbb{Z}$.
If we fix a generator $[F]$ for $H_2(F;\mathbb{Z}^w)$ we may define
the Euler number of the bundle by $e=e(p)\cap[F]$.
(We may change the sign of $e$ by reversing the orientation of $E$.)
Let $M(g;(1,-e))$ and $M(-c;(1,-e))$ be the total spaces of the $S^1$-bundles
with base $T_g$ and $\#^cRP^2$ (respectively), and Euler number $e$.
The {\it generalized Euler number\/} of a Seifert fibration $p:M\to{B}$ is
\[
\varepsilon_S=-\Sigma_{i=1}^r\frac{\beta_i}{\alpha_i}.
\]
%Two Seifert data sets $S$ and $S'$ with the same base orbifolds
%give rise to homeomorphic 3-manifolds if $\varepsilon_{S'}=\varepsilon_S$ and
%$\beta_i'-\beta_i$ is divisible by $\alpha_i$, for all $i\leq{n}$.

There is an orientation- and fibre-preserving homeomorphism between 
any two Seifert manifolds with the same base orbifolds
if and only if their Seifert data are equivalent under a finite sequence
of the following operations
\begin{enumerate}
\item{}add or delete any pair $(1,0)$;
\item replace each pair $(\alpha_i,\beta_i)$ by $(\alpha_i,\beta_i+c_i\alpha_i)$,
where $\Sigma_{i=1}^rc_i=0$;
\item{}permute the indices.
\end{enumerate}
Every  Seifert data set is equivalent to one of the  form
$S=S'\cup\{(1,-e)\}$, where $S'=((\alpha_1,\beta_1),\dots,(\alpha_s,\beta_s))$
is {\it strict\/} Seifert data,
with $0<\beta_i<\alpha_i$ for all $i\leq{s}$.
After reversing the orientation of the general fibre,
if necessary,  we may assume that $\varepsilon_S\geq0$.
In virtually all the cases of interest to us the Seifert fibration is
unique, and so we shall usually write $\varepsilon(M)$ for $\varepsilon_S$.

Lens spaces,
the manifolds $M(-1;(\alpha,\beta))\cong{M(0;(2,1),(2,-1),(-\beta,\alpha))}$
and the flat 3-manifold $M(-2;(1,0))\cong{M(0;(2,1),(2,1),(2,-1),(2,-1))}$
admit more than one Seifert fibration \cite{JN, Orl}.

Seifert manifolds $M=M(k;S)$ with $\varepsilon(M)=0$
are generically $\mathbb{H}^2\times\mathbb{E}^1$-manifolds,
while those with $\varepsilon(M)\not=0$ are generically $\widetilde{\mathbb{SL}}$-manifolds.
The exceptions have virtually solvable fundamental groups
and small Seifert data; 
the base orbifold has at most 4 singularities.
They are the spherical manifolds $S^3/G$, the two $\mathbb{S}^2\times\mathbb{E}^1$-manifolds $S^2\times{S^1}$ and $\mathbb{RP}^3\#\mathbb{RP}^3$, 
six (orientable) flat 3-manifolds and the $\mathbb{N}il^3$-manifolds.
We shall settle the question of which of these embed in $S^4$ in Chapter 5.

The {\it fibred sum\/} of two Seifert manifolds $M=M(k;S)$ and $M'=M(k';S')$ 
is defined as follows.
Let $N$ and $N'$ be regular neighbourhoods of regular fibres in $M$ and $M'$,
and let $h:\partial{N}\to\partial{N'}$ be an orientation-reversing, fibre-preserving homeomorphism.
Then  $M\sharp_fM'=(M\setminus{int(N)})\cup_h(M'\setminus{int(N')}$ is Seifert fibred
over the connected sum of the base orbifolds.
In fact $M\sharp_fM'=M(k_\#;S\sqcup{S'})$,
where $k_\#=k+k'$ if $k$ and $k'$ have the same sign,
and $k_\#=-2k-k'$ if $k>0$ and $k'<0$.
If $M'=M(0;(\alpha_i,\beta_i),(\alpha_i,-\beta_i))$ for some $(\alpha_i,\beta_i)\in{S}$
then  $M\sharp_fM'$ is obtained from $M$ by {\it expansion\/}, in the terminology of \cite{IM20}.

As $B$ is the connected sum of $|B|$ and $S^2(\alpha_1,\dots,\alpha_r)$,
the Seifert manifold $M$ is a fibre sum $M=M(0;S)\sharp_fM(k;\emptyset)$.
(If $B$ is orientable then $M(g;\emptyset)\cong{T_g\times{S^1}}$.)
If the Seifert data is given in normalized form $S=S'\cup\{(1,e)\}$ 
then we also have $M=M(0;S')\#_fM(k;(1,e))$, 
where $M(k;(1,e))$ is the total space of an $S^1$-bundle  over $F$,
with Euler number $e$.

If $h$ is the image of the regular fibre in $\pi=\pi_1M$
then the subgroup generated by $h$ is normal in $\pi$,
and $\pi^{orb}(B)\cong\pi/\langle{h}\rangle$.
An orientable 3-manifold admits a fixed-point free $S^1$-action 
if and only if it is Seifert fibred over an orientable base orbifold,
and then $h$ is central in $\pi$ \cite{Orl}.
All Seifert manifolds considered in this chapter are of this type.

\section{An easy embedding}

\begin{defn} 
The Seifert data $S$  is {\it skew-symmetric\/} if it is equivalent to Seifert data 
of the form $((\alpha_1,\beta_1),\dots,(\alpha_r,\beta_r))$, 
where $r$ is even, $\alpha_{2i-1}=\alpha_{2i}$,
$0<\beta_i<\alpha_i$ for all $i$ and $\beta_{2j-1}=-\beta_{2j}$ for $1\leq{j}\leq{r/2}$.
\end{defn}

Our goal is to show that skew-symmetry of the Seifert data is a necessary and sufficient condition for a Seifert manifold $M$ with $\varepsilon(M)=0$ to embed in $S^4$.

\begin{lemma}
\label{CH-lem3.1}
Let $S$ be skew-symmetric Seifert data for which all cone point orders 
$\alpha_i$ are odd.
Then $M(0;S)$ embeds smoothly in $S^4$.
\end{lemma}

\begin{proof}
Let $L_i=M(0;(\alpha_i,\beta_i), (0,1))$. 
(This is the lens space $L(\alpha_i,\beta_i)$, 
expressed as a generalised Seifert manifold.)
Let $L_{io}$ be the complement of a small open ball about a point 
on the fibre of type $(0,1)$,
and let $L_i^*$ be the complement of an open regular neighbourhood
of the whole fibre.
Then $L_{io}\cong{L_i^*}\cup_{f\times{I}}{D^2\times{I}}$,
where $f$ is a regular fibre of $L_i^*$, $I=[0,1]$ 
and the identification maps $\partial{D^2}$ to $f$.

Since $\alpha_i$ is odd, $L_i$ is a 2-fold branched cover of $S^3$, branched over a knot
$k_i$, and so $L_{io}$ embeds smoothly in $S^4$ as the fibre of the 2-twist spin  
$\tau_2k_i$ \cite{Ze65}.
In fact $L_{io}$ has a product neighbourhood 
$L_{io}\times[-\epsilon,\epsilon]$ in the knot exterior $X(\tau_2k_i)$.
This neighbourhood may be written as a union 
$L_i^*\times[-\epsilon,\epsilon]\cup(D^2\times{I}\times[-\epsilon,\epsilon])$,
and so we obtain a smooth embedding of 
$M(S^2;((\alpha_i,\beta_i),(\alpha_i,-\beta_i)\cong
\partial(L_i^*\times[-\epsilon,\epsilon])$ in $S^4$.
Furthermore,  this embedding has the property that there exists a 4-ball 
$B(M_i)=D^2\times{I}\times[-\epsilon,\epsilon]$ embedded in $S^4$ such that
$M_i\cap{B(M_i)}=\partial{D^2\times{I}\times[-\epsilon,\epsilon]}$,
and each circle $\partial{D^2}\times(r,s)$ is a regular fibre of $M_i$.

Given embeddings of  two manifolds $M$ and $M'$ in $S^4$, 
each with this property,
we may construct an embedding of $M\sharp_fM'$, 
also with this property,
by taking connected sums of the 4-spheres.
Remove the balls
$\frac12B(M)=D^2\times{I}\times[-\epsilon,\epsilon]$
and $\frac12B(M')$ from the corresponding spheres,.
The sum is then the union of the closures of the remaining components 
(with embedded pieces), in the obvious way.
Thus, by induction,
we obtain a smooth embeddding of $M_1\sharp_f\dots\sharp_fM_n$.
\end{proof}

The above result holds also if $\alpha_{2i-1}=\alpha_{2i}$ is even, 
for one value of $i$ \cite[Proposition 7.8]{IM20}.

\begin{lemma}
\label{CH-lem3.2}
If $M=M(k;S)$ embeds smoothly into a $4$-manifold $W$,
then so does $M\sharp_fM(g;\emptyset)$.
\end{lemma}

\begin{proof}
Let $f$ be a regular fibre of $M$ smoothly embedded in $W$, and let $D$ be a 
3-disc with equatorial circle $C\subset\partial{D}$.
Then $f$ has a neighbourhood $D\times{f}$ in $W$ which is consistently fibred 
with $M$ so that $N=M\cap(D\times{f})$ is a neighbourhood of $f$ in $M$ 
with boundary $\partial{N}=C\times{f}$.
Write $W^*=W\setminus{int(D\times{f})}$ and $M^*=M\setminus{int(N)}$.
Then we have a smooth embedding $\phi_M:M^*\to{W^*}$
such that the restriction $\partial{M^*}\to\partial{W^*}=\partial{D}\times{f}$
is a fibre-preserving map with image $C\times{f}$.
On the other hand, write $T_g^*=T_g\setminus{int(D^2)}$.
Since there are smooth embeddings of $T_g^*$ in $D$ with 
$\mathrm{Im}(\partial{T_g^*} )=C$,
there is a smooth fibre-preserving embedding
$\phi_T:T_g^*\times{f}\to{D\times{f}}$ with $\mathrm{Im}(\partial{T_g^*}\times{f})=C\times{f}$.
Together $\phi_M$ and $\phi_T$ give a smooth embedding of
$M\sharp_fM(g;\emptyset)$ into $W=W^*\cup(D\times{f})$.
\end{proof}

See \cite[Proposition 7.2]{IM20} for an alternative account (ascribed to Donald), which uses Kirby calculus to show that $M(g+1;S)$ embeds in $M(g;S)\times[-1,1]$.

It would be convenient to have a converse to this stabilization result.
However, when the base is non-orientable there is no such cancellation.
See Chapter 3.

Together Lemmas \ref{CH-lem3.1} and \ref{CH-lem3.2} imply the following.

\begin{cor}
Let $S$ be skew-symmetric Seifert data for which all cone point orders 
$\alpha_i$ are odd.
Then $M(g;S)$ embeds smoothly in $S^4$ for all $g\geq0$.
\qed
\end{cor}

\begin{lemma}
\label{CH-lem3.3}
If $g\geq0$ and $e=0$ or $\pm1$ then  $M(g;(1,e))$ embeds smoothly in $S^4$.
\end{lemma}

\begin{proof}
The product $M(g;(1,0))\cong{T_g\times{S^1}}$ embeds smoothly 
as the boundary of any smooth embedding of $T_g$ in $S^4$.
Since $M(0;(1,\pm1))\cong{S^3}$, it embeds smoothly as the equator of $S^4$.
The lemma now follows from Lemma \ref{CH-lem3.2}.
\end{proof}

\section{$\mathbb{H}^2\times\mathbb{E}^1$-manifolds with $g=0$}

When the Seifert data is skew-symmetric, all cone point orders are odd 
and $\varepsilon(M)=0$ then $M$ embeds smoothly in $S^4$, by Lemma \ref{CH-lem3.1}.
We shall show that when $g=0$, $\varepsilon(M)=0$ and all cone point orders 
are odd then skew-symmetry is a necessary condition for $M(0;S)$ 
to embed in a homology 4-sphere. 
If $r\leq2$ then $M$ must be $S^3$ or $S^2\times{S^1}$, 
while if $r\geq4$ then $M$ is a $\mathbb{H}^2\times\mathbb{E}^1$-manifold.

If the $S^1$-action extends to a fixed-point free action onto 
a complementary region $W$ then there is a direct, geometric argument.
For the exceptional orbits with non-trivial isotropy subgroup 
have even codimension and are foliated by circles.
Therefore they are tori (in the interior of the region) and annuli
with boundary components exceptional fibres of $M$, and $\chi(W)=0$.
Consideration of relative orientations implies that the Seifert data of
the boundary components of such an annulus have the form 
$\{(\alpha,\beta),(\alpha,-\beta)\}$.

\begin{theorem}
\label{Hi09-thm4.1}
Let $M=M(g;S)$, 
and let $\phi:\pi_1M\to\mathbb{Z}$ be an epimorphism
such that $\phi(h)\not=0$.
If $b_\phi$ is neutral and all cone point orders  are odd then $S$ is skew-symmetric.
\end{theorem}

\begin{proof}
We note first that $\varepsilon(M)=0$ since the image of $h$ in $H_1(M)$ has infinite order.
Let $\sigma=|\phi(h)|$.
Since $h$ is central 
$\phi^{-1}(\sigma\mathbb{Z})\cong\mathrm{Ker}(\phi)\times\mathbb{Z}$, 
and so the covering space associated to this subgroup 
is a product $F\times{S^1}$,
where $F$ is a closed surface.
Then $M$ fibres over $S^1$ with fibre $F$ and monodromy $\theta$ 
of order $\sigma$,
and so the base orbifold $B$ is the quotient of $F$ by an effective action of 
$G=\mathbb{Z}/\sigma\mathbb{Z}$.

Let $\varphi:\pi^{orb}(B)=\pi_1M/\langle{h}\rangle\to\mathbb{Z}/\sigma\mathbb{Z}$ 
be the epimorphism induced by $\phi$,
and let $g\in{G}$ have image $\varphi(g)=[1]$.
Then $F$ is the covering space associated to $\mathrm{Ker}(\varphi)$.
The points $P$ with non-trivial isotropy subgroup $G_P$ 
lie above the cone points of $B$, 
and the representation of the isotropy subgroup 
$G_P=\langle{g}^{\sigma/\alpha_i}\rangle$ 
on the tangent space $T_P$ determines and is determined by the 
Seifert invariant $(\alpha_i,\beta_i)$,
corresponding to the $i$th cone point of $B$,
since $\varphi(q_i)=-[\frac\sigma\alpha_i\beta_i]$,
for all $i\leq{r}$.

Since $M_\phi\cong{F}\times\mathbb{R}$ the Blanchfield pairing $b_\phi$ 
on $H_1(M_\phi;\mathbb{Q})$ reduces to the intersection pairing $I_F$ 
on the fibre $F$, 
together with the isometric action of $G=Aut(F/B)\cong\mathbb{Z}/\sigma\mathbb{Z}$.
Let $s(n,k)=|\{ j:\alpha_j=n, \beta_j\equiv{k}~mod~(n)\}|$ and 
$K(m)=\{k:1\leq{k}<m/2, (k,m)=1\}$.
Then the equivariant signatures of the $G$-Signature Theorem are given by 
\[
sign(I_F,g^{\sigma/m})=i\Sigma_{k\in{K(m)}}(s(m,m-k)-s(m,k))\cot(\frac{\pi{k}}m).
\]
\cite[Theorem 6.27]{AB68}. (See also \cite{Go86} and \S1.9 above.)
 
If $b_\phi$ is neutral, its image in the Witt group 
$W_+(\mathbb{Q}(t),\mathbb{Q}\Lambda)$ is trivial.
The equivariant signature $sign(I_F,g^{\sigma/m})$ 
is an invariant of the Witt class of $b_\phi$, for each $m|\sigma$.
(See page 75 of \cite{Ne}.)
Since $b_\phi$ is neutral these signatures are 0.
If $m>2$, the algebraic numbers $\{\cot(\frac{\pi{k}}m):k\in{K(m)}\}$ 
are linearly independent over $\mathbb{Q}$ \cite{EE}.
Hence $s(m,m-k)=s(m,k)$ for all $k\in{K(m)}$ and $m>2$.
Since we may modify each $\beta_i$ by multiples of $\alpha_i$,
subject to $\varepsilon_S=0$, the Seifert data is skew-symmetric.
\end{proof}

The converse is also true, but we shall only sketch the argument, 
as we do not need the result.
Suppose that $S=((\alpha_1,\beta_1),\dots,(\alpha_{2s},\beta_{2s}))$ 
is skew-symmetric, and let $T=\{(\alpha_{2j},\beta_{2j})\mid j\leq{s}\}$.
Then $M\cong{N}\#_f-N\#_fM(g;\emptyset)$, where $N=M(0;T)$,
and $\phi$ induces nonzero homomorphisms $\phi_N:\pi_1N\to\mathbb{Z}$ 
and $\psi:\pi_1M(g;\emptyset)\to\mathbb{Z}$.
Let $n(f)$ be an open tubular neighbourhood of a regular fibre of $N$, 
and let $N_0=N\setminus{n(f)}$.
Then $N\#_f-N=N_0\cup_\partial-N_0$.
If $g=0$, then $\phi_N$ is an epimorphism, 
and the diagonal copy of $H_1(N_0;\mathbb{Q}\Lambda)$ in 
$H_1(M;\mathbb{Q}\Lambda)\cong
{H_1(N_0;\mathbb{Q}\Lambda)}\oplus{H_1(N_0;\mathbb{Q}\Lambda)}$
is a maximal self-annihilating submodule,
so $b_\phi$ is neutral.
In general, $\phi_N$ and $\psi$ need not be onto,
and we must allow for infinite cyclic covering spaces with finitely many
components.

The argument for the following Lemma extends that of \cite[Theorem 2.4]{BB22}.

\begin{theorem}
\label{Kawabeta1}
Let $W$ be a compact $4$-manifold with connected boundary $M$,
and let $\pi=\pi_1W$.
Suppose that $H^1(W)\cong{H^1(M)}\cong\mathbb{Z}$ and that $H_2(W)=0$.
Let $\phi:\pi\to\mathbb{Z}$  be an epimorphism representing a generator of $H^1(W)$.
Then the Blanchfield pairing on $H_1(M_\phi;\mathbb{Q})=H_1(M;\mathbb{Q}\Lambda)$ is neutral.
\end{theorem}

\begin{proof}
Consideration of the Wang sequences associated to the infinite cyclic covers
of $M$, $W$ and $(W,M)$ shows that $t-1$ acts epimorphically 
on $H_1(M;\mathbb{Q}\Lambda)$,
$H_1(W;\mathbb{Q}\Lambda)$ and $H_2(W;\mathbb{Q}\Lambda)$.
Since  $\mathbb{Q}\Lambda$ is noetherian and $W$ and $M$ are compact 
these modules are finitely generated,
and so are $\mathbb{Q}\Lambda$-torsion modules on which $t-1$ acts invertibly.
Poincar\'e duality then implies that the same is true for all the homology 
and cohomology modules of $M$, $W$ and $(W,M)$
with coefficients $\mathbb{Q}\Lambda$.

Let $j$ be the inclusion of $M$ as the boundary of $W$,
and let $\delta:H_2(W,M;\mathbb{Q}\Lambda)\to
{H=H_1(M;\mathbb{Q}\Lambda)}$ be the connecting homomorphism
in the long exact sequence of homology.
Let $P=\mathrm{Im}(\delta)$ be the image of $H_2(W,M;\mathbb{Q}\Lambda)$ in 
$H$.
Poincar\'e duality combined with the Universal Coefficient Theorem gives us isomorphisms of the three short exact sequences:
\begin{equation*}
\begin{CD}
0\to{P}@>>>{H=H_1(M;\mathbb{Q}\Lambda)}@>>>\mathrm{Im}(j_*)\cong{H/P}\to0\\
@VV{PD}V @VV{PD}V @VV{PD}V\\
0\to\overline{\mathrm{Im}(j^*)}@>>>\overline{H^2(M;\mathbb{Q}\Lambda)}@>>>
\overline{\mathrm{Im}(\delta^*)}\to0\\
@VV{UCT}V @VV{UCT}V @VV{UCT}V\\
0\to\overline{\mathrm{Im}(\mathrm{Ext}(j_*))}@>>>
\overline{\mathrm{Ext}_{\mathbb{Q}\Lambda}(H,\mathbb{Q}\Lambda)}@>>>\overline{\mathrm{Im}(\mathrm{Ext}(\delta))}\to0.
\end{CD}
\end{equation*}
Hence $\Delta_0(H/P)=\overline{\Delta_0(P)}$.
This is enough to show that the ``Alexander polynomial"
$\Delta_0(H)$ is a product $\lambda\overline\lambda$, where $\lambda=\Delta_0(P)$.

The central isomorphism from $H$ to 
$\overline{\mathrm{Ext}_{\mathbb{Q}\Lambda}(H,\mathbb{Q}\Lambda)}=
\mathrm{Hom}(H,\mathbb{Q}(t)/\mathbb{Q}\Lambda)$
is the adjoint $\tilde{b}_\phi$ of the Blanchfield pairing.
Let 
\[
P^\perp=\{u\in{H}\mid b_\phi(u,p)=0~\forall{p}\in{P}\}.
\]
Then we have another commutative diagram
\begin{equation*}
\begin{CD}
0\to{P^\perp}@>>>{H}@>>>{H/P^\perp}\to0\\
@VVV @VV{\tilde{b}_\phi}V @VVV\\
0\to\overline{\mathrm{Ext}_{\mathbb{Q}\Lambda}(H/P,\mathbb{Q}\Lambda)}
@>>>
\overline{\mathrm{Ext}_{\mathbb{Q}\Lambda}(H,\mathbb{Q}\Lambda)}@>>>\overline{\mathrm{Ext}_{\mathbb{Q}\Lambda}(P,\mathbb{Q}\Lambda)}\to0.
\end{CD}
\end{equation*}
Hence $\Delta_0(H/P^\perp)=\overline{\Delta_0(P)}$. 
Comparing these equations, and using the multiplicity of orders in short exact sequences, 
we see that $\Delta_0(P^\perp)=\Delta_0(P)$.

If we use the geometric formulation of the Blanchfield pairing as in the proof of \cite[Theorem 2.4]{AIL}, 
we see easily that  $P\leq{P^\perp}$.
Since $\mathbb{Q}\Lambda$ is a PID and these are torsion modules 
of the same order it follows that $P=P^\perp$.
Hence $b_\phi$ is neutral.
\end{proof}

\begin{cor}
%formerly\label{Hi09-cor4.3}
\label{g=0skewsym}
Let $M=M(0;S)$, where $\varepsilon(M)=0$ and all cone point orders 
$\alpha_i$ are odd. 
Then $M$ embeds in $S^4$ if and only if $S$ is skew-symmetric.
\end{cor}

\begin{proof}
In this case $\beta=1$, and so $H^1(X)\cong{H^1(M)}\cong\mathbb{Z}$.
Let $\phi:\pi_X\to\mathbb{Z}$ be an epimorphism.
The restriction $\phi:\pi=\pi_1M\to\mathbb{Z}$ is also an epimorphism, 
since $H_1(M)$ maps onto $H_1(X)$,
and $H_1(M;\mathbb{Q}(t))=0$.
Thus $b_\phi$ is neutral, by Theorem \ref{Kawabeta1}, 
and so $S$ is skew-symmetric, 
by Theorem \ref{Hi09-thm4.1}.
\end{proof}

This settles the question of which Seifert fibred rational homology 
$S^2\times{S^1}$s embed, when all cone point orders are odd; 
the answer is not known in general.

\section{An open question}

Corollary \ref{g=0skewsym} suggests the question:
{\it if $M(g;S)$ embeds in $S^4$ and $\varepsilon(M)=0$,
must $S$ be skew-symmetric?} 
This is true if $g=0$ and all cone point orders are odd,
by Theorem \ref{g=0skewsym}.
However,
skew-symmetry alone is not enough to ensure an embedding.
The 2-torsion must be homogeneous.
(See Appendix A.)

We consider some hindrances in extending our argument.

If there is an embedding with $H_2(X)=0$ then 
the inclusions of $M$ and of $\vee^\beta{S^1}$ into $X$ induce isomorphisms on all 
the rational lower central series quotients of the fundamental groups \cite{St65}.
Hence these quotients are those of the free group $F(\beta)$.
This is never true for $M$ an $\mathbb{H}^2\times\mathbb{E}^1$-manifold
with $\beta>1$, since the centre has nonzero image in $H_1(M;\mathbb{Q})$.
It follows immediately from Lemma \ref{Hi17-lem2.1} that if $g\leq1$
then there is a complementary region $X$ with $\chi(X)=0$.
Does $h$ have nonzero image in $H_1(X;\mathbb{Q})$ when $\chi(X)=0$ and $g=1$?

If $g>1$, the condition $\chi(X)=0$ holds for neither  complementary region
of $M(g;\emptyset)=T_g\times{S^1}$, when embedded in $S^4$ 
as the boundary of a regular neighbourhood of an embedding of $T_g$.
It remains possible that $b_\phi$ is neutral
when $\Phi(h)\not=0$ for some $\Phi$ as above.
(This would follow by the argument of \cite[Theorem 2.3]{AIL},
if the torsion submodule of $H_2(X,M;\mathbb{Q}\Lambda)$ maps onto 
the image of $H_2(X,M;\mathbb{Q}\Lambda)$ in $H_1(M;\mathbb{Q}\Lambda)$.)

When $r=2$ and $\varepsilon(M)=0$ the Seifert data $S$ is skew-symmetric,
and $M=M(0;S)\cong{S^2\times{S^1}}$,
with $S^1$-action given by $u.(w,z)=(u^\alpha{w},u^\beta{z})$ 
for $u,z\in{S^1}$ and $w\in\widehat{\mathbb{C}}\cong{S^2}$.
If a regular fibre of $M$ bounds a locally flat disc 
in one complementary component of some embedding of $M$ in $S^4$, 
ambient surgery gives an embedding of $L\#-\!L$ in $S^4$, 
where $L=L(\alpha,\beta)$.
This is only possible if $\alpha$ is odd \cite{KK80}.
Thus it is not clear whether a fibre sum construction can be used to
build up embeddings of other $\mathbb{H}^2\times\mathbb{E}^1$-manifolds 
with some exceptional fibres of even multiplicity.

If $M=M(0;(3,1),(5,-2),(15,1))$, then the $\alpha_i$s are odd,
$\varepsilon_S=0$ and $\tau_M=0$, so $\ell_M$ is hyperbolic. 
Thus these conditions alone do not imply that $S$ must be skew-symmetric.

\section{The torsion subgroup}

In this section we shall describe the torsion subgroup $\tau_M$
in terms of the Seifert invariants of $M=M(g;S)$, 
when $g\geq0$.

Let $N_i$ be a regular neighbourhood of the $i$th exceptional fibre, 
and let $B_o$ 
be a section of the restriction of the Seifert fibration to $\overline{M\setminus\cup\{N_i\}}$.
Then $B_o$ is homeomorphic to $T_g$ with $r$ open 2-discs deleted,
and $\overline{M\setminus\cup\{N_i\}}\cong{B_o\times{S^1}}$.
Let $\xi_i$ and $\theta_i$ be simple closed curves on $\partial{N_i}$ corresponding to the 
$i$th boundary component of $B_o$ and a regular fibre respectively.
The fibres are naturally oriented as orbits of the $S^1$-action.
We may assume that $B_o$ is oriented so that regular fibres have negative intersection 
with $B_o$, and the $\xi_i$ are oriented compatibly with $\partial{B_o}=\Sigma\xi_i$.
Then there are 2-discs $D_i$ in $N_i$ such that $\partial{D_i}=\alpha_i\xi_i+\beta_i\theta_i$.
It follows from Van Kampen's Theorem that $\pi_1M$ has a presentation
\[
\langle{a_1},b_1,\dots,a_g,b_g,q_1,\dots,q_r,h\mid
\Pi[a_i,b_i]\Pi{q_j}=1,~q_i^{\alpha_i}h^{\beta_i}=1,~h~central\rangle,
\]
where $\{a_i,b_i\}$ are the images of curves in $B_o$ which form a canonical basis for $|B|$,
$q_j$ is the image of $\xi_j$ and $h$ is the image of regular fibres such as $\theta_j$.

\begin{theorem}
\label{Hi09-thm1}
Let $M=M(g;S)$ be a Seifert manifold with orientable base orbifold.
Then $H_1(M)\cong\mathbb{Z}^{2g}\oplus
(\bigoplus_{i\geq0}(\mathbb{Z}/\lambda_i\mathbb{Z}))$,
where $\lambda_i$ is determined by $\{\alpha_1,\dots,\alpha_r\}$ 
and is nonzero, 
for all $i>0$.
Moreover, $\lambda_{i+1}$ divides $\lambda_i$, for all $i\geq0$,
and $|\varepsilon(M)|\Pi_{i=1}^n\alpha_i=\lambda_0\Pi_{j\geq1}\lambda_j$.
\end{theorem}

\begin{proof}
On abelianizing the above presentation for $\pi_1M$,
we see that $H_1(M)\cong\mathbb{Z}^{2g}\oplus\mathrm{Cok}(A)$, 
where $A$ is the matrix
\[ A=
\left(
\begin{matrix}
0& 1 & \dots & 1\\
\beta_1 & \alpha_1 &\dots &0\\
\vdots & 0 & \ddots & \vdots\\
\beta_r& 0& \dots &\alpha_r
\end{matrix}
\right).
\]
Let $E_i(A)$ be the ideal generated by the $(r+1-i)\times(r+1-i)$ subdeterminants of $A$, 
and let $\delta_i$ be the positive generator of $E_i(A)$.
Then $\Delta_0=|det(A)|=|\varepsilon(M)|\Pi_{i=1}^n\alpha_i$.
Since the elements of each row are relatively prime, 
$\Delta_i$  is the highest common factor of the $(r-i-1)$-fold
products of distinct $\alpha_j$s, if $0<i<r$, 
and $\Delta_i=1$ if $i\geq\max\{r-1,1\}$.
(In particular, if $r>2$ then $\Delta_{r-2}=\mathrm{hcf}(\alpha_1,\dots,\alpha_r)$.)
Thus $\Delta_i$ depends only on $\{\alpha_1,\dots,\alpha_r\}$ and is nonzero, for all $i>0$.
If we set $\lambda_i=\Delta_i/\Delta_{i+1}$ for $i\geq0$
then $\mathrm{Cok}(A)\cong\bigoplus_{i\geq0}(\mathbb{Z}/\lambda_i\mathbb{Z})$, 
and $\lambda_{i+1}$ divides $\lambda_i$, for all $i\geq0$,
by the Elementary Divisor Theorem.
In particular, $|\varepsilon(M)|\Pi_{i=1}^n\alpha_i=\lambda_0\Pi_{j\geq1}\lambda_j$.
\end{proof}

Note that $\tau_M\cong\tau_{M(0;S)}$ and the image of $h$ is in $\tau_M$ 
if and only if $\varepsilon(M)\not=0$.

\begin{cor}
\label{Hi09-cor3.2}
If $\Delta_1=1$ then $\tau_M$ is cyclic, and $\tau_M=0$ if and only if $\varepsilon(M)=0$ or $\pm1/\Pi_{i=1}^n\alpha_i$.
If $\Delta_1>1$ then $\tau_M\not=0$.
Given $\{\alpha_1,\dots,\alpha_r\}$ such that $\Delta_1>1$, 
there is at most one value of $|\varepsilon|$ for which $\tau_M$ is a direct double.
\end{cor}

\begin{proof}
As $\Delta_1=\Pi_{i\geq1}\lambda_i$ divides the order of $\tau_M$,
this group is nonzero unless $\Delta_1=1$.
If $\varepsilon(M)\not=0$ then
$\tau_M\cong\bigoplus_{i\geq0}(\mathbb{Z}/\lambda_i\mathbb{Z})$,
and so is a direct double if and only if $\lambda_{2i}=\lambda_{2i+1}$ for all $i\geq0$.
In particular, $\varepsilon(M)=(\Delta_1)^2/\Pi_{i=1}^n\alpha_i\Delta_2$,
and so is determined by  $\{\alpha_1,\dots,\alpha_r\}$.
If $\varepsilon(M)=0$ then 
$\tau_M\cong\bigoplus_{i\geq1}(\mathbb{Z}/\lambda_i\mathbb{Z})$,
and so is a direct double if and only if $\lambda_{2i-1}=\lambda_{2i}$ for all $i>0$.

Clearly these two systems of equations can both be satisfied 
only if $\lambda_i=1$ for all $i>0$ and $\lambda_0=0$ or 1,
in which case $\tau_M=0$.
\end{proof}

If $M$ is Seifert fibred over a non-orientable base orbifold then $\tau_M\not=0$,
as we shall recall in Lemma \ref{CH-lem3.4}.

It follows from Theorem \ref{Hi09-thm1} that a Seifert manifold $M$ is a homology sphere 
if and only if $M=M(0;S)$ for some Seifert data $S$ with 
$\varepsilon(M)\Pi_{i=1}^n\alpha_i=\pm1$.
In particular,
$\mathrm{hcf}\{\alpha_i,\alpha_j\}=1$ for all $i<j\leq{r}$.
It is easy to find such examples for any $r\geq1$.
Thus there is no reason to expect parity constraints for 
embedding such manifolds, since every homology sphere embeds in $S^4$ \cite{Fr82}.

If $M=M(g;S',(1,-e))$ is a homology sphere then $g=0$ and 
$\varepsilon(M)\Pi_{i=1}^n\alpha_i=1$.
Hence $e=\frac1{\Pi\alpha_i}+\Sigma_{i=1}^k\frac{\beta_i}{\alpha_i}$.
It is easy to see that since $e$ is an integer we must have $e<k$.
This bound is best possible, by the following purely arithmetic lemma.

\begin{lemma}
\label{arithbound}
Let $\alpha_1,\dots,\alpha_k$ be relatively prime integers greater than 1,
and let $\Pi=\Pi_{i=1}^k\alpha_i$.
If $k\geq2$ then there are integers $\beta_1,\dots,\beta_k$ such that
$0<\beta_i<\alpha_i$ and $(\alpha_i,\beta_i)=1$, for $i\leq{k}$,
and
$
\frac1{\Pi\alpha_i}+\Sigma_{i=1}^k\frac{\beta_i}{\alpha_i}=k-1.
$
\end{lemma}

\begin{proof}
We shall induct on $k$.
If $k=2$ we may  choose $\beta_1$ so that $0<\beta_1<\alpha_1$ and 
$1+\beta_1\alpha_2=m\alpha_1$, for some $m\in\mathbb{Z}$.
Clearly $0<m<\alpha_2$, and so $0<\beta_2=\alpha_2-m<\alpha_2$.
We then have 
\[
\frac1{\alpha_1\alpha_2}+\frac{\beta_1}{\alpha_1}+\frac{\beta_2}{\alpha_2}=1.
\]

Suppose the result holds for $k\leq{j}$.
Let $\alpha_1,\dots,\alpha_{k+1}$ be relatively prime integers,
and let $A=\alpha_k\alpha_{k+1}$.
Let $\Pi=(\Pi_{j=1}^{k-1}\alpha_j)A$.
Then there are integers $\beta_1,\dots,\beta_{k-1},B$ such that 
$0<\beta_i<\alpha_i$ and $(\alpha_i,\beta_i)=1$, for all $i<k$,
$0<B<A$,$(A,B)=1$ and
$
\frac1\Pi+\Sigma_{i=1}^{k-1}\frac{\beta_i}{\alpha_i}+\frac{B}A=k-1,
$
by the inductive hypothesis.

If $p, q$ are natural numbers then the additive subsemigroup of $\mathbb{N}$ 
generated by $\{p,q\}$ contains every integer $\geq(p-1)(q-1)$.
Since $\alpha_k$ and $\alpha_{k+1}$ are relatively prime, and
$A+B>\alpha_k\alpha_{k+1}>(\alpha_k-1)(\alpha_{k+1}-1)$,
we may write $A+B=x\alpha_{k+1}+y\alpha_{k+2}$, where $x,y>0$.
Since $A+B<2A$ and $(A,B)=1$, we have $x<\alpha_k$ and $y<\alpha_{k+1}$,
and $(x,\alpha_1)=(y,\alpha_2)=1$.
Let $\beta_k=x$ and $\beta_{k+1}=y$.
Then $\frac{\beta_k}{\alpha_k}+\frac{\beta_{k+1}}{\alpha_{k+1}}=\frac{B}A+1$, and so
\[
\frac1\Pi+\Sigma_{i=1}^{k+1}\frac{\beta_i}{\alpha_i}=
\frac1\Pi+\Sigma_{i=1}^{k-1}\frac{\beta_i}{\alpha_i}+\frac{B}A+1=k-1.
\]
Since $\Pi_{i=1}^{k+1}\alpha_i=\Pi_{j=1}^{k-1}\alpha_jA=\Pi$, this proves the lemma.
\end{proof}

There is a similar criterion for a Seifert manifold to be a homology handle.

\begin{cor}
\label{seifert homhandle}
Let $M=M(g;S)$.
Then there is a $\mathbb{Z}$-homology equivalence $f:M\to{S^2}\times{S^1}$
if and only if $\varepsilon(M)=0$ and $\mathrm{hcf}\{\alpha_i,\alpha_j,\alpha_k\}=1$ 
for all $i<j<k\leq{r}$.
If there are $r>2$ cone points then no such map is a $\Lambda$-homology equivalence.
\end{cor}

\begin{proof}
The first assertion follows from Theorem 3.5.
If $M$ is Seifert fibred then the base orbifold is orientable,
since $M$ is orientable and $\pi/\pi'$ is torsion free.
Since $\varepsilon(M)=0$, $\beta$ is odd, and $M$ is a mapping torus. 
If $r>2$ then $\pi_1M\not\cong\mathbb{Z}$, and so the infinite cyclic cover of $M$ has positive genus.
\end{proof}

The elementary divisors $\lambda_i$ may be determined more explicitly by localization.
If $p$ is a prime, an integer $m$ has {\it $p$-adic valuation\/} 
$val_p(m)=v$ if $m=p^vq$, 
where $p$ does not divide $q$.

\begin{cor}
\label{Hi09-cor3.3}
Let $p$ be a prime and let $v_i=val_p(\alpha_i)$, for $i\geq1$.
Assume that the indexing is such that $v_i\geq{v_{i+1}}$ for all $i$, 
and that $\tau_M$ is a direct double. Then
\begin{enumerate}
\item{}if $\varepsilon(M)=0$ then $v_{2j-1}=v_{2j}$ for all $j\geq1$;
\item{}if $\varepsilon(M)\not=0$ then $val_p(\lambda_0)=v_3$, 
and $v_{2j}=v_{2j+1}$ for all $j\geq2$.
\end{enumerate}
\end{cor}

\begin{proof}
It is clear from the theorem that $val_p(\Delta_i)=\Sigma_{j\geq{i+2}}v_j$,
and so $val_p(\lambda_i)=v_{i+2}$, for all $i\geq1$.
If $\varepsilon(M)=0$  then $v_1=v_2$ follows immediately from the fact that 
$p^{v_2}\varepsilon(M)$ is an integer.
If also $\bigoplus_{i\geq1}(\mathbb{Z}/\lambda_i\mathbb{Z})$ is a direct double
then $v_{2j-1}=v_{2j}$ for all $j\geq2$.

Suppose now that $\varepsilon(M)\not=0$.
Then $val_p(\varepsilon(M)))\geq-v_1$, 
and $val_p(\varepsilon(M))=-v_1$ if $v_1>v_2$.
Since $val_p(\Pi)=\Sigma_{i=1}^nv_i$ and $\Pi_{j\geq1}\lambda_j=\Delta_1$,
so $val_p(\Pi_{j\geq1}\lambda_j)=val_p(\Delta_1)$,
we have $val_p(\lambda_0)=v_1+v_2+val_p(\varepsilon(M))\geq{v_2}$
(with equality if $v_1>v_2$).
If $\tau_M$ is a direct double we must have 
$val_p(\lambda_0)=v_2=v_3$ and $v_{2j}=v_{2j+1}$ for all $j\geq2$.
\end{proof}

We note here a related observation of Issa and McCoy.

\begin{lemma}
\cite[Lemma 4.4]{IM20}
\label{IM4.4}
Let $M=M(g; \frac{\beta_1}{\alpha_1}, \dots, \frac{\beta_k}{\alpha_k}, (1,e))$, 
where $0<\beta_i<\alpha_i$ for all $i$ and $\varepsilon(M)>0$, 
and suppose that $\tau_M\cong{A}\oplus{A}$
for some finite abelian group $A$.
If the index set $\{1,\dots,k\}$ is partitioned into $n\leq{e}$ classes $\{C_i\}$
such that $\Sigma_{j\in{C_i}}\frac{\beta_j}{\alpha_j}\leq1$ for all $i\leq{n}$ then $n=e$ and
there is precisely one value of $i$ for which $\Sigma_{j\in{C_i}}\frac{\beta_j}{\alpha_j}<1$.
For this value we have 
$1-\Sigma_{j\in{C_i}}\frac{\beta_j}{\alpha_j}=\frac1{\mathrm{lcm}(a_1,\dots,a_k)}$.
\end{lemma}

\begin{proof}
The $p$-primary part of $\tau_M\cong{A}\oplus{A}$is a direct double, for each prime $p$.
We may assume the cone points are labelled so that the $p$-adic valuations of their orders are
$v_k\geq{v_{k-1}}\geq\dots\leq{v_1}$. 
Let $w$ be the $p$-adic valuation of $\varepsilon(M)$.
The $p$-primary part of $\tau_M$ is then 
$\bigoplus_{i=3}^k(\mathbb{Z}/p^{v_i}\mathbb{Z})\oplus \mathbb{Z}/p^v\mathbb{Z}$,
where $v=v_1+v_2+w$, and $v\geq{v_2}$.
Since there must be an even number of summands,
either $k$ is odd or $v_1=0$.
Since the exponents $v_i$ are increasing, 
this can be a direct double only if $v=v_1=v_2$.
But then $w=-v_1$ is the $p$-adic valuation of $\mathrm{lcm}(\alpha_1,\dots,\alpha_k)$.

This argument applies to each prime,
and so 
\[
\varepsilon(M)=\frac1{\mathrm{lcm}(\alpha_1,\dots,\alpha_k)}.
\]
If $k$ is even then $\mathrm{hcf}(\alpha_1,\dots,\alpha_k)=1$.

Now suppose that the index set $\{1,\dots,k\}$ is partitioned into $n\leq{e}$ classes $\{C_i\}$
such that $\Sigma_{j\in{C_i}}\frac{\beta_j}{\alpha_j}\leq1$ for all $i\leq{n}$.
Then we may write
\[
\varepsilon(M)=e-n+\Sigma_{j=1}^n\sigma_j
=(1-\Sigma_{i\in{C_j}}\frac{\beta_i}{\alpha_i})=
\frac1{\mathrm{lcm}(\alpha_1,\dots,\alpha_k)},
\]
where $\sigma_j=1-\Sigma_{i\in{C_j}}\frac{\beta_i}{\alpha_i})\geq0$ for all $j$.
Then we must have $e-n\leq0$ and so $n=e$.
If $\sigma_j>0$ 
then $\sigma_j\geq\frac1{\mathrm{lcm}(\alpha_1,\dots,\alpha_k)}$,
and so $\sigma_h=0$ for all $h\not=j$.
\end{proof}

The following bound is an immediate consequence.
(Note that this result bounds $e$ in terms of the number of singular fibres, 
rather than in terms of the full Seifert data for these singular fibres,
as in Corollary \ref{Hi09-cor3.2}.)

\begin{theorem}
\cite[Prop. 4.6]{IM20}
\label{IM4.6}
Let $M=M(g; \frac{\beta_1}{\alpha_1}, \dots, \frac{\beta_k}{\alpha_k}, (1,-e))$, 
where $k>1$,
$0<\beta_i<\alpha_i$  for all $i$ and $\varepsilon(M)>0$.
If $M$ embeds  in $S^4$ then $e\leq{k-1}$.
\end{theorem}

\begin{proof}
If $M$ embeds in $S^4$ then $\tau_M$ is a direct double.
If $e\geq{k}$ then the partition of $\{1,\dots,k\}$ into $k$ singletons would violate the conclusion of Lemma \ref{IM4.4}, 
since $\frac{\beta_i}{\alpha_i}<1$ for all $i\leq{k}$, and $k>1$.
\end{proof}

For every $k\geq3$ there are Seifert fibred $\mathbb{Z}$-homology 3-spheres of the form $M(0; \frac{\beta_1}{\alpha_1}, \dots, \frac{\beta_k}{\alpha_k}, (1,1-k))$, by Lemma \ref{arithbound}.
Since these embed in $S^4$, by Freedman's result, 
the bound in this theorem is best possible.

Let $\ell$ be a linking pairing on a finite abelian group of order a
power of an odd prime $p$.
Then $\ell\cong\perp_{j=1}^t\ell_j$, where $\ell_j$ is a pairing on 
$(\mathbb{Z}/p^{k_j}\mathbb{Z})^{\rho_j}$ and $k_1>k_2>\dots>k_t\geq1$.
If $\ell$ is hyperbolic then each of the summands $\ell_j$ is also hyperbolic, 
and so the ranks $\rho_j$ are all even.
Let $r_1=\frac12{\rho_1}+1$ and $r_i=\frac12{\rho_i}$ for $2\leq{j}\leq{t}$,
and let $S(p)$ be the concatenation of $r_j$
copies of $((p^{k_j},1),(p^{k_j},-1))$, for $1\leq{j}\leq{t}$.
Then $M=M(0;S(p))$ embeds in $S^4$, by Lemma 3.2,  and $\ell_M\cong\ell$, 
since $\tau_M\cong\oplus_{j=1}^t(\mathbb{Z}/p^{k_j}\mathbb{Z})^{\rho_j}$
and $\ell_M$ is hyperbolic.
This argument extends to show that every hyperbolic linking pairing 
on a finite abelian group of odd order is realized by some Seifert 3-manifold $M$ 
with $\varepsilon(M)=0$ and which embeds in $S^4$.

In Appendix A we analyze the linking pairings of
orientable 3-manifolds with fixed-point free $S^1$-actions.
In particular, 
Theorem A10 of the Appendix implies that the linking pairing is even 
if and only if all even cone point orders $\alpha_i$ have the same 
2-adic valuation, and if there is one such point there are at least 3.

The following lemma  supports the conjectures at the end of the next section.

\begin{lemma}
\cite[Lemma 1.6]{IM20}
If $M$ embeds in $S^4$ and $M_1$ is obtained from $M$ by expansion then $M_1$
embeds in $S^4$.
\end{lemma}

\begin{proof}
Suppose that $M_1=M\sharp_fM'$, 
where $M'=M(0;(\alpha_1,\beta_1),(\alpha_1,-\beta_1))$ 
with $(\alpha_i,\beta_1)\in{S}$ and
relatively prime $0<\beta<\alpha$.
Let $N_1$ be a Seifert fibred neighbourhood of the first exceptional fibre of $M$,
and let $U=N_1\times[-\frac12,\frac12]$.
Then $\partial{U}\cong{S^2}\times{S^1}$ has a natural Seifert structure,
so that $\partial{U}=M'$.
The restrictions of the fibration to $N_1\times\{\pm\frac12\}$
are copies of the fibration of $N_1$, and is the obvious product structure on 
$\partial{N_1}\times[-\frac12,\frac12]\cong{T}\times[-\frac12,\frac12]$.
Now let $N_2\subset{N_1}$ be a Seifert fibred neighbourhood of a regular fibre.
Let \[
V=(\overline{U\setminus{N_2}})\times\{0\}\cup\partial{N_2}\times[-1,-\frac12]\cup
(\overline{\partial{U}\setminus{N_2}}\times\{\frac12\}).
\]
Then $V\cong{M\#_fU}\cong{M\#_fM'}$. 
By smoothing the corners we obtain a smooth embedding of  $V$ into $M\times[-1,1]$.
Thus if $M$ embeds in $S^4$ so does $M_1$.
\end{proof}

This lemma does not require that the base orbifold be orientable.

\section{Smooth embeddings}

We conclude this chapter by summarizing the work of Donald and of Issa and McCoy 
on smooth embeddings.
It is not yet known whether there are similar restrictions on locally flat embeddings,
and so these may suggest further work.

If the Seifert data for $M$ is skew-symmetric and all the cone point orders are odd 
then $M$ embeds smoothly in $S^4$, by Lemmas \ref{CH-lem3.1} and  \ref{CH-lem3.2}.
This also holds if there is also one pair $\{(\alpha,\pm\beta)\}$ in $S$ with $\alpha$ even  \cite[Proposition 7.8]{IM20}.
However this is not yet known if more than one pair has even cone point order.
Donald has shown that skew-symmetry is a necessary condition for a
$\mathbb{H}^2\times\mathbb{E}^1$-manifold to embed smoothly \cite{Do15}.
(The cone point orders are not assumed to be odd here.)

\begin{thm}
\cite[Theorem 1.3]{Do15}
Let $M$ be a Seifert manifold with orientable base orbifold and $\varepsilon(M)=0$.
If $M$ embeds smoothly in $S^4$ then the Seifert data for $M$ is skew-symmetric.
\qed
\end{thm}

Issa and McCoy give strong constraints on which $\widetilde{\mathbb{SL}}$-manifolds 
with orientable base orbifold embed smoothly in $S^4$.
Assume henceforth that $M=M(g;S',(1,-e)$,
where $S'$ is strict Seifert data with index set $\{1,\dots,k\}$
and $e>\Sigma_{i=1}^k\frac{\beta_i}{\alpha_i}$, so that $\varepsilon(M)>0$.
(We have used our notational conventions for the Seifert data in what follows.)
Their main technical result involves the following notion.

The Seifert manifold $M$ is {\it partitionable\/} if $\tau_M$ 
is a direct double and the index set $\{1,\dots,k\}$
has two partitions $\mathcal{P}$ and $\mathcal{P}'$ into $e$ classes
$\{C_1,\dots,C_e\}$ such that 
\begin{enumerate}
\item$\Sigma_{i\in{C_1}}\frac{\beta_i}{\alpha_i}=1-\frac1{\mathrm{lcm}(a_1,\dots,a_k)}$;
\item $\Sigma_{i\in{C_j}}\frac{\beta_i}{\alpha_i}=1$ for $j>1$; and
\item{}
no non-empty union of a proper subset of classes in $\mathcal{P}'$ is a union 
of classes in $\mathcal{P}$. 
\end{enumerate}
The notion of partitionable Seifert manifold was introduced by Issa and McCoy,
and has a key role in their results on smooth embeddings \cite{IM20}.

\begin{thm}
\cite[Theorem 1.4]{IM20}
If  $M$ embeds smoothly in $S^4$ then $M$ is partitionable.
\qed
\end{thm}

From this they derive the following results (among others).

\begin{thm}
\cite[Theorem 1.1]{IM20}
If  $M$ embeds smoothly in $S^4$ then $2e\leq{k+1}$.
Moreover, if $k=2e-1$ is odd then $M$ embeds smoothly if and only if
\[S'=\{e(\alpha,\alpha-1), (e-1)((\alpha,1)\}~(and~so~\varepsilon(M)=\frac1\alpha).
\quad\quad\quad\quad\qed
\]
\end{thm}

\begin{thm}
\cite[Theorem 1.12]{IM20}
If  $k=2e$ and $M$ embeds smoothly in $S^4$ then 
there are integers $p,q,r,s$ with $0<q<p$, $0<s<r$, $(p,q)=(r,s)=1$
and $ps+qr+1=pr$ such that $S'$ is either
\begin{enumerate}
\item$\{(x(p,q),y(r,s), (x-1)(p,p-q),(y-1)(r, r-s)\}$,  where $x+y=k-2$; or
\item$\{(x(p,q),y(r,s), (x-1)(p,p-q),(y-1)(r, r-s),z(pr,1), z((pr,pr-1)\}$, 
where $x+y+z=k-2$ and $z>0$.
\end{enumerate}
Moreover, in case (1) each such $M$ embeds smoothly in $S^4$.
\qed
\end{thm}

The argument for the next result uses the Rochlin invariant $\bar\mu(M;\mathfrak{s})$.

\begin{thm}
\cite[Theorem 1.10]{IM20}.
If $g=0$ and $M$ embeds smoothly in $S^4$ then $\beta_1(M;\mathbb{F}_2)\leq2e$.
\qed
\end{thm}

They make the following conjectures, and prove the second of these under
the additional assumption that all cone point orders are even.

\smallskip
\noindent\cite[Conjecture 1.8]{IM20}. 
{\sl A Seifert manifold $M(0;S',(1,-e))$ with $\varepsilon(M)>0$
embeds smoothly in $S^4$
if and only if it is obtained by a (possibly empty) sequence of expansions 
from some $M(0;S',(1,-1))$ with $S'$ strict Seifert data 
which also embeds smoothly in $S^4$.}

\smallskip
\noindent\cite[Conjecture 1.9]{IM20}. 
{\sl If $M=M(0;S',(1,-e))$ embeds smoothly in $S^4$ and $\varepsilon(M)>0$ 
then $k-1\leq2e\leq{k+1}$.}

%% file: e4.tex
\chapter{Seifert manifolds with non-orientable base orbifolds}

In this chapter we shall consider Seifert manifolds with non-orientable base orbifold.

W. S. Massey settled an old question of Whitney
when he showed that the normal bundles to smooth
embeddings of $\#^k\mathbb{RP}^2$ in $S^4$ have Euler number 
$e\in\{-2k,4-2k,\dots,2k\}$ \cite{Ma69},
This was an early application of the $G$-Signature Theorem
to low dimensional topology.
We shall use this theorem in a similar manner in proving 
the main result of this chapter, 
which is a bound on the range of values for $\varepsilon(M)$,
where $M$ embeds in $S^4$.

\section{The torsion subgroup}

We begin by determining the torsion $\tau_M$ for Seifert manifolds $M$
with non-orientable base orbifold.

Let $N_i$ be a regular neighbourhood of the $i$th exceptional fibre, 
and let $B_o$ 
be a section of the restriction of the Seifert fibration to $\overline{M\setminus\cup\{N_i\}}$.
Then $B_o$ is homeomorphic to $\#^kRP^2$ with $r$ open 2-discs deleted,
and $\overline{M\setminus\cup\{N_i\}}\cong{B_o\times{S^1}}$.
Let $\xi_i$ and $\theta_i$ be simple closed curves on $\partial{N_i}$ corresponding to the 
$i$th boundary component of $B_o$ and a regular fibre respectively.
Since $B$ is non-orientable $\pi_1M$ has a presentation
\[
\langle{u_1,\dots,u_k, q_1,\dots,q_n, h}\mid\Pi{u}_i^2\Pi{q_j}=1,
~q_i^{\alpha_i}h^{\beta_i}=1,~u_jhu_j^{-1}=h^{-1}\rangle,
\]
where the $u_i$s are orientation-reversing loops in the $\mathbb{RP}^2$ 
summands of $|B|$,
$q_j$ is the image of $\xi_j$ and $h$ is the image of regular fibres such as $\theta_j$.
In this case $u_jhu_j^{-1}=h^{-1}$ for all $j$, since $M$ is orientable.

\begin{lemma}
\label{CH-lem3.4}
Let $M=M(-k;S)$ be a Seifert manifold with non-orientable base orbifold 
(so $k\geq1$).
Let $r$ be the maximal index such  that $t_i=t_1$ for $i\leq{r}$.
Let $c=r$ if $t_1\not=0$ and let $c=\Sigma\beta_i$ if $t_1=0$
(i.e., if all the $\alpha_i$ are odd). 
Then 
\[
\tau_M\cong\mathbb{Z}/2\alpha_1\mathbb{Z}\oplus\mathbb{Z}/2\alpha_2\mathbb{Z}\oplus(\bigoplus_{i=3}^n\mathbb{Z}/\alpha_i\mathbb{Z})\quad\mathrm{if}~c~\mathrm{is~even},
\]
and 
\[
\tau_M\cong\mathbb{Z}/4\alpha_1\mathbb{Z}\oplus(\bigoplus_{i=2}^n\mathbb{Z}/\alpha_i\mathbb{Z})\quad\mathrm{if}~c~\mathrm{is~odd}.
\]
\end{lemma}

\begin{proof}
On abelianizing the above presentation for $\pi_1M$,
we see that $H_1(M)$ has the presentation
\[
\langle{a_1,\dots,a_k,q_1,\dots,q_n,h}\mid2h=0,~\alpha_jq_j+\beta_jh=0,~
2\Sigma{a_i}+\Sigma{q_j}=0\rangle,
\]
where $h=[f]$, $a_i=[u_i]$ and $q_j=[\xi_j]$, 
for $1\leq{i}\leq{k}$ and $1\leq{j}\leq{n}$.
Since $h$,  $q$ and $a=\Sigma{a_i}$ are clearly torsion elements, 
we see that $H_1(M)\cong\mathbb{Z}^{k-1}\oplus\tau_M$,
where $\tau_M$ has the presentation
\[
\langle{a,h,q_1\dots,q_n}\mid2h=0,~\alpha_jq_j+\beta_jh=0,~
2a+\Sigma{q_j}=0\rangle.
\]
We shall write the cone point orders as $\alpha_i=2^{t_i}s_i$, 
with $s_i$ odd,  and in descending order of 2-adic valuation $t_i$.

Suppose first that $t_1>0$.
Then after modifying the Seifert data, if necessary,
we may assume that every $\beta_i$ is odd, 
and so rewrite the appropriate relations in the above presentation for $\tau_M$ 
as $\alpha_jq_j=h$.
Replace $q_j$ by $\bar{q}_j=q_j+2^{-t_j}\alpha_1q_1$, and replace $a$ with
\[
\bar{a}=a-\left(\Sigma_{j=2}^n\frac{\alpha_1}{2^{t_j+1}}\right)q_1+
\Sigma_{j=3}^n\frac12(1-s_j)\bar{Q}_j+\omega
\]
where $\omega=\frac12q_1$ if $r$ is even and $\omega=\frac12(1-s_2)\bar{q}_2$
if $r$ is odd.
It is easily checked that the coefficient of $q_1$ in this expression is an integer,
in either case.
Now the presentation for $\tau_M$ reduces to one of the following:
\[
\langle{h, \bar{a},q_1,\bar{q}_2,\dots,\bar{q}_n}\mid2h=0,~\alpha_1q_1=h,~
\alpha_j\bar{q}_j=0,~2\bar{a}+\bar{q}_2+\Sigma_{j=3}^ns_j\bar{q}_j=0\rangle,
\]
or
\[
\langle{h, \bar{a},q_1,\bar{q}_2,\dots,\bar{q}_n}\mid2h=0,~\alpha_1q_1=h,~
\alpha_j\bar{q}_j=0,~2\bar{a}+\bar{q}_1+\Sigma_{j=2}^ns_j\bar{q}_j=0\rangle,
\]
according as $r$ is even or odd.

If $r$ is even, replace $\bar{q}_2$ by 
$\bar{\bar{q}}_2=\bar{q}_2+\Sigma_{j=3}^ns_j\bar{q}_j$.
With this, the relation $\alpha_2\bar{q}_2$ may be replaced by $\alpha_2\bar{\bar{q}}_2=0$
(noting that $\alpha_j$ divides $\alpha_2s_j$ for $j\geq3$), and the presentation 
for $\tau_M$ becomes
\[
\langle{h, \bar{a},q_1,\bar{\bar{q}}_2,\bar{q}_3,\dots,\bar{q}_n}\mid2h=0,~\alpha_1q_1=h,~\alpha_2\bar{\bar{q}}_2=\alpha_j\bar{q}_j=0,~2\bar{a}+\bar{\bar{q}}_2=0
\rangle,
\]
which reduces to
\[
\langle{\bar{a},q_1,\bar{q}_3,\dots,\bar{q}_n}\mid2\alpha_1q_1=0,~
2\alpha_2\bar{a}=0,~\alpha_j\bar{q}_j=0\rangle,
\]
and so $\tau_M\cong\mathbb{Z}/2\alpha_1\mathbb{Z}\oplus\mathbb{Z}/2\alpha_2\mathbb{Z}
\oplus(\bigoplus_{i=3}^n\mathbb{Z}/\alpha_i\mathbb{Z})$.

If $r$ is odd, replace $q_1$ by $\bar{q}_1=q_1+\Sigma_{j=2}^ns_j\bar{q}_j$.
With this, the relation $\alpha_1q_1=h$ may be replaced by $\alpha_1\bar{q}_1=h$,
and the presentation for $\tau_M$ becomes 
\[
\langle{h, \bar{a},\bar{q}_1,\dots,\bar{q}_n}\mid2h=0,~\alpha_1\bar{q}_1=h,~
\alpha_j\bar{q}_j=0 ~(j\geq2),~2\bar{a}+\bar{q}_1=0\rangle,
\]
which reduces to
\[
\langle{\bar{a},\bar{q}_1,\dots,\bar{q}_n}\mid
4\alpha_1\bar{q}_1=0,~
\alpha_j\bar{q}_j=0 ~(j\geq2)\rangle,
\]
and so $\tau_M\cong\mathbb{Z}/4\alpha_1\mathbb{Z}\oplus
(\bigoplus_{i=2}^n\mathbb{Z}/\alpha_i\mathbb{Z})$.

Finally, suppose that $t_1=0$ and every $\alpha_j$ is odd.
After reindexing, if necessary, 
we may assume that $\beta_j$ is odd for $j\leq{s}$ and is even for $j>s$,
for some $s\geq0$.
Then $\alpha_jq_j=h$ for $j\leq{s}$ and $\alpha_jq_j=0$ for $j>s$, 
and $s\equiv{c}=\Sigma_{j=1}^n\beta_j$ {\it mod} (2).
Let 
\[
\bar{q}_j=q_j+\alpha_1q_1=q_j+h\quad\mathrm{for}~2\leq{j}\leq{s}\quad
\mathrm{and}\quad\bar{q}_j=q_j\quad\mathrm{for}~{j>s},
\]
and let 
\[\bar{a}=a-\left(\Sigma_{j=2}^s\frac12{\alpha_1}\right)q_1+
\Sigma_{j=3}^n\frac12(1-s_j)\bar{Q}_j+\omega
\]
where $\omega=\frac12q_1$ if $c$ is even and $\omega=\frac12(1-s_2)\bar{q}_2$
if $c$ is odd.
Then the rest of the argument follows exactly as in the preceding cases.
\end{proof}

If $M$ is as in the above Lemma then $\tau_M$ is a direct double only if the 
prime power factors of the $\alpha_i$s occur in pairs. 
If $\alpha_i$ is odd  for all $i$ then 
we need also that $\Sigma_{i=1}^n\beta_i$ is even.
In particular, this requires that there are an even number $2b$ (possibly 0)
of cone points with $\alpha_i$ even.
As we shall see in Theorem \ref{CH-thm3.7},
hyperbolicity of the torsion linking form requires furthermore 
that the non-zero $t_i$s must all be equal, or, 
when all the $\alpha_i$s are odd, that
$\Sigma_{i=1}^n\alpha_i\beta_i\equiv2k$ {\it mod} (4).

If $M$ is Seifert fibred over a nonorientable base then $\tau_M\not=0$, 
by Lemma \ref{CH-lem3.4}, and so $M$ is neither a homology sphere
nor a homology $S^2\times{S^1}$.

\section{The linking pairing}

In order to determine the linking pairing we shall have to look 
more carefully at geometric representatives for the homology classes.

We begin by finding a convenient choice of generators for $\tau_M$.
Fix a base-point for $\#^k\mathbb{RP}^2$ and let $\sigma_1,\dots,\sigma_k$ be orientation-reversing loops corresponding to the distinct summands.
Choose loops $\xi_1,\dots,\xi_n$ at $*$ which bound discs $N_i$, which meet only at $*$.
Orient and label these loops (clockwise) as in Figure 4.1.

\medskip
%4.1 shift frame .8 R (i.e add +2.8 to $X$ coordinate of bottom left corner)
% done

\setlength{\unitlength}{1mm}
\begin{picture}(100,68)(-38.5,4)

\put(17,38.2){$\bigstar$}

\put(20,48){$\sigma_1$}
\put(32,39){$\sigma_2$}
\put(32,24){$\sigma_k$}

\put(32,31){$\vdots$}
\put(42,29){$\vdots$}
\put(8,34){$\ddots$}

\put(0,41){$\xi_n$}
\put(10,25){$\xi_1$}
\put(8,62){$f$}

\put(4,61){$\nwarrow$}
\put(21,56.7){$\blacktriangleright$}
\put(23,37.1){$\blacktriangleright$}
\put(35.1,15.8){$\blacktriangleleft$}
\put(5,46){$\blacktriangleright$}
\put(16.5,23.8){$\blacktriangleleft$}

\thinlines

\qbezier(19,39)(10,39)(5,42)
\qbezier(5,42)(3,44)(5,46)
\qbezier(5,46)(8,50)(19,39)

\qbezier(19,39)(18,56)(21,57)
\qbezier(21,57)(22,58)(24.5,57.5)
\qbezier(19,39)(23,39)(26,45)
\qbezier(26,45)(28,50)(27.5,54)
\qbezier(27.5,54)(27.5,56)(26.5,56.8)
\qbezier(19,39)(21,38)(24,38)
\qbezier(24,38)(30,38)(33.5,44)
\qbezier(33.5,44)(36,49)(40,47.7)
\qbezier(19,39)(22,36)(26,36)
\qbezier(26,36)(34,35)(38,38)
\qbezier(38,38)(43,42)(41.5,46.5)

\qbezier(19,39)(16,38)(15,32)
\qbezier(19,39)(21,38)(21,30)
\qbezier(15,32)(14,26)(16,25)
\qbezier(21,30)(21,26)(19,25)
\qbezier(16,25)(17.5,24)(19,25)

\qbezier(19,39)(30,28)(35,28)
\qbezier(35,28)(43,27)(42.3,20)
\qbezier(19,39)(28,30)(28,27)
\qbezier(28,27)(28,20)(32,18)
\qbezier(32,18)(35,16)(39,17)

\qbezier(43,39) (42,37)(42,35)
\qbezier(43.5,23)(43,25)(43,26)

\linethickness{1pt}

\qbezier(19,40)(8,65)(0,65)

\qbezier(-10,20)(0,15)(20,17)
\qbezier(-10,20)(-13,25)(-10,28)
\qbezier(-10,28)(-5, 35)(-8,40)
\qbezier(-8,40)(-10,45)(-6,50)
\qbezier(-6,50)(0,57)(8,57)
\qbezier(12,57)(14,57)(16,58)
\qbezier(16,58)(22,61)(23,60)

\qbezier(23,60)(26,60)(26,56)

\qbezier(26,56)(25,52)(22,52)
\qbezier(22,52)(21,52)(21,54)
\qbezier(21,54)(22,56)(25,57)

\qbezier(27,57.5)(29,58)(29,54)
\qbezier(29,54)(29,48)(32,48)
\qbezier(32,48)(34,47)(38,49)
\qbezier(38,49)(41,50)(41,47)
\qbezier(41,47)(40,43)(37,43)
\qbezier(37,43)(36,43)(36,45)
\qbezier(36,45)(37,47)(40,47)

\qbezier(42,47)(44,47)(44,43)
\qbezier(44,43)(44,42)(43,39)

\qbezier(20,17)(27,20)(35,15)
\qbezier(35,15)(38,13)(40,17)
\qbezier(37,19)(39,17)(43,19)
\qbezier(37,19)(36,20)(37,22)
\qbezier(37,22)(38,24)(40,23)
\qbezier(40,23)(42,21)(41,19)
\qbezier(43,19)(44,20)(43.5,23)

\put(9.4,8){Figure 4.1}

\end{picture}

Let $P_{k,n*}=\#^k\mathbb{RP}^2\setminus\cup\,{int{N_i}}$. 
Then $P_{k,n*}$ deformation retracts onto a wedge of circles and thus there is an unique orientable $S^1$-bundle $P_{k,n*}\tilde\times{S^1}$ over $P_{k,n*}$, 
with a section which carries the loops
$\sigma_i$ and $\xi_j$ into $P_{k,n*}\tilde\times{S^1}$.
Let $f$ denote the oriented fibre over $*\in{P_{k,n*}}$.
Finally let $N(\gamma_i)$ be solid tori with cores $\gamma_i$, for $1\leq{i}\leq{n}$, 
and write 
\[M=P_{k,n*}\tilde\times{S^1}\cup_{\cup{h_i}}(\bigcup_{i=1}^nN(\gamma_i)).
\]
where each attaching homomorphism $h_i$ identifiers the boundary of 
a meridianal disc in $N(\gamma_i)$ with a curve representing 
the homology class $\alpha_i[\xi_i]+\beta_i[f]$.
Then $M=M(-k;S)$, where $S=\{(\alpha_i,\beta_i)\mid1\leq{i}\leq{n}\}$, 
and the cores $\gamma_i$ represent the singular fibres.

Now consider the loops $\sigma_i$, $\xi_j$ and $f$ of Figure 1 to 
be 1-cycles in $M$.
Let $\sigma=\Sigma_{i=1}^k\sigma_i$.
Then $\sigma$ has image $a=[\sigma]$ in $\tau_M$.
The relations in our original presentation for $\tau_M$ correspond to 
the presence of the following 2-chains in $M$.
The union of fibres over $\sigma_1$ may be considered as a 2-chain $A$ (with support 
the image of an annulus), with $\partial{A}=2f$.
(Note that the $\sigma_i$s are orientation-reversing loops!)
For each $i\leq{n}$ there is a 2-chain $D_i$ in $N(\gamma_i)$ with $\partial{D_i}=\alpha_i\xi_i+\beta_if$ in $M$. 
(Thus $D_i$ is homologous to a meridianal disc.)
Finally, we may consider $P_{k,n*}$ as a 2-chain in $M$ with $\partial{P_{k,n*}}=2\sigma+\Sigma_{i=1}^n\xi_i$.
Note that we have described each of $\partial{A}$, $\partial{P_{k,n*}}$ 
and the $\partial{D_i}$s explicitly  as 1-cycles.
We may now compute various linkings.

\begin{lemma}
\label{CH-lem3.6}
Using the above generators for $\tau_M$, we have 
\[
\ell_M(q_i,q_i)=\frac{\beta_i}{\alpha_i},\quad
\ell_M(q_i,q_j)=0~\mathrm{ if}~j\not=i,\quad\ell_M(q_i,h)=\ell_M(h,h)=0,
\]
\[
\ell_M(a,q_i)=-\frac{\beta_i}{2\alpha_i},\quad
\ell_M(a,a)=\frac14(2k-\varepsilon(M))\quad\mathrm{and}\quad\ell_M(a,h)=\frac12.
\]
\end{lemma}

\begin{proof}
The 1-cycles corresponding to the generators are shown in Figure 1 with $f$ directed out of the page, and we may use the right-hand rule to determine a positive intersection.
we may perturb each $\xi_i$ to a homologous 1-cycle $\xi_i'$ just inside $N(\gamma_i)$,
and perturb $\sigma$ to $\sigma'$ as illustrated in Figure 4.2.

\setlength{\unitlength}{1mm}
\begin{picture}(110,60)(-33.4,4)

\put(9,37){.}
\put(10,35){.}
\put(11,33){.}

\put(39,34){.}
\put(39,31.5){.}
\put(39,29){.}

\put(-3,43){$\xi_n'$}
\put(8,28){$\xi_1'$}
\put(23.5,16){$\sigma'$}

\put(22,57.1){$\blacktriangleright$}
\put(37,46.2){$\blacktriangleright$}
\put(6,46.1){$\blacktriangleright$}
\put(15.8,23.7){$\blacktriangleleft$}
\put(36.5,16){$\blacktriangleleft$}

\put(32,32.6){.}
\put(33,33.7){.}
\put(34.2,34.4){.}

\put(11.8,28){$\to$}

\put(21.7,25){$\nwarrow$}
\put(26.6,16){$\to$}
\put(1.5,43){$\to$}

\thinlines

\qbezier(15,35)(12,39)(15,44)
\qbezier(15,44)(16,47)(19,49)
\qbezier(19,49)(21,52)(23,55)
\qbezier(23,55)(26, 60)(29,55)
\qbezier(29,55)(30,52)(28,44)
\qbezier(28,44)(28,40)(32,40)
\qbezier(32,40)(34,40)(36,44)
\qbezier(36,44)(39,48)(42,44)
\qbezier(42,44)(43,42)(42,40)
\qbezier(42,40)(39, 36)(36,35)

\qbezier(15,35)(16,32)(19,30)
\qbezier(19,30)(26,25)(28,20)
\qbezier(28,20)(35,12)(39,20)
\qbezier(39,20)(41,23)(38,26)
\qbezier(38,26)(36,28)(33,30)
\qbezier(33,30)(32,31)(32,32)

\linethickness{1pt}

\qbezier(6,42.5)(4.5,44)(6,45.5)
\qbezier(6,45.5)(8,47)(10,45)
\qbezier(6,42.5)(8,41)(10,42)
\qbezier(10,42)(12, 43.5)(10,45)

\qbezier(16,26)(17.5,24.5)(19,26)
\qbezier(16,26)(15,29)(16,32)
\qbezier(19,26)(21,29)(20,32)
\qbezier(16,32)(18,35)(20,32)

\qbezier(19,39)(10,39)(5,42)
\qbezier(5,42)(3,44)(5,46)
\qbezier(5,46)(8,50)(19,39)
\qbezier(19,39)(18,56)(21,57)
\qbezier(21,57)(23,59)(26,57)
\qbezier(19,39)(23,39)(26,45)
\qbezier(26,45)(28,50)(27.5,54)
\qbezier(27.5,54)(27.5,56)(26,57)
\qbezier(19,39)(21,38)(24,38)
\qbezier(24,38)(30,38)(33.5,44)

\qbezier(33.5,44)(36,49)(41,46.5)
\qbezier(19,39)(22,36)(26,36)
\qbezier(26,36)(34,35)(38,38)
\qbezier(38,38)(44,42)(41,46.5)

\qbezier(19,39)(16,38)(15,32)
\qbezier(19,39)(21,38)(21,30)
\qbezier(15,32)(14,26)(16,25)
\qbezier(21,30)(21,26)(19,25)
\qbezier(16,25)(17.5,24)(19,25)

\qbezier(19,39)(30,28)(35,28)
\qbezier(35,28)(43,27)(41,19)
\qbezier(19,39)(28,30)(28,27)
\qbezier(28,27)(28,20)(32,18)
\qbezier(32,18)(39,15)(41,19)

\put(13,8){Figure 4.2}

\end{picture}

Then we have intersection numbers $\xi_i'\bullet{A}=\xi_i'\bullet{P_{n*}}=0$,
$\xi_i'\bullet{D_i}=\beta_i$ and $\xi_i'\bullet{D_j}=0$ if $j\not=i$,
while $\sigma'\bullet{A}=-1$  and $\sigma'\bullet{D_i}=-\beta_i$.
We also have $\sigma'\bullet{P_{n*}}=k$ (with the right choice of $\sigma'$,
for a different choice may give $-k$).
Now let $C_j=2D_j-\beta_jA$, so that $\partial{C_j}=2\alpha_j\xi_j$.
Then $\xi_i'\bullet{C_j}=2\xi_i'\bullet{D_j}=\delta_{ij}2\beta_i$,
and consequently $\ell_M(q_i,q_j)=\delta_{ij}\frac{\beta_i}{\alpha_i}$.
Also $\sigma'\bullet{C_i}=-2\beta_i+\beta_i=-\beta_i$, so that
$\ell_M(a,q_i)=-\frac{\beta_i}{2\alpha_i}$.
Now write $C=\Sigma_{i=1}^n\widehat\alpha_iC_i$, 
where $\widehat\alpha_i=\alpha_i^{-1}\Pi_{j=1}^n\alpha_j$, and let $R=2(\Pi_{j=1}^n\alpha_j)P_{n*}-C$.
Then $\partial{R}=4(\Pi_{j=1}^n\alpha_j)\sigma$.
Also
\[
\sigma'\bullet{R}=2(\Pi_{j=1}^n\alpha_j)k+\Sigma_{i=1}^n\widehat\alpha_i\beta_i=
(\Pi_{j=1}^n\alpha_j)(2k-\varepsilon(M)).
\]
Thus $\ell_M(a,a)=\frac14(2k-\varepsilon(M)$.
(Note that we may also use $R$ to check that
$\ell_M(q_i,a)=\frac{-\beta_i}{2\alpha_i}=\ell_M(a,q_i)$,
which also follows by the symmetry of the form.)
Finally, $\xi_i'\bullet{A}=0$ implies that $\ell_M(q_i,h)=0$ for $i\leq{n}$,
and $\sigma'\bullet{A}=-1$ implies that $\ell_M(a,h)=-\frac12=\frac12$, 
since $\partial{A}=2f$.
Also we may perturb $f$ to $f'$ so that $f'\bullet{A}=0$.
Thus $\ell_M(h,h)=0$.
\end{proof}

\begin{cor}
\label{P(2,2)e}
If  $M=M(-1;((2,1),(2,-1),(1,-e))$ then $\ell_M$ is hyperbolic if and only if 
$e\equiv2$  mod $(4)$.
\end{cor}

\begin{proof}
In this case $\tau_M\cong(\mathbb{Z}/4\mathbb{Z})^2$, with generators $a$ and $q_1$,
and the matrix of $\ell_M$ with respect to this basis is 
\[
L=\left(\smallmatrix\frac14(2-e)&-\frac14\\
-\frac14&\frac12\endsmallmatrix\right).
\]
It is easily checked that there is a matrix $A\in{GL(2,\mathbb{Z}/4\mathbb{Z})}$
such that $A^{tr}LA=\left(\smallmatrix0&1\\1&0\endsmallmatrix\right)$
if and only if $e\equiv2$  mod $(4)$.
\end{proof}

In particular,  the Hantzsche-Wendt flat manifold $M(-1;(2,1),(2,-1))$ does not embed in any homology 4-sphere. 
On the other hand,  0-framed surgery on the link $8_2^2$ gives an embedding of the $\mathbb{N}il^3$-manifold $M(-1;(2,1),(2,-1),(1,2))$.

\begin{theorem}
\label{CH-thm3.7}
Let $M=M(-k;S)$.
If $\ell_M$ is hyperbolic then all the even cone point orders $\alpha_i$ 
must have the same $2$-adic valuation.
If all the $\alpha_i$ are odd then we must also have
$\eta=\Sigma\alpha_i\beta_i\equiv2k~mod~(4)$.
\end{theorem}

Remark. The quantity $\eta$ is invariant {\it mod} (4) under Seifert data equivalences.

\begin{proof}
Enlarge the set of generators to $\{c_1,\dots,c_n,d_1,\dots,d_n,h,\hat{a}\}$,
where $c_i=2^{t_i}q_i-\beta_ih$ and $d_i=s_iq_i$ for $i\leq{n}$,
and $\hat{a}=a+\Sigma_{i=1}^n\frac12\rho_i(s_i+1)c_i$,
with $\rho_i$ chosen so that $2^{t_i}\rho_i\equiv1$ {\it mod} $(s_i)$.
This corresponds to a decomposition of $H_1(M)$
into a direct sum of $n$ cyclic factors $\langle{c_i}\rangle$ of odd order $s_i$,
for $i\leq{n}$, and a 2-group.
By using Lemma \ref{CH-lem3.6}, one can check that,
with respect to this set of generators, 
the linking matrix for $\ell_M$ is a block sum of the diagonal matrix
\[
\ell_M(c_i,c_j)=\mathrm{diag}
\left(\frac{2^{t_1}\beta_1}{s_1},\dots,\frac{2^{t_n}\beta_n}{s_n}\right)
\]
with the 2-group block
\[
\left(
\begin{matrix}
\frac{s_1\beta_1}{2^{t_1}}&0&\dots&0&-\frac{\beta_1}{2^{t_1+1}}\\
0&\frac{s_2\beta_2}{2^{t_2}}&\dots&0&-\frac{\beta_2}{2^{t_2+1}}\\
\vdots & \vdots & \ddots & \vdots&\vdots\\
0&0&\dots&0&\frac12\\
-\frac{\beta_1}{2^{t_1+1}}&-\frac{\beta_2}{2^{2_1+1}}&\dots&\frac12&{x}
\end{matrix}
\right),
\]
where $x=\ell_M(\hat{a},\hat{a})$.
When every $\alpha_i$ is odd this block simplifies to
$\left(\smallmatrix0&\frac12\\\frac12&\frac14(2k+\eta)\endsmallmatrix\right)$,
which is itself hyperbolic if and only if $2k+\eta\equiv0$ {\it mod} (4),
which gives the second part of the theorem.

Now suppose that, contrary to the conclusion of the first part of the theorem, 
there is some $m\in\{2,\dots,n\}$ such that $0<t_m<t_1=t$.
It remains to show in this case that $\ell_M$ is not hyperbolic.
Note that $\alpha_1$ and $\alpha_m$ are even, so that $\beta_1$ and $\beta_m$ are odd
(as well as $s_1$ and $s_m$ being odd).
Replace the generator $d_m$ with $\bar{d}_m=d_m+2^{t-t_m}s_md_1$
(or just $\bar{d}_m=s_m\bar{q}_m$ for $\bar{q}_m$ as in the proof of Lemma
\ref{CH-lem3.4}.
Then $\ell_M(\bar{d}_m,h)=0$, $\ell_M(\bar{d}_m,d_i)=0$ for $i\not=1$ or $m$ and
\[
\ell_M(\bar{d}_m,\bar{d}_m)=\ell_M(d_m,d_m)+(2^{t-t_M}s_m)^2\ell_M(d_1,d_1)
=\frac{u}{2^{t_m}},
\]
where $u=s_m\beta_m+2^{t-t_m}s_m^2t_1\beta_1$ is odd, 
since $t_m<t$ and $s_m\beta_m$ is odd.
Moreover
\[
\ell_m(\bar{d}_m,d_1)=\frac{v}{2^{t_m}}
\]
where $v=s_ms_1\beta_1$, and
\[
\ell_M(\bar{d}_m,\hat{a})=-\frac{\beta_m}{2^{t_m+1}}-2^{t-t_m}s-m\frac{\beta_1}{2^{t+1}}+
\frac{w}{2^{t_m}},
\]
where $w=-\frac12(\beta_m+s_m\beta_1)$ is also an integer.
Now choosing $r$ such that $ru\equiv1$ {\it mod} $(2^{t_m})$, 
replace $d_1$ and $\hat{a}$ with $\bar{d}_1=d_1-vr\bar{d}_m$ and
$\bar{a}=\hat{a}-wr\bar{d}_m$, respectively.
Then $\ell_M(\bar{d}_m,\bar{d}_1)=0$ and $\ell_M(\bar{d}_m,\bar{a})=0$.
Thus $\bar{d}_m$ generates an orthogonal direct summand of the 2-group.
But a hyperbolic linking on a 2-group admits no such summand
\cite[Theorem 0.1]{KK80}.
\end{proof}

Using Wall's presentation of the semigroup of linkings on $p$-groups
(for a fixed odd prime $p$) \cite{KK80,Wa64}, 
it is a routine matter to check whether or not the odd-order part of 
a given torsion linking pairing is hyperbolic.
In the case of Proposition \ref{CH-prop1.2},
where the strictly singular fibres occur in opposite pairs, 
we can easily check that $\eta\equiv-\varepsilon~mod~(4)$.
Moreover, 
since
\[
\left(\smallmatrix\frac{\beta_i}{\alpha_i}&0\\0&-\frac{\beta_i}{\alpha_i}
\endsmallmatrix\right)\sim
\left(\smallmatrix0&\frac1{\alpha_i}~\\
\frac1{\alpha_i}&0\endmatrix\right),
\]
the linking pairing is clearly hyperbolic  in this case if and only if
$\varepsilon\equiv2k$ {\it mod} (4).
Manifolds with all cone point orders $\alpha_i$ even may be treated case by case with
the help of the presentation of the semigroup of linkings on 2-groups given in \cite{KK80}.

%For instance, we may show that the manifolds fibred over $P(2,2)$
%(i.e., those of the form $M(-1;((2,1),(2,-1),(-e))$ for some $e$) have hyperbolic linking pairing %if and only if $e\equiv2$ {\it mod} (4).

\section{The main result}

 We shall assume that $\alpha_i=2^{t_i}s_i$ with $s_i$ odd, for all $i$.
We may also assume that there is an $r\geq0$ such that $t_1=\dots=t_r$,
and $t_i=0$ if $i>r$, by Theorem \ref{CH-thm3.7}.
In what follows we shall write $t=t_1$.
Since $\ell_M$ must be hyperbolic there is an even number, $2b$ say,
of cone points with $\alpha_i$ even.
Let $b_S=b-1$ if $b>0$ and $b_S=0$ if all the $\alpha_i$s are odd.

\begin{theorem}
\label{CH-thm1.1}
Let $S$ be strict Seifert data.
If $M_1=M(-k_1;S,(1,e_1))$ and $M_2=M(-k_2;S,(1,e_2))$ each embed in homology 
$4$-spheres then 
\[
(\varepsilon(M_1)-\varepsilon(M_2))\in\mathbb{Z}\cap2^{-t}\{-2K, 4-2K, \dots,2K-4,2K\},
\]
where $K=2^t(k_1+k_2+4b_S)+b_S$.
\end{theorem}

In particular, if the cone point orders $\alpha_i$ are all odd (so $t=b_S=0$)
and $M(-k; S, (1,e))$ embeds in a homology 4-sphere then there are
at most $k$ other values of $e$ for which the corresponding 3-manifold
embeds in a homology 4-sphere.

\begin{cor}
\label{CH-prop1.2}
Suppose that $S=\{\gamma_i(\alpha_i,\beta_i),\gamma_i(\alpha_i,\alpha_i-\beta_i)\mid 1\leq{i}\leq{n}\}$,
where each $\alpha_i$ is odd, and let $M=M(-c;S,(1,e))$.
Then $\varepsilon(M)=-n-e$, and $M$ embeds in a homology $4$-sphere 
if and only if 
$-2c\leq\varepsilon(M)\leq{2c}$ and $\varepsilon(M)\equiv{2c}\mod~(4)$.
Every such manifold embeds smoothly in $S^4$.
\end{cor}

\begin{proof}
On the one hand, there are at most $c+1$ possible values of $e$ corresponding to 3-manifolds which embed in homology 4-spheres, by Theorem \ref{CH-thm1.1}.

On the other hand,  we may adapt the construction of Lemma \ref{CH-lem3.1}.
There we saw that $M(S^2;((\alpha_i,\beta_i),(\alpha_i,-\beta_i)\cong\partial(L_i^*\times[-\epsilon,\epsilon])$ has a smooth embedding in $S^4$ 
for which there exists a 4-ball $B(M_i)=D^2\times{I}\times[-\epsilon,\epsilon]$ 
embedded in $S^4$ such that
$M_i\cap{B(M_i)}=\partial{D^2\times{I}\times[-\epsilon,\epsilon]}$,
and each circle $\partial{D^2}\times(r,s)$ is a regular fibre of $M_i$.
The embeddings of $M(-c;(1,e))$ (for $e$ in the above range) due to Massey \cite{Ma69} 
also have this property, since in fact each one bounds a disc bundle.
As in Lemma \ref{CH-lem3.1} we may form embedded fibre sums to obtain
the desired embedding.
These must exhaust the possibilities.
\end{proof}

Of course we may change the ambient homology 4-sphere by taking connected sums.

\begin{cor}
\label{CH-prop1.3}
Let $M$ be an $S^1$ bundle over a closed surface $F$, with Euler number $e$.
Then $M$ embeds in a homology $4$-sphere if and only if either
\begin{enumerate}
\item$F$ is orientable and $e=0$ or $\pm1$; or
\item$F=\#^c\mathbb{RP}^2$ and $e\in\{-2c,4-2c,\dots,2c\}$.
\end{enumerate}
In all cases $M$ embeds smoothly in $S^4$.
\end{cor}

\begin{proof}
The case when $F$ is non-orientable follows immediately from Corollary
\ref{CH-prop1.2}.
If $F$ is orientable of genus $g$  then 
$H_1(M)\cong\mathbb{Z}^{2g}\oplus(\mathbb{Z}/e\mathbb{Z})$, 
and so $\tau_M$ is a direct double if and only if $e=0$ or $\pm1$.
Lemma \ref{CH-lem3.3} gives smooth embeddings in $S^4$ in each of these cases.
\end{proof}

Thus orientable $S^1$-bundle spaces over non-orientable bases embed 
if and only if they embed smoothly, 
as boundaries of regular neighbourhoods of smooth embeddings of the base.
Note also that although $M(-3;(1,6))=M(-1;(1,6)\#_fT\times{S^1}$ embeds in $S^4$,
we cannot cancel the $T\times{S^1}$ term as $M(-1;(1,6))$ does not embed.

An involution of a connected orientable manifold $W$ is an orientation-preserving self-homeomorphism of $W$ of order 2.

\begin{lemma}
\label{CH-lem3.8}
Let $M=M(-k;S)$.
Then one of the complementary regions, $X$ say,
has a $2^{t+1}$-fold covering $\psi:\widetilde{X}\to{X}$
such that the restriction of the covering involution to $\widetilde{M}$
acts by rotating the fibres through $\pi$ radians.
(Here we assume that $\widetilde{M}$ has the Seifert structure induced by the covering
$\psi|_M$.)
\end{lemma}

\begin{proof}
Since $\tau_M$ must be a direct double,
we may assume by Lemma \ref{CH-lem3.4} that it has a direct summand 
$\mathbb{Z}/2\alpha_1\mathbb{Z}$ generated by $q_1$,
as in the proof.
In fact $\widehat{q}=s_1q_1$ generates a direct summand of $\tau_M$ of order $2^{t+1}$.
Note also that $h=\alpha_1q_1=2^t\widehat{q}$, 
where $h$ is the element of order 2 in $H_1(M)$ which is the image 
of a regular fibre $f$ of $M$.
Define an epimorphism $\phi:H_1(M)\to\mathbb{Z}/2^{t+1}\mathbb{Z}$ by projection onto the summand of $\tau_M$ generated by $\widehat{q}$.
Now $\widehat{q}$ may be identified with an element 
$(x_0,y_0)\in{H_1(X)}\oplus{H_1(Y)}$,
and one of $(x_0,0)$ or $(0,y_0)$ must have order $2^{t+1}$ under the quotient map $\phi$.
Suppose that $x_0$ satisfies this condition.
Thus $\phi$ defines a surjection
$\phi_X:H_1(X)\to\mathbb{Z}/2^{t+1}\mathbb{Z}$
such that $x\mapsto\phi(x,0)$, 
and there exists a $2^{t+1}$-fold covering 
$\psi:\widetilde{X}\to{X}$ with $\pi_1\widetilde{X}=\mathrm{Ker}(\phi_X\circ\mathrm{ab})$.

Since the inclusion $M\subset{X}$ induces the projection $p_X:H_1(M)\to{H_1(X)}$,
the regular fibre $f$ of $M$ represents the element $[f]=p_X(h)=2^tx_0$ in $H_1(X)$,
and since $\phi_X(x_0)$ has order $2^{t+1}$,
$\phi_X([f])$ will have order 2 in $\mathbb{Z}/2^{t+1}\mathbb{Z}$.
Thus $\psi|_M^{-1}(f)$ is a collection of $2^t$ disjoint circles, 
each of which doubly covers $f$.
On the other hand, if we take
\[
\overline{\phi}_X:H_1(X)\to\mathbb{Z}/2^{t+1}\mathbb{Z}\to
\mathbb{Z}/2^t\mathbb{Z}
\]
such that $[f]\mapsto0$, then we obtain an intermediate covering 
$\overline{\psi}_X:\overline{X}\to{X}$
where each fibre of $M$ lifts to a fibre in $\overline{M}$.
Thus $\widetilde{X}$ is a double cover of $\overline{X}$ with covering automorphism as required.
That is, each regular fibre of $\widetilde{M}$ doubly covers a fibre of $\overline{M}$,
the same being necessarily true for singular fibres as well.
\end{proof}

\begin{lemma}
\label{CH-lem3.9}
Let $M$ be as in the previous lemma. 
Then $\widetilde{M}$ is fibred over $\#^c\mathbb{RP}^2$,
where $c=2^t(k+b_S)-2b_S$ and $\varepsilon(\widetilde{M})=2^{t-1}\varepsilon(M)$.
\end{lemma}

\begin{proof}
Clearly $\widetilde{M}$ and $\overline{M}$ each have the same base orbifold,
$\widetilde{B}$, say.
Observe also that the cone point orders of $\widetilde{B}$ must all be odd since,
given a fibre $\gamma$ in $\overline{M}$ of type $(\alpha,\beta)$, 
$\beta$ must be even as $\gamma$ is covered by a fibre $\tilde\gamma$ 
of type $(\alpha, \frac12\beta+c\alpha)$ \cite[1.3]{NR78}.
Now, in the covering $\overline{M}\to{M}$ each regular fibre of $M$ lifts to $2^t$ distinct fibres, as do the singular fibres of odd multiplicity, while each of the remaining fibres, 
having multiplicity $2^ts_i$ for some odd $s_i$, 
must be $2^t$-fold covered by a single fibre of odd multiplicity.
The covering $\widetilde{M}\to\overline{M}$ induces a covering of orbifolds,
or a $2^t$-fold branched covering of surfaces $|\widetilde{B}|\to|B|$ with a branch set consisting of the $2b$ points corresponding to the $2b$ singular fibres with even multiplicity.
Thus $|\widetilde{B}|$ is also a non-orientable surface,
and so $|\widetilde{B}|\cong\sharp^{c}\mathbb{RP}^2$, 
for some $c>0$.
Recall that by definition of $b_S$, 
if $t\not=0$ then there are $2b_S+2$ fibres of even multiplicity.
In this case $\chi(|\widetilde{B}|)=2^t(\chi(|B|)-2b_S-2)+2b_S+2$, 
which gives $c=2^t(k+2b_S)-2b_s$.
But this formula also holds when $t=0$, in which case $c=k$.
Finally, $\varepsilon(\widetilde{M})=\frac12\varepsilon(\overline{M})=2^{t-1}\varepsilon(M)$,
as in \cite[1.2]{NR78} or \cite{JN}.
\end{proof}

We aim to construct involutions $g$ on a 4-manifold $W$ with fixed point set 
a closed connected surface $F$, smoothly embedded in $W$, and such that $g$ acts smoothly on a neighbourhood of $F$.
That is, $F$ will have a smooth $D^2$-bundle neighbourhood $N(F)\subset{W}$,
on which $g$ acts by rotation of the fibres.
We shall call such a surface a ``good" fixed point set.
Given such an $F$,
let $e(F)$ be the normal Euler number, which is just the Euler number of $\partial{N(F)}$,
considered as a circle bundle over $F$.
According to the $G$-Signature Theorem, 
if $g$ is an involution of a closed connected 4-manifold with ``good" fixed point set $F$ then $\mathrm{sign}(g,W)=e(F)$ \cite{AS68,Go86}. 
In proving Theorem \ref{CH-thm1.1}, we shall use also the fact that 
$|\mathrm{sign}(g,W)|\leq\beta_2(W)$,
and is congruent to $\beta_2(W)$ {\it mod} (2), 
for $W$ a closed 4-manifold.
(See also \S1.9.)

If we take $M$ to be the circle bundle $M(F;(1,-e))$ then,
in Lemma \ref{CH-lem3.8}, $\widetilde{M}$ is also a circle bundle, 
with Euler number $\tilde{e}=\frac12e$ and associated disc bundle $D_{\widetilde{M}}$.
In this case the covering involution of $\widetilde{X}$ given by  Lemma \ref{CH-lem3.8},
extends to an involution $g$ of the closed connected 4-manifold
$W_{\widetilde{M}}=\widetilde{X}\cup_{\widetilde{M}}D_{\widetilde{M}}$,
with ``good" fixed point set $\widetilde{F}\cong{F}$.
Thus $\mathrm{sign}(g,W)=e(\widetilde{F})=\tilde{e}$.

More generally, given a Seifert manifold $\widetilde{M}$ with fibration 
$\tilde\pi:\widetilde{M}\to\widetilde{F}$,
the ``associated disc bundle", or more precisely the mapping cylinder
\[
C_{\widetilde{M}}=\widetilde{M}\times[0,1]/(x_1,0)\sim
(x_2,0)~\mathrm{whenever}~\tilde\pi(x_1)=\tilde\pi(x_2)
\]
is in fact not necessarily a manifold.
It ceases to be a manifold precisely at the isolated points $p_i=(x_i,0)$ 
for $x_i$ on a strictly singular fibre.
However, we may modify $C_{\widetilde{M}}$ by removing a neighbourhood which contains all such points.
There exists a closed disc $\delta\subset\widetilde{F}$ such that the point
$\tilde\pi(\tilde\gamma)\in\mathrm{int}(\delta)$ for each singular fibre $\tilde\gamma$ of
$\widetilde{M}$.
Let $E=\tilde\pi^{-1}(\delta)\times[0,\frac12)/(\sim)\subset{C_{\widetilde{M}}}$ .
This defines a neighbourhood which meets $\widetilde{F}$ in $\delta$. 
(Note that $C_{\widetilde{M}}$  contains $\widetilde{F}$ as 
$\{(x,0)\in{C_{\widetilde{M}}}\}$.)
Define the ``modified associated disc-bundle" to be
$D_{\widetilde{M}}=C_{\widetilde{M}}\setminus\mathrm{int}(E)$.

Note that in the above construction $E$ is contractible to $\delta\sim*$,
so that $\chi(E)=1$.
Also its boundary $\partial{E}$ must be a rational homology sphere.
This is true since $\partial{E}$  may be wriiten as a generalized Seifert manifold
$M(0;(0,1),(\tilde\alpha_1,\tilde\beta_1),\dots,
(\tilde\alpha_m,\tilde\beta_m))$,
which has  finite first homology, the rest following by Poincar\'e duality,
given that $\partial{E}$  is closed, connected and orientable.

As in the statement of the theorem, for fixed $S$, 
let $M_1=M(k_1,S, (1,e_1))$ and $M_2=M(k_2,S, (1,e_2))$ 
be Seifert manifolds which embed in  homology 4-spheres $\Sigma_1$ and $\Sigma_2$, respectively.
Then we may write $\Sigma_1=X_1\cup_{M_1}Y_1$ and
$\Sigma_2=X_2\cup_{M_2}Y_2$ , where we may suppose there are coverings 
$\psi_1:\widetilde{X}_1\to{X_1}$ and $\psi_2:\widetilde{X}_2\to{X_2}$, 
as described in Lemma \ref{CH-lem3.8}.

Consider the manifold $W_1=\widetilde{X}_1\cup_{\widetilde{M}_1}D_{\widetilde{M}_1}$,
with boundary $\partial{E_1}$.
By Lemma \ref{CH-lem3.8},
there is a covering involution of $\widetilde{X}_1$ which extends to 
an involution $g_1$ of $W_1$ in the obvious way, 
with fixed point set $\widetilde{F}_1\setminus\mathrm{int}(\delta_1)$.
There is a similar pair $(g_2,W_2)$.
Since $\widetilde{M}_1$ and $\widetilde{M}_2$ have the same set $\widetilde{S}$ of strictly singular fibres, we may write $\widetilde{M}_i=
M(-\tilde{k_i};(\tilde\alpha_1, \tilde\beta_1),\dots, (\tilde\alpha_m, \tilde\beta_m),(1,-e_i))$,
and we may choose $E_1$ and $E_2$ to have exactly the same fibred structure.
That is,  there is a fibre-preserving homeomorphism $E_1\cong{E_2}$ which respects
the sets $\widetilde{F}_i\cap{E_i}$ and the involution on $\partial{E_i}$.
Now, if we write
\[
W=W_1\cup_{\partial{E}}W_2
\]
then $W$ is a closed connected 4-manifold,
and $g_1$ and $g_2$ together define an involution $g$ on $W$ with ``good" fixed
point set $F=\widetilde{F}_1\sharp\widetilde{F}_2$.
Furthermore $F$ has normal Euler number $e(F)=e_1-e_2=\varepsilon(\widetilde{M}_1)-\varepsilon(\widetilde{M}_2)$,
which is simply $2^{t-1}(\varepsilon(M_1)-\varepsilon(M_2))$, by Lemma \ref{CH-lem3.9}.
Note that $\partial(-W_2)\cong-\partial{E}$, 
while $\partial(W_1)\cong\partial{E}$.
However, in order to orient the union of two oriented manifolds consistently, 
we require that the attaching homeomorphism to be orientation-reversing, 
and here it is taken to be the obvious one.
Now, 
by the $G$-Index Theorem we may write
\[
\mathrm{sign}(g,W)=2^{t-1}(\varepsilon(M_1)-\varepsilon(M_2)).
\]
This signature is bounded by the value of $\beta_2(W)$, independently of the $\varepsilon_i$.
It remains now to express this constraint in terms of $k_1$ and $k_2$, and the constants $t$ and $b_S$, which depend only on $S$. We begin  by establishing a bound on $\beta_1(W_i)$ (for $i=1,2$), for which we need the next lemma.

\begin{lemma}
\label{CH-lem3.10}
Let $p$ be a prime. 
If $f:\widetilde{X}\to{X}$ is a $p^s$-fold cyclic covering then 
$\beta_1(\widetilde{X})\leq{p^s\beta_1(X;\mathbb{F}_p)}$.
%where $R$ is bounded in terms of an epimorphism from $H_1(X;\mathbb{Z})$ 
%to a group of the form $\langle\zeta,(\mathbb{Z}/p\mathbb{Z})^R\mid
%\zeta^{p^s}\in(\mathbb{Z}/p\mathbb{Z})^R\rangle$.
\end{lemma}

\begin{proof}
The covering gives an exact sequence
\[
1\to{K_0}\to{G_0}\to\mathbb{Z}/p^s\mathbb{Z}\to1,
\]
where $K_0=\pi_1\widetilde{X}$ and $G_0=\pi_1X$.
Since $K_0$ is finitely generated,  
$\overline{K}=K_0/I(K_0)$ is free abelian of rank 
$r=\beta_1(\widetilde{X})$.
Since $I(K_0)$ is characteristic in $K_0$ it is normal in $G_0$,
and so we have an exact sequence
\[1\to\overline{K}\to{G_0/I(K_0)}\to\mathbb{Z}/p^s\mathbb{Z}\to1.
\]
Let $K=\overline{K}/p\overline{K}\cong(\mathbb{Z}/p\mathbb{Z})^r$.
Then we get another exact sequence
\[
1\to{K}\to{G}\to\mathbb{Z}/p^s\mathbb{Z}\to1,
\]
where $G$ is the quotient of $G_0/I(K_0)$ by $p\overline{K}$.
Note that $G^{ab}$ is a a quotient of $H_1(X;\mathbb{Z})$.

Now 
$G\cong\langle\zeta,K\mid\zeta^{p^s}\in{K},~\zeta^{-1}x\zeta=Ax,~\forall{x}\in{K}\rangle$,
where $\zeta$ acts on $K$ via some $A\in{GL(r,p)}$.
Let $\Gamma=\mathbb{F}_p[T]$.
Then $K$ is a $\Gamma$-module, with $T$ acting via $t.x=Ax$, for all $x\in{K}$.
Since $\zeta^{p^s}\in{K}$ and $K$ is abelian, $T^{p^s}x=x$, for all $x\in{K}$,
and so $K$ is annihilated by $T^{p^s}-I=(T-I)^{p^s}$.
Since $\Gamma$ is a PID we  have $K\cong\oplus_{i=1}^m(\Gamma/(T-I)^i)^{e_i}$, 
for some $e_i\geq0$ and $m\leq{p^s}$.
Hence $r=\Sigma_{i=1}^mie_i\leq{p^s}\Sigma{e_i}$.
On the other hand,
$G^{ab}=\langle\zeta,K/(T-I)K\mid\zeta^{p^s}\in{K/(T-I)K}\rangle$, 
and $K/(T-I)K\cong\oplus_{i=1}^m(\Gamma/(T-I))^{e_i}\cong\mathbb{F}_p^R$,
where $R=\Sigma{e_i}$.
Thus $\beta_1(\widetilde{X})\leq{p^sR}$
and $R\leq\beta_1(X;\mathbb{F}_p)$.
% and $G^{ab}=\langle\zeta,(\mathbb{Z}/p\mathbb{Z})^R\rangle$.
\end{proof}

\begin{lemma}
\label{CH-lem3.11}
$\beta_1(W_i)\leq\beta_1(\widetilde{X}_i)\leq2^{t+1}(\beta_1(X_i)+b_S)$.
\end{lemma}

\begin{proof}
For convenience, we shall drop the subscript $i$ throughout, except to distinguish $W_i$ from $W$.
Apply Lemma \ref{CH-lem3.10} to the covering $\psi$ of Lemma \ref{CH-lem3.8},
where $p^s=2^{t+1}$.
The image of $\zeta$ in $G^{ab}$ has order $2^{t+1}$ or $2^{t+2}$.
But $\tau_X$ has no elements of order $2^{t+2}$, so $\zeta$ has order $2^{t+1}$
and $G^{ab}\cong\mathbb{Z}/2^{t+1}\mathbb{Z}\oplus(\mathbb{Z}/2\mathbb{Z})^R$,
where $R$ is as defined in Lemma \ref{CH-lem3.10}.
Since $\tau_M\cong\tau_X\oplus\tau_Y$ is a direct double,
$\tau_X/2\tau_X$ has dimension $b_S+1$, by Lemma \ref{CH-lem3.4}.
Hence we get an epimorphism from
$(\mathbb{Z}/2\mathbb{Z})^{\beta_1(X)+b_S+1}$ to $(\mathbb{Z}/2\mathbb{Z})^{R+1}$,
and hence $R\leq\beta_1(X)+b_S$.
But then $\beta_1(\widetilde{X})\leq2^{t+1}(\beta_1(X)+b_S)$.

Write 
$W^0=W_i\cup_{\partial{E}}E=\widetilde{X}\cup_{\widetilde{M}}C_{\widetilde{M}}$.
(Recall that $C_{\widetilde{M}}$ is contractible to $\widetilde{F}$.)
Then we have the Mayer-Vietoris sequence
\begin{equation*}
\begin{CD}
H_1(\widetilde{M})@>\phi>>H_1(\widetilde{X})\oplus{H_1(\widetilde{F})}\to{H_1(W^0)}\to0
\end{CD}
\end{equation*}
where $\phi$ is induced by the inclusions of $\widetilde{M}$ into 
$\widetilde{X}$ and $\widetilde{C}_{\widetilde{M}}$.
But the map to $H_1(\widetilde{F})$ is surjective,
so $\dim(\mathrm{Im}(\phi))\geq\beta_1(F)$,
and hence 
$\beta_1(W^0)=\beta_1(\widetilde{X})+\beta_1(\widetilde{F})-\dim(\mathrm{Im}(\phi))\leq
\beta_1(\widetilde{X})$.
Finally, since $\beta_1(E)=0$ and $\partial{E}$ is a rational homology sphere, the Mayer-Vietoris sequence (for $W^0=W_i\cup_{\partial{E}}E$) gives $\beta_1(W_i)=\beta_1(W^0)$,
and we have proved the lemma.
\end{proof}

We may now  prove Theorem \ref{CH-thm1.1}.

\begin{proof}
For  each $i=1,2$ the union $C_{\widetilde{M}_i}=D_{\widetilde{M}_i}\cup{E_i}$
gives $\chi(C_{\widetilde{M}_i})=\chi(D_{\widetilde{M}_i})+\chi(E_i)$.
That is, $\chi(D_{\widetilde{M}_i})=\chi(\widetilde{F}_i)-1$,
since $C_{\widetilde{M}_i}$ is contractible to $\widetilde{F}_i$ and $\chi(E_i)=1$.
Thus since $\Sigma=
\widetilde{X}_1\cup{D_{\widetilde{M}_1}}\cup{D_{\widetilde{M}_2}}\cup\widetilde{X}_2$,
where the unions are along closed 3-manifolds, we have
\[
\chi(W)=2^{t+1}\chi(X_1)+2^{t+1}\chi(X_2)+\chi(\widetilde{F}_1)+\chi(\widetilde{F}_2)-2.
\]
By Poincar\'e duality, we may then write
\[
\beta_2(W)=2(\beta_1(W_1)+\beta_1(W_2)+2^{t+1}(\chi(X_1)+\chi(X_2))+\chi(\widetilde{F}_1)+\chi(\widetilde{F}_2)-4
\]
\[
=\Xi_1+\Xi_2, \quad\mathrm{where}\quad\Xi_i=2\beta_1(W_i)+2^{t+1}\chi(X_i)+\chi(\widetilde{F}_i)-2,
\]
where the Mayer-Vietoris sequence (for $W=W_1\cup_{\partial{E}}W_2$)
gives $\beta_1(W)=\beta_1(W_1)+\beta_1(W_2)$,
since $\partial{E}$ is a rational homology sphere.

Since  $X_i$ is a complementary region for an embedding of $M_i$ 
in a homology 4-sphere $\Sigma_i$,
$\beta_1(X_i)+\beta_2(X_i)=\beta_1(M_i)=k_i-1$, 
while $\beta_j(X_i)=0$ for $j>2$.
The inequality of Lemma \ref{CH-lem3.11} and these relations now give
\[
\Xi_i\leq2^{t+2}(\beta_1(X_i)+b_S)+2^{t+1}\chi(X_i)+\chi(\widetilde{F}_i)-2
\]
\[
=2^{t+1}(1+\beta_1(X_i)+\beta_2(X_i)+2b_S)+\chi(\widetilde{F}_i)-2
\]
\[
=2^{t+1}(k_i+2b_S)-\tilde{k}_i=K_i, \quad\mathrm{say}.
\]
Consequently, putting $K=K_1+K_2$, we have
$|\mathrm{sign}(g,W)|\leq\beta_2(W)\leq{K}$.
Also $\Xi_i\equiv\chi(\widetilde{F}_i)
\equiv\tilde{k_i}\equiv{K_i}$ {\it mod} (2),
so that $\mathrm{sign}(g,W)\equiv\beta_2(W)\equiv{K}$ {\it mod} (2).
Thus  $2^{t-1}(\varepsilon(M_1)-\varepsilon(M_2))=
\mathrm{sign}(g,W)\in\{-K,2-K,\dots ,K\}$.
Hence, since $\varepsilon(M_1)-\varepsilon(M_2)$ must be an integer, 
it follows that
\[
\varepsilon(M_1)-\varepsilon(M_2)\in\{ -\frac{2K}{2^t}, \frac{4-2K}{2^t},\dots,\frac{2k}{2^t}\}\cap\mathbb{Z}.
\]
Finally, note that since $\tilde{k}_i=2^t(k_i+2b_S)-2b_S$, by Lemma \ref{CH-lem3.9},
we have $K_i=2^2(k_i+2b_S)=2b_S$, and so $K=2^t(k_1+k_2+4b_S)+4b_s$,
which completes the proof.
\end{proof}

\section{Some further remarks}

Let $M=M(-c;S)$ be a Seifert manifold
with base $B=\#^k\mathbb{RP}^2(\alpha_1,\dots,\alpha_r)$,
where $c>0$ and all cone point orders $\alpha_i$ are odd.
(Then $M$ is an $\mathbb{H}^2\times\mathbb{E}^1$-manifold or a 
$\widetilde{\mathbb{SL}}$-manifold unless $c+r\leq2$.)
Since $\varepsilon(M)$ is a rational number with odd denominator  
it has a well-defined image in $\mathbb{Z}/2^s\mathbb{Z}$, for any $s\geq0$.
In particular,
$\varepsilon(M)\equiv\Sigma\beta_i$ {\it mod\/} $(2)$ and 
$\varepsilon(M)\equiv-\Sigma\alpha_i\beta_i$ {\it mod\/} $(4)$,
since $\alpha_i\equiv1$ {\it mod\/} $(2)$ and 
$\alpha_i\equiv\alpha_i^{-1}$ {\it mod\/} $(4)$ if $\alpha_i$ is odd.
Thus the invariants $c$ and $\eta$ used in 
Lemma \ref{CH-lem3.4} and Theorem \ref{CH-thm3.7}
are just the images of $-\varepsilon(M)$ in $\mathbb{Z}/2\mathbb{Z}$ and 
$\mathbb{Z}/4\mathbb{Z}$ (respectively).
Hence $\varepsilon(M)\equiv2k$ {\it mod\/} $(4)$ if $\ell_M$ is hyperbolic, 
by 
Theorem \ref{CH-thm3.7}.

If $\varepsilon(M)=0$, 
then $\tau_M\cong(\mathbb{Z}/2\mathbb{Z})^2\oplus
\bigoplus_{i\geq1}\mathbb{Z}/\alpha_i\mathbb{Z}$,
by Theorem \ref{CH-thm3.7}.
Therefore if $\tau_M$ is a direct double the numbers $\#\{i:v_p(\alpha_i)=j\}$ 
are even for all odd primes $p$ and exponents $j\geq1$.
If, moreover, $r=3$, then it follows from Lemma \ref{CH-lem3.6} 
that $\ell_M$ is hyperbolic.
Must $S$ be skew-symmetric if $\varepsilon(M)=0$ and $M(-k,S)$ embeds in $S^4$?
The first cases to test are perhaps those with base $P_2(p,q,pq)$,
where $p$ and $q$ are distinct odd primes.
(When $k=2$ one complementary domain, $X$ say, 
has $H_1(M;\mathbb{Q})\cong{H_1(X;\mathbb{Q})}$ and $\chi(X)=0$.
However the argument of Theorem \ref{Hi09-thm4.1} does not appear to apply usefully here,
as we must first pass to the 2-fold cover of $M$ induced by the orientation 
cover of the base $B$ before continuing to an infinite cyclic cover 
homeomorphic to $F\times\mathbb{R}$.
There is no obvious reason that the Blanchfield pairing associated to 
the latter cover should be neutral.)

Donald shows that if $S'$ is strict Seifert data and $M=M(-c;S', (1,-e))$
embeds smoothly in $S^4$ then the Seifert invariants must occur in pairs 
$\{(\alpha, \beta),(\alpha,\beta')\}$ where $\beta'=\alpha-\beta$ or $\beta'=\alpha-\beta^{[-1]}$,
with $\beta.\beta^{[-1]}\equiv1$ {\it mod} $(\alpha)$.
Moreover, if any of the cone point orders are even then all such cone point orders are equal,
and if $(\alpha, \beta)$ and $(\alpha, \beta')$ are Seifert invariants with $\alpha$ even then
$\beta'=\beta$, $\alpha-\beta$, 
$\beta^{[-1]}$ or $\alpha-\beta^{[-1]}$
\cite[Theorem 1.2]{Do15}.

%% file: e5.tex
\chapter{3-manifolds with restrained fundamental group}

In this chapter we shall show first that there are just thirteen 3-manifolds with restrained fundamental group which embed in homology 4-spheres, and all embed in $S^4$.
The most difficult part of the argument involves consideration of 3-manifolds which are the union of two copies of the mapping cylinder $N$ of the orientation cover of the Klein bottle.
(This is an extension of the work in Chapter 4.)
We then show that if there is a degree-1 map $f:M\to{P}$ which induces isomorphisms 
on $\mathbb{Z}[\pi_1P]$-homology and $\pi_1P$ is torsion-free and solvable
then $M$ embeds in $S^4$ if and only if $P$ does.
(This result involves 4-dimensional surgery over $\pi_1P$.)

\section{Statement of the main result on virtually solvable groups}

If the fundamental group of a 3-manifold $M$ is restrained then it is either finite or solvable,
and $M$ has one of the geometries $\mathbb{S}^3$, 
$\mathbb{S}^2\times\mathbb{E}^1$,
$\mathbb{E}^3$, $\mathbb{N}il^3$ or $\mathbb{S}ol^3$.
In the first four cases $M$ is a Seifert manifold.

The only infinite solvable 3-manifold group with torsion is $D_\infty$,
which we may exclude by the next result.

\begin{lemma}
\label{Dinfty}
If $\pi_1M\cong{D_\infty}$ then $M$
does not embed in any homology $4$-sphere.
\end{lemma}

\begin{proof}
If $\pi_1M\cong{D_\infty}$ then $M\cong\mathbb{RP}^3\#\mathbb{RP}^3$,
and so $\tau_M=H_1(M)$ is generated by elements $x_1$ and $x_2$ such that 
$\ell_M(x_i,x_i)\not=0$. 
The lemma now follows immediately from Corollary \ref{KKcor}.
\end{proof}

If $M$ is flat, a $\mathbb{N}il^3$-manifold or $\mathbb{S}ol^3$-manifold 
and $\pi/\pi'$ is infinite then $M$ is the total space of a torus bundle over $S^1$,
and conversely every such bundle has one of these geometries. 

If $M$ has one of these geometries and $\pi/\pi'$ is finite then either
$M$ is the Hantzsche-Wendt flat manifold,
or $M$ is a $\mathbb{N}il^3$-manifold and is Seifert fibred over one of
$S(2,2,2,2)$, $S(2,3,6)$, $S(2,4,4)$, $S(3,3,3)$ or $P(2,2)$, 
or $M$ is a $\mathbb{S}ol^3$-manifold.
The Hantzsche-Wendt flat manifold has fundamental group $G_6$ with presentation 
 \[
 \langle{x,y}\mid{xy^2x^{-1}=y^{-2},~
yx^2y^{-1}=x^{-2}}\rangle,
\]
and $G_6/\langle\langle{x^2,y^2}\rangle\rangle\cong{D_\infty}$.
The orbifold fundamental groups of $S(2,2,2,2)$ and $P(2,2)$ each map onto $D_\infty$.
If $\pi$ is the group of a $\mathbb{S}ol^3$-manifold then $\sqrt\pi\cong\mathbb{Z}^2$, 
and so $\pi/\sqrt\pi$ is virtually $\mathbb{Z}$.
Moreover, $\pi/\sqrt\pi$ has no non-trivial finite normal subgroup,
 and so $\pi/\sqrt\pi\cong{D_\infty}$, if $\pi/\pi'$ is finite.
In all cases, 
excepting only $\mathbb{N}il^3$-manifolds which have base orbifold  
$S(2,3,6)$, $S(2,4,4)$ or  $S(3,3,3)$, 
if $\pi/\pi'$ is finite then $\pi$ maps onto $D_\infty$, with kernel $\mathbb{Z}^2$.
We then have $\pi\cong{A*CB}$, 
where $A$ and $B$ are torsion-free, $[A:C]=[B:C]=2$ and $C\cong\mathbb{Z}^2$.
Since $\pi/\pi'$ is finite the subgroups $A$ and $B$ cannot be $\mathbb{Z}^2$.

We may assume that $A$ and $B$ have presentations 
$\langle{u,v}\mid{vuv^{-1}=u^{-1}}\rangle$ and
$\langle{x,y}\mid{yxy^{-1}=x^{-1}}\rangle$, respectively.
Let $C<A$ have basis $\{u,v^2\}$.
Then we may define a monomorphism $\phi:C\to{B}$ by
 $\phi(u)= x^ay^{2b}$ and $\phi(v^2)=x^cy^{2d}$,  where $ad-bc=1$.
Let $G=A*_\phi{B}$ be the resulting generalized free product with amalgamation.
Then $G/G'$ is infinite if and only if $c=0$.
Otherwise, $G$ is virtually abelian if $a=d=0$, is virtually nilpotent
if and only if exactly one of $a,b$ or $d$ is 0,
and is solvable but not virtually nilpotent if all the entries are non-zero.

Such 3-manifolds have a corresponding decomposition as unions of two
copies of the mapping cylinder $N$ of the orientation cover of $Kb$.
This 3-manifold $N$ is orientable and $\partial{N}\cong{T}$,
and $\pi_1N=\pi_1Kb$ has generators $x$ and $y$ such that $xyx^{-1c}=y^{-1}$.
Fix a homeomorphism $h:\partial{N}\to{T}$ which carries $x^2$ and $y$ to
the standard basis of $\pi_1T=\mathbb{Z}^2$.
Then isotopy classes of self-homeomorphisms of $\partial{N}$ correspond to
automorphisms of $\mathbb{Z}^2$.

Every 3-manifold $M$ with a solvable Lie geometry ($\mathbb{E}^3$,
$\mathbb{N}il^3$ or $\mathbb{S}ol^3$) and such that $\pi_1M$ maps onto $D_\infty$
is a union $N\cup_\phi{N}$,  for some $\phi$ is a self homeomorphism of $N$.
It is easily seen that the automorphism 
$\left(\smallmatrix1&0\\0&-1\endsmallmatrix\right)$ is induced by 
a self-homeomorphism of $N$, and so we may assume that $\det\phi=1$.

We shall state our main result now, but defer the more difficult parts of the proof.

\begin{theorem}
\label{CH-thm2.2}
If $\pi$ is finite or solvable and $M$ embeds in a homology $4$-sphere then $M$ 
is either the Poincar\'e homology sphere or is one of the twelve $3$-manifolds 
listed above.
More precisely, either
\begin{enumerate}
\item$\pi=1$, $Q(8)$ or $I^*=SL(2,5)$ and $M$ is an $\mathbb{S}^3$-manifold; or
\item$\pi\cong\mathbb{Z}$ and $M\cong{S^2\times{S^1}}$; or
\item$\pi\cong\mathbb{Z}^3$ or $G_2=\mathbb{Z}^2\rtimes_{-I}\mathbb{Z}$, 
and $M$ is a flat $3$-manifold; or
\item$\pi\cong\mathbb{Z}^2\rtimes_A\mathbb{Z}$, where
$A=\left(\smallmatrix 1&1\\
0&1\endsmallmatrix\right)$ 
or $A=\left(\smallmatrix -1&4\\
0&-1\endsmallmatrix\right)$ 
and $M$ is a $\mathbb{N}il^3$-manifold; 
\item$M\cong{N\cup_\phi{N}}$, where $\phi=
\left(\smallmatrix2&-1\\1&0\endsmallmatrix\right)$,
and $M$ is a $\mathbb{N}il^3$-manifold;  or
\item$M\cong{N\cup_\phi{N}}$, where $\phi=
\left(\smallmatrix2&-3\\1&2\endsmallmatrix\right)$,
$\left(\smallmatrix2&-5\\1&-2\endsmallmatrix\right)$ or
$\left(\smallmatrix2&-9\\1&-4\endsmallmatrix\right)$,
and $M$ is a $\mathbb{S}ol^3$-manifold; 
\end{enumerate}
\end{theorem}

\begin{proof}
Suppose first that $\pi$ is finite. 
Then it is a finite subgroup of $SO(4)$ which acts freely on $S^3$.
If $M$ embeds in $S^4$ then $\pi/\pi'$ is a direct double, 
and so $\pi=1$, $I^*$ or is a generalized quaternionic group $Q(8k)$, where $k\geq1$. 
The 3-sphere $S^3$ is the equator of $S^4$,
while the homology sphere $S^3/I^*$ embeds in $S^4$ by Freedman's result.

The quotients $S^3/Q(8k)$ with $k\geq1$ are the total spaces of $S^1$ bundles over 
$\mathbb{R}^2$ with normal Euler number $2k$.
In particular,  $S^3/Q(8)$ is the boundary of a regular neighbourhood of a
smooth embedding of $\mathbb{RP}^2$ in $S^4$.
However none of the others embed.
See Corollary \ref{CH-prop1.3}.

If $\pi$ has two ends then $\pi\cong\mathbb{Z}$,
since $D_\infty$ is excluded by Lemma \ref{Dinfty}.
Hence $M\cong{S^2\times{S^1}}$,
which is the boundary of a regular neighbourhood of a
smooth embedding of $S^2$ in $S^4$.

If $\pi$ is not $\mathbb{Z}$ and $\pi/\pi'$ is infinite then 
$\beta=\beta_1(M)=1,2$ or 3.
If $\beta=1$ then $\rho$ is uniquely a semidirect product 
$\mathbb{Z}^2\rtimes_A\mathbb{Z}$, for some $A\in{SL(2,\mathbb{Z})}$.
Let $\Delta_A(t)=\det(A-tI)$ be the characteristic polynomial.
Then $\Delta_A(1)\not=0$, since $\beta=1$.
If $\Delta_A(-1)\not=0$ also then Kawauchi's Theorem applies,
and $\Delta_A(t)$ is irreducible, giving a contradiction.
Hence we must have $\Delta_A(t)=(t+1)^2$, 
and so $A=\left(\smallmatrix -1&b\\
0&-1\endsmallmatrix\right)$ for some $b$.
The mapping torus $T\rtimes_AS^1$ is also the total space of an $S^1$-bundle 
over $Kb$ with Euler number $b$.
This embeds in $S^4$ only if $b=0$ or $\pm4$, by Corollary \ref{CH-prop1.3}.

When $b=0$ we have $\pi=\mathbb{Z}^2\rtimes_{-I}\mathbb{Z}$
and $M$ is the flat 3-manifold with holonomy of order 2.
(The corresponding link is $8^3_9$.)
This is the boundary of a regular neighbourhood of 
a smooth embedding of $Kb$ with normal Euler number 0.
When $b=4$ the corresponding link is $9^3_{19}$.

If $\beta=2$ then $M$ is a $\mathbb{N}il^3$-manifold,
and $\pi/\pi'\cong\mathbb{Z}^2\oplus(\mathbb{Z}/q\mathbb{Z})$,
for some $q\geq1$.
Clearly the torsion is a direct double only when $q=1$.
In this case $M$ is the Heisenberg 3-manifold, 
with corresponding link $5^2_1$.

If $\beta=3$ then $\pi\cong\mathbb{Z}^3$ and $M$ is the 3-torus,
which embeds smoothly as the boundary of a regular neighbourhood of a
smooth embedding of $T$ in $S^4$ with a product neighbourhood.
(The corresponding link is $6^3_1$.)

We may assume henceforth that $\pi$ has one end and $\pi/\pi'$ is finite.
The only such flat 3-manifold group is the group of the Hantzsche-Wendt
flat 3-manifold with holonomy $(\mathbb{Z}/2\mathbb{Z})^2$.
This does not embed as its torsion linking form is not hyperbolic,
by Corollary \ref{P(2,2)e}.

If $M$ is a $\mathbb{N}il^3$-manifold with
Seifert base $S(3,3,3)$ and $\pi/\pi'$ is a direct double 
then $M\cong{M(0;(3,1),(3,1),(3,-1))}$, with corresponding link $6^2_1$.

The remaining cases all involve 3-manifolds of the form $M\cong{N\cup_\phi{N}}$.
We defer discussion of this final group.
\end{proof}

Eleven of the twelve such 3-manifolds which embed in homology 4-spheres 
have presentations in terms of 0-framed bipartedly trivial links
(and thus embed smoothly in $S^4$),
namely six relatively familar links:

$U=0_1$ (the unknot), giving $S^2\times{S^1}$;

$Ho=2^2_1$ (the Hopf link), giving $S^3$;

$4^2_1$ (the $(2,4)$-torus link), giving $S^3/Q(8)$;

$Wh=5^2_1$ (the Whitehead link), giving the $\mathbb{N}il^3$-manifold
$M(1;(1,-1))$;

$6^2_1$ (the $(2,6)$-torus link), giving the $\mathbb{N}il^3$-manifold
$M(0;(3,1),(3,1),(3,-1))$;

$Bo=6^3_2$ (the Borromean rings), giving the 3-torus $S^1\times{S^1}\times{S^1}$;\\
and five without popular names:
$8^3_9,8^2_2, 9^3_{19}, 9^2_{53}$ and $9^2_{61}$, giving the half-turn flat 3-manifold, two $\mathbb{N}il^3$-manifolds and two $\mathbb{S}ol^3$-manifolds,
respectively.
One further $\mathbb{S}ol^3$-manifold is given by a bipartedly 
ribbon link $(8_{20},U)$.

The thirteenth 3-manifold with $\pi$ restrained and which embeds is
the Poincar\'e homology sphere $S^3/I^*$.
This does not embed smoothly, as it bounds a 1-connected smooth 4-manifold with intersection form $E_8$ and so has non-zero Rochlin invariant.
(Hence $S^3/I^*$ does not have a bipartedly slice link presentation.)

Seven of the first twelve also arise as the boundaries of 
regular neighbourhoods 
of smooth embeddings of the 2-sphere  $S^2$,  the torus $T$, 
the projective plane $\mathbb{RP}^2$ or the Klein bottle $Kb$ in $S^4$. 
Closed orientable surfaces embedded in $S^4$ have product neighbourhoods
(normal Euler number 0), while the normal Euler number of an embedding of $\#^c\mathbb{RP}^2$
may have any value congruent to $2c~mod~(4)$ and between $-2c$ and $2c$  \cite{Ma69}.

\section{Unions of mapping cylinders}

The hard work is in dealing with the final two cases.
We need some further preparatory discussion and several lemmas before we can complete the proof of the main result (in Theorem \ref{CH-thm4.4}).

\begin{theorem}
\label{CH-thm4.1}
Let $M=N\cup_\phi{N}$,  where $N$ and $\phi=\left(\smallmatrix{a}&b\\c&d\endsmallmatrix\right)$ are as above. Then
\begin{enumerate}
\item{}if $c=0$ then $\ell_M$ is hyperbolic if and only if $4$ divides $b$;
\item{}if $c\not=0$ then $\ell_M$ is hyperbolic if and only if $b$ is odd, $c=1$ and $a$
and $d$ are both even, but NOT both divisible by $4$.
\end{enumerate}
\end{theorem}

\begin{proof}
We shall identify the Klein bottle $Kb$ with its image in $N$.
Then the inclusion of $\partial{N}\cong{T}$ into $N\simeq{Kb}$ is homotopic 
to the orientation cover $\psi:T\to{Kb}$.
The Klein bottle is also a twisted $S^1$ bundle over the circle.
Let $f$ and $\zeta$ be 1-cycles in $Kb$ corresponding to a fibre and section of this bundle.
We take $\xi=\psi^{-1}(\zeta)$ and a component $f'$ of $\psi^{-1}(f)$ as 1-cycles in 
$\partial{N}$.
Note that the bundle structure on $Kb$ induces a bundle structure on $N$ 
with typical fibres $f$ and $f'$ and section the mapping cylinder of $\pi|\xi$.
This section is a M\"obius band with centreline $\zeta$ and boundary $\xi$, which may be regarded as a 2-chain $C$ with $\partial{C}=2\zeta-\xi$,
The Klein bottle itself may also be considered as a 2-chain with $\partial{Kb}=2f$.
Also, the fibres $f$ and $f'$ cobound an annular 2-chain $A$ such that, 
putting $S=Kb+2A$, we have $\partial{S}=2f'$.
Perturbing $\zeta$ and $f'$ to intersect these 2-chains transversely, we see that $\zeta\bullet{S}=\pm1$, $\zeta\bullet{C}=\pm1$, $f'\bullet{S}=0$ and $f'\bullet{C}=\pm1$.

Now $H_1(N)\cong\langle{x,y}\mid 2y=0\rangle$, where $x=[\zeta]$ and $y=[f']$,
and $([\xi],[f'])=(2x,y)$ provides the standard framing of $\partial{N}$.
Thus $H_1(M)$ has the presentation
\[
\langle{x_1,y_1,x_2,y_2}\mid2y_1=2y_2=0,~2ax_1+by_1=2x_2,~2cx_1+dy_1=y_2\rangle,
\]
where we use subscripts to distinguish objects related to the two copies of $N$ forming $M$.
It is easily seen that if $c=0$ then $H_1(M)$ is the direct sum of $\mathbb{Z}$ 
with a group of order 4, and the torsion subgroup is a direct double if and only if $b$ is even, while if $c\not=0$ then $H_1(M)$ is finite, of order $16c$, 
and is a direct double if and only if $c=1$ and $b$ is odd.
We shall now determine $\ell_M$ in these two cases.

If $c=0$ and $b$ is even then $a=d=\pm1$, 
and $\tau_M$ is generated by $y_1$ and $z=ax_1-x_2$.
The latter element is represented by the 1-cycle $\zeta'=a\zeta_1-\zeta_2$.
Under the identification map $\phi$, the 1-cycles $a\xi_!+bf+1'$ and $\xi_2$
are homologous and so bound a 2-chain $\Delta\subset\partial{N}$,
and the chain $E=aC_1-C_2-\frac12bS_1+\Delta$ has boundary $2\zeta'$.
Hence
$\ell_M(y_1,y_1)=\frac12(f_1'\bullet{S_1})=0$,
$\ell_M(y_1,z)=\frac12(f_1'\bullet{E})=\frac12(\zeta'\bullet{S_1})=\frac12$
and $\ell_M(z,z)=\frac12(\zeta'\bullet{E})=\pm\frac14b$ in $\mathbb{Q}/\mathbb{Z}$,
and (i) follows easily.

If $c=1$ and $b$ is odd then $H_1(M)$ is generated by $x_1$ and $x_2$.
There are 2-chains $\Delta_1,\Delta_2\subset\partial{N}$ such that the chains
$G_1= 2C_1-dS_1+S_2+\Delta_1$ and $G_2= 2C_2-S_1+aS_2+\Delta_2$ 
have boundaries $4x_1$ and $4x_2$, respectively.
(Here we have used both identifications $a\xi_1+bf_1'\sim\xi_2$ and $\xi_1+df_1'\sim{f_2'}$
due to $\phi$, as well as the fact that $\det(\phi)=ad-bc=1$.)
Hence $\ell_M(x_1,x_1)=\frac14(\zeta_1\bullet{G_1})=\frac14(2\pm{d})$,
$\ell_M(x_1,x_2)=\frac14(\zeta_1\bullet{G_2})=\pm\frac14$ and
$\ell_M(x_2,x_2)=\frac14(\zeta_2\bullet{G_2})=\frac14(2\pm{a})$
in $\mathbb{Q}/\mathbb{Z}$,
and (ii) follows easily.
\end{proof}

It follows from Theorem \ref{CH-thm4.1} that if $M=N\cup_\phi{N}$ embeds in a homology
4-sphere then either $c=0$, in which case it is a torus bundle over $S^1$ and thus one of the manifolds listed in Theorem \ref{CH-thm2.2}, or $c=1$.
In the second case we introduce the notation $M_{m,n}$ for the union corresponding to 
$\phi=\left(\smallmatrix{m}&{mn-1}\\1&{n}\endsmallmatrix\right)$.

The group $\pi_1M_{0,n}$ has the presentation $\langle{x,u}\mid{x^2=(xu^2)^2=1}\rangle$.
The quotient by the infinite cyclic normal subgroup generated  by $x^2$ is $\pi^{orb}P(2,2)$.
Hence $M_{0,n}$ is Seifert fibred over $P(2,2)$,
with generalized Euler invariant $n$. 
Thus $M_{0,n}$ has hyperbolic linking form 
if and only if $n\equiv2$ {\it mod} (4),
by Corollary \ref{P(2,2)e}.

Let $N(e)$ denote the orientable $S^1$-bundle over $\mathbb{RP}^2$ 
with Euler number $e$,  and let $N^*(e)$  denote the framed exterior of a fibre $f$ in $N(e)$, with framing $([\mu],[f'])$ given by a meridian $\mu$ and a neighbouring fibre $f'$.
For example, the mapping cylinder with the $S^1$-bundle structure and framing
$([\xi],[f']_=(2x,y)$ used in the preceding proof is precisely the bundle $N^*(0)$.
An alternative description for $M_{m,n}$ is $N\cup_{\phi'}N$,
where $\phi'=\left(\smallmatrix{m}&1-mn\\1&-n\endsmallmatrix\right)$.
Now $\phi'=\left(\smallmatrix1&{m}\\0&1\endsmallmatrix\right)
\left(\smallmatrix0&1\\1&0\endsmallmatrix\right)
\left(\smallmatrix1&-n\\0&1\endsmallmatrix\right)$.
Hence we may write
\[
M_{m,n}=N^*(m)\cup_{\left(\smallmatrix0&1\\1&0\endsmallmatrix\right)}N^*(n).
\]
That is, $M_{m,n}$ may be described by the plumbing diagram

\setlength{\unitlength}{1mm}
\begin{picture}(80,20)(-16,0)
\linethickness{2pt}
\put(30,10.1){$\bullet$}
\put(30,13){$m$}
\put(28,5){$[-1]$}
\put(50,10.1){$\bullet$}
\put(50,13){$n$}
\put(48,5){$[-1]$}
\put(31,11){\line(1,0){20}}
\end{picture}\\
as in \cite{Ne81}.
It is clear from this description firstly that $M_{m,n}\cong{M_{n,m}}$ and secondly that reversing the orientation of $M_{m,n}$ corresponds to changing the sign 
of both $m$ and $n$.

By using an argument similar to that for the proof of Theorem \ref{CH-thm1.1},
 we are able to put strong homological constraints on embeddings of the manifolds $M_{m,n}$, and hence complete the classification of those unions $N\cup_\phi{N}$ which embed.

Suppose that $M=M_{m,n}$ embeds in a homology 4-sphere $\Sigma$,
with complementary regions $X$ and $Y$.
Then each summand of  the finite group $H_1(M)\cong{H_1(X)\oplus{H_1(Y)}}$, 
is self-annihilating with respect to $\ell_M$, by Lemma \ref{hyperbolic lp}.
We shall relabel  the generators $x_1,y_1,x_2,y_2$ as $x,y,u,v$, respectively.
Then $H_1(M)\cong(\mathbb{Z}/4\mathbb{Z})^2$, with generators $x$ and $u$.
We also have $y=2u$ and $v=2x$.
By the proof of Theorem \ref{CH-thm4.1} we may write the matrix for $\ell_M$ 
with respect to the basis $\{x,u\}$ as
\[
\ell_M=\left(\smallmatrix\frac14(2+m)&\frac14\\ \frac14&\frac14(2+n)\endsmallmatrix\right)
=\left(\smallmatrix0&\frac14\\ \frac14&\frac14(2+n)\endsmallmatrix\right),\quad\mathrm{if}~m\equiv2~mod~(4).
\]
Thus, in the case that $m\equiv{n}\equiv2$ {\it mod} (4), there are four possible self-orthogonal decompositions 
for $H_1(X)\oplus{H_1(Y)}$:
\[\langle{x}\rangle\oplus\langle{u}\rangle,~
\langle{x}\rangle\oplus\langle{u+2x}\rangle,~
\langle{x}+2u\rangle\oplus\langle{u}\rangle, ~\mathrm{or}~
\langle{x+2u}\rangle\oplus\langle{u+2x}\rangle.
\]
If $m\equiv2$ {\it mod} (4) and $n\equiv0$ {\it mod} (4) the four possibilities are
\[\langle{x}\rangle\oplus\langle{u-x}\rangle,~
\langle{x}\rangle\oplus\langle{u+x}\rangle,~
\langle{-x+2u}\rangle\oplus\langle{u-x}\rangle, ~\mathrm{or}~
\langle{-x+2u}\rangle\oplus\langle{u+x}\rangle.
\]

\begin{lemma}
\label{CH-lem4.2}
If $M=M_{m,n}$ embeds in a homology $4$-sphere $\Sigma$ and $m\equiv2$ mod $(4)$
then $m=\pm2$.
\end{lemma}

\begin{proof}
The maps $H_1(Y)\cong\mathbb{Z}/4\mathbb{Z}\to\mathbb{Z}/2\mathbb{Z}$
determine 2-fold coverings $\widetilde{Y}'\to\widetilde{Y}\to{Y}$ which induce coverings
$\widetilde{M}'\to\widetilde{M}\to{M}$
on the boundary corresponding to the maps
\begin{equation*}
\begin{CD}
H_1(M)@>p_Y>>\mathbb{Z}/4\mathbb{Z}\to\mathbb{Z}/2\mathbb{Z},
\end{CD}
\end{equation*}
where $p_Y$ is the projection onto $H_1(Y)$.
Now given any one of the eight possible self-orthogonal decompositions listed above,
these maps are such that
\[
x\mapsto0~\mathrm{or}~2\mapsto0,\quad{y=2u}\mapsto2\mapsto0,
\]
\[
u\mapsto\pm1\mapsto1,\quad\quad{v=2x}\mapsto0\mapsto0.
\]
Now $M=N^*(m)\cup{N^*(n)}$, 
and we may write $\widetilde{M}=\widetilde{N}^*(m)\cup\widetilde{N}^*(n)$.
The piece $\widetilde{N}^*(m)$ is determined by the map
\begin{equation*}
\begin{CD}
H_1(N^*(m))@>i_*>>H_1(M)\to\mathbb{Z}/2\mathbb{Z}.
\end{CD}
\end{equation*}
But since $H_1(N^*(m))$ is generated by $x$ and $y$, this map is trivial.
Consequently $\widetilde{N}^*(m)$ consists of two disjoint copies of $N^*(m)$.
On the other hand we can see that since $u$ (represented by an orientation reversing path in the base space of $N^*(n)$ is mapped to 1, while $v$ (represented by a fibre of $N^*(n)$)
is mapped to 0,
then $\widetilde{N}^*(n)$ is just a framed $S^1$-bundle over the annulus,
or $M(0;(1,2n))$ with two regular fibres removed.
Using the plumbing notation of \cite{Ne81}, 
we have

\setlength{\unitlength}{1mm}
\begin{picture}(80,20)(-18.3,0)
\linethickness{2pt}
\put(7,10){$\tilde{M}=$}
\put(20,10){$\bullet$}
\put(20,13){$m$}
\put(18,5){$[-1]$}
\put(40,10){$\bullet$}
\put(39,13){$2n$}
\put (60,10){$\bullet$}
\put(60,13){$n$}
\put(58,5){$[-1]$}
\put(21,11){\line(1,0){40}}
\put(65,10.5){.}
\end{picture}\\
The boundary of $N^*(m)$ has framing $([\mu],[f])$
where $[\mu]=my+2x$ and $[f]=y$.
Since $y\mapsto2$ under the map onto $\mathbb{Z}/4\mathbb{Z}$,
in the covering map $\widetilde{M}'\to\widetilde{M}$ each copy of $N(m)$ is 
covered by $N(\frac12m)$, and the covering projection,
$\theta$ say, acts by a half-turn rotation of each fibre.
On the other hand, $\mu$ is identified (in $\widetilde{M}$) with a fibre of $\widetilde{N}^*(n)$, and since $[\mu]\mapsto0$ in $\mathbb{Z}/4\mathbb{Z}$ (since $m$ is even),
this piece is itself covered by another $S^1$-bundle over the annulus, 
$S_A$ say, where here $\theta$ acts on the base.
Thus we have the plumbing description

\setlength{\unitlength}{1mm}
\begin{picture}(80,20)(-18.3,0)
\linethickness{2pt}
\put(7,10){$\tilde{M}'=$}
\put(20,10.1){$\bullet$}
\put(19,14){$\frac12m$}
\put(18,5){$[-1]$}
\put(40,10.1){$\bullet$}
\put(39,14){$4n$}
\put(60,10.1){$\bullet$}
\put(59,14){$\frac12n$}
\put(58,5){$[-1]$}
\put(21,11){\line(1,0){40}}
\put(65,10.5){.}
\end{picture}\\
Note that since $\theta$ acts on the base of $S_A$, 
it extends to a fixed point free involution of the associated disc bundle $D_A$.
Thus writing $W^0=\widetilde{Y}'\cup_{S_A}D_A$, we see that the covering
involution of $\widetilde{Y}'$ extends to a fixed point free involution $\theta^0$
on $W^0$.
Also observe that $\partial{W^0}=N(\frac{m}2)\sqcup{N(\frac{m}2)}$ (disjointly),
since attaching the discs of $D_A$ corresponds to a trivial filling of the boundary of each $N^*(\frac{m}2)$ piece.
Writing $DN(\frac{m}2)$ for the disc bundle associated to $N(\frac{m}2)$, 
we may now construct the closed 4-manifold
\[
W=
W^0\cup_{N(\frac{m}2)\sqcup{N(\frac{m}2)}}(DN(\frac{m}2)\sqcup{DN(\frac{m}2))}.
\]
Now $\theta^0$ extends to an involution $\theta_W$ of $W$ with fixed point set the disjoint union of two copies of $\mathbb{RP}^2$,
and the total normal Euler number of the fixed point set is $\frac{m}2+\frac{m}2=m$.
Thus by the $\mathbb{Z}/2\mathbb{Z}$-Index Theorem,
$\mathrm{Sign}(\theta_W,W)=m$.
We now calculate $\beta_2(W)$.

Since $H_1(M)$ is finite it is a $\mathbb{Q}$-homology sphere, and so the complementary regions$X$ and $Y$ are  $\mathbb{Q}$-homology balls.
In particular, $\chi(Y)=1$.
Hence $\chi(W^0)=\chi(\widetilde{Y}')=4\chi(Y)=4$,
since $\chi(S_A)=\chi(T)=0$ and $\chi(D_A)=\chi(S^1)=0$.
Thus $\chi(W)=\chi(W^0)+2\chi(DN(\frac{m}2))=4+2\chi(\mathbb{RP}^2)=6$.
But by Poincar'e duality this gives $\beta_2(W)=4+2\beta_1(W)$.

By Lemma \ref{CH-lem3.10} the fact that $H_1(Y)=\mathbb{Z}/4\mathbb{Z}$
implies that $\beta_1(\widetilde{Y}')=0$.
Also, since the map on $H_1$ induced by the inclusion of an $S^1$-bundle into 
its associated disc bundle is an epimorphism, a Mayer-Vietoris argument shows that 
attaching disc bundles, as we have, does not increase $\beta_1$.
Thus $\beta_1(W)=0$ as well, and so $\beta_2(W)=4$.
Since $|\mathrm{Sign}(\theta_W,W)|$ is bounded above by this rank, we must 
have $m=\pm2$.
\end{proof}

\begin{lemma}
\label{CH-lem4.3}
If $M=M_{m,n}$ embeds in a homology $4$-sphere $\Sigma$ and $n\equiv0$ mod $(4)$
(and so $m\equiv2$ mod $(4)$ necessarily) then $m+n=\pm2$.
\end{lemma}

\begin{proof}
This time we consider the 2-fold coverings $\overline{X}'\to\overline{X}\to{X}$ 
which induce coverings $\overline{M}'\to\overline{M}\to{M}$ corresponding to the maps
\begin{equation*}
\begin{CD}
H_1(M)@>p_X>>\mathbb{Z}/4\mathbb{Z}\to\mathbb{Z}/2\mathbb{Z}.
\end{CD}
\end{equation*}
Here $p_X$ is the projection onto the first factor in one of the four posssible
self orthogonal decompositions listed for the case  $m\equiv2$ {\it mod} $(4)$ and
$n\equiv0$ {\it mod} $(4)$.
Thus the above maps are given by
\[
x\mapsto1\mapsto1,\quad{y=2u}\mapsto2\mapsto0,
\]
\[
u\mapsto\pm1\mapsto1,\quad\quad{v=2x}\mapsto2\mapsto0.
\]
Again $M=N^*(m)\cup{N^*(n)}$ and we write 
$\overline{M}=\overline{N}^*(m)\cup\overline{N}^*(n)$.
this time, since both $x$ and $u$ are mapped to 1 while $y,v\mapsto0$ in
$\mathbb{Z}/2\mathbb{Z}$,
both $\overline{N}^*(m)$ and $\overline{N}^*(n)$ are $S^1$-bundles over annuli, and $\overline{M}$ has plumbing diagram

\setlength{\unitlength}{1mm}
\begin{picture}(90,20)(-17,0)
\linethickness{2pt}
\put(30,10.1){$\bullet$}
\put(22,10){$2m$}
\qbezier(31,11)(40,17)(51,11)
\qbezier(31,11)(40,5)(51,11)
\put(50,10.1){$\bullet$}
\put(54,10){$2n$}
\end{picture}\\
That is,
$\overline{M}=N^{**}(2m)\cup_{T_a\sqcup{T_b}} {N}^{**}(2n)$, 
where the piece $N^{**}(2m)$ is just the framed exterior of a pair of fibres 
$f_a$ and $f_b$ in the bundle $M(0;(1,2m))$ with Euler number $m$,
and similarly for  $N^{**}(2n)$.
The boundary between the two pieces is a disjoint union $T_a\sqcup{T_b}$ of tori
with framings $(y_a,v_a)$ and  $(y_b,v_b)$, respectively, such that 
when $T_a\sqcup{T_b}$ covers the torus $T$ in $M=N^*(m)\cup{N^*(n)}$ we have
$y_a,y_b\mapsto{y}$ and $v_a,v_b\mapsto{v}$.
The framings of $T_a$ and $T_b$  given here correspond to the usual framings 
of $f_a$ and $f_b$, where for $N^{**}(2m)$ the fibres represent $y_a$ and $-y_b$, 
while for $N^{**}(2n)$ the fibres represent $v_a$ and $-v_b$, respectively.

Now, since both $y,v\mapsto2$ in $\mathbb{Z}/4\mathbb{Z}$, 
both the fibre of $N^{**}(2m)$ and the fibre of $N^{**}(2n)$ are ``unwrapped" 
in the covering $\overline{M}'\to\overline{M}$.
Thus, $\overline{M}'=N^{**}(m)\cup_{T_a\sqcup{T_b}}N^{**}(n)$,
and we may now view $\overline{X}'$ as a 4-manifold with (covering) involution $\theta$
which acts on each fibred piece of the boundary by 
Observe also that the behaviour of $\theta$ on a neighbourhood
$(T_a'\cup{T_b'})\times[-1,1]$ of $T_a'\sqcup{T_b'}$ is independent of $m$ and $n$.
Thus, if we begin with  $M_1=M_{m+1,n_1}$ in $\Sigma_1$ and $M_2=M_{m_2,n_2}$
in $\Sigma_2$, each of which satisfies our hypotheses,
and construct $\overline{X}_1'$ and $\overline{X}_2'$ as above,
and define $W^0=\overline{X}_1'\cup_{(T_a'\cup{T_b'})\times[-1,1] }-\overline{X}_2'$, 
then the covering involutions $\theta_1$ and $\theta_2$ behave identically on 
$(T_a'\cup{T_b'})\times[-1,1]$, 
and together define a fixed point free involution $\theta^0$ of $W^0$.
Also $\partial{W^0}$ is the disjoint union  of $S^1$-bundles over tori
$M(1;(1,m_1-m_2))\sqcup{M(1;(1,(n_1-n_2))}$,
and $\theta$ acts on each boundary component by a half-turn rotation of the fibre.
Attaching the associated disc bundles, we obtain a closed 4-manifold $W$ with involution 
$\theta$ which has fixed point set $F=T_x\sqcup{T_u}$ with total normal Euler number $e(F)=(m_1-m_2)+(n_1-n_2)$.
Now, supposing that $M=M_{m,n}$ embeds in $\Sigma$, wecmay take $M_1=M$ and $M_2=-M=M_{-m,-n}$ (embedded in $-\Sigma)$ so that, by the
$\mathbb{Z}/2\mathbb{Z}$-Index Theorem,
\[
\mathrm{sign}(\theta,W)=e(F)=2(m+n).
\]
Consequently, $|m+n|\leq\frac12\beta_2(W)$.

We have $\chi(W)=\chi(W^0)=2\chi(\overline{X'})=8\chi(X)=8$,
by arguments similar to those of Lemma \ref{CH-lem4.2}.
Poincar\'e duality gives $\beta_2(W)=6+2\beta_1(W)$.
But $\beta_1(W)\leq\beta_1(W^0)+1$, by a Mayer-Vietoris argument, since once again $\beta_1(\overline{X}')=0$.
Thus $\beta_2(W)=8$, and so $|m+n|\leq4$.
Hence we must have $m+n=\pm2$.
\end{proof}

We may now complete the proof of Theorem \ref{CH-thm2.2}.

\begin{theorem}
\label{CH-thm4.4}
If  $M=N\cup_\phi{N}$ embeds in a homology $4$-sphere then either
it is a torus bundle or $\phi$ is one of 
$\left(\smallmatrix2&-1\\1&0\endsmallmatrix\right)$,
$\left(\smallmatrix2&-3\\1&2\endsmallmatrix\right)$,
$\left(\smallmatrix2&-5\\1&-2\endsmallmatrix\right)$ or
$\left(\smallmatrix2&-9\\1&-4\endsmallmatrix\right)$.
\end{theorem}

\begin{proof}
It is clear from Theorem \ref{CH-thm4.1} that $\beta_1(M)>0$ if and only if $c=0$,
and then $M$ is a torus bundle.
If $c\not=0$ and $\ell_M$ is hyperbolic then $c=1$ and $M=M_{m,n}$ 
for some even integers $m,n$,  which are not both divisible by 4.
Since $M_{n,m}\cong{M_{m,n}}$, we may assume that $m\equiv2$ {\it mod} (4).
We then have $m=\pm2$ and either $n=\pm2$ or $n\equiv0$ {\it mod} (4).
In the latter case $m+n=\pm2$, by Lemma \ref{CH-lem4.3}.
After changing the orientation, if necessary, we may assume that $m=2$,
and so $n=0$ or $-4$.
\end{proof}

\section{A link presentation for $M_{2,4}$}

In Figure 5.1 we have redrawn the diagram from \cite[Figure 1]{CH98}
so that the non-trivial component is visibly a ribbon knot.
In \cite{CH98} the Kirby calculus was used to show that $M_{2,4}\cong{M(L)}$.
Here we shall instead use the Wirtinger presentation
for the link group $\pi{L}$ to show that $\pi_1M(L)\cong\pi_1M_{2,4}$.
(The group $\pi_1M(L)$ is the quotient of $\pi{L}$ by
the normal closure of the longitudes.)

\setlength{\unitlength}{1mm}
\begin{picture}(110,75)(-29.6,2)

\put(72.7,37.5){$\vartriangle$}
\put(76,38){$a$}
\put(24.6,27){$\triangledown$}
\put(21.5,27.2){$b$}
\put(47,23){$\vartriangleleft$}
\put(46,25){$c$}
\put(68.6,24){$\vartriangle$}
\put(72,25){$s$}
\put(25,69){$\vartriangleleft$}
\put(23,71){$t$}
\put(17.5, 14){$\vartriangleright$}
\put(20.5,16){$u$}
\put(36.6,27){$\vartriangle$}
\put(40,27){$v$}
\put(17.5,39){$\vartriangleleft$}
\put(16,41.5){$w$}
\put(0.6,26.3){$\triangledown$}
\put(-2,26.8){$y$}
\put(47,34.1){$\vartriangleleft$}
\put(46,36.5){$x$}
\put(29,19.1){$\vartriangleright$}
\put(32,21){$z$}
\put(25.5,63.5){$\bullet$}

\thinlines
\put(-15,70){\line(1,0){80}}
\put(-20,30){\line(0,1){35}}
\qbezier(-20,65)(-20,70)(-15,70)

\put(-15,35){\line(0,1){25}}
\qbezier(-15,60)(-15,65)(-10,65)
\put(-10,65){\line(1,0){70}}

\put(-15,25){\line(1,0){15}}
\put(-10,30){\line(1,0){10}}
\qbezier(-10,30)(-15,30)(-15,35)
\qbezier(-15,25)(-20,25)(-20,30)
\qbezier(2,25)(2,20)(7,20)
\put(2,25){\line(0,1){5}}
\qbezier(2,30)(3,35)(7,35)
\qbezier(4,25)(9,25)(9,20)
\qbezier(9,20)(9,15)(14,15)

\put(14,15){\line(1,0){20}}
\qbezier(34,15)(38,15)(38,19)

\qbezier(4,30)(9,30)(9,35)
\qbezier(9,35)(9,40)(14,40)

\put(14,40){\line(1,0){20}}
\qbezier(34,40)(38,40)(38,36)

\put(11,20){\line(1,0){45}}
\put(11,35){\line(1,0){13}}
\put(28,35){\line(1,0){32}}

\put(38,21){\line(0,1){2}}
\put(38, 25){\line(0,1){9}}

\put(65,40){\line(0,1){10}}
\qbezier(60,35)(65,35)(65,40)
\put(60,20){\line(1,0){5}}
\qbezier(65,20)(70,20)(70,25)

\put(70,25){\line(0,1){19.5}}
\put(70,47.5){\line(0,1){2.5}}
\put(70,60){\line(0,1){5}}
\qbezier(65,70)(70,70)(70,65)
\qbezier(60,65)(65,65)(65,60)

\put(65,50){\line(1,1){5}}
\put(65,55){\line(1,1){5}}
\qbezier(65,60)(65.8,59.2)(66.6,58.4)
\qbezier(65,55)(65.8,54.2)(66.6,53.4)
\qbezier(68.4,51.6)(69.2,50.8)(70,50)
\qbezier(68.4,56.6)(69.2,55.8)(70,55)

\linethickness{1pt}
\put(31,46){\line(1,0){33}}
\qbezier(26,41)(26,46)(31,46)
\put(66,46){\line(1,0){3}}
\qbezier(69,46)(74,46)(74,41)

\qbezier(26,19)(26,17.5)(27.5,17.5)
\put(27.5,17.5){\line(1,0){8.5}}
\put(39,17.5){\line(1,0){15.75}}
\put(26,21){\line(0,1){17.5}}

\put(37,24){\line(1,0){17.75}}
\qbezier(37,31)(33.5,31)(33.5,27.5)
\qbezier(33.5,27.5)(33.5,24)(37,24)
\qbezier(54.75,24)(58,24)(58,20.75)
\qbezier(54.75,17.5)(58,17.5)(58,20.75)
\put(39,31){\line(1,0){30}}
\qbezier(71,31)(74,31)(74,34)
\put(74,34){\line(0,1){7}}

\put(7,6){Figure 5.1. \quad$L=(8_{20},{U})$}
\end{picture}\\
The group $\pi{L}$ has the Wirtinger presentation
\[
\langle{a,b,c,s,t,u,v,w,x,y,z}\mid
axtx^{-1}a^{-1}=s,~bwyw^{-1}b^{-1}=x,
\]
\[csc^{-1}=z,~uyu^{-1}=z,~vsas^{-1}v^{-1}=c,~wx^{-1}axw^{-1}=b,~xwx^{-1}=v,
\]
\[
ytxt^{-1}y^{-1}=w, ~yty^{-1}=u, ~zucu^{-1}z^{-1}=b,~zuz^{-1}=cvc^{-1}
\rangle
\]
The longitudes associated with the meridianal generators $a$ and $x$ are
$\ell_a=xw^{-1}zuvs$
and $\ell_s=axy^{-1}z^{-1}cxytbu^{-1}cs^{-1}$.
We may use the relations $z=csc^{-1}=uyu^{-1}$, $v=xwx^{-1}$ and $u=yty^{-1}$
to eliminate these generators, to obtain the presentation
\[
\langle{a,b,c,s,t,w,x,y,}\mid
axt=sax,~bwy=xbw,~ytx=wyt, ~ytc=byt,
\]
\[csc^{-1}=ytyt^{-1}y^{-1},~xwx^{-1}sas^{-1}xw^{-1}x^{-1}=c,~wx^{-1}axw^{-1}=b,
\]
\[
~sc^{-1}yty^{-1}cs^{-1}=xwx^{-1}
\rangle.
\]
Let $B=bw$ and $E=yt$.
Then  $\pi{L}$ has the presentation 
\[
\langle{a,B,c,s,w,x,y,E}\mid
{axy^{-1}E=sax,~By=xB,~Ex=wE, }
\]
\[Ec=Bw^{-1}E,~csc^{-1}=EyE^{-1},~xwx^{-1}sas^{-1}xw^{-1}x^{-1}=c,
\]
\[
wx^{-1}ax=B,~sc^{-1}Ey^{-1}cs^{-1}=xwx^{-1}\rangle.
\]
The longitudes are now
$\ell_a=xw^{-1}zuvs=xw^{-1}Exwx^{-1}s$
and\\ $\ell_s=axy^{-1}z^{-1}cxytbwu^{-1}cs^{-2}=axy^{-1}cs^{-1}xEByE^{-1}cs^{-2}$.

Eliminating $y=B^{-1}xB$ and adjoining the relations $\ell_a=\ell_s=1$ 
gives a presentation for $\pi=\pi_1M(L)$.
\[
\langle{a,B,c,s,w,x,E}\mid
axB^{-1}x^{-1}BE=sax,~Ex=wE, ~Ec=Bw^{-1}E,
\]
\[csc^{-1}=EB^{-1}xBE^{-1},~xwx^{-1}sas^{-1}xw^{-1}x^{-1}=c,~a=xw^{-1}Bx^{-1},
\]
\[
~sc^{-1}EB^{-1}x^{-1}Bcs^{-1}=xwx^{-1},~xw^{-1}Exwx^{-1}s=1,
\]
\[
axB^{-1}x^{-1}Bcs^{-1}xExBE^{-1}cs^{-2}=1
\rangle.
\]
Eliminating $a=xw^{-1}Bx^{-1}$ gives
\[
\langle{B,c,s,w,x,E}\mid
xw^{-1}x^{-1}BE=sxw^{-1}B,~Ex=wE, ~Ecx=BE,
\]
\[csc^{-1}=EB^{-1}xBE^{-1},~BE=cxwx^{-1}sx,~Exwx^{-1}sx=w,
\]
\[
~sc^{-1}EB^{-1}x^{-1}Bcs^{-1}=xwx^{-1},~
xw^{-1}x^{-1}Bcs^{-1}xExBE^{-1}cs^{-2}=1
\rangle.
\]
The second and sixth relations give $Ewx^{-1}sx=1$, and with the first we get 
$BE=xwx^{-1}sxw^{-1}B=xE^{-1}w^{-1}B=E^{-1}B$. 
%Hence (3) gives $cx=BE^2$.
Using the fourth and sixth relations,
the seventh simplifies to $c^{-1}Ec=xwx^{-1}s=xE^{-1}x^{-1}$,
or $EcxE=cx$,
and so the presentation now becomes
\[
\langle{B,c,s,w,x,E}\mid
{BE=E^{-1}B},~Ex=wE, ~Ecx=BE,~
\]
\[csc^{-1}=EB^{-1}xBE^{-1},~BE=cxwx^{-1}sx,~Ewx^{-1}sx=1,
\]
\[
~EcxE=cx,~xw^{-1}x^{-1}Bcs^{-1}xExBE^{-1}cs^{-2}=1
\rangle.
\]
%In fact relation (7) is a consequence of (1) and (3).
Eliminating $w=ExE^{-1}$ gives:
\[
\langle{B,c,s,x,E}\mid
{BE=E^{-1}B},~Ecx=BE,~csc^{-1}=EB^{-1}xBE^{-1},~
\]
\[BE=cxExE^{-1}x^{-1}sx,~xE^{-1}x^{-1}sx=E^{-2},
\]
\[
xEx^{-1}E^{-1}x^{-1}Bcs^{-1}xExBE^{-1}cs^{-2}=1:~~cx=BE^2,~Ecx=cxE^{-1}
\rangle,
\]
in which the last two relations are redundant.
Applying the fifth relation to the fourth gives $BE=cxE^{-1}$ or $BE^2=cx$,  
and rewriting the sixth as
\[
E^{-1}x^{-1}Bcs^{-1}xExBE^{-1}cs^{-1}=xE^{-1}x^{-1}s=E^{-2}x^{-1}
\]
gives
\[
\langle{B,c,s,x,E}\mid
{BE=E^{-1}B},~Ecx=BE,~s=c^{-1}EB^{-1}xBE^{-1}c,~
\]
\[
xE^{-1}x^{-1}sx=E^{-2},~Ex^{-1}Bcs^{-1}xExBE^{-1}cs^{-1}x=1\rangle.
\]
Eliminating $s=c^{-1}EB^{-1}xBE^{-1}c$ gives
\[
\langle{B,c,x,E}\mid
{BE=E^{-1}B},~Ecx=BE,~
\]
\[xE^{-1}x^{-1}c^{-1}EB^{-1}xBE^{-1}cx=E^{-2},~
Ex^{-1}E^{-1}x^{-1}BE^{-1}cxE^2Bcx=1\rangle.
\]
Eliminating $c=BE^2x^{-1}$ and using $BE=E^{-1}B$ to simplify gives
\[
\langle{B,x,E}\mid{BE=E^{-1}B},~
xB^{-2}E^{-4}xB^2E^5=1,~x^{-1}E^{-1}x^{-1}B^4E^8=1\rangle.
\]
Since $xB^2E^4x^{-1}=xB^2E^5x(xEx)^{-1}=B^{-2}E^{-4}$
and hence\\ $x^2E=xB^4E^8x^{-1}=B^{-4}E^{-8}$, this becomes
\[
\langle{B,x,E}\mid
{BE=E^{-1}B},~x^2=B^{-4}E^{-9},~xEx^{-1}=B^8E^{17},
\]
\[ xB^2E^4x^{-1}=(B^2E^4)^{-1}\rangle.
\]
We shall now relabel the generators to bring out the similarities with
the coordinates used in describing $N\cup_\phi{N}$.
Let $X=E, Y=B^{-1},U=B^{-2}E^{-4}$ and $V=x$.
Then we have the presentation
\[
\langle{X,Y,U,V}\mid{XYX^{-1}=Y^{-1}}, ~UVU^{-1}=V^{-1}, 
\]
\[
~U^2=(X^2)^2Y^{-9},~V=(X^2)Y^{-4}\rangle,
\]
and so $\pi\cong\pi_1M_{2,4}=\pi_1N*_\phi\pi_1N$, where 
$\phi=\left(\smallmatrix2&-9\\1&-4\endsmallmatrix\right)$.

\section{Homologically related 3-manifolds}

The question of which homology 3-spheres embed smoothly in $S^4$ is delicate, 
and of continuing interest.
However Freedman showed that the corresponding question for
TOP locally flat embeddings has a dramatically simpler answer:
every homology 3-sphere embeds in $S^4$ \cite[Corollary 9.3c]{FQ}.
Inspired by this, 
we may consider the question of locally flat embeddings of 3-manifolds 
with a given homology type.
A homology 3-sphere is a 3-manifold $\Sigma$ with a map $f:\Sigma\to 
{S^3}$ which induces an isomorphism on integral homology.
Wen the model 3-manifold is not $S^3$, 
we must use {\it local\/} coefficients in formulating the appropriate homological relation.
For instance, the manifolds $M(K)$ obtained by 0-framed surgery on knots 
$K$ in $S^3$ admit maps $f:M(K)\to{S^2\times{S^1}}$ which induce 
isomorphisms on homology with simple coefficients $\mathbb{Z}$. 
If the knot $K$ is not algebraically slice then $M(K)$ does not even 
admit a Poincar\'e embedding into $S^4$.
On the other hand, if $f$ induces isomorphisms on homology with local coefficients
$\mathbb{Z}[\mathbb{Z}]$ then (as we shall see) $M(K)$ embeds in $S^4$.
Our strategy involves 4-dimensional surgery, 
and so we must assume that the fundamental group of the model manifold 
is ``good" in the sense of \cite{FQ}.
At the time of writing, 
all known good 3-manifold groups are either solvable or finite.

\begin{lemma}
\label{Hi96-1}
Let $W$ be an $(n+1)$-dimensional $H$-cobordism over $\mathbb{Z}$ with
$\partial{W}=M_1\sqcup{M_2}$, where $M_1$ and $M_2$ are closed $n$-manifolds.
Then $M_1$ embeds in a homology $(n+1)$-sphere if and only if $M_2$
embeds in a homology $(n+1)$-sphere.
\end{lemma}

\begin{proof}
Suppose that $M_1$ is a locally flat submanifold of the homology $(n+1)$-sphere $\Sigma$,
and let $X$ and $Y$ be the complementary regions, so that $\Sigma=X\cup_{M_1}Y$.
Let $2W=W\cup_{M_2}W$ be the union of two copies of $W$, doubled along $M_2$.
Then $\Sigma'=X\cup_{M_1}2W\cup_{M_1}Y$ is an $(n+1)$-manifold 
containing $M_2$ as a locally flat submanifold, 
and an easy Mayer-Vietoris argument shows that $\Sigma'$ is a homology sphere.
(Note that the inclusions of $M_1$ into $W$ at either end of $2W$ induce 
the {\it same\/} isomorphism on homology.)
\end{proof}

By retaining more control over the fundamental group we shall be able to refine this argument to obtain embeddings of certain 3-manifolds in $S^4$.
However Lemma \ref{Hi96-1} may be used to obtain non-embedding results.

A $PD_3^+$-group is an orientable 3-dimensional Poincar\'e duality group.

\begin{lemma}
\label{Hi96-2}
Let $M$ be a  $3$-manifold and $\nu$ a perfect normal subgroup of
$\pi=\pi_1M$ such that $\rho=\pi/\nu$ has finitely many ends. Then either
\begin{enumerate}
\item$\rho$ is finite and has cohomological period dividing $4$; or
\item$\rho\cong\mathbb{Z}$ or $D_\infty$; or
\item$\rho$ is a $PD_3^+$-group.
\end{enumerate}
\end{lemma}

\begin{proof}
Let $M_\nu$ be the covering space associated to $\nu$.

If $\rho$ is finite then $M_\nu$ is a homology sphere, since $\nu$ is perfect.
Since $\rho$ acts freely on $M_\nu$ it has cohomological period dividing 4.

If $\rho$ has two ends then it has a finite normal subgroup $F$ such that
$\rho/F\cong\mathbb{Z}$ or $D_\infty$.
Moreover, $M_\nu$ has the homology of $S^2$.
Hence the subgroup of $\rho=Aut(M/M_\nu)$ that acts trivially on 
$H_2(M_\nu)\cong\mathbb{Z}$ has index at most 2.
If $g$ is an element of finite order in this group then $g=1$.
Hence $\rho\cong\mathbb{Z}$, $\mathbb{Z}\oplus(\mathbb{Z}/2\mathbb{Z})$ or $D_\infty$.
(Compare \cite[Theorem 4.4]{Wa67}.)
Since $M$ is orientable we may exclude the possibility
$\mathbb{Z}\oplus(\mathbb{Z}/2\mathbb{Z})$.

If $\rho$ has one end then $M_\nu$ is acyclic,
We may assume that $M_\nu$ has the $\rho$-equivariant triangulation
lifted from a triangulation of $M$.
The cellular chain complex of $M_\nu$ is now a finitely generated free 
$\mathbb{Z}[\rho]$-complex, of length 3.
Since it is a resolution of the augmentation $\mathbb{Z}[\rho]$-module $\mathbb{Z}$,
we see that $c.d.\rho\leq3$.
The universal coefficient spectral sequence for $M$ with coefficients $\mathbb{Z}[\rho]$
collapses to give isomorphisms $H^i(\rho;\mathbb{Z}[\rho])\cong{H^i(M;\mathbb{Z}[\rho])}$.
As this is in turn isomorphic to $H_{3-i}(M;\mathbb{Z}[\rho])=H_{3-i}(M_\nu)$,
by Poincar\'e duality,  $\rho$ is a $PD_3^+$-group.
\end{proof}

This lemma is enough for our purposes, 
since virtually solvable groups have finitely many ends.

\begin{lemma}
\label{Hi96-3}
Let $M$ be a $3$-manifold such that $\pi=\pi_1M$  is an extension 
of a torsion-free solvable group $\rho$ by a perfect normal subgroup $\nu$.
Then either $\rho=1$ and $M$ is a homology sphere,
or $\rho\cong\mathbb{Z}$ and there is a $\Lambda$-homology
isomomorphism from $M$ to $S^2\times{S^1}$,
or $\rho$ is the fundamental group of an aspherical  $3$-manifold $P$ 
and there is a $\mathbb{Z}[\rho]$-homology isomorphism from $M$ to $P$. 
\end{lemma}

\begin{proof}
Since $\rho$ is torsion free and has finitely many ends it is either 1, $\mathbb{Z}$ or is a $PD_3^+$-group.
The result is clear if $\rho=1$.

If $\rho=\mathbb{Z}$ then there is a $\mathbb{Z}$-homology equivalence
$h:M\to{S^2\times{S^1}}$, by Lemma \ref{homhandlemap}.
Let $\tilde{h}:M_\nu\to{S^2}\times\mathbb{R}$ be a lift of $h$ 
to the infinite cyclic covering spaces.
Since $\nu$ is perfect it follows easily from the Wang sequence 
for the covering projection that $\tilde{h}$ is a homology isomorphism, 
and hence that $h$ is a $\mathbb{Z}[\rho]$-homology isomorphism.

Suppose now that $\rho$ is a $PD_3^+$-group.
As it is virtually solvable it is virtually poly-$\mathbb{Z}$, 
of Hirsch length 3 \cite[Theorem 9.23]{Bie}.
Hence it is the fundamental group of a 3-manifold $P$,
which is either flat (if $\rho$ is virtually abelian), a  $\mathbb{N}il^3$-manifold
(if $\rho$ is virtually nilpotent but not virtually abelian) or a $\mathbb{S}ol^3$-manifold
(if $\rho$ is not virtually nilpotent).
In all cases, every automorphism of $\rho$ is realizable by some based
self-homeomorphism of $P$, 
and so the epimorphism from $\pi$ to $\rho$ may be realized by a map $h:M\to{P}$. 
Let $\sigma$ be a normal subgroup of finite index in $\rho$
which is  a poly-$\mathbb{Z}$ group, and let $\lambda:\sigma\to\mathbb{Z}$ be an epimorphism. 
Then $\kappa=\mathrm{Ker}(\lambda)$ is poly-$\mathbb{Z}$ of Hirsch length 2, 
and so is isomorphic to $\mathbb{Z}^2$, since $\rho$ is orientable.
Let $h_\kappa:M_\kappa\to{P_\kappa}\simeq{S^1\times{S^1}\times\mathbb{R}}$ 
be a lift of $h$ to the covering spaces associated to $\kappa$ and its preimage in $\pi$.
The cohomology rings of $M_\kappa$ and $P_\kappa$ are generated in degree 1
and so $h_\kappa$ induces a (co)homology isomorphism.
It follows easily that $h$ is a $\mathbb{Z}[\rho]$-homology isomorphism.
\end{proof}

In the final case $P$ is either flat, a $\mathbb{N}il^3$-manifold or a $\mathbb{S}ol^3$-manifold.

\newpage

\section{Application of surgery}

The symbols $M$,  $P$, $\rho$ and $h$ shall retain their meaning from 
Lemma \ref{Hi96-3} in the following lemmas,
except that we shall allow $P$ to denote $S^3$ or $S^2\times{S^1}$ also, 
where appropriate.

\begin{lemma}
\label{Hi96-4}
The maps $h$ and $id_P$ are normally cobordant.
\end{lemma}

\begin{proof}
The result is clear if $P=S^3$.
Suppose that $P=S^2\times{S^1}$.
Then $\mathbb{Z}[\rho]\cong\Lambda$.
There is a normal map $n:T=S^1\times{S^1}\to{S^2}$ with Arf invariant 1,
and $n\times{id_S^1}$
is a normal map with non-trivial surgery obstruction in 
$L_3(\Lambda)\cong\mathbb{Z}/2\mathbb{Z}$.
Thus the surgery obstruction map $\sigma_3(S^2\times{S^1})$ from $\mathcal{N}(S^2\times{S^1})=\mathbb{Z}/2\mathbb{Z}$
to $L_3(\Lambda)$ is bijective.

If $\rho$ is virtually poly-$\mathbb{Z}$ we may appeal to \cite{FJ88} 
instead to see that $Wh(\rho)=0$ and 
$\sigma_3(P):\mathcal{N}(P)\to{L_3(\mathbb{Z}[\rho])}$ is an isomorphism.

Since in all cases $h$ is a $\mathbb{Z}[\rho]$-homology isomorphism,
it has trivial surgery obstruction and so is normally cobordant to $id_P$.
\end{proof}

%This lemma can be extended to other model 3-manifolds $P$.

\begin{lemma}
\label{Hi96-5}
There is an $H$-cobordism $W$ over $\mathbb{Z}[\rho]$ from $M$ to $P$ with
$\pi_1W=\rho$.
\end{lemma}

\begin{proof}
Suppose first that $\rho\cong\mathbb{Z}$.
Let $F:N\to{S^2}\times{S^1}\times[0,1]$ be a normal cobordism from $h$ to $id_{S^2\times{S^1}}$.
The surgery obstruction of $F$ is determined by the signature of $N$, 
since the natural homomorphism from $L_4(\mathbb{Z})$ to $L_4(\mathbb{Z}[\mathbb{Z}])$
is an isomorphism.
By forming the connected sum of $F$ with an appropriate multiple of the
degree-1 map from the $E_8$-manifold to $S^4$ we may obtain a
normal cobordism with trivial surgery obstruction.

If $\rho$ is virtually poly-$\mathbb{Z}$ then we may appeal to \cite{FJ88} instead,
to see that $\sigma_4(P\times[0,1],\partial)$ is an isomorphism.
Thus we may in all cases assume that there is a normal cobordism from $h$ to $id_P$ with trivial surgery obstruction.

Since $\rho$ is ``good",
we may apply \cite[Theorem 11.3A]{FQ} to complete surgery
relative to the boundary to obtain a simple homotopy equivalence $f:W\to{P\times[0,1]}$, 
where $\partial{W}={M}\sqcup{P}$, $f|_M=h$ and $f|_P=id_P$.
Such a 4-manifold is clearly an $H$-cobordism over $\mathbb{Z}[\rho]$.
\end{proof}

The argument of the first paragraph of Lemma \ref{Hi96-4} together with 
Lemma \ref{Hi96-5} shows that any 3-manifold with the same homology as 
$S^2\times{S^1}$ is normally cobordant to $S^2\times{S^1}$.
However in general the normal cobordism cannot be improved to 
an $H$-cobordism over $\mathbb{Z}$. See \cite[\S11.8]{FQ}.

\begin{theorem}
\label{Hi96-thm1}
Let $P$ be a $3$-manifold which embeds in $S^4$ and such that 
every automorphism of $\rho=\pi_1P$ is induced by some 
(base-point preserving) self homeomorphism of $P$.
Let $W$ be a cobordism with $\partial{W}=P\sqcup{M}$ such that 
the inclusion of $P$ into $W$ induces an isomorphism $\rho\cong\pi_1W$, 
and which is an $H$-cobordism over $\mathbb{Z}[\rho]$. 
Then $M$ also embeds in $S^4$.
\end{theorem}

\begin{proof}
Let $Z=W\cup_{M}W$ be the union of two copies of $W$, doubled along $M'$, 
and let $j_1$ and $j_2$ be the natural identifications of $P$ with the two components of
$\partial{Z}$.
Since the inclusion of $M$ into $W$ induces an epimorphism on fundamental groups, 
the maps $\pi(j_i)$ are isomorphisms, for $i=1,2$, by Van Kampen's Theorem.
Let $\psi$ be a (basepoint preserving) self homeomorphism of $P$ inducing the 
isomorphism $\pi(j_2)^{-1}\pi(j_1)$.
Let $X$ and $Y$ be the complementary regions of an embedding of $P$ in $S^4$,
let $i_X:P\to\partial{X}$ and $i_Y:P\to\partial{Y}$ be the natural inclusions,
and let $\Sigma=X\cup{Z}\cup{Y}$, 
where we identify $i_X(p)$ with $j_1(p)$ and $i_Y(p)$ with $j_2\psi(p)$, 
for all $p\in {P}$.
Then $\Sigma$ is simply connected and $\chi(\Sigma)=\chi(X)+\chi(Y)=\chi(S^4)$, 
so $\Sigma\cong{S^4}$.
As $\Sigma$ contains $M$ as a locally flat submanifold this proves the theorem.
\end{proof}

There is a parallel argument on the homotopy level.
If $h:M\to{P}$ is a $\mathbb{Z}[\rho]$-homology isomorphism
where $\rho=\pi_1P$, and $Z(h)$ is the mapping cylinder of $h$ then 
$(Z(h);m,P)$ is a Poincar\'e duality triad with $\pi_1Z(h)\cong\rho$.
Hence if $P$ Poincar\'e embeds in $S^4$ and automorphisms of $\rho$
are realizable by self-homotopy equivalences of $P$ then $M$ also
Poincar\'e embeds in $S^4$.

It remains possible that a model manifold $P$ might not embed in $S^4$,
while $M$ does. 
However, we may settle the question for 3-manifolds 
with models having virtually solvable fundamental group.

\begin{cor}
\label{solvmodel}
If $\pi$ is an extension of an infinite solvable group $\rho$ 
by a perfect normal subgroup then $M$ embeds in $S^4$ if and only if 
$\rho$ is one of the $10$ torsion-free infinite groups
corresponding to cases (2) to (6) of Theorem \ref{CH-thm2.2}.
\end{cor}

\begin{proof}
As observed in Lemma \ref{Hi96-3}, solvable $PD_3$-groups are 3-manifold groups.
The condition is necessary, by Theorem \ref{CH-thm2.2} and Lemma \ref{Hcoblem}.
In each case, 
every automorphism of $\rho$ is induced by a self-homeomorphism of $P$,
and $P$ embeds smoothly in $S^4$.
Therefore the condition is also sufficient, by Lemmas \ref{Hi96-3}, \ref{Hi96-4} and \ref{Hi96-5}
and Theorem \ref{Hi96-thm1}.
\end{proof}

The connected sum of a homology sphere with $S^2\times{S^1}$ 
satisfies the hypotheses of Theorem \ref{Hi96-thm1}, 
but in this case the embeddability follows immediately from the corresponding result for homology spheres.
We may construct more interesting examples from knot manifolds.
Let $M=M(K)$, where $K$ is a knot with Alexander polynomial $\Delta_1(K)=1$,
and let $\pi=\pi_1M(K)$.
Then $\pi'$ is perfect,  and so $M$ embeds, by the above corollary.
If $K$ is non-trivial then $M$ is aspherical \cite{Ga87},
and thus is not a connected sum of $S^2\times{S^1}$ with a homology sphere.

There is another argument for the case $\rho=\mathbb{Z}$.
For a 3-manifold $M$ such that $\pi_1M$ is an extension of 
$\mathbb{Z}$ by a perfect normal subgroup $\nu$ may be obtained by 
0-framed surgery on a knot $K$ in an homology 3-sphere $\Sigma$.
Then $\Sigma=\partial{C}$ for some contractible 4-manifold $C$.
Since $\nu$ is perfect $K$ has Alexander polynomial 1, 
and so it bounds a disc $D$ in $C$ \cite[Theorem 11.7B]{FQ}.
A mild variation on the 0-framed surgery construction shows that $M$ embeds in
$C\cup_\Sigma{C}$, which is a homotopy 4-sphere, and so is homeomorphic to $S^4$.

The elementary argument sketched above for the case $\rho\cong\mathbb{Z}$
may be extended to realize the other possibilities considered here by 
examples based on links obtained by tying Alexander polynomial 1 
knots along various components of links representing the model 3-manifolds.
We do not know whether there are other $\mathbb{Z}[\rho]$-homology equivalent 3-manifolds which cannot be handled in this way, but think it likely. 

There is a similar result for the other torsion-free solvable groups corresponding to the manifolds listed in Theorem \ref{CH-thm2.2}.

\section{Extensions of the argument}

We may ask for what other groups $\rho$ does the present approach work.
When $\rho=D_\infty$ no such 3-manifold embeds, by Lemmas \ref{Hcoblem} 
and \ref{Dinfty}, and so the situation is clear in this case.

We are left with finite groups with cohomological period dividing 4.
If such a group has abelianization a direct double then it is 1, 
a generalized quaternionic group $Q(8k)$ (with $k\geq1$),
the binary icosahedral group $I^*=SL(2,5)$, or one of the groups $Q(2^na,b,c)$ 
(where $a,b,c$ are odd and relatively prime, and either $n=3$ and at most one
of $a,b,c$ is 1, or $n>3$ and $bc>1$).
Closed 3-manifolds with finite fundamental group are Seifert fibred,
and the only finite groups realized by 3-manifolds which embed in $S^4$
are 1,  $Q(8)$ and $I^*$. 
However in so far as the argument rests upon the 
$\mathbb{Z}/2\mathbb{Z}$-Index Theorem (and thus on the geometry of 
$\mathbb{Z}/2\mathbb{Z}$-actions on 4-manifolds) and is not purely homological, 
it remains possible that some 3-manifold with fundamental group an extension of 
$\rho$ by a perfect normal subgroup might embed.

The spherical space form $S^3/Q(8)$ embeds smoothly in $S^4$ as the boundary 
of a regular neighbourhood of an embedding of $\mathbb{RP}^2$.
Every automorphism of $Q(8)$ is realizable by a self homeomorphism of $S^3/Q(8)$ \cite{Pr77}.
Let $M$ be a closed 3-manifold such that $\pi=\pi_1M$ is an extension of 
$Q(8)$ by a perfect normal subgroup $\nu$,
and let $C_*$ be the cellular chain complex of the universal cover $\widetilde{M}$.
Since $M/\nu$ is an homology 3-sphere and $S^3/Q(8)$ is the unique finite Swan complex for $Q(8)$, the complex $\mathbb{Z}[Q(8)]\otimes_{\mathbb{Z}[\pi]}C_*$
is $\mathbb{Z}[Q(8)]$-chain homotopy equivalent to the cellular chain complex for the
universal cover of $S^3/Q(8)$.
Any such chain homotopy equivalence may be realized by a $\mathbb{Z}[Q(8)]$-homology equivalence,
since $M$ and $S^3/Q(8)$ each have dimension $\leq3$.
However there are difficulties in carrying through our strategy 
(i.e., in extending Lemma \ref{Hi96-5})
for this case, as $L_4(Q(8))$ has rank 5.

Since $S/I^*$ is a homology 3-sphere it embeds in $S^4$.
Although $L_4(I^*)$ has rank 9 there are no obstructions to embedding 3-manifolds
with fundamental group an extension of $I^*$ by a perfect normal subgroup,
for this case may be subsumed into the case $\rho=1$ settled by Freedman.

Although the groups $Q(2^na,b,c)$ are not always 3-manifold groups 
many act freely on homology 3-spheres \cite{HM86}.
However it is not known which of the corresponding quotients embed in $S^4$.
%(The arguments which exclude 3-manifolds with fundamental group $Q(8k)$ 
%for some $k>1$ rely on the $\mathbb{Z}/2\mathbb{Z}$-Index Theorem,
%and there is no obvious analogue for $PD_3$-complexes.)

%% file: e6.tex
\chapter{The complementary regions}
             
In this chapter we turn our attention to the variety of possible embeddings.
We consider here $\chi(W)$ and $\pi_1W$, 
for $W$  a complementary region of an embedding of $M$ in $S^4$.
Our examples mostly involve Seifert manifolds $M$,
and the obstructions to embeddings derive from the lower central series 
for $\pi$ and its dual manifestation in terms of (Massey) products of classes 
in $H^1(M;\mathbb{Q})$.

We begin with a proof of the Generalized Schoenflies Theorem.
This result is not specifically about embeddings in $S^4$, 
but the argument is so simple that it deserves a place here.
The original proof of Aitchison's Theorem on smooth embeddings 
of $S^2\times{S^1}$ used the Generalized Schoenflies Theorem.
We sketch this briefly,  and use surgery to give an argument
for locally flat embeddings.

In the next three sections we use the Massey product structure in $H^*(M)$ 
to show that if $M$ is a Seifert manifold with orientable base orbifold 
and non-zero Euler number then $\chi(X)=\chi(Y)=1$ is the only possibility.
On the other hand, 
all values except for $\chi(X)=1-\beta$ and $\chi(Y)=1+\beta$
are realized by embeddings of $T_g\times{S^1}$.

Sections  6--8 lead to a criterion for a complementary region to be aspherical 
and of cohomological dimension at most 2.
When $M=F\times{S^1}$ or when $M$ is the total space of an $S^1$-bundle
with non-orientable base the simplest embeddings of $M$ have 
one complementary component $X\simeq{F}$ 
and the other with cyclic fundamental group.
In \S9 we sketch how surgery may be used to identify such embeddings 
(up to $s$-cobordism).
(No such argument is yet available when $M$ fibres over an orientable 
base with Euler number 1.)

\section{The Generalized Schoenflies Theorem}

The arguments of Brown and Mazur for the Generalized Schoenflies Theorem
each involve limiting processes in an essential way.
Brown used ideas from decomposition theory (``Bing topology") involving
collapsing cellular sets.
The key idea in Mazur's argument is an ingenious regrouping of an infinite ``sum",
and we shall outline this version (as completed by M. Morse).
Both versions are presented in \cite{Pu24}, 
and our account is based on this.

If $M$ and $N$ are two $n$-manifolds with fixed homeomorphisms
$\partial{M}\cong\partial{I^n}$ and $\partial{N}\cong\partial{I^n}$,
we let $M+N$ be the boundary connected sum, 
whereby we identify $(1,x_2,\dots,x_n)\in\partial{M}$ with 
$(0,x_2,\dots,x_n)\in\partial{N}$, 
for all $0\leq{x_2,\dots,x_n}\leq1$.
Clearly $M+D^n\cong{M}$,
$M+N\cong{N+M}$, and stacking with respect to the first coordinate
gives a natural identification of $\partial(M+N)$ with $\partial{I^n}$.

Define $kM$ inductively by $1M=M$ and $(k+1)M=kM+M$, for $k\geq1$.
Then $kM$ is embedded in $(k+1)M$, 
and $LM=\cup_{k\geq1}kM$ is a (non-compact) $n$-manifold 
with boundary homeomorphic to $\mathbb{R}^{n-1}$.
In particular,  
$LD^n\cong[0,\infty)\times{I^{n-1}}\cong\mathbb{R}^n_+$.
We shall extend the definition of $+$ slightly by identifying
$\{0\}\times{I^{n-1}}\subset\partial(1M)\subset{LM}$ with
$\{1\}\times{I^{n-1}}\subset{M}$, and then we see that $M+LM\cong{LM}$.
We also see that $L(M+N)\cong{M}+L(N+M)$.

\begin{theorem}
[Brown-Mazur]
\label{GST}
Let $j$ be a locally flat embedding of $S^{n-1}$ into $S^n$.
Then there is a homeomorphism $h$ of $S^n$ such that $h\circ{j}$ is the equatorial embedding.
\end{theorem}

\begin{proof}
Since $j$ is locally flat it extends to an embedding $J:E=S^{n-1}\times[-1,1]\to{S^n}$
of a neighbourhood of the equator $S^{n-1}\times\{0\}\subset{S^n}$,
by Brown's Collaring Theorem \ref{Brown}.
Let $X$ be the complementary region containing $J(S^{n-1}\times\{1\})$
and fix homeomorphisms of $\partial{X}$ and $\partial{Y}$ with $\partial{I^n}$.
We shall show that $X$ and $Y$ are $n$-discs.

After composition with a rotation of $S^n$, if necessary, 
we may assume that $J(P)=P$, for some $P\in{S^{n-1}}\times\{\frac12\}$.
Let $U=B_r(P)$ be an open ball in $J(S^{n-1}\times(0,1))$,
with respect to the standard metric, 
and let $Q\in{S^n}\setminus{X}$ be another point.
Note that $X\cong{X}\setminus(J(S^{n-1}\times(0,\frac12))\cup{U})$
and $Y\cong{Y}\cup{J(S^{n-1}\times(0,\frac12}))\setminus{U}$,
and so $X+Y\cong{S^n}\setminus{U}\cong{D^n}$.

There is a homeomorphism $c$ from $S^n\setminus\{Q\}$ onto $U$ 
which is the identity on the closed ball $\overline{B_{r/2}(P)}$.
Let $V=\overline{B_s(P)}$ be another closed ball centred on $P$ 
such that $J(V)\subset\overline{B_{r/2}(P)}$.
Let $W=J^{-1}(U)$, and let $f:{S^n}\setminus\{Q\}\to{S^n}$ be
the composite of $(J|_W)^{-1}\circ{c}$ with the inclusion of $U$ into $S^n$.
The composite embedding $J'=f\circ{J}$ then has the property
that there is an $n$-disc $D\subset {S^{n-1}}\times(0,1)\subset{E}$ 
such that the closed complements $\overline{S^n\setminus{D}}$ and 
$\overline{S^n\setminus{J'(D)}}$ are each homeomorphic to $D^n$.
(Establishing this was the contribution of Morse.)
It shall suffice to show that $f(X)\cong{D^n}$, since $X\subset{{S^n}\setminus\{Q\}\ }$
and $f$ is a homeomorphism onto its image.
We shall assume henceforth that the original embedding $J$ has the above property.

Since $LD^n\cong\mathbb{R}^n_+$,
the 1-point compactification $LD^n\cup\{\infty\}$ is homeomorphic to $D^n$.
Now $X+L(Y+X)=L(X+Y)=LD^n$, 
and so  $X\cong{X+D^n}\cong{X+L(Y+X)}\cup\{\infty\}\cong{D^n}$.
A similar argument shows that $Y\cong{D^n}$.

Every self-homeomorphism $\phi$ of $S^{n-1} $ extends to 
a self-homeomorphism $\Phi$ of $D^n$ by setting $\Phi(r.s)=r\phi(s)$ for 
all $0\leq{r}\leq1$ and $s\in {S^{n-1}}$.
We may extend $j$ in this manner across each hemisphere to obtain a self-homeomorphism of $S^n$. 
The inverse homeomorphism $h$ is the equatorial embedding.
\end{proof}

In this case there is no advantage in considering only 3-spheres in $S^4$.
This result is also a consequence of TOP surgery, 
since the complementary regions must be contractible, 
by van Kampen's Theorem and Alexander duality.
However this is an unnecessarily blunt instrument here.

\section{$S^2\times{S^1}$ and Aitchison's Theorem}

Since $S^2\times{S^1}$ may be obtained by 0-framed surgery on the unknot, 
it has a standard abelian embedding
with $X\cong{S^1\times{D^3}}$ and $Y\cong{D^2}\times{S^2}$.
In fact $Y\cong{D^2}\times{S^2}$ whenever $M=S^2\times{S^1}$,
by a result of Aitchison.
This result predates topological surgery.
The proof given in \cite{Ru80} uses the ``Dehn's Lemma" of R. A. Norman
\cite{No69} together with the  Generalized Schoenflies Theorem 
to show that one complementary region of a {\it smooth\/} embedding 
of $S^2\times{S^1}$ in $S^4$ must be homeomorphic to $S^2\times{D^2}$.
We shall outline this proof and then give one which applies to all 
(TOP locally flat) embeddings.

\begin{lem}
[Norman]
Let $D$ be an $n$-point smoothly immersed $2$-disc in a $4$-manifold $P$.
Suppose that there is an embedded $2$-sphere $\sigma\subset{int\,P}$ 
with trivial normal bundle and such that $D\cap\sigma$ is a single point of transverse intersection.
Then there is an embedded $2$-disc $\Delta\subset{P}$ with
$\partial\Delta=\partial{P}$.
\qed
\end{lem}

If $j:S^2\times{S^1}\to{S^4}$ is an embedding then $\pi_Y=1$, 
by Lemma \ref{VKsplitmono}, and so $C=j(\{*\}\times{S^1})$
is null-homotopic in $Y$.
Hence $C$ bounds a smooth embedded disc $\Delta$ in $Y$, 
by Norman's Lemma.
Let $N(\Delta)$ be a tubular neighbourhood of $\Delta$ in $Y$.
Then $\overline{Y\setminus{N(\Delta)}}$ is homeomorphic to $D^4$, 
by the TOP Schoenflies Theorem, and so $Y\cong{D^2\times{S^2}}$.

While the extension of transversality to the 4-dimensional TOP setting 
probably allows for a corresponding extension of the applicability 
of Norman's argument,
we shall instead appeal to 1-connected TOP surgery.

\begin{theorem}[Aitchison]
\label{aitch}
If $M=S^2\times{S^1}$ is embedded in $S^4$ then one complementary region is homeomorphic to $S^2\times{D^2}$.
\end{theorem}

\begin{proof}
Let $j:M\to{S^4}$ be an embedding.
Since $j_{X*}=\pi_1j_X$ is a split monomorphism, $\pi_Y=1$, 
by Lemma \ref{VKsplitmono}
and so $Y\simeq{S^2}$. 
Let $f:\partial{Y}\to{S^2}\times{S^1}$ be a homeomorphism.
Then $f$ extends to a map from $Y$ to $S^2\times{D^2}$.
Thus we obtain a map of pairs $F:(Y,\partial{Y})\to(S^2\times{D^2},S^2\times{S^1})$ 
which is a homotopy equivalence and a homeomorphism on the boundary.
Hence $Y$ is homeomorphic to $S^2\times{D^2}$, by 1-connected surgery
\cite[Theorem 11.6A]{FQ}.
\end{proof}

If $\pi_X\cong\mathbb{Z}$ then $X\cong{D^3\times{S^1}}$ 
and the embedding is equivalent to $j_U$,
by Freedman's Unknotting Theorem for 2-knots \cite[Theorem 11.7A]{FQ}.
(This is also a special case of Theorem 7.7 below.)

We shall give a partial extension of Aitchison's Theorem to embeddings of 
$\#^\beta(S^2\times{S^1})$ at the end of this chapter.

\section{Massey products}

Massey products provide further obstructions to finding embeddings 
with given $\chi(X)$. 
For instance, if $H^2(X;\mathbb{Q})\cong\mathbb{Q}$ or 0 
then all triple Massey products $\langle{a},b,c\rangle$ of elements 
$a,b,c\in{H^1}(X;\mathbb{Q})$ are proportional.
Stallings' Theorem can be refined to relate ``freeness" of quotients 
of the lower central series and its rational and $mod~(p)$ analogues
to the vanishing of higher Massey products \cite{Dw75, St65}.

The $\mathbb{N}il^3$-manifold $M=M(Wh)$ has fundamental group 
$\pi\cong{F(2)/\gamma_3F(2)}$,
with a presentation
\[
\pi=\langle{x,y,z}\mid {z=xyx^{-1}y^{-1}},~xz=zx,~yz=zy\rangle.
\]
Every element of $\pi$ has an unique normal form $x^my^nz^p$.
The images $X,Y$ of $x,y$ in 
$H_1(\pi)\cong H_1(T)$
form a (symplectic) basis.
Let $\xi,\eta$ be the Kronecker dual basis for $H^1(\pi)$.
Define functions $\phi_\xi,\phi_\eta$ and $\theta:\pi\to\mathbb{Z}$
by 
\[
\phi_\xi(x^my^nz^p)=\frac{m(1-m)}2,~
\phi_\eta(x^my^nz^p)=\frac{n(1-n)}2~\mathrm{and}~
\theta(x^my^nz^p)=-mn-p,
\]
for all $x^my^nz^p\in\pi$.
(We consider these as inhomogeneous 1-cochains with values 
in the trivial $\pi$-module $\mathbb{Z}$.)
Then 
\[
\delta\phi_\xi(g,h)=\xi(g)\xi(h),\quad 
\delta\phi_\eta(g,h)=\eta(g)\eta(h)\quad\mathrm{and}\quad
\delta\theta(g,h)=\xi(g)\eta(h),
\]
for all $g,h\in\pi$.
Thus $\xi^2=\eta^2=\xi\cup\eta=0$, and the Massey triple products 
$\langle\xi,\xi,\eta\rangle$ and $\langle\xi,\eta,\eta\rangle$ 
are represented by the 2-cocycles $\phi_\xi\eta+\xi\theta$ 
and $\theta\eta+\xi\phi_\eta$, respectively.
On restricting these to the subgroups generated by $\{x,z\}$ and $\{y,z\}$, 
we see that they are linearly independent.

In fact, $\langle\xi,\xi,\eta\rangle\cup\eta$ and
$\langle\xi,\eta,\eta\rangle\cup\xi$ each 
generate $H^3(\pi;\mathbb{Q})$.
This is best seen topologically.
Let $p:M\to{T}$ be the natural  fibration of $M$ over the torus,
and let $x$ and $y$ be simple closed curves in $T$ which represent a basis for 
$\pi_1Tcong\mathbb{Z}^2$.
The group $H_2(M)\cong\mathbb{Z}^2$ is generated by the images of
the fundamental classes of the tori $T_x=p^{-1}(x)$ and $T_y=p^{-1}(y)$.
If we fix sections in $M$ for the loops $x$ and $y$ we see that 
$[T_x]\bullet{x}=[T_y]\bullet{y}=0$ while $|[T_x\bullet{y}|=|T_y\bullet{x}|=1$.
Hence $[T_x]$ and $[T_y]$ are Poincar\'e dual to $\eta$ and $\xi$, respectively.
Since $\langle\xi,\xi,\eta\rangle$ restricts non-trivially to $T_x$ 
and trivially to $T_y$ we must have
$\langle\xi,\xi,\eta\rangle\cup\eta\not=0$, 
and similarly $\langle\xi,\eta,\eta\rangle\cup\xi\not=0$.
Similarly,
$\langle\xi,\xi,\eta\rangle\cup\xi=\langle\xi,\eta,\eta\rangle\cup\eta=0$.
Thus these Massey products are the Poincar\'e duals of $Y$ and $X$,
respectively.

Since the components of $Wh$ are unknotted, $M$ embeds in $S^4$, 
with $\chi(X)=\chi(Y)=1$, and $\mu_M=0$, since $\beta=2$.
On the other hand, $M$ has no embedding with $\chi(X)=-1$,
for otherwise $H^3(X)$ would contain $\langle\xi,\xi,\eta\rangle\cup\eta$,
and so be non-trivial. 

A similar strategy may be used for $M=M(g;(1,e))$ and $\pi=\pi_1M$, 
when $g>1$.

\begin{lemma}  
Let $p:M\to{T_g}$ be the projection of an $S^1$-bundle with non-zero Euler Number.
Then $H^1(p;\mathbb{Q})$ is an isomorphism and $H^2(p;\mathbb{Q})=0$,
and so $\mu_M=0$.
Let $\{\alpha_1,\beta_1,\dots,\alpha_g,\beta_g\}$ 
be the basis for $H=H^1(M;\mathbb{Q})$ which is Kronecker dual 
to a symplectic basis for $H_1(M;\mathbb{Q})\cong{H_1(T_g;\mathbb{Q})}$.
Then the Massey triple products $\langle\alpha_i,\alpha_i,\beta_i\rangle$
and $\langle\alpha_i,\beta_i,\beta_i\rangle$ (for $1\leq{i}\leq{g}$)
form a basis for $H^2(\pi;\mathbb{Q})$ which is Poincar\'e dual to the given basis for $H_1(\pi;\mathbb{Q})$.
\end{lemma}

\begin{proof}
The first assertion follows from the Gysin sequence for the bundle 
\cite[Theorem 5.7.11]{Span}.
We may assume that $\pi=\pi_1M$ has a presentation $\langle{x_1,y_1,\dots,x_g,y_g}\mid\Pi[x_i,y_i]=h^e,~h~central\rangle$.
Thus the images of the generators $\{x_1,y_1,\dots,x_g,y_g\}$ in $H_1(B)$,
represent a standard symplectic basis, and so determine a basis for $H_1(M;\mathbb{Q})\cong{H_1(B;\mathbb{Q})}$.
The argument then follows as in the case of $M(1,(1,-1))$ discussed above.
\end{proof}

We shall extend these results to the Seifert case in the next section.

\section{Seifert manifolds}

In this section we shall use cup products and Massey products to restrict
the possible values of $\chi(W)$ for $W$ a complementary region of an embedding of a Seifert manifold $M$ with orientable base orbifold.

\begin{lemma}
\cite{CT}
\label{Hi17-lem8.1}
Let $M$ be a Seifert manifold.
If the base $B$ is nonorientable or if  $\varepsilon(M)\not=0$
then $H^*(M;\mathbb{Q})\cong{H^*}(\#^\beta{S^2\times{S^1}};\mathbb{Q})$.
Otherwise, the image of $h$ in $H_1(M;\mathbb{Q})$ is nonzero, 
and $H^*(M;\mathbb{Q})\cong{H^*(|B|\times{S^1};\mathbb{Q})}$.
\end{lemma}

\begin{proof}
There is a finite regular covering $q:\widehat{M}\to{M}$,
where $\widehat{M}$ is an $S^1$-bundle space with orientable 
base $\widehat{B}$, say.
Let $G=Aut(q)$.
Then $H^*(M;\mathbb{Q})\cong{H^*(\widehat{M};\mathbb{Q})^G}$.
If $B$ is nonorientable or if $\varepsilon(M)\not=0$
then the regular fibre has image 0 in $H_1(M;\mathbb{Q})$,
and so $H^*(\widehat{B};\mathbb{Q})$ maps onto $H^*(\widehat{M};\mathbb{Q})$.
Hence all triple cup products of classes in
$H^1(\widehat{M};\mathbb{Q})$ are 0.
Therefore all pairwise cup products of such classes are also 0, 
by the non-degeneracy of Poincar\'e duality,
and so $H^*(M;\mathbb{Q})\cong{H^*}(\#^\beta{S^2\times{S^1}};\mathbb{Q})$.
Otherwise, $\widehat{M}\cong\widehat{B}\times{S^1}$ 
and $G$ acts orientably on each of $S^1$ and $\widehat{B}$.
Hence the image of $h$ in $H_1(M;\mathbb{Q})$ is nonzero 
and $H^*(M;\mathbb{Q})\cong{H^*(|B|\times{S^1};\mathbb{Q})}$.
\end{proof}

We may use the observations on cup product from Lemma \ref{Hi17-lem8.1}
 to extract some information on the image of the regular fibre 
under the maps $H_1(j_X)$ and $H_1(j_Y)$, when $M$ is Seifert fibred.

\begin{theorem}
\label{Hi17-thm8.2}
Let $M=M(g;S)$ where $g\geq1$ and $\varepsilon(M)=0$.
If $M$ embeds in $S^4$ then $\chi(X)>1-\beta=-2g$ and $\chi(Y)<1+\beta=2g+2$.
If $\chi(X)<0$ then the image of $h$ in $H_1(Y;\mathbb{Q})$ is non-trivial.
\end{theorem}

\begin{proof}
Let $\{a_i^*,b_i^*;1\leq{i}\leq{g}\}$ be the images in $H^1(M;\mathbb{Q})$ 
of a symplectic basis for $H^1(|B|;\mathbb{Q})$. 
Then $a_i^*(h)=b_i^*(h)=0$ for all $i$.
Let $\theta\in{H^1(M;\mathbb{Q})}$ be such that $\theta(h)\not=0$.
By Lemma \ref{Hi17-lem8.1} we have
\[
H^*(M;\mathbb{Q})\cong{H^*(|B|\times{S^1};\mathbb{Q})}\cong
\mathbb{Q}[\theta,a_i^*,b_i^*,~\forall~i\leq{g}]/I,
\]
where $I$ is the ideal $(\theta^2,a_i^{*2}, b_i^{*2},
\theta{a_i^*b_i^*}-\theta{a_j^*b_j^*},
a_i^*a_j^*,b_i^*b_j^*, ~\forall~1\leq{i}<j\leq{g})$.

Since $\theta{a_1^*b_1^*}\not=0$ the triple product $\mu_M\not=0$,
and so $M$ has no embedding with $\beta_2(Y)=0$, 
by Lemma \ref{Hi17-lem2.2}.
Hence $\chi(X)=1-\beta$ ($\Leftrightarrow\chi(Y)=1+\beta$)
is impossible.

If $\chi(X)<0$ then $\beta_1(X)>g+1$,
and so the image of $H^1(X;\mathbb{Q})$ in $H^1(M;\mathbb{Q})$ 
must contain some pair of classes from the image of $H^1(|B|;\mathbb{Q})$
with nonzero product.
But then it cannot also contain $\theta$, 
since all triple products of classes in $H^1(X;\mathbb{Q})$ are 0.
Thus the image of $H^1(Y;\mathbb{Q})$ must contain a class 
which is non-trivial on $h$, 
and so $j_Y(h)\not=0$ in $H_1(Y;\mathbb{Q})$.
\end{proof}

In particular, if $g=1$ then $\chi(X)=0$ and $\chi(Y)=2$.

Theorem \ref{Hi17-thm8.2} also follows from Lemma \ref{Hi17-lem4.1},
since the centre of $\pi$ is not contained in the commutator subgroup $\pi'$.

If the base orbifold $B$ is nonorientable or if $\varepsilon(M)\not=0$
then $\mu_M=0$, by Lemma \ref{Hi17-lem8.1},
and so the argument of Theorem  \ref{Hi17-thm8.2} 
does not extend to these cases.
However, Lemma \ref{Hi17-lem8.1}  also suggests that when 
$\varepsilon(M)\not=0$ we should be able to use Massey product arguments 
as in \S6.3 above (where we considered the case $S=\emptyset$).

\begin{theorem}
\label{Hi17-thm8.3}
Let $M=M(g;S)$, where $g\geq0$ and $\varepsilon(M)\not=0$.
If $M$ embeds in $S^4$ with complementary regions $X$ and $Y$ then $\chi(X)=\chi(Y)=1$.
\end{theorem}

\begin{proof}
The group $\pi=\pi_1M(g;S)$ has a presentation
\[
\langle{x_1,y_1,\dots,x_g,y_g, c_1,\dots,c_r,h}\mid
\Pi[a_i,b_i]\Pi{c_j}=1,~c_i^{\alpha_i}h^{\beta_i}=1,~h~central\rangle.
\]
We may assume that $g\geq1$, 
for if $g=0$ then $M$ is a $\mathbb{Q}$-homology 3-sphere 
and the result is clear.
As in Lemma \ref{Hi17-lem8.1}, there is a finite regular covering
$q:\widehat{M}\to{M}$,
where $\widehat{M}$ is an $S^1$-bundle space with orientable base 
$\widehat{B}$.
Let $G=Aut(q)$.
Then $H^*(M;\mathbb{Q})\cong{H^*(\widehat{M};\mathbb{Q})^G}$.
Let  $\{\alpha_2,\beta_2,\dots,\alpha_h,\beta_h\}$ be a basis for $H^1(\widehat{M};\mathbb{Q})$ which is Kronecker dual to 
the image of a symplectic basis for $H_1(\widehat{B})$ in $H_1(\widehat{M};\mathbb{Q})$.

The homomorphism $H^1(p;\mathbb{Q})$ from $H^1(|B|;\mathbb{Q})$
to $H^1(M;\mathbb{Q})$ induced by the Seifert fibration $p:M\to{B}$.
If $j:M\to{S^4}$ is an embedding then $H=H^1(M;\mathbb{Q})=A\oplus\Omega$,
where $A$ and $\Omega$ are self-annihilating with respect to cup product. 
If $L\leq{H}$ is a direct summand of rank $>g$ then there are $a\in{L\cap{A}}$ and $b\in{L\setminus{A}}$ such that $a\cup{b}\not=0$ in $H^2(|B|;\mathbb{Q})$.
We consider their images under $H^1(q;\mathbb{Q})$.
Since $\widehat{M}\cong\widehat{B}_o\times{S^1}\cup{D^2\times{S^1}}$,
it is easy to see that every self-homeomorphism of $\widehat{B}$ lifts
to a fibre-preserving self-homeomorphism of $\widehat{M}$.
Thus (after multiplication by factors in $\mathbb{Q}^\times$,
if necessary) we may assume that $a=\alpha_1$ and then $b=\beta_1+b'$, 
where $b'$ is in the span of $\{\alpha_2,\beta_2,\dots,\alpha_h,\beta_h\}$.
We now view these classes as classes in $H^1(\widehat{B};\mathbb{Q})$.
Since $a\cup{b'}=0$ there is a map $f:\widehat{B}\to{V=S^1\vee{S^1}}$,
such that $a$ and $b'$ are in the image of $H^1(f;\mathbb{Q})$.
Hence $\langle{a},{a},b'\rangle=0$.
The topological argument used in \S6.2 shows that
$\langle{a,a,\beta_1}\rangle\cup{b'}=0$ also.
Therefore $\langle{a},{a},b\rangle\cup{b}=
\langle\alpha_1,\alpha_1,\beta_1\rangle\cup\beta_1$ is non-zero.
But this contradicts  the fact that $H^3(X)=H^3(Y)=0$.
Therefore
$H^1(X)$ and $H^1(Y)$ each have rank at most $g$,
and so $\chi(X)=\chi(Y)=1$.
\end{proof}

If $\chi(X)=0$, all cone point orders are odd and $h$ has nonzero image in $H_1(X;\mathbb{Q})$ then $S$ is skew-symmetric, 
by Theorem \ref{g=0skewsym}.
(In particular, this must be the case if $g$ and $\varepsilon_S$ are 0.)
Conversely, if $S$ is skew-symmetric and all cone point orders 
$a_i$ are odd then $M(0;S)$ embeds smoothly.
Since $\beta=1$ we must have $\chi(X)=0$ and $H_1(Y;\mathbb{Q})=0$.
(In fact, for the embedding constructed on page 693 of \cite{CH98}
the component $X$ has a fixed point free $S^1$-action.)
Hence also $M(g;S)$ embeds smoothly, as in Lemma \ref{CH-lem3.2},
which gives embeddings with $\chi(X)=0$.
Is there a natural choice of 0-framed bipartedly slice link representing $M(0;S)$?
Are all values of $\chi(X)$ consistent with Theorem \ref{Hi17-thm8.2} 
possible for $M(g;S)$?

However, even if $\chi(X)=0$ the other hypothesis of Theorem \ref{g=0skewsym}
need not hold.
For instance, the standard 0-framed link representing $M=T_2\times{S^1}$ 
has two essentially different partitions
into 3- and 2-component trivial sublinks.
For one, 
$\pi_X\cong\mathbb{Z}\times{F(2)}$ and $\pi_Y\cong{F(2)}$,
while for the other $\pi_X\cong\mathbb{Z}*\mathbb{Z}^2$ and $\pi_Y\cong\mathbb{Z}^2$.

If $\ell_M$ is hyperbolic then all even cone point orders have 
the same 2-adic valuation,
by Theorem \ref{CH-thm3.7} (when $g<0$) and Lemma 6 of 
Appendix A (when $g\geq0$).

\section{$S^1$-bundle spaces}

The bundle space $E={M(g;(1,e))}$ can only embed in $S^4$ if $e=0$ or $\pm1$,
since $\tau_E=0$ if $e=0$ and is cyclic of order $e$ otherwise.
The 3-torus $M(1;(1,0))$ is also $M(Bo)$.
Since $M(g;(1,0))\cong{T_g\times{S^1}}$ is an iterated fibre sum 
of copies of $T\times{S^1}$,
it may be obtained by 0-framed surgery on the $(2g+1)$-component link 
of Figure 6.1,
which shares some of the Brunnian properties of $Bo$.

\setlength{\unitlength}{1mm}
\begin{picture}(90,52)(-12,-13)

\qbezier(15,0)(50,-5)(75,0)
\qbezier(15,0)(0,4)(10.8,12.7)
\qbezier(75,0)(90,4)(79.2,12.7)

\qbezier(14.3,23)(12.3,19)(11.7,16)
\qbezier(11.7,16)(11.5,14.7)(11.4,14.3)
\qbezier(11.4,14.3)(12.7,3)(19.6,14.3)

\qbezier(15.6,20.8)(13.8,17)(13.3,14)

\qbezier(19,20.5)(18.3,17)(17.8,15.5)
\qbezier(21,20.5)(20.2,18)(20,16.2)

\qbezier(14,13.7)(18,15)(28,16.5)

\qbezier(13.3,14)(13.7,8)(17.3,13.8)

\qbezier(14.7,22)(17,20.5)(20.2,21.2)

\qbezier(14.3,23)(16.3,25)(16.1,22.5)
\qbezier(19.5,22)(20.5,24)(21.2,22)
\qbezier(21.2,22)(21.1,21)(21,20.5)

\qbezier(13,23)(11,25.5)(13,28)
\qbezier(22,22)(24,25)(22,28)
\qbezier(13,28)(14.5,29.5)(17.5,30)
\qbezier(17.5,30)(20,30)(22,28)

\qbezier(75.7,23)(77.7,19)(78.3,16)
\qbezier(78.3,16)(78.5,14.7)(78.6,14.3)
\qbezier(78.6,14.3)(77.3,3)(70.4,14.3)

\qbezier(74.4,20.8)(76.2,17)(76.7,14)

\qbezier(71,20.5)(71.7,17)(72.2,15.5)
\qbezier(69,20.5)(69.8,18)(70,16.2)

\qbezier(76,13.7)(72,15)(62,16.5)

\qbezier(76.7,14)(76.3,8)(72.7,13.8)

\qbezier(75.3,22)(73,20.5)(69.8,21.2)

\qbezier(75.7,23)(73.7,25)(73.9,22.5)
\qbezier(70.5,22)(69.5,24)(68.8,22)
\qbezier(68.8,22)(68.9,21)(69,20.5)

\qbezier(77,23)(79,25.5)(77,28)
\qbezier(68,22)(66,25)(68,28)
\qbezier(77,28)(75.5,29.5)(72.5,30)
\qbezier(72.5,30)(70,30)(68,28)

\put(7,17){$L_1$}
\put(7,27){$L_2$}
\put(79,17){$L_{2g-1}$}
\put(79,27){$L_{2g}$}
\put(25,1){$L_0$}
\put(38,30){($g$ copies)}
\put(39,23){.\quad.\quad.\quad.}
\put(19,-10){Figure 6.1. A link giving $T_g\times{S^1}$}

\end{picture}

This 3-manifold has an embedding as the boundary of $T_g\times{D^2}$,
the regular neighbourhood of the unknotted embedding of $T_g$ in $S^4$, 
with the other complementary region having fundamental group $\mathbb{Z}$.
It is easy to see that if $g\geq1$ then $T_g\times{S^1}$ has 
other embeddings with $\chi(X)$ realizing each even value $>1-\beta$.
On the other hand,  $\mu_{T_g\times{S^1}}\not=0$,
and so no embedding has a complementary region $Y$ with $\beta_1(Y)=0$.

Changing the framing on one component of $Bo$ to 1,
and applying a Kirby move to isolate this component
gives the disjoint union of the Whitehead link $Wh$ and the unknot.
Since the linking numbers are 0 the framings are unchanged, 
and we may delete the isolated 1-framed unknot.
Thus $M(1;(1,1))$ may be obtained by 0-framed surgery on $Wh$.
The corresponding modification of the standard 0-framed $(2g+1)$-component
link $L$ representing $T_g\times{S^1}$ involves changing the framing of the component 
$L_{2g+1}$ whose meridian represents the central factor of $\pi$.
Performing a Kirby move and deleting an isolated 1-framed unknot gives
a 0-framed $2g$-component link representing $M(g;(1,1))$.
%(See Figure 2.)
%\bigskip
%Figure [Hi17-2]

Since the original link had partitions into two trivial links 
with $g+1$ and $g$ components respectively, 
the new link has a partition into two trivial $g$-component links.
However this is the only partition into slice sublinks,
for as we shall see consideration of the Massey product structure 
shows that all embeddings of $M(g;(1,1))$ have $\chi(X)=\chi(Y)=1$.

Suppose now that $F$ is nonorientable.
Then $M(-c;(1,e))$ embeds if and only if it embeds as 
the boundary of a regular neighbourhood of an embedding 
of $\#^c\mathbb{RP}^2$ with normal Euler number $e$.
We must have $e\leq2c$ and $e\equiv2c$ {\it mod} (4),
by Corollary \ref{CH-prop1.3}.
The standard embedding of $\mathbb{RP}^2$ in $S^4$ is determined up to 
composition with a reflection of $S^4$.
The complementary regions are each homeomorphic to a disc bundle over $\mathbb{RP}^2$
with normal Euler number 2, and so have fundamental group $\mathbb{Z}/2\mathbb{Z}$.
The standard embeddings of $\#^c\mathbb{RP}^2$ are obtained by taking iterated
connected sums of these building blocks $\pm(S^4,\mathbb{RP}^2)$, 
and in each case the exterior has fundamental group $\mathbb{Z}/2\mathbb{Z}$.
The regular neighbourhoods of $\#^c\mathbb{RP}^2$ are disc bundles 
with boundary $M(-c;(1,e))$.
Thus $M(-c;(1,e))$ has an embedding with one complementary component 
$X_{c,e}$ a disc bundle over $\#^c\mathbb{RP}^2$ 
and the other component $Y_{c,e}$
having fundamental group $\mathbb{Z}/2\mathbb{Z}$.

This embedding arises from a 0-framed $(c+1)$-component link
assembled from copies of the $(2,4)$-torus link $4^2_1$ and its reflection.
This is the union of an unknot and a trivial $c$-component link, 
but has no other partitions into slice links.
However, we can do better if we recall that 
$\#^c\mathbb{RP}^2\cong(\#^{c-2g}\mathbb{RP}^2)\#T_g$ for any $g$ 
such that $2g<c$.
The 3-manifold obtained by 0-framed surgery on the $(a+b+2g+1)$-component link 
of Figure 6.2 is $M(-c;(1,e))$, where $c=a+b+2g$ and $e=\pm2(a-b)$, 
with the sign depending on the orientation chosen.
(See also \cite[Figure A.3 ]{CH98}.)

This link has partitions into trivial sublinks corresponding 
to all the values $2-c\leq\chi(X)\leq\min\{2-\frac{|e|}2,1\}$
such that $\chi(X)\equiv{c}$ {\it mod} (2).
Are any other values realized?
In particular, does $M(-3;(1,6))$ embed with $\chi(X)=\chi(Y)=1$?

\setlength{\unitlength}{1mm}
\begin{picture}(90,60)(-5,-27)

\qbezier(15,0)(0,4)(10.5,12.5)
\qbezier(88,0)(105,4)(94.5,12.5)
\qbezier(15,0)(40,-4)(54,-3.5)
\qbezier(58,-3.3)(62,-3.2)(63,-3)
\qbezier(75,-2)(79,-2)(83,-1)

\qbezier(56,-6)(53,1)(51, -2.5)
\qbezier(58,-4.8)(53,4.5)(48.7,-2.5)
\qbezier(48.6,-4.5)(48.6, -6.5)(49.2, -6.5)
\qbezier(50.8,-4.4)(51,-6.5)(51.2,-6.5)
\qbezier(51.15,-6.5)(51.1,-7.8)(50.6,-7.8)
\qbezier(50.5, -6.5)(49,-8)(50,-10)
\qbezier(52, -6.1)(54,-5)(55,-5.3)
\qbezier(56.3,-5)(58,-5)(59,-8)
\qbezier(56,-6)(56.3,-7)(57.5,-7)
\qbezier(50,-10)(51,-12)(53,-12)
\qbezier(57,-11.5)(59,-11)(59,-8)
\qbezier(53,-12)(56, -12)(57, -11.5)

\qbezier(86,-4)(83,3)(81, -0.5)
\qbezier(88,-2.8)(83,6.5)(78.7,-0.5)
\qbezier(78.6,-2.5)(78.6, -4.5)(79.2, -4.5)
\qbezier(80.8,-2.4)(81,-4.5)(81.2,-4.5)
\qbezier(81.15,-4.5)(81.1,-5.8)(80.6,-5.8)
\qbezier(80.5, -4.5)(79,-6)(80,-8)
\qbezier(82, -4.1)(84,-3)(85,-3.3)
\qbezier(86.3,-3)(88,-3)(89,-6)
\qbezier(86,-4)(86.3,-5)(87.5,-5)
\qbezier(80,-8)(81,-10)(83,-10)
\qbezier(87,-9.5)(89,-9)(89,-6)
\qbezier(83,-10)(86, -10)(87, -9.5)

\qbezier(10.5,12.5)(11.5, 13)(13.3, 13.4)
\qbezier(14.8,13.8)(16.5,14.2)(18.8,14.5)
\qbezier(20.5,14.7)(21,15)(23.5,15.3)
\qbezier(31,16.3)(33,16.7)(34.5,16.8)
\qbezier(36.5,17)(38, 17.2)(40,17.3)

\qbezier(94.5,12.5)(93.5, 13.3)(91.2, 14)
\qbezier(89.7,14.3)(88.3,14.5)(85.5,14.8)
\qbezier(83.6,15.1)(82.5,15.3)(81,15.4)
\qbezier(68,16.7)(66,16.9)(64,17)
\qbezier(70,16.5)(73,16.3)(74,16.2)

\qbezier(11,12)(13,10)(14.2,14)
\qbezier(14.2,14)(15.2,16)(16.6,14.7)
\qbezier(17,13.5)(19.3,10.5)(19.8,15)
\qbezier(11,14)(11.5, 17)(13,18)
\qbezier(19.8,15)(19.9,15.5)(19.5, 16.5)
\qbezier(13,18)(16,20)(18,18.5)
\qbezier(18,18.5)(19,17.5)(19.5,16.5)

\qbezier(32.5,16)(34.5,13)(35.7,18)
\qbezier(35.7,18)(37,19.2)(38.1,17.8)
\qbezier(38.5,16.5)(41,13.5)(41.2,18.5)
\qbezier(32.5,17.5)(33, 20.5)(34.5,21.5)
\qbezier(41.2,18.5)(41.3,19)(40.9, 20)
\qbezier(34.5,21.5)(37.5,23.5)(39.5,22)
\qbezier(39.5,22)(40.5,21)(40.9,20)

\qbezier(93.5,12.6)(91,10.6)(90.3,14.8)
\qbezier(90.3,14.8)(89.3,16.8)(87.9,15.3)
\qbezier(87.5,14.1)(85.2,11.1)(84.7,15.6)
\qbezier(94,14)(95,17)(93,18.6)
\qbezier(84.7,15.6)(84.8,16)(86,18)
\qbezier(86,18)(90,22)(93,18.6)

%% NEW MATERIAL 13 April 2024

\qbezier(72,15.5)(69.5,13)(68.5,17.5)
\qbezier(68.5, 17.5)(67.7,19)(66.3,17.5)
\qbezier(66,16)(63.5,13.5)(62.5,18)
\qbezier(62.5,18)(62.2,19)(62.5,20)
\qbezier(62.5,20)(63,21)(64,22)
\qbezier(64,22)(65,23)(66,23)
\qbezier(66,23)(67,23.5)(68,23)
\qbezier(68,23)(69,23)(70,22.5)
\qbezier(70,22.5)(71,22)(71.5,21)
\qbezier(71.5,21)(72.2,20)(72.3,19)
\qbezier(72.3,19)(72.5,18)(72.4,17)

\qbezier(42.3,17.5)(52,18)(61.2,17.2)

\put(17,24.5){($a$ copies)}
\put(72,24.5){($b$ copies)}
\put(62,-10){($g$ copies)}
\put(65,-3.1){.}
\put(68.5,-2.9){.}
\put(72,-2.5){.}
\put(24.5,15.3){.}
\put(26.5,15.6){.}
\put(28.6,15.9){.}
\put(75,15.9){.}
\put(77,15.7){.}
\put(79,15.5){.}

\put(4,17){$L_1$}
\put(90,-10){$L_c$}
\put(25,1){$L_0$}

\put(21.7,-20){Figure 6.2.  A link giving $M(-c;(1,e))$.}

\end{picture}
 
If we move beyond the class of $S^1$-bundle spaces, 
we may give an example of ``intermediate" behaviour.
It is not hard to show that if $H\cong\mathbb{Z}^\beta$ with $\beta\leq5$
then for every $\mu:\wedge^3H\to\mathbb{Z}$ there is an epimorphism 
$\lambda:H\to\mathbb{Z}$ such that $\mu$ is 0 on the image of 
$\wedge^3\mathrm{Ker}(\lambda)$.
Hence there are splittings $H\cong{A}\oplus{B}$ with $A$ of rank 3 or 4
such that $\mu$ restricts to 0 on each of $\wedge^3A$ and $\wedge^3B$.
However if $\beta =6$ this fails for 
\[
\mu=
e_1\wedge{e_2}\wedge{e_3}+e_1\wedge{e_5}\wedge{e_6}+e_2\wedge{e_4}\wedge{e_5}.
\]
(Here $\{e_i\}$ is the basis for $Hom(H,\mathbb{Z})$ 
which is Kronecker dual to the standard basis 
of $H\cong\mathbb{Z}^6$.)
For every epimorphism $\lambda:\mathbb{Z}^6\to\mathbb{Z}$ 
there is a rank 3 direct summand $A$ of $\mathrm{Ker}(\lambda)$ 
such that $\mu$ is non-trivial on $\wedge^3A$.
[This requires a little calculation.
Suppose that $\lambda=\Sigma\lambda_ie_i^*$.
If $\lambda_6\not=0$ then we may take $A$ to be the direct summand 
containing $\langle{f_1,f_2,f_3}\rangle$,
where $f_j=\lambda_6e_j-\lambda_je_6$, for $1\leq{j}\leq3$,
for then $\mu(f_1\wedge{f_2}\wedge{f_3})=\lambda_6^3\not=0$.
Similarly if $\lambda_3$ or $\lambda_4$ is nonzero.
If $\lambda_3=\lambda_4=\lambda_6=0$ but $\lambda_1\not=0$
then we may take $A$ to be the direct summand
containing $\langle{g_2,e_4,g_5}\rangle$,
where $g_2=\lambda_1e_2-\lambda_2e_1$ and $g_5=\lambda_1e_5-\lambda_5e_1$.
Similarly if $\lambda_2$ or $\lambda_5$ is nonzero.]

This example arose in a somewhat different context.
It is the cup product 3-form of the 3-manifold $M$ given by 0-framed
surgery on the 6-component link of  \cite[Figure 6.1]{DH17}.
This link has certain ``Brunnian" properties.
All the 2-component sublinks, all but three of the 3-component sublinks
and six of the 4-component sublinks are trivial.
Thus $M$ has embeddings in $S^4$ with $\chi(X)=-1$ or 1,
corresponding to partitions of $L$ into a pair of trivial sublinks,
but there are no embeddings with $\chi(X)=-5$ or $-3$,
since $\mu_M$ does not satisfy the second assertion of Lemma 2.

\section{Homotopy types of pairs}

We shall next give some lemmas on recognizing 
the homotopy types of certain spaces and pairs of spaces arising later.

\begin{theorem}
\label{2con}
Let $U$ and $V$ be connected finite cell complexes such that $c.d.U\leq2$ and $c.d.V\leq2$.
If $f:U\to{V}$ is a $2$-connected map then $\chi(U)\geq\chi(V)$,
with equality if and only if $f$ is a homotopy equivalence.
\end{theorem}

\begin{proof}
Up to homotopy, we may assume that $f$ is a cellular inclusion,
and that $V$ has dimension $\leq3$.
Let $\pi=\pi_1U$ and let $C_*=C_*(\widetilde{V},\widetilde{U})$. 
Then $H_q(C_*)=0$ if $q\leq2$, since $f$ is 2-connected,
and $H_q(C_*)=0$ if $q>3$, since $c.d.U$ and $c.d.V\leq2$.
Hence $H_3(C_*)\oplus{C_2}\oplus{C_0}\cong{C_3}\oplus{C_1}$,
by Schanuel's Lemma,
and so $H_3(C_*)$ is a stably free $\mathbb{Z}[\pi]$-module 
of rank $-\chi(C_*)=\chi(U)-\chi(V)$.
Hence $\chi(U)\geq\chi(V)$, with equality if and only if
$H_3(C_*)=0$, since group rings are weakly finite,
by a theorem of Kaplansky.
(See \cite{Ros84} for a proof of Kaplansky's result.)
The result follows from the long exact sequence of the pair
$(\widetilde{Y},\widetilde{X})$ and the theorems of Hurewicz and Whitehead.
\end{proof}

If $c.d.X\leq2$ then $C_*(\widetilde{X})$ is chain homotopy equivalent 
to a finite projective complex of length 2,
which is a partial resolution of the augmentation module $\mathbb{Z}$.
Chain homotopy classes of such partial resolutions 
are classified by $Ext^3_{\mathbb{Z}[\pi]}(\mathbb{Z},\Pi)=H^3(\pi;\Pi)$,
where $\Pi$ is the module of 2-cycles.

\begin{cor}
\label{2concor}
If $U$ is a connected finite complex such that $c.d.U\leq2$ and 
$\pi_1U\cong\mathbb{Z}$
then $U\simeq{S^1}\vee\bigvee^{\chi(U)}S^2$.
\end{cor}

\begin{proof}
Since $c.d.U\leq2$ and projective $\mathbb{Z}[\pi_1U]$-modules are free,
$C_*(\widetilde{U})$ is chain homotopy equivalent to 
a finite free $\mathbb{Z}[\pi_1U]$-complex $P_*$ of length $\leq2$,
and $\chi(U)=\Sigma(-1)^irank(P_i)$.
Since $\pi_2U\cong{H_2(U;\mathbb{Z}[\pi_1U])}$
is the module of 2-cycles in  $C_*(\widetilde{U})$, it is free of rank $\chi(U)$.
Let $f:{S^1\vee\bigvee^{\chi(U)}S^2\to{U}}$ be the map determined by a
generator for $\pi_1U$ and representatives of a basis for $\pi_2U$.
Then $f$ is a homotopy equivalence, by the theorem.
\end{proof}

Theorem 3.2 of \cite{FMGK} gives an analogue of Theorem \ref{2con} for maps between closed 4-manifolds.
The argument extends to the following relative version.

\begin{lemma}
\label{4man}
Let $f:(X_1,A_1)\to(X_2,A_2)$ be a map of orientable $PD_4$-pairs such that
$f|_{A_1}:A_1\to{A_2}$ is a homotopy equivalence. Then $f$ is a homotopy
equivalence of pairs if and only if $\pi_1f$ is an isomorphism and $\chi(X_1)=\chi(X_2)$.
\end{lemma}

\begin{proof}
Since $f|_{A_1}:A_1\to{A_2}$ is a homotopy equivalence,
$f$ has degree 1, and hence is 2-connected as a map from $X_1$ to $X_2$.
The rest  of the argument is as in \cite[Theorem 2]{FMGK}.
\end{proof} 

In certain cases we can identify the homotopy type of a pair.

\begin{lemma}
\label{asph}
Let $(X,A)$ and $(X',A')$ be pairs such that the inclusions 
${\iota_A:A\to{X}}$ and $\iota_{A'}:A'\to{X'}$ induce epimorphisms 
on fundamental groups.
If $X$ and $X'$ are aspherical and $f:A\to{A'}$ 
is a homotopy equivalence such that
$\pi_1f(\mathrm{Ker}(\pi_1\iota_A))=\mathrm{Ker}(\pi_1\iota_{A'})$
then $f$ extends to a homotopy equivalence of pairs $(X,A)\simeq(X',A')$.
\end{lemma}

\begin{proof}
The fundamental group conditions imply that $g=\iota_{A'}f$ 
extends to a map from the relative 2-skeleton $X^{[2]}\cup{A}$.
The further obstructions to extending $g$ to a map from $X$ to $X'$ lie in $H^{q+1}(X,A;\pi_q(X'))$,
for $q\geq2$. Since $X'$ is aspherical these groups are 0.
The other hypotheses imply that any extension $h:X\to{X'}$ induces 
an isomorphism on fundamental groups,
and hence is a homotopy equivalence.
\end{proof}

\section{Cohomological dimension and fundamental group} 

Since the complementary regions are 4-manifolds with non-empty boundary
they are homotopy equivalent to 3-dimensional complexes.
However, when such a space is homotopically 2-dimensional remains 
an open question, in general.
We shall say that $c.d.W\leq{n}$ if the equivariant chain complex 
of the universal cover $\widetilde{W}$ is chain homotopy equivalent 
to a complex of projective $\mathbb{Z}[\pi_1W]$-modules 
of length $\leq{n}$.

In the next theorem we refer to the Bass Conjectures \cite{Ba76},
which we outline very briefly.
If $P$ is a finitely generated projective $\mathbb{Z}[\pi]$-module then
it is the image of an idempotent $n\times{n}$-matrix $A$ with entries in $\mathbb{Z}[\pi]$,
for some $n\geq0$.
The Kaplansky rank $\kappa(P)$ is the coefficient of 1 in the trace of $A$. 
The weak Bass Conjecture is the assertion that $\kappa(P)$ equals the 
``naive" rank  $\dim_\mathbb{Q}\mathbb{Q}\otimes_{\mathbb{Z}[\pi]}P$.
An analytic argument originally due to Kaplansky shows
that $\kappa(P)>0$ if $P\not=0$.

\begin{theorem}
\label{Hi17-thm5.1}
Let $W$ be a complementary region of an embedding of $M$ in $S^4$.
Then $c.d.W\leq2$ if and only if $j_{W*}=\pi_1j_W$ is an epimorphism.
If so, then $W$ is aspherical if and only if $c.d.\pi_1W\leq2$ and $\chi(W)=\chi(\pi_1W)$.  
\end{theorem}

\begin{proof}
Let $\Gamma=\mathbb{Z}[\pi_1W]$.
Then there are Poincar\'e-Lefshetz duality isomorphisms
$H_i(W;\Gamma)\cong{H^{4-i}}(W,\partial{W};\Gamma)$
and $H^j(W;\Gamma)\cong{H_{4-j}}(W,\partial{W};\Gamma)$,
for all $i,j\leq4$.

If $c.d.W\leq2$ then $H_i(\widetilde{W},\partial\widetilde{W})=H_i(W,\partial{W};\Gamma)=0$ for $i\leq1$,
and so $\partial\widetilde{W}$ is connected.
Therefore $j_{W*}$ must be surjective.
Conversely, if $j_{W*}$ is an epimorphism then we may assume that 
$W$ may be obtained from $M$ (up to homotopy)
by adjoining cells of dimension $\geq2$.
Hence $H_i(W,\partial{W};\Gamma)$ 
and $H^j(W,\partial{W};\Gamma)$ are 0 for $i,j\leq1$.
Therefore $H_q(W;\Gamma)=H^q(W;\Gamma)=0$ for all $q>2$,
and so $C_*=C_*(W;\Gamma)$ is chain homotopy equivalent to a complex $P_*$
of finitely generated projective $\Gamma$-modules of length at most 2,
by Wall's finiteness criteria \cite{Wl66}.

If $W$ is aspherical then $c.d.\pi_1W\leq2$,
and we must have $\chi(W)=\chi(\pi_1W)$.
Conversely, if $j_{W*}$ is onto then $\Pi=H_2(P_*)\cong\pi_2W$ 
is the only obstruction to asphericity.
If, moreover, $c.d.\pi_1W\leq2$ we may apply Schanuel's Lemma, 
to see that $P_*$ splits as 
\[
P_*=\Pi\oplus(Z_1\to{P_1}\to{P_0}),
\]
where $\pi$ is concentrated in degree 2,
$Z_1$ is the submodule of 1-cycles and $Z_1\to{P_1}\to{P_0}$ 
is a resolution of the augmentation module $\mathbb{Z}=H_0(P_*)$.
Now $\mathbb{Z}\otimes_\Gamma\Pi\cong{H_2(W)}$ 
is a free abelian group of rank $\chi(W)-\chi(\pi_1W)$.
If, moreover, $\chi(W)=\chi(\pi_1W)$ then $\Pi=0$, and so $W$ is aspherical,
since the weak Bass Conjecture holds for groups of cohomological dimension $\leq2$ \cite{Ec}.
\end{proof}

\begin{cor}
The augmentation ideal of the group ring $\mathbb{Z}[\pi_X]$ 
has a square presentation matrix.
\end{cor}

\begin{proof}
If $c.d.X\leq2$ then $C_*(X;\mathbb{Z}[\pi])$ 
is chain homotopy equivalent to a finite free $\mathbb{Z}[\pi_X]$-complex
of length 2.
Hence the augmentation ideal of the group ring $\mathbb{Z}[\pi_X]$ 
has a square presentation matrix, since $\chi(X)\leq1$.
\end{proof}

The property that the augmentation ideal has a square presentation matrix
interpolates between $\pi_X$ having a balanced presentation
and being homologically balanced.
The stronger condition (having a balanced presentation) would hold if $X$ 
were homotopy equivalent to a finite 2-dimensional cell complex.

In our applications of Theorem \ref{Hi17-thm5.1} below,
$\pi_1W$ is either free, free abelian or the fundamental group 
of an aspherical surface. 
Hence all projective $\Gamma$-modules are stably free.
A stably free $\Gamma$-module $P$ is trivial if and only if
$\mathbb{Z}\otimes_\Gamma{P}=0$,  by an old result of Kaplansky
(see \cite{Ros84} for a proof), 
and we could use this instead of invoking \cite{Ec}.
A similar argument may be used to show that, in general,
$W$ is aspherical if and only if $c.d.\pi_1W\leq3$, 
$\pi_1W$ is of type $FF$, $\chi(W)=\chi(\pi_1W)$ and $\pi_2W=0$.

Let $K$ be the Artin spin of a non-trivial classical knot, 
and let $X=X(K)$ be the exterior of 
a tubular neighbourhood of $K$ in $S^4$.
Then $\pi_1X\cong\pi{K}$, the knot group, 
and $M=\partial{X}\cong{S^2}\times{S^1}$.
In this case $c.d.\pi{K}=2$ and $\chi(X)=\chi(\pi{K})=0$,
but $j_{X*}$ is not onto, and $X$ is not aspherical.
(Thus $c.d.X=3$.)

There are two essentially different partitions of the standard link 
representing $T_g\times{S^1}$ into moieties with $g+1$ and $g$ components.
For one, $X\cong{S^1}\times(\natural^g(D^2\times{S^1})$,
which is aspherical (as to be expected from Theorem 4);
for the other, $\pi_X\cong\mathbb{Z}^2*F(g-1)$, and $X$ is not aspherical.
(In neither case is $Y$ aspherical.)

\section{Aspherical embeddings}

Suppose that a complementary region $W$ is aspherical.
Then $c.d.\pi_W\leq3$, and $\pi_W$ cannot be a
$PD_3$-group since $H^3(W)=0$.
If the embedding is bi-epic then $c.d.\pi_W\leq2$, by Theorem \ref{Hi17-thm5.1}.
Conversely, if $W$ is aspherical and $c.d.\pi_W\leq2$ then $j_{W*}$ is an epimorphism, 
by equivariant Poincar\'e duality \cite{DH23}.
Thus if $X$ and $Y$ are each aspherical then
$c.d.\pi_X, c.d.\pi_Y\leq2\Leftrightarrow{j}$ {\it is bi-epic}.

Every homology sphere has an aspherical embedding, 
since it bounds a contractible 4-manifold.
Sums of aspherical embeddings are again aspherical.

If $W$ aspherical and $\pi_W$ is elementary amenable (but is not a $PD_3$-group)
then either $W\simeq*$ or $\pi_W\cong\mathbb{Z}$ or $BS(1,m)$ \cite{DH23}.
There are no other known (finitely presentable) restrained groups $G$ with $c.d.G=2$,
and certainly no others which are almost coherent 
and have infinite abelianization \cite[Theorem 2.6]{FMGK}.

If $G$ and $H$ are two groups of cohomological dimension 2 with
balanced presentations and isomorphic abelianizations then 
there is an embedding of a 3-manifold such that the complementary regions
$X$ and $Y$ are each homotopy equivalent to finite 2-complexes and
$\pi_X\cong{G}$ and $\pi_Y\cong{H}$ \cite{Li04}.  
Moreover $\chi(X)=\chi(\pi_X)=1$ and $\chi(Y)=\chi(\pi_Y)=1$.
If, moreover, $\beta_2(G)=\beta_1(G)$
and $\beta_2(H)=\beta_1(H)$ then $X$ and $Y$ are aspherical
\cite[Theorem 2.8]{FMGK}.

The simplest example of this type that we know of are based 
on the groups with presentations of the form
\[
\langle{a,b,c,d}\mid
~a^m=bab^{-1},~b^m=cbc^{-1},~c^m=dcd^{-1},~d^m=ada^{-1}\rangle.
\]
We may construct such groups by assembking copies of $BS(1,m)$ over
free subgroups.
(See \cite[Exercise 6.4.15]{Rob}.)
Hence they all have cohomological dimension 2.
In particular, $m=1$ gives the product $F(2)\times{F(2)}$,
while $m=2$ gives the Higman group $Hig$, with $H_1(Hig)=Hig^{ab}=1$.
Since the presentations are balanced, 
it follows that $H_2(Hig)=0$,
and so $Hig$ is superperfect.

Figure 6.3 is symmetric under quarter-turn rotations
around the axis through the central point.

\setlength{\unitlength}{1mm}
\begin{picture}(95,102)(-15.6,-12)

\put(40,40){$.$}
\put(10,79.1){$\vartriangleleft$}
\put(8,81){$x$}
\put(70,79.1){$\vartriangleleft$}
\put(68,81){$w$}
\put(10,-0.8){$\vartriangleright$}
\put(8,2){$y$}
\put(70,-0.8){$\vartriangleright$}
\put(68,2){$z$}

\put(4,40){$\bullet$}
\put(2,42){$b$}
\put(40.5,75){$\bullet$}
\put(38,77){$a$}
\put(40.5,3){$\bullet$}
\put(38,1){$c$}
\put(76.2,40){$\bullet$}
\put(78,38){$d$}

\put(40.5,40){\circle{10}}
\put(39.2,44){$\vartriangleleft$}
\put(46,35){$\mathbb{Z}/4\mathbb{Z}$}

\linethickness{1pt}
\put(4,80){\line(1,0){14}}
\put(1,63){\line(0,1){14}}
\put(4,60){\line(1,0){2}}
\put(8,60){\line(1,0){10}}
\qbezier(1,77)(1,80)(4,80)
\qbezier(1,63)(1,60)(4,60)
\qbezier(18,80)(21,80)(21,77)
\qbezier(18,60)(21,60)(21,63)

\put(21,64){\line(0,1){2}}
\put(21,67.1){\line(0,1){3.2}}
\put(21,71.5){\line(0,1){4}}
\put(21,7){\line(0,1){10}}
\put(21,3){\line(0,1){2}}

\put(5.5,20){\line(1,0){4}}
\put(10.7,20){\line(1,0){3.2}}
\put(15,20){\line(1,0){2}}

\put(61,14){\line(0,1){2}}
\put(61, 9.7){\line(0,1){3.2}}
\put(61,4.5){\line(0,1){4}}

\put(65,60){\line(1,0){2}}
\put(68.1,60){\line(1,0){3.2}}
\put(72.5,60){\line(1,0){4}}

\put(64,80){\line(1,0){14}}
\put(81,63){\line(0,1){14}}
\qbezier(78,80)(81,80)(81,77)
\qbezier(78,60)(81,60)(81,63)
\qbezier(61,77)(61,80)(64,80)

\put(4,0){\line(1,0){14}}
\put(1,3){\line(0,1){14}}
\qbezier(1,17)(1,20)(4,20)
\qbezier(1,3)(1,0)(4,0)
\qbezier(18,0)(21,0)(21,3)
\qbezier(18,20)(21,20)(21,17)

\put(64,20){\line(1,0){10}}
\put(76,20){\line(1,0){2}}
\put(81,3){\line(0,1){14}}
\put(64,0){\line(1,0){14}}
\qbezier(78,20)(81,20)(81,17)
\qbezier(78,0)(81,0)(81,3)
\qbezier(61,3)(61,0)(64,0)
\qbezier(61,17)(61,20)(64,20)

\put(61,75){\line(0,1){2}}
\qbezier(61,63)(61,60)(64,60)
\put(61,63){\line(0,1){10}}

\thinlines
\put(5,19){\line(0,1){40}}
\qbezier(5,19)(5,18)(6,18)
\qbezier(6,18)(7,18)(7,19)
\put(7,21){\line(0,1){16}}
\qbezier(7,37)(7,39)(9,39)
\qbezier(9,39)(11,39)(11,41)
\put(11,41){\line(0,1){9}}
\qbezier(11,50)(11,53)(14,53)
\put(14,53){\line(1,0){13}}
\qbezier(27,53)(30,53)(30,56)
\put(30,56){\line(0,1){11.5}}

\put(20,76){\line(1,0){40}}
\qbezier(20,76)(19,76)(19,75)
\qbezier(19,75)(19,74)(20,74)
\put(22,74){\line(1,0){16}}
\qbezier(38,74)(40,74)(40,72)
\qbezier(40,72)(40,70)(42,70)
\put(42,70){\line(1,0){9}}
\qbezier(51,70)(54,70)(54,67)
\put(54,54){\line(0,1){13}}
\qbezier(54,54)(54,51)(57,51)
\put(57,51){\line(1,0){11.5}}

\put(77,21){\line(0,1){40}}
\qbezier(77,61)(77,62)(76,62)
\qbezier(76,62)(75,62)(75,61)
\put(75,43){\line(0,1){16}}
\qbezier(75,43)(75,41)(73,41)
\qbezier(73,41)(71,41)(71,39)
\put(71,30){\line(0,1){9}}
\qbezier(71,30)(71,27)(68,27)
\put(55,27){\line(1,0){13}}
\qbezier(55,27)(52,27)(52,24)
\put(52,12.5){\line(0,1){11.5}}

\put(22,4){\line(1,0){40}}
\qbezier(62,4)(63,4)(63,5)
\qbezier(63,5)(63,6)(62,6)
\put(44,6){\line(1,0){16}}
\qbezier(44,6)(42,6)(42,8)
\qbezier(42,8)(42,10)(40,10)
\put(31,10){\line(1,0){9}}
\qbezier(31,10)(28,10)(28,13)
\put(28,13){\line(0,1){13}}
\qbezier(28,26)(28,29)(25,29)
\put(13.5,29){\line(1,0){11.5}}

\put(33,68){\line(1,0){17}}
\qbezier(50,68)(52,68)(52,66)
\put(52,52){\line(0,1){14}}
\qbezier(52,52)(52,49)(55,49)
\put(55,49){\line(1,0){14}}

\put(69,31){\line(0,1){17}}
\qbezier(69,31)(69,29)(67,29)
\put(53,29){\line(1,0){14}}
\qbezier(53,29)(50,29)(50,26)
\put(50,12){\line(0,1){14}}

\put(32,12){\line(1,0){17}}
\qbezier(32,12)(30,12)(30,14)
\put(30,14){\line(0,1){14}}
\qbezier(30,28)(30,31)(27,31)
\put(13,31){\line(1,0){14}}

\put(13,32){\line(0,1){17}}
\qbezier(13,49)(13,51)(15,51)
\put(15,51){\line(1,0){14}}
\qbezier(29,51)(32,51)(32,54)
\put(32,54){\line(0,1){14}}

\put(20,71){\line(1,0){9}}
\qbezier(20,71)(19.25,71)(19.25,70.25)
\qbezier(20,69.5)(19.25,69.5)(19.25,70.25)
\put(22,69.5){\line(1,0){6}}
\qbezier(28,69.5)(30,69.5)(30,67.5)

\put(72,52){\line(0,1){9}}
\qbezier(72,61)(72,61.75)(71.25,61.75)
\qbezier(70.5,61)(70.5,61.75)(71.25,61.75)
\put(70.5,53){\line(0,1){6}}
\qbezier(70.5,53)(70.5,51)(68.5,51)

\put(53,9){\line(1,0){9}}
\qbezier(62,9)(62.75,9)(62.75,9.75)
\qbezier(62,10.5)(62.75,10.5)(62.75,9.75)
\put(54,10.5){\line(1,0){6}}
\qbezier(54,10.5)(52,10.5)(52,12.5)

\put(10,19){\line(0,1){9}}
\qbezier(10,19)(10,18.25)(10.75,18.25)
\qbezier(11.5,19)(11.5,18.25)(10.75,18.25)
\put(11.5,21){\line(0,1){6}}
\qbezier(11.5,27)(11.5,29)(13.5,29)

\qbezier(10,28)(10,31)(13,31)
\qbezier(29,71)(32,71)(32,68)
\qbezier(72,52)(72,49)(69,49)
\qbezier(53,9)(50,9)(50,12)

\qbezier(5,61)(5,62)(6,62)
\qbezier(6,62)(7,62)(7,61)
\put(7,48){\line(0,1){13}}
\qbezier(7,48)(7,46)(9,46)
\put(9,46){\line(1,0){1}}

\qbezier(62,76)(63,76)(63,75)
\qbezier(63,75)(63,74)(62,74)
\put(49,74){\line(1,0){13}}
\qbezier(49,74)(47,74)(47,72)
\put(47,71){\line(0,1){1}}

\qbezier(77,19)(77,18)(76,18)
\qbezier(76,18)(75,18)(75,19)
\put(75,19){\line(0,1){13}}
\qbezier(75,32)(75,34)(73,34)
\put(72,34){\line(1,0){1}}

\qbezier(20,4)(19,4)(19,5)
\qbezier(19,5)(19,6)(20,6)
\put(20,6){\line(1,0){13}}
\qbezier(33,6)(35,6)(35,8)
\put(35,8){\line(0,1){1}}

\qbezier(13,19)(13, 18.25)(13.75,18.25)
\qbezier(13.75,18.25)(14.5,18.25)(14.5,19)
\put(14.5,19){\line(0,1){2}}
\qbezier(14.5,21)(14.5, 21.75)(15.25,21.75)
\qbezier(15.25,21.75)(16, 21.75)(16,21)

\qbezier(20,68)(19.25,68)(19.25,67.25)
\qbezier(19.25,67.25)(19.25,66.5)(20,66.5)
\put(20,66.5){\line(1,0){2}}
\qbezier(22,66.5)(22.75, 66.5)(22.75,65.75)
\qbezier(22.75,65.25)(22.75,65)(22,65)

\qbezier(69,61)(69,61.75)(68.25,61.75)
\qbezier(68.25,61.75)(67.5,61.75)(67.5,61)
\put(67.5,59){\line(0,1){2}}
\qbezier(67.5,59)(67.5,58.25)(66.75,58.25)
\qbezier(66.75,58.25)(66,58.25)(66,59)

\qbezier(62,12)(62.75,12)(62.75,12.75)
\qbezier(62.75,12.75)(62.75,13.5)(62,13.5)
\put(60,13.5){\line(1,0){2}}
\qbezier(60,13.5)(59.25,13.5)(59.25,14.25)
\qbezier(59.25,14.25)(59.25,15)(60,15)

\qbezier(16,19)(16,18.25)(16.75,18.25)
\qbezier(16.75,18.25)(17.5,18.25)(17.5,19)
\qbezier(20,65)(19.25,65)(19.25,64.25)
\qbezier(19.25,64.25)(19.25,63.5)(20,63.5)
\qbezier(66,61)(66,61.75)(65.25,61.75)
\qbezier(65.25,61.75)(64.5,61.75)(64.5,61)
\qbezier(62,15)(62.75,15)(62.75,15.75)
\qbezier(62.75,15.75)(62.75,16.5)(62,16.5)

\put(20,63.5){\line(1,0){8.5}}
\put(64.5,52.5){\line(0,1){8.5}}
\put(53.5,16.5){\line(1,0){8.5}}
\put(17.5,19){\line(0,1){8.5}}

\put(22,68){\line(1,0){6.5}}
\put(69,52.5){\line(0,1){6.5}}
\put(53.5,12){\line(1,0){6.5}}
\put(13,21){\line(0,1){6.5}}

\put(17.5,32){\line(0,1){10.5}}
\qbezier(17.5,42.5)(17.5, 46)(14,46)
\put(33,63.5){\line(1,0){10.5}}
\qbezier(43.5,63.5)(47,63.5)(47,67)
\put(64.5,37.5){\line(0,1){10.5}}
\qbezier(64.5,37.5)(64.5,34)(68,34)
\put(38.5,16.5){\line(1,0){10.5}}
\qbezier(38.5,16.5)(35,16.5)(35,13)

\put(17.1,-8){Figure 6.3. $X\cong{Y}\simeq{K(Hig,1)}$}

\end{picture}

Consider the embedding corresponding to the link in Figure 6.4, 
in which the strands in the box have $m$ full twists, 
and the central component represents the word $A=uvu^{-1}v^{-m}$ 
in the meridians $u,v$ for the other components.
When $m=0$ the link is the split union of an unknot and the Hopf link,
$M\cong{S^2}\times{S^1}$,
$X\cong{D^3}\times{S^1}$ and $Y\cong{S^2}\times{D^2}$.
When $m=1$ the link is the Borromean rings $Bo$,
and $X$ is a regular neighbourhood of the unknotted embedding of the torus
$T$ in $S^4$.
When $m=-1$, the link is $8^3_9$,
and $X$ is a regular neighbourhood of the unknotted embeding of the 
Klein bottle $Kb$ in $S^4$ with normal Euler number 0.
In general, $X$ is aspherical, $\pi_X\cong{BS(1,m)}$ and $\pi_Y\cong\mathbb{Z}/(m-1)\mathbb{Z}$.
(Note however that the boundary of a regular neighbourhood of the Fox 2-knot 
with group $BS(1,2)$ gives an embedding of $S^2\times{S^1}$ with
$\pi_X\cong{BS(1,2)}$ and $\chi(X)=0$, 
but this embedding is not bi-epic and $X$ is not aspherical.)

\setlength{\unitlength}{1mm}
\begin{picture}(90,45)(-26.2,8.4)

\put(-2,19.1){$\vartriangleright$}
\put(-4,22.1){$u$}
\put(63,19.1){$\vartriangleright$}
\put(61,22.1){$v$}
\put(29,46.2){$\vartriangleright$}
\put(27,48.2){$a$}
\put(48,33.4){$m$}
\put(17,33.3){$A=uvu^{-1}v^{-m}$}
\put(-15.9,34){$\bullet$}
\put(75.1,34){$\bullet$}

\thinlines
\put(-10,20){\line(1,0){18}}
\put(-15,25){\line(0,1){20}}
\put(-10,50){\line(1,0){15}}
\qbezier(-15,25)(-15,20)(-10,20)
\qbezier(-15,45)(-15,50)(-10,50)
\qbezier(5,50)(10,50)(10,45)
\qbezier(8,20)(10,20)(10,22)
\put(10,24){\line(0,1){18}}
\put(10,44){\line(0,1){1}}

\put(51,39){\line(0,1){7}}
\put(51,25){\line(0,1){4}}
\qbezier(51,25)(51,20)(56,20)
\put(56,20){\line(1,0){15}}
\qbezier(71,20)(76,20)(76,25)
\put(76,25){\line(0,1){20}}
\put(53,50){\line(1,0){18}}
\qbezier(51,48)(51,50)(53,50)
\qbezier(71,50)(76,50)(76,45)

\linethickness{1pt}
\qbezier(5,43.5)(5,47)(8.5,47)
\put(10,23){\line(1,0){33}}
\qbezier(5,28)(5,23)(10,23)
\put(5,28){\line(0,1){15.5}}

\put(9,43){\line(1,0){35}}
\qbezier(9,43)(8,43)(8,44)
\qbezier(8,44)(8,45)(9,45)

\put(11,47){\line(1,0){41}}
\qbezier(52,47)(53,47)(53,46)
\qbezier(52,45)(53,45)(53,46)
\put(11,45){\line(1,0){39}}

\put(48,28){\line(0,1){1}}
\qbezier(43,23)(48,23)(48,28)

\qbezier(44,43)(48,43)(48,39)
\qbezier(47,39)(49,39)(52,39)
\qbezier(47,29)(49,29)(52,29)
\qbezier(47,39)(47,34)(47,29)
\qbezier(52,39)(52,34)(52,29)

\put(9, 12){Figure 6.4.\quad $\pi_X\cong{BS(1,m)}$}

\end{picture}

If $W$ is aspherical and $\pi_W$ is a non-trivial amenable group then $\chi(W)=0$,
while if $\pi_W=F(r)$ for some $r>0$ then $\chi(W)\leq0$.
In either of these cases $W=X$,  by our convention on labelling the regions.
If the complementary region $Y$ is aspherical then $\chi(\pi_Y)>0$,
and so $\pi_Y$ is neither amenable nor free.

If $X$ is aspherical and $\pi_X\cong{F(r)}$ is free then $H_1(X)$ is torsion-free
and $H_2(X)=0$, and so $H_1(Y)=0$ and $H_2(Y)\cong\mathbb{Z}^r$.
This situation is realized by the standard embedding of $\#^r(S^2\times{S^1})$.
Are there examples with $Y$ also aspherical?

\section{Recognizing the simplest embeddings}

The simplest 3-manifolds to consider in the present context are perhaps 
the total spaces of $S^1$-bundles over orientable surfaces.
Most of those which embed have canonical ``simplest" embeddings.
We give some evidence that these may be characterized 
up to $s$-concordance by the conditions $\pi_X\cong\pi_1F$, 
where $F$ is the base, and $\pi_Y$ is abelian.

Suppose first that $M\cong{T_g}\times{S^1}$.
There is a canonical embedding $j_g:M\to{S^4}$, 
as the boundary of a regular neighbourhood of the standard 
smooth embedding $T_g\subset{S^3}\subset{S^4}$.
Let $X_g$ and $Y_g$ be the complementary components.
Then $X_g\cong{T_g}\times{D^2}$ and $Y_g\simeq{S^1}\vee\bigvee^{2g}{S^2}$,
and so $\pi_{Y_g}\cong\mathbb{Z}$.

We shall assume henceforth that $g\geq1$,
since embeddings of $S^2\times{S^1}$ and $S^3=M(0;(1,1))$ 
may be considered well understood.
Let $h$ be the image of the fibre in $\pi$.

\begin{lemma}
\label{Hi17-lem9.1}
Let $j:T_g\times{S^1}\to{S^4}$ be an embedding such that $\pi_X\cong\pi_1T_g$.
Then $X$ is $s$-cobordant rel $\partial$ to $X_g=T_g\times{D^2}$.
\end{lemma}

\begin{proof}
Since $H^2(X)\cong\mathbb{Z}$ is a direct summand of $H^2(M)$ 
and is generated by cup products of classes from $H^1(X)$,
the image of $j_{X*}$ cannot be a free group.
Therefore it has finite index, $d$ say, 
and so $\chi(\mathrm{Im}(j_{X*}))=d\chi(F)$.
Since $\mathrm{Im}(j_{X*})$ is an orientable surface group,
it requires at least $2-d\chi(F)=2(gd-d+1)$ generators.
On the other hand, $\pi$ needs just $2g+1$ generators.
Thus if $g>1$ we must have $d=1$, and so $j_{X*}$ is onto.
This is also clear if $g=1$, 
for then $\pi_X\cong{H_1(X)}$ 
is a direct summand of $H_1(M)$.
In all cases, we may apply Theorem \ref{Hi17-thm5.1} to conclude that $X$ is aspherical.

Any homeomorphism from $\partial{X}$ to $\partial{X_g}$ which 
preserves the product structure extends to a homotopy equivalence of pairs 
$(X,\partial{X})\simeq(X_g,\partial{X_g})$.
Now $L_5(\pi_1T_g)$ acts trivially on the $s$-cobordism 
structure set $\mathcal{S}_{TOP}^s(X_g,\partial{X_g})$,
by Theorem 6.7 and Lemma 6.9 of \cite{FMGK}.
Since $T_g$ has a metric of constant non-positive curvature the integral Novikov
conjecture holds for $\pi_X\cong\pi_1T_g$, 
and since $X$ is aspherical it follows that $\sigma_4(X,\partial{X})$ is an isomorphism.
(See \cite{KL}, especially \SS 9.2, 20.2 and 24.1.)
Therefore $X$ and $X_g$ are TOP $s$-cobordant (rel $\partial$).
\end{proof}

If $\pi_Y\cong\mathbb{Z}$ then $\Sigma=Y\cup(T_g\times{D^2})$ is 1-connected,
since $\pi_Y$ is generated by the image of $h$,
and $\chi(\Sigma)=2$.
Hence $\Sigma$ is a homotopy 4-sphere, 
containing a locally flat copy of $T_g$ with exterior $Y$.

\begin{lemma}
\label{Hi17-lem9.2}
If there is a map $f:Y\to{Y_g}$ which extends a homeomorphism 
of the boundaries then $Y$ is homeomorphic to $Y_g$.
\end{lemma} 

\begin{proof}
Let $t$ be a generator $t$ for $\pi_Y$.
Then $\mathbb{Z}[\pi_Y]\cong\Lambda=\mathbb{Z}[t,t^{-1}]$.
Let $\Pi=\pi_2Y$.
As in Theorem \ref{Hi17-thm5.1}, $H_q(Y;\Lambda)=H^q(Y;\Lambda)=0$ for $q>2$,
and the equivariant chain complex for $\widetilde{Y}$
is chain homotopy equivalent to a finite projective $\Lambda$-complex
\[
Q_*=\Pi\oplus(Z_1\to{Q_1}\to{Q_0})
\]
of length 2, with $Z_1\to{Q_1}\to{Q_0}$ a resolution of $\mathbb{Z}$.
The alternating sum of the ranks of the modules $Q_i$ is $\chi(Y)=2g$.
Hence $\Pi\cong\Lambda^{2g}$,
since projective $\Lambda$-modules are free. 
In particular, this holds also for $Y_g$.

If $f:Y\to{Y_g}$ restricts to a homeomorphism 
of the boundaries then $\pi_1f$ is an isomorphism.
Comparison of the long exact sequences of the pairs shows that 
$f$ induces an isomorphism $H_4(Y,\partial{Y})\cong
{H_4(Y,\partial{Y})}$, and so has degree 1.
Therefore $\pi_2f=H_2(f;\Lambda)$ is onto, by Poincar\'e-Lefshetz duality.
Since $\pi_2Y$ and $\pi_2Y_g$ are each free of rank $2g$,
it follows that $\pi_2f$ is an isomorphism,
and so $f$ is a homotopy equivalence, by the Whitehead and Hurewicz Theorems.

Thus $f$ is a homotopy equivalence {\it rel} $\partial$, by the HEP,
and so it determines an element of the structure set $\mathcal{S}_{TOP}(Y_g,\partial{Y_g})$.
The group $L_5(\mathbb{Z})$ acts trivially on the structure set, 
as in Lemma \ref{Hi17-lem9.1},
and so the normal invariant gives a bjection
$\mathcal{S}_{TOP}(Y_g,\partial{Y_g})\cong{H^2(Y_g,\partial{Y_g};\mathbb{F}_2)}
\cong{H_2(Y_g;\mathbb{F}_2)}$.
Since $H_2(\mathbb{Z};\mathbb{F}_2)=0$ the Hurewicz homomorphism maps $\pi_2Y_g$ onto $H_2(Y_g;\mathbb{F}_2)$.
Therefore there is an $\alpha\in\pi_2Y_g$ whose image in $H_2(Y_g;\mathbb{F}_2)$ 
is the Poincar\'e dual of the normal invariant of $f$.
Let $f_\alpha$ be the composite of the map from $Y_g$ to $Y_g\vee{S^4}$
which collapses the boundary of a 4-disc in the interior of $Y_g$ with 
$id_{Y_g}\vee\alpha\eta^2$, where $\eta^2$ is the generator of $\pi_4(S^2)$.
Then $f_\alpha$ is a self homotopy equivalence of $(Y_g,\partial{Y_g})$
whose normal invariant agrees with that of $f$ \cite{CH90}.
Therefore $f$ is homotopic to a homeomorphism $Y\cong{Y_g}$.
\end{proof}

However, finding such a map $f$ to begin with seems difficult.
Can we somehow use the fact that $Y$ and $Y_g$ are subsets of $S^4$?

Suppose now that $W$ is an $s$-cobordism {\it rel} $\partial$ 
from $X$ to $X_g$, and that $Y\cong{Y_g}$.
Since $g\geq1$ the 3-manifold $T_g\times{S^1}$ is irreducible 
and sufficiently large.
Therefore $\pi_0(Homeo(T_g\times{S^1}))\cong{Out}(\pi)$ \cite{Wd}.
If $g>1$ then $\pi_1T_g$ has trivial centre,
and so $Out(\pi)\cong
\left(\begin{smallmatrix}
Out(\pi_1T_g)&0\\ \mathbb{Z}^{2g}&\mathbb{Z}^\times
\end{smallmatrix}\right)$. 
It follows easily that every self homeomorphism of $T_g\times{S^1}$ 
extends to a self homeomorphism of $T_g\times{D^2}$.
Attaching $Y\times[0,1]\cong{Y_g\times[0,1]}$ to $W$ 
along $T_g\times{S^1}\times[0,1]$ gives an $s$-concordance from $j$ to $j_g$.

If $g=1$ then $X\cong{T\times{D^2}}$ and $Out(\pi)\cong{GL(3,\mathbb{Z})}$.
Automorphisms of $\pi$ are generated by those which may be realized 
by homeomorphisms of $T\times{D^2}$ together with those 
that may be realized by homeomorphisms of $Y_1$ \cite{Mo83}.
Thus if embeddings of $T$ with group $\mathbb{Z}$ are standard
so are embeddings of $S^1\times{S^1}\times{S^1}$ with both complementary components having abelian fundamental groups.

The situation is less clear for bundles over $T_g$ with Euler 
number $\pm1$.
We may construct embeddings of such manifolds
by fibre sum of an embedding of $T_g\times{S^1}$
with the Hopf bundle $\eta:S^3\to{S^2}$.
However, it is not clear how the complements change under this operation.
There are natural 0-framed links representing such bundle spaces.
Since the Whitehead link $Wh$ is an interchangeable 2-component link,
$M(1;(1,1))=M(Wh)$ has an embedding with $X\cong{Y}\simeq{S^1}\vee{S^2}$ 
and $\pi_X\cong\pi_Y\cong\mathbb{Z}$.
Is this embedding characterized by these conditions?
(Once again, it is enough to find a map which restricts 
to a homeomorphism on boundaries.)

The 3-manifold $\#^\beta(S^2\times{S^1})$ is the result of 
0-framed surgery on the $\beta$-component trivial link,
and so has embeddings realizing all the possibilities 
for Euler characteristics allowed by Lemma \ref{Hi17-lem2.1}.
In particular, it has an embedding with complementary regions
$X\cong\natural^\beta(D^3\times{S^1})$ and 
$Y\cong\natural^\beta(S^2\times{D^2})$. 
(In this case $\mu_M=0$.)

\begin{theorem}
Let $M=\#^\beta(S^2\times{S^1})$ and let $j:M\to{S^4}$ be an embedding such that $\chi(X)=1-\beta$ and $j_{X*}$ is an epimorphism.
Then $X\simeq\vee^\beta{S^1}$ and $Y\cong\natural^\beta(S^2\times{D^2})$.
\end{theorem}

\begin{proof}
Since $\pi\cong{F(\beta)}$ is a free group and $H_2(X)=0$, 
the  homomorphism from $\pi$ to $\pi_X/\cap_{n\geq1}\gamma_n\pi_X$ is a monomorphism, by Lemma  \ref{Hi17-lem4.1}
and the residual nilpotence of free groups \cite[6.10.1]{Rob}.
Hence  $j_{X*}$ is an isomorphism,  
so $\pi_Y=1$,  by Lemma \ref{VKsplitmono}, and $j$ is biepic.
Therefore $c.d.X\leq2$ and $c.d.Y\leq2$, by Theorem \ref{Hi17-thm5.1}.
Since $\chi(X)=1-\beta$ it follows easily that $X$ is aspherical,
and hence that $X\simeq\vee^\beta{S^1}$.

The inclusion of  $M=\#^\beta(S^2\times{S^1})$ into
$\natural^\beta(S^2\times{D^2})\simeq\vee^\beta{S^2}$ induces an
isomorphism $H^2(\vee^\beta{S^2})\cong{H^2(M)}$,
and we see easily that it extends to a map $g:(Y,M)\to
(\natural^\beta(S^2\times{D^2}),\#^\beta(S^2\times{S^1})$.
Since $Y$ is 1-connected, $c.d.Y\leq2$ and $H^2(g)$ is an isomorphism,
$g$ is a homotopy equivalence.
Moreover, $w_2(Y)=0$ and $Y$ has signature 0, since it is a subset of $S^4$.
As in Theorem \ref{aitch},
any such map is homotopic {\it rel} $ \partial$ to a homeomorphism, 
by 1-connected surgery, and so  $Y\cong\natural^\beta(S^2\times{D^2})$.
\end{proof}

The conclusion may be strengthened slightly, 
to show that $(X,M)$ is $s$-cobordant {\it rel} $\partial$ to 
$(\natural^\beta(D^3\times{S^1}),\#^\beta(S^2\times{S^1}))$.
(See \cite[Theorem 11.6A]{FQ}.)
Every self-homeomorphism of $\#^\beta(S^2\times{S^1})$ extends across
$\natural^\beta(D^3\times{S^1})$, 
and so $j$ is $s$-concordant to the standard embedding $j_{\beta{U}}$,
where $\beta{U}$ is the trivial $\beta$-component link. 

We remark that if $M=M_1\#{M_2}$ is a proper connected sum 
of 3-manifolds which embed in $S^4$ and one of the
summands has embeddings with differing values of $\chi(X)$
then so does $M$.

%% file: e7.tex
\chapter{Abelian embeddings}

We begin this chapter with some observations on restrained embeddings, 
and then narrow our focus to abelian embeddings.
Homology 3-spheres have essentially unique abelian embeddings
(although they may have other embeddings).
This is also known for $S^2\times{S^1}$ and $S^3/Q(8)$,
by results of Aitchison \cite{Ru80} and Lawson \cite{La84},
respectively.
In Theorems \ref{neutral} and \ref{homhandle} below we show that 
if $M$ is an orientable homology handle
(i.e., such that $H_1(M)\cong\mathbb{Z}$)
then it has an abelian embedding if and only if $\pi_1M$ 
has perfect commutator subgroup, 
and then the abelian embedding is essentially unique.
(There are homology handles which do not embed in $S^4$ at all!)
The 3-manifolds obtained by 0-framed surgery 
on 2-component links with unknotted components
always have abelian embeddings, 
and the complementary regions for such embeddings 
are homotopy equivalent to standard 2-complexes.
These shall be our main source of examples.
In particular,
we shall give an example in which $X\simeq{Y}\simeq{S^1\vee{S^2}}$,
but the pairs $(X,M)$ and $(Y,M)$ are not homotopy equivalent.
We do not yet have examples of a 3-manifold with several 
inequivalent abelian embeddings.

\section{Restrained embeddings}

With our present understanding, the application of surgery in dimension 4 
is limited to situations where the relevant fundamental group 
is in the class $SA$ \cite{FT95}.
In particular, all such groups are restrained.

\begin{theorem}
\label{biepic restrained}
Let $L$ be a $0$-framed bipartedly trivial link and let
$j:M\to{S^4}$ be the associated bi-epic embedding.
Suppose that $\pi_X$ and $\pi_Y$ are restrained.

If $\beta=\beta_1(M;\mathbb{Q})$ is odd then $\chi(X)=0$ and $\chi(Y)=2$,
and $X$ is aspherical.
If, moreover,  $\pi_X$ is almost coherent or elementary amenable then
$\pi_X\cong\mathbb{Z}$ or $BS(1,m)$, for some $m\not=0$,
and $\beta=1$ or $3$.

If $\beta$ is even then $\chi(X)=1$, and so $\pi_X$ has a balanced presentation, since $X$ is homotopy equivalent to a finite $2$-complex.
Similarly for $\pi_Y$.
\end{theorem}  

\begin{proof}
Since the embedding derives from a bipartedly trivial
link the complementary regions are each homotopy equivalent to finite 2-complexes.
Therefore $\chi(X),\chi(Y)\geq0$, since $\pi_X$ and $\pi_Y$ are each restrained
\cite{BP78}.
Hence $\chi(X)$ and $\chi(Y)$ are determined by the parity of $\beta$,
since $0\leq\chi(X)\leq\chi(Y)\leq2$ and $\chi(X)\equiv\chi(Y)\equiv1+\beta$ 
{\it mod} (2).

If $\chi(X)=0$ and $\pi_X$ is restrained then $X$ is aspherical,
by \cite[Theorem 2.5]{FMGK}.
If, moreover, $\pi_X$ is elementary amenable or almost coherent 
then $\pi_X\cong\mathbb{Z}$ or $BS(1,m)$ for some $m\not=0$, 
by \cite[Corollary 2.6.1]{FMGK}.
Hence $\beta=\beta_1(X;\mathbb{Q})+\beta_2(X;\mathbb{Q})=1$,
if $\pi_X\not\cong{BS(1,1)}=\mathbb{Z}^2$,
and $\beta=3$ if $\pi_X\cong\mathbb{Z}^2$.

If $\beta$ is even then $\chi(X)=1$, 
and so  $\pi_X$ has a balanced presentation.
Similarly for $\pi_Y$.
\end{proof}

Does the conclusion still hold for any bi-epic restrained embedding?

There are examples of each type. (See below).
There is also a partial converse. 
If $\chi(X)=0$ and $\pi_X\cong{BS(1,m)}$ for some $m\not=0$ then 
$X$ is aspherical and $j_{X*}$ is an epimorphism, by Theorem \ref{Hi17-thm5.1}.

The fact that the higher $L^2$ Betti numbers of amenable groups vanish
give a simple but useful consequence.
(See \cite{Lue} for details of the von Neumann algebra 
$\mathcal{N}(\pi_X)$ invoked below.)

\begin{lemma}
\label{L2}
If $c.d.X\leq2$ and $\beta^{(2)}_1(\pi_X)=0$ then 
either $\chi(X)=0$ and $X$ is aspherical or $\chi(X)>0$.
\end{lemma}

\begin{proof}
This follows from a mild extension of \cite[Theorem 2.5]{FMGK}.
Since $c.d.X\leq2$ and $X$ is homotopy equivalent to a finite 3-complex,
$C_*(\widetilde{X})$ is chain homotopy equivalent to 
a finite free $\mathbb{Z}[\pi_X]$-complex $D_*$ of length at most 2.
If $\beta_1^{(2)}(\pi_X)=0$ then $\chi(X)=\chi(D_*)=
dim_{\mathcal{N}(\pi_X)}
H_2(\mathcal{N}(\pi_X)\otimes_{\mathbb{Z}[\pi]}D_*)\geq0$,
with equality only if $D_*$ is acyclic,
in which case $X$ is aspherical.
\end{proof}

In particular, if $\pi_X$ is elementary amenable and $\chi(X)=0$
then $\pi_X\cong\mathbb{Z}$ or $BS(1,m)$, for some $m\not=0$.
(See \cite[Corollary 2.6.1]{FMGK}.)

\begin{ex}
If $M=M(-2;(1,0))$ or $M(-2;(1,4))$ and $j$ is bi-epic then $X\simeq{Kb}$.
\end{ex}
In each case $\pi$ is polycyclic and $\pi/\pi'\cong\mathbb{Z}\oplus(\mathbb{Z}/2\mathbb{Z})^2$.
Hence $\chi(X)=0$, and so $c.d.\pi_X\leq2$.
Since $\pi_X$ is a quotient of $\pi$ and $\pi_X/\pi_X'\cong\mathbb{Z}\oplus\mathbb{Z}/2\mathbb{Z}$
we must have $\pi_X\cong\mathbb{Z}\rtimes_{-1}\!\mathbb{Z}$.
Since $c.d.X\leq2$ and $\chi(X)=0$ the classifying map $c_X:X\to{Kb}=K(\mathbb{Z}\rtimes_{-1}\!\mathbb{Z},1)$
is a homotopy equivalence.

\section{Constraints on the invariants} 

In this section we shall show that manifolds with embeddings for which $\pi_X$ is abelian are severely constrained.

%In the statement of Theorem 7.1, the condition 
%``If $\pi_1(X)$ is abelian" should be imposed in the second sentence, 
%rather than in the final sentence. (??)

\begin{theorem}
\label{Hi17-thm7.1}
Suppose $M$ has an embedding in $S^4$ for which one complementary region $X$
has $\chi(X)\leq1$ and $\gamma_2\pi_X=\gamma_3\pi_X$.
Then either $\beta\leq4$ or $\beta=6$.
If $\beta=0$ or $2$ then $\pi_X\cong\mathbb{Z}/n\mathbb{Z}$
or $\mathbb{Z}\oplus\mathbb{Z}/n\mathbb{Z}$, respectively, for some $n\geq1$,
while if $\beta=1$, $3$, $4$ or $6$ then $\pi_X\cong\mathbb{Z}^\gamma$,
where $\gamma=\lfloor\frac{\beta+1}2\rfloor$. 
If $\pi_X$ is abelian and $\beta=1$ or $3$ then $X$ is aspherical.
\end{theorem}

\begin{proof}
Let $\gamma=\beta_1(X)$ and $A=H_1(X)$.
Then $2\gamma\geq\beta$ and $A\cong\mathbb{Z}^\gamma\oplus\tau_X$.
Since $A$ is abelian,  $H_2(A)=
A\wedge{A}\cong\mathbb{Z}^{\binom\gamma2}\oplus(\tau_X)^\gamma
\oplus(\tau_X\wedge\tau_X)$.

If $\gamma_2\pi_X=\gamma_3\pi_X$ then $H_2(A)$ 
is a quotient of $H_2(\pi_X)$,
by the 5-term exact sequence of low degree for $\pi_X$ as an extension of $A$. 
This in turn is a quotient of $H_2(X)\cong\mathbb{Z}^{\beta-\gamma}$,
by Hopf's Theorem.
Hence $\binom\gamma2\leq\beta-\gamma\leq\gamma$, 
and so $\gamma\leq3$.
If $\tau_X\not=0$ then either $\gamma=\beta=0$ and $\tau_X\wedge\tau_X=0$, 
or $\gamma=1$, $\beta=2$ and $\tau_X\wedge\tau_X=0$.
In either case, $\tau_X$ is (finite) cyclic.
If $\beta\not=0$ or 2 then $\tau_X=0$ and either $\gamma=\beta=1$,
or $\gamma=2$ and $\beta=3$ or 4, or $\gamma=3$ and $\beta=6$.
The final assertion follows immediately from Theorem \ref{Hi17-thm5.1}, 
since $cd\pi_X=\gamma\leq2$ and $\chi(X)$ must be 0.
\end{proof}

If $\pi_X$ is abelian, $\gamma=\beta=0$ and $\tau_M=0$ then $X$ is contractible.
In the remaining cases $X$ cannot be aspherical,
since either $\pi_X$ has non-trivial torsion (if $\beta=0$),
or $H_2(X)$ is too big (if $\beta=2$ or 4),
or $H_3(X)$ is too small (if $\beta=6$).

If we assume merely that $\gamma_2\pi_X/\gamma_3\pi_X$ is finite 
(i.e., that the rational lower central series stabilizes after one step) 
then $\cup_X:\wedge_2H^1(X;\mathbb{Q})\to{H^2(X;\mathbb{Q})}$ is injective
\cite{Su75}, 
and a similar calculation gives the same restrictions on $\beta$.

\begin{cor}
Let $M$ be a $3$-manifold with an abelian embedding in $S^4$.
Then
\begin{enumerate}
\item{if} $\beta=3$ then there is a $\mathbb{Z}$-homology isomorphism $h:M\to{T^3}$;
\item{if} $\beta=4$ or $6$ then $\mu_M\not=0$.
\end{enumerate}
\end{cor}

\begin{proof}
The homomorphism induced by cup product from $\wedge_2H^1(X)$ to $H^2(X)$
is injective, since $\pi_X$ is abelian \cite{Su75}.

If $\beta=3$ then $\pi_X\cong\mathbb{Z}^2$ and $X$ is aspherical, 
by the theorem.
Hence $X\simeq{T^2}$, and so $H^1(X)$ has a basis $a,b$ such that
$a\cup{b}$ generates a free summand of $H^2(M)$.
We may then choose $c\in{H^1(M)}$ so that $(a\cup{b}\cup{c})[M]=1$.
If we identify $H^1(M)$ with the homotopy set $[M;S^1]$
we see that $a,b,c$ together determine a map $h$ from $M$ to $T^3$ which is an isomorphism on $H^1(M)$ and on $H^3(M)$.
This is clearly a $\mathbb{Z}$-homology isomorphism.

If $\beta\geq3$ then $\beta_1(X)\geq2$, 
and so there are $a, b\in{H^1(X)}$ 
with $a\cup{b}\not=0$ in $H^2(M)$.
Hence there is a $c\in{H^1(M)}$ such that $a\cup{b}\cup{c}\not=0$,
by Poincar\'e duality for $M$.
\end{proof}

In particular, 
$\#^\beta(S^2\times{S^1})$ has an abelian embedding if and only if 
$\beta\leq2$.

Embeddings with $\pi_X$ abelian and $\beta\leq4$ realizing these possibilities 
may be easily found.
(If $\pi_X\not=1$ then 2-knot surgery gives further examples with 
$\pi_X$ nonabelian and $\gamma_2\pi_X=\gamma_3\pi_X$.)
The simplest examples are for $\beta=0,1$ or 3,
with $M\cong{S^3}$, $M=S^2\times{S^1}$ or 
$S^1\times{S^1}\times{S^1}$ 
the boundary of a regular neighbourhood of a point or of
the standard unknotted embedding of $S^2$ or $T$ in $S^4$,
respectively.

Other examples may be given in terms of representative links.
When $\beta=0$ the $(2,2n)$ torus link gives examples with $X\cong{Y}$ and 
$\pi_X\cong\mathbb{Z}/n\mathbb{Z}$.
When $\beta=1$ we may use any knot which bounds a slice disc $D\subset{D^4}$
such that $\pi_1(D^4\setminus{D})\cong\mathbb{Z}$, 
such as the unknot or the Kinoshita-Terasaka knot $11_{n42}$.
(All such knots have Alexander polynomial 1.
Conversely every Alexander polynomial 1 knot bounds a TOP locally flat
slice disc with group $\mathbb{Z}$, by a striking result of Freedman.)
The links $8^3_5$ and $8^3_6$ give further simple examples. 
(These each have a trivial 2-component sublink and an unknotted third
component which represents a meridian of the first component or the
product of meridians of the first two components, respectively.)
When $\beta=2$ any 2-component link with unknotted components 
and linking number 0, such as the trivial 2-component link or $Wh$,
gives examples with $\pi_X\cong\mathbb{Z}$.
We may construct examples realizing $\mathbb{Z}\oplus\mathbb{Z}/n\mathbb{Z}$ from 
the 4-component link obtained from the Borromean rings $Bo$ by 
replacing one component by its $(2,2n)$ cable.
When $\beta=3$ we may use the links $Bo$, $9^3_9$ or $9^3_{18}$.
(These each have a trivial 2-component sublink and an unknotted third
component which represents the commutator of the meridians
of the first two components. 
However neither of the latter two links is Brunnian.)
 
Let $L$ be the 4-component link obtained from  $Bo$
by adjoining a parallel to the third component,
and let $M$ be the 3-manifold $M$ obtained by 0-framed surgery on $L$.
Then the meridians of $L$ represent a basis $\{e_i\}$ for
$H_1(M)\cong\mathbb{Z}^4$, and 
$\mu_M=e_1\wedge{e_2}\wedge{e_3}+e_1\wedge{e_2}\wedge{e_4}$.
This link may be partitioned into the union of two trivial 2-component 
links in two essentially different ways,
and ambient surgery gives two essentially different embeddings of $M$.
If the sublinks are $\{L_1,L_2\}$ and $\{L_3,L_4\}$ then the complementary components 
have fundamental groups $\mathbb{Z}^2$ and $F(2)$.
Otherwise, 
the complementary components are homeomorphic 
and have fundamental group $\mathbb{Z}^2$.

The case $\beta=6$ follows from Theorem \ref{LQ}.
We may construct an explicit example as follows.
The  link diagram in Figure 7.1 is a projection of $Wh$.
The pre-images of this link in the 2- and 3-fold branched cyclic covers of $S^3$,
branched over the trivial knot represented in cross-section by $\bullet$,
are 4- and 6-component bipartedly trivial links in which each component represents the commutator of meridians of the adjacent component (as in $Bo$).
The corresponding 3-manifolds have $\beta=4$ and 6, respectively, 
and the associated embeddings are abelian.
These 3-manifolds are branched cyclic covers of $M(Wh)$,
but we are not aware of any other features which are of particular interest.

\setlength{\unitlength}{1mm}
\begin{picture}(95,65)(-21,5)

\linethickness{1pt}
\put(4,65){\line(1,0){14}}
\put(1,23){\line(0,1){39}}
\put(4,20){\line(1,0){14}}
\qbezier(1,62)(1,65)(4,65)
\qbezier(1,23)(1,20)(4,20)
\qbezier(18,65)(21,65)(21,62)
\qbezier(18,20)(21,20)(21,23)
\put(21,25){\line(0,1){18}}
\put(21,45){\line(0,1){5}}
\put(21,52){\line(0,1){7}}
\put(21,61){\line(0,1){1}}

\thinlines
\put(19,24){\line(1,0){45}}
\put(19,27){\line(1,0){1}}
\put(22,27){\line(1,0){28}}
\qbezier(19,24)(17.5,24)(17.5,25.5)
\qbezier(19,27)(17.5,27)(17.5,25.5)

\put(19,41){\line(1,0){1}}
\put(19,44){\line(1,0){18}}
\put(22,41){\line(1,0){15}}
\put(39,41){\line(1,0){2}}
\put(43,41){\line(1,0){7}}
\qbezier(19,41)(17.5,41)(17.5,42.5)
\qbezier(19,44)(17.5,44)(17.5,42.5)

\put(19,54){\line(1,0){1}}
\put(22,54){\line(1,0){16}}
\put(19,51){\line(1,0){16}}
\qbezier(19,51)(17.5,51)(17.5,52.5)
\qbezier(19,54)(17.5,54)(17.5,52.5)

\qbezier(37,54)(42,54)(42,49)

\put(18,60){\line(1,0){46}}
\qbezier(10,52)(10,60)(18,60)

\put(18,38){\line(1,0){2}}
\put(22,38){\line(1,0){15}}
\put(39,38){\line(1,0){2}}
\put(43,38){\line(1,0){7}}

\qbezier(10,46)(10,38)(18,38)
\put(10,46){\line(0,1){6}}

\qbezier(62,24)(72,24)(72,34)
\qbezier(62,60)(72,60)(72,50)
\put(72,34){\line(0,1){16}}

\put(39,44){\line(1,0){2}}
\put(43,44){\line(1,0){7}}
\put(38,36){\line(0,1){12}}
\qbezier(35,51)(38,51)(38,48)

\put(42,34){\line(0,1){15}}
\qbezier(42,34)(42,30)(46,30)

\put(48,33){\line(1,0){2}}
\qbezier(50,33)(52.5,33)(52.5,35.5)
\qbezier(50,38)(52.5,38)(52.5,35.5)

\qbezier (38,36)(38,33)(41,33)
\put(43,33){\line(1,0){5}}
\put(46,30){\line(1,0){4}}

\qbezier(50,30)(55.5,30)(55.5,35.5)
\qbezier(50,41)(55.5,41)(55.5,35.5)

\qbezier(50,44)(58.5,44)(58.5,35.5)
\qbezier(50,27)(58.5,27)(58.5,35.5)

\put(48,34.4){$\bullet$}

\put(-10,10){Figure 7.1.  A version of $Wh$ with branching axis through $\bullet$}

\end{picture}

In all of the above examples except for one $\pi_Y$ is also abelian. 
Note that Theorem \ref{Hi17-thm7.1} does {\it not\/} apply to $\pi_Y$,
as it uses the hypothesis $\beta_1(X)\geq\frac12\beta$.

\begin{lemma}
\label{finite}
If $\pi_X$ is abelian of rank at most $1$ 
then $X$ is homotopy equivalent to a finite $2$-complex.
If $\pi_X$ is cyclic then $X\simeq{S^1}$, 
$S^1\vee{S^2}$ or $P_\ell=S^1\vee_\ell{e^2}$, 
for some $\ell\not=0$.
\end{lemma}

\begin{proof}
The first assertion follows from the facts that $c.d.X\leq2$,
by Theorem \ref{Hi17-thm5.1}, 
and that the $\mathcal{D}(2)$ property holds for
cyclic groups (see \cite[page 235]{Jo}) 
and for the groups $\mathbb{Z}\oplus\mathbb{Z}/\ell\mathbb{Z}$ \cite{Ed}. 
If $\pi_X\cong\mathbb{Z}/\ell\mathbb{Z}$ is cyclic then $X\simeq{S^1}$ or $S^1\vee{S^2}$,
if $\ell=0$, or $P_\ell=S^1\vee_\ell{e^2}$, if $\ell\not=0$ \cite{DS}.
\end{proof}

We shall show later that a similar result holds when $\pi_X\cong\mathbb{Z}^2$.

Ten of the thirteen 3-manifolds with elementary amenable 
fundamental groups and which embed in $S^4$ have abelian embeddings.
In at least four cases ($S^3$, $S^3/Q(8)$, $S^3/I^*$ and $S^2\times{S^1}$)
the abelian embedding is essentially unique.
The  flat 3-manifold $M(-2;(1,0))$ 
and the $\mathbb{N}il^3$-manifold $M(-2;(1,4))$ 
bound regular neighbourhoods of embeddings 
of the Klein bottle $Kb$ in $S^4$,
but have no abelian embeddings.
The status of one $\mathbb{S}ol^3$-manifold is not yet known.

When $\beta\leq1$ we must have $\chi(X)=1-\beta$.
If $L$ is a 2-component slice link with unknotted components 
(such as the trivial 2-component link, or the Milnor boundary link)
and $M=M(L)$ then $\beta=2$ and $M$ has an abelian embedding 
(with $\chi(X)=\chi(Y)=1$), 
and also an embedding with $\chi(X)=1-\beta=-1$ and $\pi_Y=1$.
However it shall follow from the next lemma that if $\beta>2$ then $M$ 
cannot have both an abelian embedding and also one with $\chi(X)=1-\beta$.

\begin{lemma}
 \label{completion}
Let $j:M\to{S^4}$ be an embedding $j$ such that $H_1(Y)=0$,
and let $S\subset\Lambda_\beta=\mathbb{Z}[\pi/\pi']$ be the multiplicative
system consisting of all elements $s$ with augmentation $\varepsilon(s)=1$.
If the augmentation homomorphism $\varepsilon:\Lambda_{\beta{S}}\to\mathbb{Z}$
factors through an integral domain $R\not=\mathbb{Z}$ then $H_1(M;R)$ has rank $\beta-1$ as an $R$-module.
\end{lemma}

\begin{proof}
Let  $*$ be a basepoint for $M$ and $A(\pi)=H_1(M,*;\Lambda_\beta)$ 
be the Alexander module of $\pi$.
(See  \cite[Chapter 4]{AIL}.)
Since $H_2(X)=0$,
the inclusion of representatives for a basis of $H_1(X)\cong\mathbb{Z}^\beta$
induces isomorphisms $F(\beta)/\gamma_nF(\beta)\cong\pi/\gamma_n\pi$, 
for all $n\geq1$,
by Lemma \ref{Hi17-lem4.1}.
Hence $A(\pi)_S\cong(\Lambda_{\beta{S}})^\beta$,
by \cite[Lemma 4.9]{AIL}.
Since $\varepsilon$ factors through $R$, 
the exact sequence of the pair $(M,*)$
with coefficients $R$ gives an exact sequence
\[
0\to{H_1(M;R)}\to{R}\otimes_{\Lambda_\beta}{A(\pi)}\cong{R^\beta}\to{R}\to{R}\otimes_{\Lambda_\beta}\mathbb{Z}=
\mathbb{Z}\to0,
\]
from which the lemma follows.
(Note that the hypotheses on $R$ imply that $\mathbb{Z}$ is an $R$-torsion module.)
\end{proof}

\section{Homology spheres}

If $M$ is a homology 3-sphere then it bounds a contractible 4-manifold,
and so has an abelian embedding with $X$ and $Y$ each contractible. 
Moreover, the contractible complementary regions are determined up to homeomorphism by their boundaries, and so are homeomorphic,
by 1-connected surgery \cite{FQ}.
Thus the abelian embedding is unique.
When $M=S^3$, the result goes back to the Generalized Schoenflies Theorem,
which does not use surgery. 
In this special case the embedding is essentially unique!
%See \cite{Pu24} for a concise, self-contained account of the
%(quite different) proofs due to Brown and Mazur.

It is not clear whether non-simply connected homology spheres must have
embeddings with one or both of $\pi_X$ and $\pi_Y$ non-trivial. 
Figure 2.2 gives an example with $\pi_X\cong\pi_Y\cong{I^*}$,
the binary icosahedral group.
In this case the homology sphere is the result of surgery on a complicated
4-component bipartedly trivial link, and probably has no simpler description.
The Poincar\'e homology sphere $S^3/I^*$ is not the result 
of 0-framed surgery on any bipartedly slice link,
since it does not embed smoothly.

We mention here several results from \cite{Liv05} on smooth embeddings
of homology spheres.
There is a homology sphere which has at least two smooth embeddings; one with 
$X$ contractible  and one with $\pi_X\not=1$.
There is a superperfect group which is not the fundamental group of a homology 4-ball.
There is a contractible codimension-0 submanifold $W\subset{S^4}$ such that $def(\pi_W)<-n$, for every $n\geq0$.
Moreover $DW$ is diffeomorphic to $S^4$.

If $S^4=DW$ is the double of a 4-manifold $W$ with connected boundary $\partial{W}=M$ then $W$ is 1-connected, 
by the Van Kampen Theorem, and so contractible.
Thus $M$ must be a homology sphere.
Conversely, if $M$ is a homology sphere and $X$ and $Y$ are contractible
then there is a homeomorphism $h:X\cong{Y}$ such that $j_Y=h\circ{j_X}$,
and so $S^4\cong{DX}$.

In general one may construct examples with $Y\cong{X}$ and $\beta$ any positive even number.
Let $p:{S^3}\to{S^3}$ be a $2d$-fold branched cyclic covering, 
with branch set the first component of the Whitehead link.
The preimage of the second component is a $2d$-component link $L$ which
is the union of two trivial $d$-component links, 
and any generator of the covering group carries one sublink onto the other.
(When $d=1$ the preimage of the whole Whitehead link is the  3-component link $8^3_9$,
which is the union of a copy of the
the $(2,4)$-torus link $4^2_1$ and an unknot.)
Thus we obtain an example with $\beta=2d$ and $X\cong{Y}$.
In such cases $S^4$ is a twisted double $X\cup_\psi{X}$, for some self-homeomorphism of $M=\partial{X}$.
See \cite{Ya97} for other representations of $S^4$ as a twisted double.

If $j:M\to{S^4}$ is an embedding such that one of the complementary regions 
is 1-connected then $H_2(X)=0$, 
and so $\pi/\gamma_n\pi\cong{F(\beta)/\gamma_nF(\beta)}$, 
for all $n$, by Lemma 2.4.
The latter condition holds if $M=M(L)$ for some slice link,
and clearly the emmbedding deriving from a partition of 
such a link $L$ as $L_+\cup{L_-}$ with $L_-$ empty has $\pi_Y=1$.

\section{Homology  handles}

If $M$ is a homology handle then there is a
$\mathbb{Z}$-homology isomorphism $f:M\to{S^2\times{S^1}}$,
by Lemma\ref{homhandle}, 
and $\pi'/\pi''$ is a finitely generated torsion module over 
$\mathbb{Z}[\pi/\pi']\cong\Lambda=\mathbb{Z}[t,t^{-1}]$.

\begin{theorem}
\label{neutral}
Let $M$ be  an orientable homology handle.
If $M$ embeds in $S^4$ then the Blanchfield pairing on 
$\pi'/\pi''=H_1(M;\mathbb{Z}[\pi/\pi'])$ is neutral.
There is an abelian embedding $j:M\to{S^4}$
if and only if $\pi'$ is perfect,
and then $X\simeq{S^1}$ and $Y\simeq{S^2}$.
\end{theorem}

\begin{proof}
The first assertion follows on applying equivariant Poincar\'e-Lefshetz 
duality to the infinite cyclic cover of the pair $(X,M)$.
(See the proof of  \cite[Theorem 2.4]{AIL}.)

If $j$ is abelian then $\pi_X\cong\mathbb{Z}$ and $\pi_Y=1$,
while $H_2(X)=0$  and $H_2(Y)\cong\mathbb{Z}$.
Since $c.d.X\leq2$ and $\pi_X\cong\mathbb{Z}$,
it follows that $\pi_2X=H_2(X;\mathbb{Z}[\pi_X])$ 
is a free $\mathbb{Z}[\pi_X]$-module of rank $\chi(X)=0$.
Hence $\pi_2X=0$, and so maps $f:S^1\to{X}$ and $g:S^2\to{Y}$ 
representing generators for $\pi_X$ and $\pi_2Y$ are homotopy equivalences.
Since $H_2(X,M;\mathbb{Z}[\pi_X])\cong\overline{H^2(X;\mathbb{Z}[\pi_X])}=0$, 
by equivariant Poincar\'e-Lefshetz duality,
$\pi'/\pi''=H_1(M;\mathbb{Z}[\pi_X])=0$, by the homology exact sequence 
for the infinite cyclic cover of the pair $(X,M)$.
Hence $\pi'$ is perfect.

Suppose, conversely, that $\pi'$ is perfect.
Then $M$ embeds in $S^4$,  
by Corollary \ref{solvmodel},
and examination of the proof shows that the embedding constructed 
in Theorem \ref{Hi96-thm1} is abelian.
\end{proof}

If a Seifert manifold $M$ is a  homology handle other than $S^2\times{S^1}$ 
then there is no $\Lambda$-homology isomorphism from $M$ to $S^2\times{S^1}$,
by Corollary \ref{seifert homhandle}.
It then follows from Theorem \ref{neutral} that if $M$ 
has an abelian embedding then $M\cong{S^2\times{S^1}}$.

If $K$ is an Alexander polynomial 1 knot then 
$M(K)$ has an abelian embedding, 
and if $K$ is a knot such that $M(K)$ embeds in $S^4$ then
$K$ is algebraically slice, by Theorem \ref{neutral}.
However if $K$ is a slice knot with non-trivial Alexander polynomial 
then $M(K)$ embeds in $S^4$ but no embedding is abelian.
There are obstructions beyond neutrality of the Blanchfield pairing
to slicing a knot, 
which probably also obstruct embeddings of homology handles.

\begin{theorem}
\label{homhandle}
Let $M$ be an orientable homology handle.
Then $M$ has at most one abelian embedding, up to equivalence.
\end{theorem}

\begin{proof}
Assume that $j$ and $j_1$ are abelian embeddings of $M$.
There is a homotopy equivalence of pairs 
$f:(X_1,M)\simeq(X,M)$ which extends $id_M$, 
by Lemma \ref{asph}.
The normal invariant map $\nu$ from $\mathcal{S}_{TOP}(X,\partial{X})$ to
$\mathcal{N}(X,\partial{X})$ is constant, 
since $H_2(X;\mathbb{F}_2)=0$, and 
the surgery obstruction group $L_5(\mathbb{Z})$ acts trivially on 
the structure set $\mathcal{S}_{TOP}(X,\partial{X})$.
(This follows from the Wall-Shaneson theorem and the existence of the
$E_8$-manifold \cite[Theorem 6.7]{FMGK}.)
Hence $f$ is homotopic {\it rel} $M$ to a homeomorphism $F$.

The other complementary components $Y$ and $Y_1$ are homotopy equivalent to $S^2$.
Since they are codimension-0 submanifolds of $S^4$ 
the intersection pairings on $H_2(Y)$ and $H_2(Y_1)$ are trivial, 
and $w_2(Y_1)=w_2(Y)=0$.
It then follows from \cite[Proposition 0.5]{Bo86} that there is a homeomorphism
$G:Y_1\cong{Y}$ such that $G\circ{j_{Y_1}}=j_Y$.
The map $h=F\cup{G}$ is a homeomorphism of $S^4$ such that $hj_1=j_2$. 
\end{proof}

An alternative approach to identifying $Y$ would be to use the fact that $w_2(Y)=0$ and 
the pinch construction based on the generator of $\pi_4(S^2)$ to show that there is a self-homotopy equivalence {\it rel} 
$\partial$ of the pair $(Y,M)$ with non-trivial normal invariant, as in \cite[\S3]{La84}.

\begin{ex}
The manifold $M=M(11_{n42})$ has an essentially unique abelian embedding,
although $M=M(K)$ for infinitely many distinct knots $K$.
\end{ex}

The knot $11_{n42}$ is the Kinoshita-Terasaka knot,
which is the simplest non-trivial knot with Alexander polynomial 1.
This bounds a smoothly embedded disc $D$ in $D^4$,
such that $\pi_1(D^4\setminus{D})\cong\mathbb{Z}$,
obtained by desingularizing a ribbon disc.
(See Figure 1.4 of \cite{AIL}.)
Hence $M$ has a smooth abelian embedding.
Since $11_{n42}$ has unkotting number 1, 
it has an annulus presentation, 
and so there are infinitely many knots $K_n$ such that $M(K_n)\cong{M}$ \cite{AJOT}. 
These knots must all have Alexander polynomial 1,
and so each determines an abelian embedding.
Are all of these embeddings smoothable,
and if so, are they smoothly equivalent?

\begin{theorem}
Let $K$ be a $1$-knot. 
Then $M=M(K)$ has a bi-epic embedding if and only if $K$ is homotopically ribbon.
\end{theorem}

\begin{proof}
If $K$ is homotopically ribbon then the embedding corresponding to
the slice disc demonstrating this property is clearly bi-epic.

Suppose that $M$ has a bi-epic embedding.
Let $W$ be the trace of 0-framed surgery on $K$.
Then $W$ is 1-connected, $\chi(W)=1$ and $\partial{W}=S^3\amalg{M}$.
Let $P=X\cup_MW$. 
Then $P$ is 1-connected, since $\pi(j_X)$ is an epimorphism,
$\chi(P)=1$, and $\partial{P}=S^3$, and so $P\cong{D^4}$.
Clearly $K$ is homotopy ribbon in $P$.
\end{proof}

If $K$ is a fibred homotopically ribbon knot then the monodromy 
for the fibration extends over a handlebody \cite{CG}.
Hence $M$ bounds a mapping torus $X$ such that 
the inclusion $M\subset{X}$ induces an epimorphism from $\pi$ to $\pi_1X$
and an isomorphism on the abelianizations.
Let $Y$ be the 4-manifold obtained by adjoining a 2-handle to $D^4$ along $K$.
Then $\Sigma=X\cup_MY$ is a homotopy 4-sphere,
and the inclusion of $M$ into $\Sigma$ is bi-epic.

For example, 
if $k$ is a fibred 1-knot with exterior $E(k)$ and genus $g$, 
then $K=k\#\!-\!k$ is a fibred ribbon knot,
and $M(K)$ bounds a thickening $X$ of $E(k)\subset{S^3}\subset{S^4}$,
which fibres over $S^1$, with fibre $\natural^g(S^1\times{D^2})$.

In the next theorem we do not assume that $M$ is a homology handle.

\begin{theorem}
\label{mtor}
Let $M$ be a $3$-manifold and  $j:M\to{S^4}$ an embedding.
If $X$ fibres over $S^1$ and $j_{X*}$ is an epimorphism 
then $M$ is a mapping torus, 
the projection $p:M\to{S^1}$ extends to a map from $X$ to $S^1$,
$\chi(X)=0$, $X$ is aspherical and $\pi_X\cong{F(r)\rtimes\mathbb{Z}}$ 
for some $r\geq0$.
Conversely, if these conditions hold and the integral Novikov conjecture holds for $\pi_X$
then there is an embedding $\widehat{j}:M\to{S^4}$ which is  $s$-concordant to $j$ 
and such that $\widehat{X}$ fibres over $S^1$.
\end{theorem}

\begin{proof}
If $X$ fibres over $S^1$, with fibre $F$, 
then $M=\partial{X}$ is the mapping torus
of a self-homeomorphism of $\partial{F}$
 and the projection $p:M\to{S^1}$ extends to a map from $X$ to $S^1$.
Moreover, $\chi(X)=0$ and $\pi_X$ is an extension of $\mathbb{Z}$
by the finitely presentable normal subgroup $\pi_1F$.
Hence $\beta_1^{(2)}(\pi_X)=0$,
by  \cite[Theorem 7.2.6]{Lue}.
If $j_{X*}$ is surjective then $c.d.X\leq2$,  and conversely,
by Theorem \ref{Hi17-thm5.1}.
Hence $X$ is aspherical, by Lemma \ref{L2}.
Conversely, if $X$ is aspherical and $c.d.X\leq2$
then $H_1(X,M:\mathbb{Z}[\pi_X])\cong{H^3(X;\mathbb{Z}[\pi_X])}=0$,
by Poincar\'e duality with coefficients $\mathbb{Z}[\pi_X]$.
Hence the preimage of $M$ in $\widetilde{X}$ is connected, 
and so $j_{X*}$ is surjective,
Hence $\pi_1F$ is free, by  \cite[Corollary 6.6]{Bie},
and so $\pi_X\cong{F(r)\rtimes\mathbb{Z}}$ for some $r\geq0$.

Suppose now that the conditions in the second sentence of the enunciation hold.
Let $X^\infty$ be the covering space associated to the subgroup $F(r)$, 
and  let $j_{X^\infty}$ be the inclusion of $M^\infty=\partial{X^\infty}$ into $X^\infty$.
Let $\tau$ be a generator of the covering group $\mathbb{Z}$.
Since $X$ is aspherical and $\pi_1X^\infty\cong{F(r)}$
there is a homotopy equivalence $h:X^\infty\to{N}=\natural^r(S^1\times{D^2})$.
Then there is a self-homeomorphism $t_N$ of $N$ such that $t_Nh\sim{h}\tau$.
Let $\theta:\partial{N}\to{N}$ be the inclusion,
and let $\widehat{X}=M(t_N)$ be the mapping torus of $t_N$.
Then there is a homotopy equivalence $\alpha:M^\infty\to\partial{N}$ 
such that $\theta\alpha\sim{hj_{X^\infty}}$,
by a result of Stallings and Zieschang. (See \cite[Theorem 2]{GK}.)
We may modify $h$ on a collar neighbourhood of $\partial{X^\infty}$ so that 
$h|_{\partial{X^\infty}}=\alpha$.
Hence $h$ determines a homotopy equivalence of pairs 
$(X,M)\simeq(\widehat{X},\partial\widehat{X})$.
Since $M$ and $\partial\widehat{X}$ are orientable (Haken) manifolds
we may further arrange that $h|_M:M\to\partial\widehat{X}$ is a homeomorphism.
The group $L_5(\mathbb{Z}[\pi_X])$ acts trivially on 
the $s$-cobordism structure set 
$\mathcal{S}^s_{TOP}(\widehat{X},\partial\widehat{X})$,
by \cite[Theorem 6.7]{FMGK}, and $\sigma_4(X,M)$ is an isomorphism
(as in Lemma \ref{Hi17-lem9.1}), 
since $X$ is aspherical and the integral Novikov conjecture holds for $\pi_X$.
Hence $X$ and $\widehat{X}$ are $s$-cobordant {\it rel} $\partial$,
by hypothesis.
The union $\Sigma=\widehat{X}\cup_MY$ is an homotopy 4-sphere, 
and so is homeomorphic to $S^4$.
The composite $\widehat{j}:M\subset\widehat{X}\subset\Sigma\cong{S^4}$
is clearly $s$-concordant to $j$ and $X(\widehat{j})=\widehat{X}$ fibres over $S^1$.
\end{proof}

In particular, if $\beta=1$ then $\chi(X)=0$ and $M$ is a rational homology handle.
In this case $\widehat{j}$ is equivalent to $j$, 
since the $s$-cobordism theorem holds over $\mathbb{Z}$.

\section{$\pi/\pi'\cong\mathbb{Z}^2$}

When  $\pi/\pi'\cong\mathbb{Z}^2$ there is again a simple necessary 
condition for $M$ to have an abelian embedding.

\begin{theorem}
\label{Z^2}
Let $M$ be a $3$-manifold with fundamental group $\pi$ 
such that $\pi/\pi'\cong\mathbb{Z}^2$.
If $j:M\to{S^4}$ is an abelian embedding then 
$X\simeq{Y}\simeq{S^1\vee{S^2}}$, and
$H_1(M;\mathbb{Z}[\pi_X])$ and $H_1(M;\mathbb{Z}[\pi_Y])$ 
are cyclic $\mathbb{Z}[\pi_X]$- and $\mathbb{Z}[\pi_Y]$-modules
(respectively), of projective dimension $\leq1$.
\end{theorem}

\begin{proof}
Since $j$ is abelian $\pi_X\cong\pi_Y\cong\mathbb{Z}$ and $\chi(X)=\chi(Y)=1$.
Moreover,
since $j_{X*}$ and $j_{Y*}$ are epimorphisms $c.d.X\leq2$ and $c.d.Y\leq2$,
by  Theorem \ref{Hi17-thm5.1}.
Hence $X\simeq{Y}\simeq{S^1\vee{S^2}}$, by Corollary \ref{2concor}.

We again consider the homology exact sequences
of the infinite cyclic covers of the pairs $(X,M)$ and $(Y,M)$,
in conjunction with equivariant Poincar\'e-Lefshetz duality.
Since $H_i(X;\mathbb{Z}[\pi_X])=0$  for $i\not=0$ or 2
and $H_2(X;\mathbb{Z}[\pi_X])\cong\mathbb{Z}[\pi_X]$,
we have $H_2(X,M;\mathbb{Z}[\pi_X])\cong
\overline{H^2(X;\mathbb{Z}[\pi_X])}\cong\mathbb{Z}[\pi_X]$ also.
Hence there is an exact sequence 
\[
0\to{H_2(M;\mathbb{Z}[\pi_X])}\to\mathbb{Z}[\pi_X]\to
\mathbb{Z}[\pi_X]\to{H_1(M;\mathbb{Z}[\pi_X])}\to0.
\]
Therefore either $H_1(M;\mathbb{Z}[\pi_X])\cong{H_2(M;\mathbb{Z}[\pi_X])}\cong\mathbb{Z}[\pi_X]$
or $H_2(M;\mathbb{Z}[\pi_X])$ is a cyclic torsion module with a short free resolution, 
and $H_2(M;\mathbb{Z}[\pi_X])=0$.
In either case $H_1(M;\mathbb{Z}[\pi_X])$ is a cyclic module of projective dimension $\leq1$.

A similar argument applies for the pair $(Y,M)$.
\end{proof}

To use Theorem \ref{Z^2} to show that some $M$ has no abelian embedding
we must consider all possible bases for $Hom(\pi,\mathbb{Z})$,
or, equivalently, for $\pi/\pi'$.

\begin{ex}
Let $L$ be the link obtained from the Whitehead link $Wh=5^2_1$ 
by tying a reef knot ($3_1\#\!-\!3_1$) 
in one component. Then no embedding of $M=M(L)$ is abelian.
\end{ex}

The link group $\pi{L}$ has the presentation
\[
\langle{a,b,c,r,s,t,u,v,w}\mid
{as^{-1}vsa^{-1}=w=brb^{-1}},~cac^{-1}=b,~rcr^{-1}=a,~wcw^{-1}=b,
\]
\[
rvr^{-1}=tut^{-1},~sts^{-1}=u,~usu^{-1}=t,~vsv^{-1}=r\rangle,
\]
and $\pi=\pi_1M\cong\pi{L}/\langle\langle\lambda_a,\lambda_r\rangle\rangle$,
where $\lambda_a=c^{-1}wr^{-1}a$ and $\lambda_r=vu^{-1}s^{-1}t^{-1}rsa^{-1}b$
are the longitudes of $L$.
Let $b=\beta{a}$, $c=\gamma{a}$ and $t=r\tau$. 
Then $w=\gamma{r}$ in $\pi$, and so  $\pi$ has the presentation
\[
\langle{a,\beta,\gamma,r,s,\tau,v}\mid
[r,a]=\gamma^{-1}\beta{r}\beta^{-1}r^{-1}=r\gamma^{-1}r^{-1},~
\gamma{a}\gamma^{-1}a^{-1}=\beta,~sr\tau{s}=r\tau{sr\tau},
\]
\[
as^{-1}vsa^{-1}=\gamma{r},~vs=rv,~
v=\tau{s}r\tau{s^{-1}}\tau^{-1}=\beta^{-1}s^{-1}\tau{s^2}r\tau{s^{-1}}\rangle.
\]
Now let $s=\sigma{r}$ and $v=\xi{r}$.
Then $\pi/\pi''$ has the metabelian presentation
\[
\langle{a,\beta,\gamma,r,\sigma,\tau,\xi}\mid
[r,a]=\gamma^{-1}\beta.{r}\beta^{-1}r^{-1}=r\gamma^{-1}r^{-1},~\gamma.{a}\gamma^{-1}a^{-1}=\beta,~
\]
\[
r^{-1}\sigma{r}.r\tau\sigma{r^{-1}}=\tau\sigma.{r^2}\tau{r^{-2}},
~ar^{-1}\sigma^{-1}\xi{ra^{-1}}.a\sigma{a^{-1}}=\gamma.r\gamma^{-1}r^{-1},~\xi={r}\xi\sigma^{-1}r^{-1},
\]
\[
\xi=\tau\sigma.{r^2}\tau{r^{-2}}.r\sigma^{-1}\tau^{-1}r^{-1}=
\beta^{-1}.r^{-1}\tau{r}.\sigma.{r^2}\tau{r^{-2}}.r\sigma^{-1}r^{-1},~[[\,,\,],[\,,\,]]=1\rangle,
\]
in which $\beta,\gamma,\sigma,\tau$ and $\xi$ represent elements of $\pi'$, 
which is the normal closure of the images of these generators.
The first relation expresses the commutator $[r,a]$ 
as a product of conjugates of these generators.
Using the third relation to eliminate $\beta$, 
we see that $\pi'/\pi''$ is generated as a module over 
$\mathbb{Z}[\pi/\pi']=\mathbb{Z}[a^\pm,r^\pm]$ by the images of 
$\gamma,\sigma,\tau$ and $\xi$, with the relations
\[
(1-r)[\gamma]=0,
\]
\[
(r^2-r+1)[\sigma]=r(r^2-r+1)[\tau]=0,
\]
\[
[\xi]=(1-r)[\sigma],
\]
and 
\[
2[\sigma]+2[\tau]=(a-1)[\gamma].
\]
If we extend coefficients to the rationals to simplify the analysis, 
we see that $P=H_1(M;\mathbb{Q}[\pi/\pi'])=\mathbb{Q}\otimes\pi'/\pi''$
is generated by $[\gamma]$ and $[\tau]$,
with the relations 
\[(1-r)[\gamma]=(r^2-r+1)[\tau]=0.
\]
Let $\{x,y\}$ be a basis for $\pi/\pi'$. Then $x=a^mr^n$ and $y=a^pr^q$,
where $|mq-np|=1$.
Let $\{x^*,y^*\}$ be the Kronecker dual basis for $Hom(\pi,\mathbb{Z})$,
and let $M_x$ and $M_y$ be the infinite cyclic covering spaces corresponding 
to $\mathrm{Ker}(x^*)$ and $\mathrm{Ker}(y^*)$, respectively.
Then $H_1(M_x;\mathbb{Q})\cong{(P/(y-1)P\oplus\langle{y}\rangle)/(x.y=y+[x,y])}$.
If this module is cyclic as a module over the PID $\mathbb{Q}[x,x^{-1}]$ then so is 
the submodule 
\[
P/(y-1)P\cong\mathbb{Q}[\pi/\pi']/(r^2-r+1,y-1)\oplus\mathbb{Q}[\pi/\pi']/(r-1,y-1).
\]
On substituting $y=a^pr^q$ we find that this is so if and only if $p=0$ and $q=\pm1$.
But then $x=a^{\pm1}$, and a similar calculation show that 
$H_1(M_y;\mathbb{Q})$ is not cyclic  as a $\mathbb{Q}[y,y^{-1}]$-module.
Thus no basis for $\pi/\pi'$ satisfies the criterion of Theorem \ref{Z^2}, 
and $M$ has no abelian embedding.

We shall assume henceforth that $M=M(L)$, 
where $L$ is a 2-component link with
components slice knots and linking number $\ell=0$.
Let $x$ and $y$ be the images of the meridians of $L$ in $\pi$, 
and let $D_x$ and $D_y$ be slice discs for the components of $L$, 
embedded on opposite sides of the equator $S^3\subset{S^4}$.
Then the complementary regions for the embedding $j_L$ determined by $L$ are
$X_L=(D^4\setminus{N(D_x)})\cup{D_y\times{D^2}}$ and
$Y_L=(D^4\setminus{N(D_y)})\cup{D_x\times{D^2}}$.
The kernels of the natural homomorphisms from $\pi$ to $\pi_{X_L}$ 
and $\pi_{Y_L}$ are the normal closures of $y$ and $x$, respectively.
If one of the components of $L$ is unknotted then 
the corresponding complementary region is a handlebody of the form $S^1\times{D^3}\cup{h^2}$.
Inverting the handle structure gives a handlebody structure
$M\times[0,1]\cup{h^2}\cup{h^3}\cup{h^4}$.

If the components of $L$ are unknotted then $j_L$ is abelian,
and $\pi_X\cong\pi_Y\cong\mathbb{Z}$.

If $L$ is interchangeable there is a self-homeomorphism of $M(L)$ 
which swaps the meridians.
Hence $X_L$ is homeomorphic to $Y_L$, and $S^4$ is a twisted double.

To find examples where the complementary regions are {\it not\/} 
homeomorphic we should start with a link $L$ which is not interchangeable.
The simplest condition that ensures that a link with unknotted components 
is not interchangeable is asymmetry of the Alexander polynomial,
and the smallest such link with linking number 0 is $8^2_{13}$.
Since $\pi=\pi_1M$ is a quotient of $\pi{L}$, 
there remains something to be checked.

\begin{ex}
The complementary regions of the embedding of  $M=M(8^2_{13})$
determined by  the link $L=8^2_{13}$ are not homeomorphic
(although they are homotopy equivalent).
\end{ex}

The link group $\pi{L}=\pi8^2_{13}$ has the presentation
\[
\langle{s,t,u,v,w,x,y,z}\mid{yv=wy}, ~zx=wz,~ty=zt,~uy=zu,~sv=us,~vs=xv,
\]
\[
wu=tw,~xs=tx\rangle
\]
and the longitudes are $u^{-1}t$ and $x^2z^{-1}ys^{-1}w^{-1}xv^{-1}$.
Hence $\pi=\pi_1M$ has the presentation
\[
\langle{s,t,v,w,x,y}\mid{yv=wy}, ~tyt^{-1}x=wtyt^{-1}\!,~x^2ty^{-1}t^{-1}ys^{-1}w^{-1}xv^{-1}=1,
\]
\[sv=ts,~vs=xv,~wt=tw,~xs=tx\rangle.
\]
Setting $s=x\alpha$, $t=x\beta$, $v=x\gamma$ and $w=x\delta$, we obtain the presentation
\[
\langle{\alpha,\beta,\gamma,\delta,x,y}\mid
{[x,y]=xy\gamma(xy)^{-1}}\!.x\delta{x^{-1}},~
\beta.{y}\beta{y^{-1}}=\delta.{x}\beta{x}^{-1}\!.xy\beta^{-1}(xy)^{-1}\!.[x,y],
\]
\[
x^2\beta{x^{-2}}\!.x^2y^{-1}\beta^{-1}yx^{-2}=\gamma\delta.x\alpha{x^{-1}}.xy^{-1}[x,y]^{-1}yx^{-1}~
\]
\[
\delta{x}\beta=\beta{x}\delta,~\alpha{x}\gamma=\beta{x}\alpha,~\gamma{x}\alpha=x\gamma,~x\alpha=\beta{x}
\rangle
\]
in which $\alpha,\beta,\gamma$ and $\delta$ represent elements of $\pi'$, 
which is the normal closure of the images of these generators.
The subquotient $\pi'/\pi''$ is generated as a module over 
$\mathbb{Z}[\pi/\pi']\cong\Lambda_2=\mathbb{Z}[x^\pm,y^\pm]$
by the images of $\gamma$ and $\delta$,
with the relations 
\[
(x+1)(y-1)(x-1)[\gamma]=xy[\gamma]-x[\delta],
\]
\[(x-1)^2[\gamma]=(x-1)[\delta],
\]
and
\[
(x^2-x+1)[\gamma]=0,
\]
since $[\alpha]=x^{-1}(x-1)[\gamma]$ and $[\beta]=(x-1)[\gamma]$.
Adding the first two relations and rearranging gives
\[
[\delta]=-((x^2-x+1)y+2-2x) [\gamma]=2(x-1)[\gamma].
\]
Hence $\pi'/\pi''\cong\Lambda_2/( x^2-x+1,3(x-1)^2)=\Lambda_2/(x^2-x+1,3)$.
As a module over the subring $\mathbb{Z}[x,x^{-1}]$,
this is infinitely generated, but as a module over 
$\mathbb{Z}[y,y^{-1}]$ it has two generators.
Therefore there is no automorphism of $\pi$ which  induces an isomorphism 
$\mathrm{Ker}(j_{X*})=\pi'\rtimes\langle{x}\rangle
\cong\mathrm{Ker}(j_{Y*})=\pi'\rtimes\langle{y}\rangle$.
Hence $(X,M)$ and $(Y,M)$ are not homotopy equivalent as pairs, 
although $X\simeq{Y}$.

Does $M$ have any other abelian embeddings 
with neither complementary component homeomorphic to $X$,
perhaps corresponding to distinct link presentations?
Is this 3-manifold homeomorphic to a 3-manifold $M(\tilde{L})$ 
via a homeomorphism which does not preserve the meridians?

There is just one 3-manifold with $\pi$ elementary amenable and
$\beta=2$ which embeds in $S^4$.
This is the $\mathbb{N}il^3$-manifold $M=M(Wh)$ and $\pi\cong{F(2)}/\gamma_3F(2)$.
The embedding $j_{Wh}$ is abelian, 
since the components of $Wh$ are unknotted, 
but the complementary regions are not aspherical and so
Lemma \ref{asph} does not apply.
All epimorphisms from $\pi$ to $\mathbb{Z}$ are equivalent 
under composition with automorphisms,
and each automorphism of $\pi$ is induced by a self-diffeomorphism of $M$, 
since $M$ is Haken.
Thus if $j:M\to{S^4}$ is another abelian embedding and we fix a homotopy equivalence
$h:X\to{X(j_{Wh})}$ then we may assume that $\pi_1(h\circ{j_X})=j_{X*}$.
Can we choose $h$ to be a homotopy equivalence of pairs, {\it rel} $M$?

We have a more general result, albeit rather weak.

\begin{theorem}
\label{Z^2rel}
Let $M$ be a $3$-manifold with fundamental group $\pi$
such that ${\pi/\pi'\cong\mathbb{Z}^2}$.
Then there are at most $4$ abelian embeddings of $M$ with given homotopy types for the pairs $(X,M)$ and $(Y,M)$.
\end{theorem}

\begin{proof}
Let $W$ be a $4$-manifold with connected boundary $M$ 
and such that $W\simeq{S^1}\vee{S^2}$.
Then $\mathbb{Z}[\pi_1W]\cong\Lambda$, 
and $L_5(\Lambda)$ acts trivially on $\mathcal{S}_{TOP}(W,M)$,
by \cite[Theorem 6.7]{FMGK}.
Since $H_2(W;\mathbb{F}_2)=\mathbb{Z}/2\mathbb{Z}$,
there are at most 2 possibilities for each complementary region
realizing the given homotopy types.
\end{proof}

\section{The higher rank cases}

Theorems \ref{neutral} and \ref{Z^2} have analogues when $\beta=3$, 4 or 6.

\begin{theorem}
\label{b=3}
Let $M$ be a $3$-manifold with fundamental group $\pi$ 
such that $\pi/\pi'\cong\mathbb{Z}^3$.
If $j:M\to{S^4}$ is an abelian embedding then $X\simeq{T}$ 
and ${Y}\simeq{S^1\vee2{S^2}}$,
while  $H_1(M;\mathbb{Z}[\pi_X])\cong\mathbb{Z}$ and 
$H_1(M;\mathbb{Z}[\pi_Y])$ is a torsion $\mathbb{Z}[\pi_Y]$-module
of projective dimension $1$ and which can be generated by two elements.
The component $X$ is determined up to homeomorphism by its boundary $M$
and the homomorphism $j_{X*}$.
\end{theorem}

\begin{proof}
The classifying map $c_X:X\to{K(\pi_X,1)}\simeq{T}$ is a homotopy equivalence,
by Theorem \ref{2con},  since $c.d.X=c.d.T=2$ and $\chi(X)=\chi(T)=0$.
The equivalence ${Y}\simeq{S^1\vee2{S^2}}$ follows from Corollary \ref{2concor},
since $\pi_Y\cong\mathbb{Z}$ and $\chi(Y)=2$.

Since $H_2(X;\mathbb{Z}[\pi_X])=0$,
the exact sequence of homology for the pair $(X,M)$ 
with coefficients $\mathbb{Z}[\pi_X]$ reduces to an isomorphism
$H_2(X,M;\mathbb{Z}[\pi_X])\cong{H_1(M;\mathbb{Z}[\pi_X])}$, 
and so $H_1(M;\mathbb{Z}[\pi_X])\cong
\overline{H^2(X;\mathbb{Z}[\pi_X])}\cong\mathbb{Z}$,
by Poincar\'e duality.

Similarly, there is an exact sequence
\[
0\to\mathbb{Z}\to{H_2(M;\mathbb{Z}[\pi_Y])}\to\mathbb{Z}[\pi_Y]^2\to\mathbb{Z}[\pi_Y]^2\to
{H_1(M;\mathbb{Z}[\pi_Y])}\to0,
\]
since $H_2(Y;\mathbb{Z}[\pi_Y])\cong\mathbb{Z}[\pi_Y]^2$ and 
$H_2(Y,M;\mathbb{Z}[\pi_Y])\cong\overline{H^2(Y;\mathbb{Z}[\pi_Y]) }$.
Let $A=\pi'/\pi''$, considered as a $\mathbb{Z}[\pi/\pi']$-module.
Then $A$ is finitely generated as a module, 
since $\mathbb{Z}[\pi/\pi']$ is a noetherian ring.
Let $\{x,y,z\}$ be a basis for $\pi/\pi'$ such that $j_{X*}(y)=0$ and 
$j_{Y*}(x)=j_{Y*}(z)=0$.
Then $H_1(M;\mathbb{Z}[\pi_X])\cong\mathbb{Z}$ is an extension of $\mathbb{Z}$ 
by $A/(y-1)A$,
and so $A=(y-1)A$.
Similarly, $H_1(M;\mathbb{Z}[\pi_Y])$ is an extension of $\mathbb{Z}^2$ by $A/(x-1,z-1)A$.
Together these observations imply that $H_1(M;\mathbb{Z}[\pi_Y])$ is a torsion $\mathbb{Z}[\pi_Y]$-module,
and so the fourth homomorphism in the above sequence is a monomorphism.
Thus $H_1(M;\mathbb{Z}[\pi_Y])$ is a torsion $\mathbb{Z}[\pi_Y]$-module
with projective dimension $\leq1$,
and is clearly generated by two elements.
(Note also that a torsion $\mathbb{Z}[\pi_Y]$-module of projective dimension 0 is 0.)

Since $X$ is aspherical, Lemma \ref{asph} applies, 
and so the homotopy type of the pair $(X,M)$ is determined 
by $M$ and the homomorphism $j_{X*}$. 
The final assertion follows (as in Lemma 6.10), 
since $L_5(\mathbb{Z}[\pi_X])$ acts trivially on $\mathcal{S}_{TOP}(X,M)$, by \cite[Theorem 6.7]{FMGK}, 
and the integral Novikov Conjecture holds for $\pi_X\cong\mathbb{Z}^2$.
\end{proof}

Since $L_5(\mathbb{Z}[\pi_Y])$ also acts trivially on 
$\mathcal{S}_{TOP}(Y,M)$, by \cite[Theorem 6.7]{FMGK}, 
and $H_2(Y;\mathbb{F}_2)\cong(\mathbb{Z}/2\mathbb{Z})^2$,
there are at most four possibilities for $Y$, 
given the homotopy type of $(Y,M)$.

The link $L=9^3_{21}$ has an unique partition as a bipartedly slice link,
and for the corresponding embedding $\pi_{X_L}\cong{F(2)}$
and $\pi_{Y_L}\cong\mathbb{Z}$.
Then $M=M(9^3_{21})\cong(S^2\times{S^1})\#M(5^2_1)$,
so $\pi\cong\mathbb{Z}*F(2)/\gamma_3F(2)$,
with presentation $\langle{x,y,z}\mid[x,y]\leftrightharpoons{x,y}\rangle$.
It is not hard to show that the kernel of any epimorphism
$\phi:\pi\to\langle{t}\rangle\cong\mathbb{Z}$ 
has rank $\geq1$ as a $\mathbb{Z}[t,t^{-1}]$-module.
Hence $M$ has no abelian embedding, by Theorem \ref{b=3}.

The 3-torus $T^3=\mathbb{R}^3/\mathbb{Z}^3$ has an abelian embedding,
as the boundary of a regular neighbourhood of an unknotted embedding of $T$ in $S^4$.
This manifold may be obtained by 0-framed surgery on the Borromean rings $Bo$,
and also  on $9^3_{18}$. 
The three bipartite partitions of $Bo$ lead to equivalent embeddings.
(However these are clearly not isotopic!)
The link $9^3_{18}$ has two bipartedly slice partitions (both bipartedly trivial).
Any such embedding of $T^3$ has 
$X\cong{T}\times{D^2}$ and $Y\simeq{S^1}\vee2S^2$.
Does $T^3$ have an essentially unique abelian embedding?

If $M$ is Seifert fibred and $\pi/\pi'\cong\mathbb{Z}^3$
then it has generalized Euler invariant $\varepsilon=0$, 
and so is a mapping torus $T_g\rtimes_\theta{S^1}$, 
with orientable base orbifold and monodromy $\theta$ of finite order.
Are there any such manifolds other than the 3-torus 
which have abelian embeddings?

Suppose that $\beta=3$ and $M$ has an embedding $j$ 
such that $H_1(Y;\mathbb{Z})=0$.
If $f:\pi\to\mathbb{Z}^2$ is an epimorphism  with kernel $\kappa$ 
and $R=\mathbb{Z}[\pi/\kappa]_{f(S)}$ then $H_1(M;R)$ has rank 2, 
by Lemma \ref{completion},
and so the condition of Theorem \ref{b=3} does not hold.
Therefore no such 3-manifold can also have an abelian embedding. 

\begin{theorem}
\label{b=4}
Let $M$ be a $3$-manifold with fundamental group $\pi$ 
such that $\pi/\pi'\cong\mathbb{Z}^4$.
If $j:M\to{S^4}$ is an abelian embedding
then $X\simeq{Y}\simeq{T\vee{S^2}}$.
Hence  $H_1(M;\mathbb{Z}[\pi_X])$ is a quotient of
$\mathbb{Z}[\pi_X]\oplus\mathbb{Z}$  by a cyclic submodule
(and similarly for $H_1(M;\mathbb{Z}[\pi_Y])$).
\end{theorem}

\begin{proof}
As in Corollary \ref{2concor}, generators
for $\pi_X\cong\mathbb{Z}^2$ and $\pi_2X\cong\mathbb{Z}[\pi_X]$ 
determine a map from $T^{[1]}\vee{S^2}$ to $X$.
This extends to a 2-connected map from $T\vee{S^2}$ to $X$,
which is a homotopy equivalence by Theorem \ref{2con}.
Hence $X\simeq{T\vee{S^2}}$.

The second assertion follows from the exact sequence 
of homology for $(X,M)$ with coefficients $\mathbb{Z}[\pi_X]$,
since $H^2(X;\mathbb{Z}[\pi_X])\cong{\mathbb{Z}[\pi_X]\oplus\mathbb{Z}}$.
Parallel arguments apply for $Y$ and $H_1(M;\mathbb{Z}[\pi_Y])$.
\end{proof}

An orientation for $V=\mathbb{Z}^4$ determines an isomorphism
$\wedge^3V\cong{Hom(V,\mathbb{Z})}$,
and it follows easily that all alternating 3-forms $\mu:V\to\mathbb{Z}$ 
with the same set of values are equivalent under the action of $Aut(V)\cong{GL(4;\mathbb{Z})}$.
If $M$ has an abelian embedding and $\beta=4$, 
we may refine Corollary 7.3.1(2)  to show that $\mu_M$ is surjective.
(This uses Poincar\'e duality and the fact that $X\simeq{Y}\simeq{T\vee{S^2}}$.)

The argument below for the final case ($\beta=6$) is adapted from Wall's proof 
that the $(n-1)$-skeleton of a $PD_n$-complex is essentially unique 
 \cite[Theorem 2.4]{Wa67}.

\begin{theorem}
\label{b=6}
Let $M$ be a $3$-manifold with fundamental group $\pi$ 
such that $\pi/\pi'\cong\mathbb{Z}^6$.
If $j:M\to{S^4}$ is an abelian embedding
then $X\simeq{Y}\simeq{T^{3[2]}}$, 
the $2$-skeleton of the $3$-torus $T^3$,
while  $H_1(M;\mathbb{Z}[\pi_X])$ and $H_1(M;\mathbb{Z}[\pi_Y])$ 
are cyclic $\mathbb{Z}[\pi_X]$- and $\mathbb{Z}[\pi_Y]$-modules
(respectively), of projective dimension $\leq1$.
\end{theorem}

\begin{proof}
Since $\beta=6$ and $j$ is abelian we may identify $\pi_X$ with $\mathbb{Z}^3$.
The first part of the second paragraph of Theorem \ref{neutral} applies to show that
$\pi_2X$ is isomorphic to $\Lambda_3=\mathbb{Z}[\mathbb{Z}^3]$.
Let $C_*$ and $D_*$ be the equivariant chain complexes
of the universal covers of $T^{3[2]}$ and $X$, respectively.
Since these are partial resolutions of $\mathbb{Z}$ there is
a chain map $f_*:C_*\to{D_*}$ such that $H_0(f)$ is an isomorphism.
Clearly $H_1(f)$ is also an isomorphism.
We shall modify our choice of $f_*$ so that it is a chain homotopy equivalence.

The $\Lambda_3$-modules $H_2(C_*)<{C_2}$ and $H_2(D_*)<D_2$ are free of rank 1.
Let $t\in{C_2}$ and $x\in{D_2}$ represent generators for these submodules,
and let $t^*$ and $x^*$ be the Kronecker dual generators of the cohomology modules 
$H^2(C^*)=Hom(H_2(C_*),\Lambda_3)\cong\Lambda_3$  and 
$H^2(D^*)=Hom(H_2(D_*),\Lambda_3)\cong\Lambda_3$, respectively.
Let $f'_i=f_i$ for $i=0,1$, and  let $f_2'(u)=f_2(u)-z^*(u)x$ for all $u\in{C_2}$,
where $z^*=H^2(f^*)(x^*)-t^*\in{Hom(H_2(C_*),\Lambda_3)}$.
Then $f'_*$ is again a chain homomorphism, and $H_2(f'_*)$ is an isomorphism.
Hence $f_*$ is a chain homotopy equivalence. 
This may be realized by a map from $T^{3[2]}$ to the 2-skeleton $X^{[2]}$,
and the composite with the inclusion $X^{[2]}\subseteq{X}$ is then
a homotopy equivalence.

A similar argument applies for $Y$.
The second assertion follows as before.
\end{proof}

We may also refine Corollary 7.3.1(2) to show that if $M$ has an abelian embedding 
and $\beta=6$ then $\mu_M$ is surjective.
In this case there may be a further constraint,
since even after extending coefficients to $\mathbb{R}$
there are 5 orbits of non-zero 3-forms on $\mathbb{R}^6$ 
under the action of $GL(6,\mathbb{R})$ \cite{Bry06}.

Lemma \ref{completion} and Theorems \ref{b=4} and \ref{b=6} again imply that 
when $\beta=4$ or 6 no 3-manifold which has an embedding $j$ such that 
$H_1(Y)=0$ can also have an abelian embedding.
However, 
if $L$ is the 4-component link obtained from $Bo$ by adjoining a parallel copy
of one component,
then $M(L)$ has an abelian embedding with $X\cong{Y}$ and $\chi(X)=1$, 
and also has an embedding with $\chi(X)=-1$.
We shall not give more details,
as no natural examples demand our attention in these cases.

\section{2-component links with $\ell\not=0$}

If $M$ is a rational homology sphere with an abelian embedding then
$\pi/\pi'\cong(\mathbb{Z}/\ell\mathbb{Z})^2$ and $\pi_X\cong\pi_Y\cong\mathbb{Z}/\ell\mathbb{Z}$, 
for some $\ell>0$.
In particular, if $L$ is a 2-component link with linking number
$\ell\not=0$ then $M(L)$ is a rational homology sphere.
If the components of $L$ are unknotted then $L$ determines an embedding 
$j_L$, which is clearly abelian.
There is again a necessary condition for the existence of an 
abelian embedding.

\begin{lemma}
\label{torsion}
Let $M$ be a $3$-manifold with fundamental group $\pi$ 
such that $\pi/\pi'\cong(\mathbb{Z}/\ell\mathbb{Z})^2$, for some $\ell>0$.
If $j:M\to{S^4}$ is an abelian embedding then $X\simeq{Y}\simeq{P_\ell}$,
and $H_1(M;\mathbb{Z}[\pi_X])$ and $H_1(M;\mathbb{Z}[\pi_Y])$ 
are cyclic $\mathbb{Z}[\pi_X]$- and $\mathbb{Z}[\pi_Y]$-modules (respectively), 
and are quotients of $\mathbb{Z}^{\ell-1}$, as abelian groups.
\end{lemma}

\begin{proof}
The first assertion holds by Lemma \ref{finite}.
The second part then follows from the exact sequences of homology for the 
universal covering spaces of the pairs $(X,M)$ and $(Y,M)$,
since $\widetilde{X}\simeq\widetilde{Y}\simeq\vee^{\ell-1}S^2$.
(Note that $H_2(\widetilde{X})$ and $H^2(\widetilde{X})$
are each isomorphic to the augmentation ideal of $\mathbb{Z}[\pi_X]$,
which is cyclic as a module and free of rank $\ell-1$ as an abelian group.)
\end{proof}

To use Theorem \ref{torsion} to show that some $M$ has no abelian embedding
we must consider all possible bases for $Hom(\pi,\mathbb{Z}/\ell\mathbb{Z})$,
or, equivalently, for $\pi/\pi'$.

Six of the eight rational homology 3-spheres with elementary amenable groups 
and which embed in $S^4$ have such link presentations,
with $\ell\leq4$.
(In particular, $S^3=M(2^2_1)$, and also $M(\emptyset)$!)
The simplest such link is $L=(2\ell)^2_1$, the $(2,2\ell)$-torus link,
for which $M(L)\cong{M(0;(\ell,1),(\ell,1),(\ell,-1))}$. 
Since $L$ is interchangeable, $X_L\cong{Y_L}$.

Since $M$ is an integral homology 3-sphere if $\ell=1$,
we may assume that $\ell>1$.

When $\ell=2$ we have $X\simeq{Y}\simeq\mathbb{RP}^2$,
and the composite $\partial\widetilde{X}\subset\widetilde{X}\simeq{S^2}$ 
induces an isomorphism on $H_2$.

The quaternion manifold $M=S^3/Q(8)=M(4^2_1)$
has an essentially unique abelian embedding.
The complementary regions are homeomorphic to the total space $N$ 
of the disc bundle over $\mathbb{RP}^2$ with Euler number 2 \cite{La84}.
(Lawson constructed a self-homotopy equivalence of $N$ 
which is the identity outside a regular neighbourhood of an essential $S^1$,
and which has non-trivial normal invariant.
Thus every element of $\mathcal{S}_{TOP}(N,\partial{N})$ is represented by a homeomorphism.
His  construction extends to all $X\simeq\mathbb{RP}^2$.
Do all the resulting self-homotopy equivalences have non-trivial normal invariant?

The links $9^2_{38}$, $9^2_{57}$ and $9^2_{58}$
each have unknotted components, 
asymmetric Alexander polynomial and linking number 2.

Let $L$ be the link obtained by tying a slice knot 
with non-trivial Alexander polynomial 
(such as the stevedore's knot $6_1$) in one component of $4^2_1$.
Then $M(L)$ embeds in $S^4$, but does not satisfy Lemma \ref{torsion},
since for two of the three 2-fold covers of $M(L)$ the first homology 
is not cyclic as an abelian group.
Hence $M(L)$ has no abelian embedding.

Suppose next that $\ell=3$.
The manifold $M(6^2_1)$ is a $\mathbb{N}il^3$-manifold
with Seifert base the flat orbifold $S(3,3,3)$,
and $X_L\cong{Y_L}$.

The most interesting example with $\ell=4$ is perhaps $M(8^2_2)$, 
which is the $\mathbb{N}il^3$-manifold $M(-1;(2,1),(2,3))$.
The link $8^2_2$ is interchangeable, 
and so $X_L\cong{Y_L}$. 
Each of the links $9^2_{53}$ and $9^2_{61}$ has  
unknotted components and $\ell=4$,
and gives a $\mathbb{S}ol^3$-manifold with an abelian embedding. 
Is either of these links interchangeable?

In the range $\ell\leq4$ the groups $Wh(\mathbb{Z}/\ell\mathbb{Z})$ and 
$L_5(\mathbb{Z}[\mathbb{Z}/\ell\mathbb{Z}])$ are each 0 \cite{Coh, Wa76},
while $H_2(X;\mathbb{F}_2)=H_2(Y;\mathbb{F}_2)=0$ if $\ell$ is odd.
Thus the main difficulty in identifying abelian embeddings of these manifolds lies in identifying the homotopy types of the pairs $(X,M)$ and $(Y,M)$.

There remains one more $\mathbb{S}ol^3$-manifold which embeds in $S^4$.
This is $M_{2,4}$, 
which arises from surgery on the link $L=(8_{20},{U})$ of Figure 5.1.
The knot $8_{20}$ bounds a slice disc $D\subset{D^4}$ obtained by desingularizing the obvious ribbon disc of the figure.
Then $M(L)$ has an embedding with one complementary region
$X$ obtained from $D^4\subset{S^4}$ by deleting a regular neighbourhood of $D$ and adding a 2-handle along the unknotted component $U$, 
and the other being $Y_L=\overline{S^4\setminus{X_L}}$.
The ``ribbon group" $\pi_1(D^4\setminus{D})$ has a presentation obtained by adjoining the relation $z=x$ to the Wirtinger presentation for $\pi8_{20}$ 
\cite[Theorem 1.15]{AIL}.
It is easily seen that this presentation reduces to $\langle{y,z}\mid{yzy=zyz}\rangle$.
Adding a 2-handle along the unknotted component kills
the $a$-longitude $\ell_a=zyzyzy^{-1}=zyyzyy^{-1}=zy^2z$,
and so we obtain the presentation $\langle{y,z}\mid{yzy=zyz},~y^2z^2=1\rangle$. 
The corresponding group is the semidirect product 
$\mathbb{Z}/3\mathbb{Z}\rtimes_{-1}\mathbb{Z}/4\mathbb{Z}$.
On the other hand, $\pi_{Y_L}\cong\mathbb{Z}/4\mathbb{Z}$.

The group $\pi=\pi_1M(L)$ has the presentation
\[
\langle{x,y,u}\mid{xyx^{-1}=y^{-1}},~ux^2y^{-4}u^{-1}=x^{-2}y^4,~
u^2=x^4y^{-9}\rangle.
\]
(See S\S2 of Chapter 5.)
Thus $\pi^{ab}$ is generated by the images of $x$ and $u$.
Let $\lambda_{i,j}:\pi\to\mathbb{Z}/4\mathbb{Z}$ be the epimorphism
sending $x, u$ to $i,j\in\mathbb{Z}/4\mathbb{Z}$, respectively.
It can be shown that the abelianization of $\mathrm{Ker}(\lambda_{i,j})$ 
is a quotient of the augmentation ideal in $\mathbb{Z}[\mathbb{Z}/4\mathbb{Z}]$, for 
$(i,j)=(1,0)$ or (2,1). Since these epimorphisms form a basis 
for $Hom(\pi,\mathbb{Z}/4\mathbb{Z})$, 
we cannot use Lemma \ref{torsion} to rule out an abelian embedding for $M_{2,4}$.
Is there a 2-component link with unknotted components which gives rise to this manifold?

\section{Some remarks on the mixed cases}

If $M$ has an abelian embedding such that 
$\pi_X\cong{G_k}=\mathbb{Z}\oplus(\mathbb{Z}/k\mathbb{Z})$,
for some $k>1$, then $\chi(X)=1$, by Lemma \ref{L2}.
Hence $\chi(Y)=1$ and  so $\pi_Y\cong{G_k}$ also.
Therefore $H_1(M)\cong\mathbb{Z}^2\oplus(\mathbb{Z}/k\mathbb{Z})^2$,
which requires four generators. 
The simplest examples may be constructed from 4-component links
obtained by replacing one component of the Borromean rings $Bo$ by its $(2k,2)$ cable.

In this case even the determination of the homotopy types of the complements is not clear.
The group $G_k$ has minimal presentations
\[
\mathcal{P}_{k,n}=\langle{a,t}\mid{a^k,~ta^n=a^nt}\rangle,
\]
where $0<n<k$ and $(n,k)=1$.
The 2-complexes $S_{k,n}=S^1\vee{P_k}\cup_{[t,a^n]}e^2$ associated to these presentations have Euler characteristic 1,
and it is easy to see that there are maps between them which induce isomorphisms on fundamental groups.
We may identify $S_{k,n}$ with $T\cup{MC}\cup{P_k}$, 
where $MC$ is the mapping cylinder of the degree-$n$ map 
$z\mapsto{z^n}$ from $\{1\}\times{S^1}\subset{T}$ to the 1-skeleton $S^1\subset{P_k}$.
In particular,  $S_k=S_{k,1}=T\cup_{a^k}e^2$ is  the 2-skeleton of $S^1\times{P_k}$.
From these descriptions it is easy to see that 
(1) automorphisms of $G_k$ which fix the torsion subgroup $A=\langle{a}\rangle$ may be realized 
by self homeomorphisms  of $S_{k,n}$ which act by reflections and Dehn twists on $T$, and fix the second 2-cell;
and (2) the automorphism which fixes $t$ and inverts $a$ is induced by an involution of $S_{k,n}$.

Let $C(k,n)_*$ be the cellular chain complex of the universal cover of $S_{k,n}$.
A choice of basepoint for $S_{n,k}$ determines lifts of the cells of $S_{k,n}$,
and hence isomorphisms $C(k,n)_0\cong\Gamma$, $C(k,n)_1\cong\Gamma^2$
and $C(k,n)_2\cong\Gamma^2$. 
The differentials are given by $\partial_1=(a-1,t-1)$ and 
$\partial_2^n=\left(\begin{smallmatrix}
(t-1)\nu_n&\rho\\ 1-a&0
\end{smallmatrix}\right)$,
where $ \nu_n=\Sigma_{0\leq{i}<n}a^i$  and $\rho=\nu_k$.
Let $\{e_1,e_2\}$ be the standard basis for $C(k,n)_2$.
Then $\Pi_{k,n}=\pi_2S_{k,n}=\mathrm{Ker}(\partial_2^n)$
is generated by $g=\rho{e_1}-n(t-1)e_2$ and $h=(a-1)e_2$,
with relations $(a-1)g=n(t-1)h$ and $\rho{h}=0$.
It can be shown that $\Pi_{k,n}\cong\alpha^*\Pi_{k,m}$,
where $\alpha$ is the automorphism of $G_k$ such that $\alpha(t)=t$
and $\alpha(a)=a^r$, where $n\equiv{rm}$ {\it mod} $k$.
Is there a chain homotopy equivalence $C(k,n)_*\simeq\alpha^*C(k,m)_*$?

Is every finite 2-complex $S$ with $\pi_1S\cong{G_k}$ and $\chi(S)=1$ 
homotopy equivalent to $S_{k,n}$, for some $n$?
The key invariants are the $\Gamma$-module $\pi_2S$ and the $k$-invariant in 
$H^3(G_k;\pi_2S)$.
Let $S_{\langle{t}\rangle}$ be the finite covering space with fundamental group
${\langle{t}\rangle}\cong\mathbb{Z}$.
If $M$ is a finitely generated submodule of a free $\Gamma$-module
then $H^i(\langle{t}\rangle;M)=0$ for $i\not=1$, 
while $H^1(\langle{t}\rangle;M)=M_t=M/(t-1)M$.
Hence the spectral sequence
\[
H^p(A;(H^q(\langle{t}\rangle;M))\Rightarrow{H^{p+q}(G_k;M)}
\] 
collapses,
to give $H^{p+1}(G_k;M)\cong{H^p(A;M_t)}$.
If $M=\pi_2S$ then $M_t\cong{H_2(S_{\langle{t}\rangle}};\mathbb{Z})$,
as a $\mathbb{Z}[A]$-module. 
When $M=\Pi_{k,n}$ it is easy to see that $M_t\cong\mathbb{Z}\oplus{I_A}$,
where $I_A$ is the augmentation ideal of $\mathbb{Z}[A]$,
and so $H^2(A;M_t)\cong{H^2(A;\mathbb{Z})}\cong\mathbb{Z}/k\mathbb{Z}$.

Let $V$ and $W$ be finite 2-complexes with $\pi_1V\cong\pi_1W\cong{G_k}$,
and let $\Gamma=\mathbb{Z}[G_k]$.
Then $\chi(V)\geq1$ and $\chi(W)\geq1$, and an application of Schanuel's Lemma
to the chain complexes of the universal covers gives
\[
\pi_2V\oplus\Gamma^{\chi(W)+n}\cong\pi_2W\oplus\Gamma^{\chi(V)+n},
\]
for $n>>0$.
Taking $W=S_{k,1}$, we see that $H^3(G;\pi_2V)\cong\mathbb{Z}/k\mathbb{Z}$,
for all such $V$.

Even if we can determine the homotopy types of the 2-complexes $S$ with 
$\pi_1S$ and $\chi(S)=1$, and the homotopy types of the pairs $(X,M)$
for a given $M$, 
the groups $L_5^s(\mathbb{Z}[G])$ are commensurable 
with $L_4(\mathbb{Z}[\mathbb{Z}/k\mathbb{Z}])$, 
which has rank $\lfloor\frac{k+1}2\rfloor$,
and so characterizing such abelian embeddings up to isotopy may be difficult.

The $S^1$-bundle spaces $M(-2;(1,0))$ (the half-turn flat 3-manifold $G_2$), 
and $M(-2;(1,4))$ (a $\mathbb{N}il^3$-manifold) are the total spaces of $S^1$-bundles over $Kb$.
They do not have abelian embeddings, 
since $\beta=1$ but $\pi/\pi'$ has non-trivial torsion.
In each case $\pi$ requires 3 generators, and so
they cannot be obtained by surgery on a 2-component link.
However, they may be obtained by 0-framed surgery on the links 
$8^3_9$ and $9^3_{19}$, respectively.
For the embeddings defined by these links $X\simeq{Kb}$ and $\pi_Y=\mathbb{Z}/2\mathbb{Z}$.
Since $X$ is aspherical we may apply Lemma \ref{asph} to show that
in each case the pair $(X,M)$ is homotopy equivalent {\it rel} $\partial$ 
to the corresponding disc bundle space.
The group $L_5(\mathbb{Z}[\mathbb{Z}\rtimes_{-1}\!\mathbb{Z}])$ 
acts trivially on the structure set $\mathcal{S}_{TOP}(X,\partial{X})$,
by \cite[Theorem 6.7]{FMGK}.
Thus there are only finitely many possible homeomorphism types for $X$.
As in Lemma \ref{finite}, 
$Y$ is homotopy equivalent to a finite 2-complex, 
and hence $Y\simeq\mathbb{RP}^2\vee{S^2}$.
However we do not yet know how to identify the pair $(Y,M)$.

%Are the corresponding embeddings of $Kb$ unknotted?

It is easy to find 3-component bipartedly trivial links $L$ such that $X_L$ 
is aspherical and $\pi_{X_L}\cong{BS(1,m)}$, for $m\not=0$. 
(If $m\not=1$ then $H_2(X)=0$, and if $m=2$ then $M$ is a homology handle.)
The group $L_5(\mathbb{Z}[\pi_X])$ acts trivially on $\mathcal{S}_{TOP}(X,M)$
\cite[Lemma 6.9]{FMGK}.
If $m$ is even then $H_2(X;\mathbb{F}_2)=0$, 
and so $X$ is determined up to homeomorphism
by $M$ and the homomorphism $j_{X*}$.
(See also \cite{DH23}.)
In this case $Y$ is homotopy equivalent to a finite 2-complex, 
by Lemma \ref{finite},
since $\pi_Y\cong\mathbb{Z}/(m-1)\mathbb{Z}$,
$\chi(Y)=2$ and $c.d.Y\leq2$.
Hence $Y\simeq{P_{m-1}}\vee{S^2}$ \cite{DS}.
However we do not yet know how to identify the pair $(Y,M)$.

%% file: e8.tex
\chapter{Nilpotent embeddings}

The broader class of nilpotent groups is of particular interest.  
If $\pi_X$ is nilpotent then $j_{X*}=\pi_1j_X$ is onto, 
since $H_1(j_X)$ is onto,
and any subset of a nilpotent group $G$ whose image generates
the abelianization $G/G'$ generates $G$.
Since $j_{X*}$ is onto, $c.d.X\leq2$, by Theorem \ref{Hi17-thm5.1}.
(Similarly, if $\pi_Y$ is nilpotent then $j_{Y*}$ is onto and $c.d.Y\leq2$.)
If $\pi_X$ and $\pi_Y$ are each nilpotent then $j$ is bi-epic,
by Lemma \ref{minimal} above.

There are also purely algebraic reasons why nilpotent groups 
should be of interest. 
Firstly, there is the well-known connection between homology,
lower central series and (Massey) products \cite{Dw75, St65}.
Secondly, if a group $G$ is finite or solvable and every homomorphism
$f:H\to{G}$ which  induces an epimorphism on abelianization is an epimorphism
then $G$ must be nilpotent. 
(See pages 132 and 460 of \cite{Rob}.)
However even the class of 2-generator nilpotent groups with balanced 
presentations is not known.
(We expect that nilpotent groups of large Hirsch length should have negative 
deficiency, and so should not arise in this context.)

We begin by showing that if $G$ is a finitely generated nilpotent group 
other than $\mathbb{Z}$ or $\mathbb{Z}^2$ then $\beta_2(G;F)\geq\beta_1(G;F)$ for some prime field $F$, and we summarize what is presently known about
(infinite) homologically balanced nilpotent groups.
(See Appendix B for more details.)
We then give strong constraints on nilpotent groups arising in our context,
parallel to Theorem \ref{Hi17-thm7.1}, and conclude with several examples of abelian and nilpotent embeddings.

\section{Wang sequence estimates}

An automorphism $\alpha$ of an abelian group $A$ is {\it unipotent\/} 
if $\alpha-id_A$ is nilpotent.
The following lemma is a particular case of a result of P. Hall \cite[5.2.1]{Rob}.

\begin{lem}
[Hall]
Let $\psi$ be an automorphism of a finitely generated nilpotent group $N$.
Then $G=N\rtimes_\psi\mathbb{Z}$ is nilpotent if and only if $\psi^{ab}$ is unipotent.
\qed
\end{lem}

We shall extend the term ``unipotent",
to say that an automorphism $\psi$ of a finitely generated 
nilpotent group is unipotent if $\psi^{ab}$ is unipotent.
Our next lemma is probably known,
but we have not found a published proof.

\begin{lemma}
\label{unipotent}
Let $\psi$ be a unipotent automorphism of a 
finitely generated nilpotent group $N$.
Then $H_i(\psi;R)$ and $H^i(\psi;R)$ are unipotent, 
for all simple coefficients $R$ and $i\geq0$.
\end{lemma}

\begin{proof}
If $N$ is cyclic then the result is clear.
In general,  $N$ has a composition series with cyclic
subquotients $\mathbb{Z}/p\mathbb{Z}$, where $p=0$ or is prime.
We shall induct on the number of terms in such a composition series.
If $N$ is infinite then $\psi$ acts unipotently on $Hom(N,\mathbb{Z})$
and so fixes an epimorphism to $\mathbb{Z}$;
if $N$ is finite then $\psi$ fixes an epimorphism to $\mathbb{Z}/p\mathbb{Z}$,
for any $p$ dividing the order of $N$.

Let $K$ be the kernel of such an epimorphism.
Then $\psi(K)=K$, by the choice of $\psi$;  let $\psi_K=\psi|_K$.
This is a unipotent automorphism of $K$, 
by Hall's Lemma \cite[5.2.10]{Rob}.
Hence the induced action of $\psi$ on $H_i(K;R)$ is unipotent,
for all $i$, by the inductive hypothesis.
Let $\Lambda=\mathbb{Z}[N/K]$ and let $B$ be a $\Lambda$-module.
Then $H_i(N/K;B)=Tor_i^\Lambda(\mathbb{Z},B)$ may be computed
from the tensor product $C_*\otimes_\mathbb{Z}B$,
where $C_*$ is a resolution of the augmentation $\Lambda$-module 
$\mathbb{Z}$.
If $B=H_i(K;R)$ then the diagonal action of $\psi$ on
each term of $C_*\otimes_\mathbb{Z}B$ is unipotent.
The result is now a straightforward consequence of the  
Lyndon-Hochschild-Serre spectral sequence for
$N$ as an extension of $N/K$ by $K$.

The argument for cohomology is similar.
\end{proof}

In fact we only need this lemma in degrees $\leq2$.
We shall usually assume that the coefficient ring is a field,
and then homology and cohomology are linear duals of each other.
Homology has an advantage deriving from the isomorphism
$G^{ab}\cong{H_1(G)}$, 
but it is often more convenient to use cohomology instead.

If $G$ is a finitely generated infinite nilpotent group then 
there is an epimorphism $f:G\to\mathbb{Z}$, 
and so $G\cong{K}\rtimes_\psi\mathbb{Z}$,
where $\psi$ is an automorphism of $K=\mathrm{Ker}(f)$ 
determined by conjugation in $G$.
The homology groups $H_i(K;R)=H_i(G;R[G/K])$ are $R[G/K]$-modules, 
with a generator $t$ of $G/K\cong\mathbb{Z}$ acting via $H_i(\psi;R)$.
The (homology) Wang sequence for $G$ has an extension of $\mathbb{Z}$ by $K$ has the form
\[
H_2(K;R)\xrightarrow{H_2(\psi;R)-I}{H_2(K;R)}\to{H_2(G;R)}\to
\]
\[
\to{H_1(K;R)}\xrightarrow{H_1(\psi;R)-I}{H_1(K;R)}\to{H_1(G;R)}\to{R}\to0.
\]
There is a similar Wang sequence for cohomology.
(These are special cases of the  Lyndon-Hochschild-Serre spectral sequences 
for the homology and cohomology with coefficients $R$ of $G$ 
as an extension of $\mathbb{Z}$ by $K$.)

\begin{lemma}
\label{wang app}
Let $G\cong{K}\rtimes_\psi\mathbb{Z}$ be a finitely generated nilpotent group,
and let $F$ be a field.
Then  
\begin{enumerate}
\item$\dim_F\mathrm{Cok}(H_2(\psi;F)-I)=
\dim_F\mathrm{Ker}(H^2(\psi;F)-I)=\beta_2(G)-\beta_1(G)+1$,
and so $\beta_2(G;F)\geq\beta_1(G;F)-1$,
with equality if and only if $\beta_2(K;F)=0$;
\item{}if $\beta_2(G;F)=\beta_1(G;F)$ then $H_2(K;F)$ is cyclic 
as a $F[G/K]$-module;
\item{}$\beta_1(G;F)=1\Leftrightarrow\beta_2(G;F)=0$, 
and then $K$ is finite,  $\beta_1(K;F)=0$, and $h(G)=1$;
\item{}if $H_2(G;\mathbb{Z})=0$ then $G\cong\mathbb{Z}$.
\end{enumerate}
\end{lemma}

\begin{proof}
Part (1) follows from the Wang sequences for the homology and cohomology of $G$ as an extension of $\mathbb{Z}$ by $K$.
The endomorphisms $H_i(\psi;F)-I$ have non-trivial kernel and cokernel if $H_i(K;F)\not=0$,  
since they are nilpotent.

The $F[G/K]$-module $H=H_2(K;F)$ is finitely generated 
and is annihilated by a power of $t-1$, 
since $H_2(\psi;F)$ is unipotent.
If $\beta_2(G;F)=\beta_1(G;F)$ then $\dim_FH/(t-1)H=1$,
by the exactness of the Wang sequence.
Since $F[G/K]\cong{F}[t,t^{-1}]$ is a PID,
it follows that $H$ is cyclic as an $F[G/K]$-module.

Let $t\in{G}$ represent a generator of $G/K$.
Then $F[G/K]\cong{F}[t,t^{-1}]$ is a PID
and $H=H_2(K;F)=H_2(G;F[G/K])$ is a finitely generated 
$F[t,t^{-1}]$-module, with $t$ acting via $H_2(\psi;F)$.
This module is annihilated by a power of $t-1$, 
since $H_2(\psi;F)$ is unipotent.
If $\beta_2(G;F)=\beta_1(G;F)$ then $\dim_FH/(t-1)H=1$,
by exactness of the Wang sequence.
Since $F[G/K]\cong{F}[t,t^{-1}]$ is a PID,
it follows that $H$ is cyclic as an $F[G/K]$-module.

If $\beta_1(G;F)=1$ then $H_1(K;F)=0$, and so $K$ is finite and $h(G)=1$.
Since $K$ is finite it is the direct product of its Sylow subgroups,
and the Sylow $p$-subgroup carries the $p$-primary homology of $K$.
Hence  if $F$ has characteristic $p>1$ and $H_1(K;F)=0$
then the Sylow $p$-subgroup is trivial and $H_i(K;F)=0$, for all $i\geq1$.
If $F$ has characteristic 0 then $H_i(K;F)=0$ for all $i\geq1$ also.
In each case, $H_i(G;F)=0$, for all $i>1$, and so $\beta_2(G;F)=0$.
Conversely, if $\beta_2(G;F)=0$ then $H_1(\psi;F)-I$ is a monomorphism.
Since $H_1(\psi;F)-I$ is nilpotent,  $H_1(K;F)=0$.
Hence $K$ is finite, so $h(G)=1$, and $\beta_1(G;F)=1$.

Part (4) is similar.
If $H_2(G)=0$ then $\psi^{ab}-I$ is a monomorphism,
and so $K^{ab}=0$.
Hence $K=1$ and $G\cong\mathbb{Z}$.
\end{proof}

In particular, if $h(G)=1$ and $T$ is the torsion subgroup of $G$ then 
$\beta_1(T;\mathbb{F}_p)>0$ if and only if $\beta_1(G;\mathbb{F}_p)>1$. 
The fact that the torsion subgroup has non-trivial image in the abelianization
does not extend to nilpotent groups $G$ with $h(G)>1$,
as may be seen from the groups
with presentation 
\[
\langle{x,y}\mid[x,[x,y]]=[y,[x,y]]=[x,y]^p=1\rangle.
\]

\begin{cor}
\label{wang cor}
Let $G$ be a finitely generated nilpotent group.
Then
\begin{enumerate}
\item{} $\beta_2(G;\mathbb{Q})<\beta_1(G;\mathbb{Q})$ if and only if
$h(G)=1$ or $2$;
\item{}
if $\beta_2(G;\mathbb{F}_p)<\beta_1(G;\mathbb{F}_p)$ for some prime $p$ then $G$ is infinite, $G$ has no $p$-torsion and
$h(G)=\beta_1(G;\mathbb{F}_p)=1$ or $2$. 
\end{enumerate}
\end{cor}

\begin{proof}
If $G$ is finite then $\beta_i(G;\mathbb{Q})=0$ for all $i>0$,
and if $p$ divides the order of $G$ then it follows from the 
Universal Coefficient Theorem that
$\beta_2(G;\mathbb{F}_p)\geq\beta_1(G;\mathbb{F}_p)$,
since $_pG^{ab}$ and $G^{ab}/pG^{ab}$ have the same dimension.

Hence we may assume that $G$ is infinite,
and so $G\cong{K}\rtimes_\psi\mathbb{Z}$,
where $K$ is a finitely generated nilpotent group
and $\psi$ is a unipotent automorphism.
Let $F$ be a field.
We may use Lemma \ref{wang app} to show first that
$\beta_2(K;F)=0$ and then that $\beta_1(K;F)\leq1$.
Hence $\beta_1(G;F)\leq2$.

If $F=\mathbb{Q}$ then either $K$ is finite and $h(G)=1$, 
or $h(K)=1$ and $h(G)=2$.
The converse is clear, since $G$ is then a finite extension of 
$\mathbb{Z}^{h(G)}$.

Suppose that $F=\mathbb{F}_p$ for some prime $p$.
If $\beta_1(G;\mathbb{F}_p)=1$ then $K$ is finite,
so $h(G)=1$, and $\beta_2(K;\mathbb{F}_p)=0$, 
so $\beta_1(K;\mathbb{F}_p)=0$ and $K$ has no $p$-torsion.
If $\beta_1(G;\mathbb{F}_p)=2$ then $\beta_1(K;\mathbb{F}_p)=1$
and $\beta_2(K;\mathbb{F}_p)=0$, 
so $h(K)=1$ and $K$ has no $p$-torsion.
Hence $h(G)=2$ and $G$ has no $p$-torsion.
\end{proof}

It follows immediately that if $\beta_2(G;F)<\beta_1(G;F)$ for
all prime fields $F$ then $G\cong\mathbb{Z}$ or $\mathbb{Z}^2$.

The next result is a corrected version of Lemma 1 of \cite{Hi22}
(in which it was inadvertently assumed that $\beta_2(G)=\beta_1(G)$
if $G$ is homologically balanced).

\begin{lemma}
\label{h=2}
Let $G$ be a finitely generated nilpotent group 
and let $\beta=\beta_1(G;\mathbb{Q})$.
Then $G$ is homologically balanced if and only if 
$H_2(G)$ is a quotient of $\mathbb{Z}^\beta$;
if $h(G)>2$ then $G$ is homologically balanced if and only if 
$H_2(G)\cong\mathbb{Z}^\beta$.
\end{lemma}

\begin{proof}
Since $G^{ab}\cong\mathbb{Z}^\beta\oplus{B}$, where $B$ is finite,
the first assertion follows from the Universal Coefficient exact sequences of \S1.
If $h(G)>2$ then $\beta_2(G;\mathbb{Q})\geq\beta$, 
by Corollary \ref{wang cor}, 
and so  $H_2(G)$ is a quotient of $\mathbb{Z}^\beta$  
if and only if $H_2(G)\cong\mathbb{Z}^\beta$.
\end{proof}

In Chapter 1 we observed that  a finite group $T$ is homologically balanced if and only if $H_2(T)=0$.
It is not known whether every homologically balanced finite nilpotent group 
has a balanced presentation.
The finite nilpotent 3-manifold groups
$Q(8k)\times\mathbb{Z}/a\mathbb{Z}$ (with  $(a,2k)=1$)
have the balanced presentations 
\[
\langle{x,y}\mid {x^{2ka}=y^2},~yxy^{-1}=x^s\rangle,
\]
where $s\equiv1~mod~(a)$ and $s\equiv-1~mod~(2k)$.
The other finite nilpotent groups $F$ with 4-periodic cohomology
(the generalized quaternionic groups 
$Q(2^na,b,c)\times\mathbb{Z}/d\mathbb{Z}$,
with $a,b,c,d$ odd and pairwise relatively prime) have $H_2(F;\mathbb{Z})=0$, 
but we do not know whether they all have balanced presentations.

We shall summarize here some of the results of Appendix B
on homologically balanced infinite nilpotent groups.

\begin{thm}
\label{h=1}
{\bf(B.6)}
Let $G\cong{T}\rtimes_\psi\mathbb{Z}$, 
where $T$ is a finite nilpotent group and $\psi$ is 
a unipotent automorphism of $T$.
If $G$ is homologically balanced then  
\begin{enumerate}
\item{}$G$ can be generated by $2$ elements;
\item{} if the Sylow $p$-subgroup of $T$ is abelian then it is cyclic;
\item{}if $T$ is abelian then 
$G\cong\mathbb{Z}/m\mathbb{Z}\rtimes_n\mathbb{Z}$, 
for some $m,n\not=0$ such that $m$ divides a power of $n-1$. \qed
\end{enumerate}
\end{thm}

Every semidirect product $\mathbb{Z}/m\mathbb{Z}\rtimes_n\mathbb{Z}$ 
has a balanced presentation
\[
\langle{a,t}\mid{a^m=1},~tat^{-1}=a^n\rangle.
\]
The simplest examples with $T$ non-abelian are the groups 
$Q(8k)\rtimes\mathbb{Z}$,
with the balanced presentations 
$\langle{t,x,y}\mid{x^{2k}=y^2},~tx=xt,~tyt^{-1}=xy\rangle$,
which simplify to
\[
\langle{t,y}\mid[t,y]^{2k}=y^2,~[t,[t,y]]=1\rangle.
\]
Let $m=p^s$ , where $p$ is a prime and $s\geq1$, and let $G$ be the group with presentation
\[
\langle{t,x,y}\mid{txt^{-1}=y},~tyt^{-1}=x^{-1}y^2,~yxy^{-1}=x^{m+1}\rangle.
\]
If we conjugate the final relation with $t$ to get the relation 
$x^{-1}yx=y^{m+1}$ then we see that the torsion subgroup $T$
has presentation $\langle{x,y}\mid{x^m=y^m},~yxy^{-1}=x^{m+1}\rangle$.
Moreover, $G$ is nilpotent, 
$\zeta{G}=\langle{x^m}\rangle$ and $G'=\langle{x^m,x^{-1}y}\rangle$ is abelian.
Hence $G$ is metabelian.
Each of the groups that we have described here
is 2-generated and its torsion subgroup is homologically balanced. 

It is not known whether there are infinitely many 
finitely generated, homologically balanced nilpotent groups $G$ with $\beta_1(G)=2$ and Hirsch length $>3$.
All known examples of homologically balanced nilpotent groups $G$ 
with $h(G)>1$ are torsion-free. 
If $G$ is such a group then either $G\cong\mathbb{Z}^3$ 
or $\beta_1(G;\mathbb{Q})\leq2$ \cite{Hi20a}.
If, moreover, $h(G)\leq5$ then $G$ is metabelian, 
and is either free abelian of rank $\leq3$ or is a $\mathbb{N}il^3$-group $\Gamma_q$ with presentation
\[
\langle{x,y,z}\mid[x,y]=z^q,~xz=zx,~yz=zy\rangle,
\]
for some $q\geq1$,
or is the $\mathbb{N}il^4$-group $\Omega$ with  presentation
\[
\langle{t,u}\mid[t,[t,[t,u]]]=[u,[t,u]]=1\rangle.
\]

\section{Constraints on the invariants}

Nilpotent embeddings are always bi-epic, 
since homomorphisms to a nilpotent group which induce epimorphisms on abelianization are epimorphisms.

\begin{theorem}
\label{nilp emb}
Let $j:M\to{S^4}$ be an embedding such that $\pi_X$ and $\pi_Y$ are nilpotent.
Then
\begin{enumerate}
\item{}
if $\beta=\beta_1(M)$ is odd and $\pi_X$ and $\pi_Y$ are nilpotent then 
$X$ is aspherical, $\chi(X)=0$
and either $\pi_X\cong\mathbb{Z}$ and $\pi_Y=1$ or
$\pi_X\cong\mathbb{Z}^2$ and $\pi_Y\cong\mathbb{Z}$;
\item{}if $\beta$ is even then $\beta=0,2,4$ or $6$,
$\chi(X)=\chi(Y)=1$ and $\pi_X$ and $\pi_Y$ can each be generated 
by $3$ elements and are homologically balanced;
\item{}if $\beta=6$ then $\pi_X\cong\pi_Y\cong\mathbb{Z}^3$.
\end{enumerate}
\end{theorem}

\begin{proof}
If $\beta$ is odd then $\chi(X)\leq0$, by Lemma 2.2.
Conversely, if $\chi(X)\leq0$ then $\beta_2(\pi_X;F)<\beta_1(\pi;F)$ for all fields $F$, 
and so $\pi\cong\mathbb{Z}$ or $\mathbb{Z}^2$,
by Corollary \ref{wang cor}
Hence $\chi(X)=0$ and $X$ is aspherical, by Lemma \ref{L2}.
Since $\pi_Y$ is nilpotent and $H_1(Y)\cong{H^2(X)}=0$
or $\mathbb{Z}$, $\pi_Y=1$ or $\mathbb{Z}$.
In particular, $\beta=1$ or 3.

Suppose now that $\beta$ is even.
Then $0<\chi(X)\leq\chi(Y)$ and $\chi(X)+\chi(Y)=2$,
so $\chi(X)=\chi(Y)=1$.
Hence $\pi_X$ and $\pi_Y$ are homologically balanced, by Lemma \ref{Hi17-lem2.1}.
Since $H_i(X;R)=0$ for $i>2$,
we have $\beta_2(X;R)=\beta_1(X;R)$,
and so $\beta_2(\pi_X;R)\leq\beta_1(\pi_X;R)$,
for any coefficient ring $R$.
Since $\pi_X$ is finitely generated and nilpotent,
there is a prime $p$ such that $\pi_X$ can be generated by $d=\beta_1(\pi_X;\mathbb{F}_p)$ elements.
Let $\widehat{\pi_X}$ be the pro-$p$ completion of $\pi_X$.
Since $\pi_X$ is nilpotent, it is $p$-good,
and so $\beta_i(\widehat{\pi_X};\mathbb{F}_p)=
\beta_i(\pi_X;\mathbb{F}_p)$, for all $i$.
The group $\widehat{\pi_X}$ is a pro-$p$ analytic group, 
and so has a minimal presentation with 
$d=\beta_1(\widehat{\pi_X};\mathbb{F}_p)$ generators and
$r=\beta_2(\widehat{\pi_X};\mathbb{F}_p)$ relators.
Since $\beta>2$, $\widehat{\pi_X}\not\cong\widehat{\mathbb{Z}}_p$,
and so $r>\frac{d^2}4$, by \cite[Theorem 2.7]{Lub83}.
(Similarly for $\pi_Y$.)
Therefore $d\leq3$ and $\beta\leq2d\leq6$.

If $\chi(X)=\chi(Y)=1$ then $\beta_2(\pi_X;F)\leq\beta_1(\pi_X;F)$ 
for all field coefficients $F$.
Hence if $\beta=6$ then $\pi_X\cong\pi_Y\cong\mathbb{Z}^3$,
by the main result of \cite{Hi22}.
\end{proof}

\begin{cor}
\label{Z1Z2Z}
If $\pi_X$ and $\pi_Y$ are nilpotent and $\chi(X)<\chi(Y)$
then either $X\simeq{S^1}$ and $Y\simeq{S^2}$ or $X\simeq{T}$ 
and $Y\simeq{S^1\vee{S^2}\vee{S^2}}$.
\end{cor}

\begin{proof}
By the theorem, $X$ is aspherical, and the embedding is abelian.
The further details are given in Theorems \ref{neutral} and \ref{b=3}.
\end{proof}

If $\pi_X$ is nilpotent and $\beta_1(X)=0$ then $\pi_X$ is finite,
while if $\beta_1(X)=1$ then $\pi_X\cong{F\rtimes\mathbb{Z}}$,
where $F$ is finite.
Thus if $\pi_X$ and $\pi_Y$ are torsion-free nilpotent and $\beta\leq3$
then $\pi_X$ and $\pi_Y$ are abelian.

\begin{theorem}
If $M$ has a nilpotent embedding and $\beta\geq3$ then there is a non-zero Massey product of length at most $4$ in $H^2(M;\mathbb{Q})$.
\end{theorem}

\begin{proof}
We may assume that $\pi_X$ is non-abelian, and so $\beta=4$,
for if $\beta\geq3$ and $\pi_X$ is abelian there are non-zero cup products in 
$H^2(X;\mathbb{Q})$ and hence in $H^2(M;\mathbb{Q})$.

Hence $\beta_1(\pi_X)=2$. 
Since $\pi_X$ is nilpotent and homologically balanced
$\pi_X/\gamma_5^\mathbb{Q}\pi_X$ is a proper quotient of
$F(2)/\gamma_5^\mathbb{Q}F(2)$ \cite[Theorem 2]{FHT97},
and so there is a non-zero Massey product of length at most 4 in
$H^2(\pi_X;\mathbb{Q})$ \cite[Proposition 4.3]{Dw75}.
Since $H^2(\pi_X;\mathbb{Q})$ maps injectively to
$H^2(X;\mathbb{Q})$ and $H^2(M;\mathbb{Q})$,
the theorem follows.
\end{proof}

In particular,  $\#^\beta(S^2\times{S^1})$ has a nilpotent embedding 
if and only if $\beta\leq2$.

\begin{theorem}
\label{finiteY}
If $\pi_X$ is nilpotent and $H_1(Y)$ is a non-trivial finite group
then $\pi_X$ is finite and $\chi(X)=\chi(Y)=1$.
If $H_1(Y)=0$ then we may also have $\pi_X\cong\mathbb{Z}$ and $\chi(X)=0$.
\end{theorem}

\begin{proof}
Since $\pi_X$ is nilpotent, $j_{X*}$ is an epimorphism,
and so $c.d.X\leq2$, by Theorem \ref{Hi17-thm5.1}.
Moreover, $\beta_1^{(2)}(\pi_X)=0$ and so either $\chi(X)=0$ and $X$ is aspherical, 
or $\chi(X)=1$, by Lemma \ref{L2}.
If $\chi(X)=0$ then $\pi_X=1$, 
$\mathbb{Z}$ or $\mathbb{Z}^2$, and $H_1(X)$ is torsion-free.
Therefore $H_1(Y)=0$ (since $H_1(Y)=\tau_Y\cong\tau_X$), and so $H_2(X)=0$.

If $\chi(X)=1$ then $\chi(Y)=1$, 
and so  $H^1(X)\cong{H_2(Y)}=0$ and $H_2(X)\cong{H^1(Y)}=0$.
Therefore $H_1(X)$ is finite.
Since $\pi_X$ is nilpotent and has finite abelianization, it is finite.
Moreover,  $H_2(\pi_X)=H_2(\pi_Y)=0$, 
since these groups are quotients of $H_2(X)$
and $H_2(Y)$, respectively.
\end{proof}

The complementary regions of the standard embedding of $S^2\times{S^1}$ 
in $S^4$ are $X=D^3\times{S^1}$ and $Y=S^2\times{D^2}$,
with fundamental groups $\mathbb{Z}$ and 1,  respectively.

It is reasonable to restrict consideration further to torsion-free nilpotent groups, as such groups satisfy the Novikov conjecture, and the surgery obstructions are maniable.

If $G$ is torsion-free nilpotent of Hirsch length $h$ then $c.d.G=h$.
The first non-abelian examples are the $\mathbb{N}il^3$-groups $\Gamma_q$.
Some of the argument of Theorem \ref{b=6} for the group $\mathbb{Z}^3$
extends to the groups $\Gamma_q$.
The homology of the pair $(X,M)$ with coefficients $\mathbb{Z}[\pi_X]$
gives an exact sequence
\[
H_2(X;\mathbb{Z}[\pi_X])\to{H^2(X;\mathbb{Z}[\pi_X])}\to
{H_1(M;\mathbb{Z}[\pi_X])}\to0.
\]
Let $K_X=\mathrm{Ker}(j_{X*})={H_1(M;\mathbb{Z}[\pi_X])}$ 
and $P=H_2(X;\mathbb{Z}[\pi_X])$.
Since $c.d.X\leq2$ and $c.d.\Gamma_q=3$,
an application of Schanuel's Lemma shows that $P$ 
is a projective $\mathbb{Z}[\Gamma_q]$-module of rank 1.
It is stably free since $\widetilde{K}_0(\mathbb{Z}[G])=0$ for
torsion-free poly-Z groups $G$, 
and $P$ has rank 1 since $\chi(X)=1$.
Since $Ext^i_{\mathbb{Z}[\pi_X])}(\mathbb{Z},\mathbb{Z}[\pi_X])=0$ 
for $i\leq2$
we then see that $H^2(X;\mathbb{Z}[\pi_X])\cong{P^\dagger}=
\overline{Hom_{\mathbb{Z}[\pi_X])}(P,\mathbb{Z}[\pi_X])}$,
and so is also stably free of rank 1.
We thus have an exact sequence 
\[
P\to{P^\dagger}\to{K_X}\to0.
\]
However it is not clear that this is as potentially useful as the analogous conditions 
on abelian embeddings given above.
Moreover, if $G$ is a nonabelian poly-$\mathbb{Z}$ group 
then there are infinitely many isomorphism classes of 
stably free $\mathbb{Z}[G]$-modules $P$ such that
${P\oplus\mathbb{Z}[G]}\cong\mathbb{Z}[G]^2$ \cite{Ar81}.
(This contrasts strongly with the case $\pi_X\cong\mathbb{Z}^3$,
for then $P$ must be a free module.)
At the end of the chapter we mention briefly an example with $\pi_X\cong\pi_1Kb$ 
and for which $\pi_2X$ is stably free but not free.

\section{Examples}

If $(n-1,\ell)=(n-1,m)$ then $\mathbb{Z}/\ell\mathbb{Z}\rtimes_n\mathbb{Z}$
and $\mathbb{Z}/m\mathbb{Z}\rtimes_n\mathbb{Z}$ 
have isomorphic abelianizations.
Since they have balanced presentations, 
every such pair of groups can be realized by an embedding,
by Theorem \ref{LQ}.

The simplest non-abelian nilpotent example corresponds to the choice
$\ell=2, m=4$ and $n=-1$.
One group is $\mathbb{Z}/4\mathbb{Z}\rtimes_{-1}\mathbb{Z}$, 
and the other is its abelianization $\mathbb{Z}\oplus\mathbb{Z}/2\mathbb{Z}$.
We shall give an explicit construction of an embedding
realizing this pair of groups.
Let $M=M(L)$, where $L$ is the 4-component bipartedly trivial link 
depicted in Figure 8.1.
If $X$ and $Y$ are the complementary regions for $j_L$ then
$\pi_X$ and $\pi_Y$ have presentations $\langle{a,b}\mid{U=V=1}\rangle$
and $\langle{u,v}\mid{A=B=1}\rangle$, respectively,
where the words $A=u^4v^2$, $B= vuv^{-1}u^{-1}$,
$U=a^4$ and $V=b^{-1}aba$, are easily read from the diagram.
Thus the embedding $j_L$ is nilpotent, 
with $\pi_X\cong\mathbb{Z}/4\mathbb{Z}\rtimes_{-1}\mathbb{Z}$,
and $\pi_Y\cong\mathbb{Z}\oplus\mathbb{Z}/2\mathbb{Z}$.

\setlength{\unitlength}{1mm}
\begin{picture}(95,75)(-32,-2.5)

\put(-5,69.05){$\vartriangleright$}
\put(-8,67){$a$}
\put(54,67){$b$}
\put(57,69.05){$\vartriangleright$}
\put(25,57.8){$\vartriangleright$}
\put(22,60){$u$}
\put(22,39){$v$}
\put(25,36.8){$\vartriangleright$}
\put(-3.9,32){$\bullet$}
\put(44.9,32){$\bullet$}

\linethickness{1pt}
\put(-15,70){\line(1,0){20}}
\put(-20,15){\line(0,1){50}}
\put(-15,10){\line(1,0){20}}
\qbezier(-20,65)(-20,70)(-15,70)
\qbezier(-20,15)(-20,10)(-15,10)

\put(10,60){\line(0,1){5}}
\qbezier(5,70)(10,70)(10,65)
\put(10,15){\line(0,1){15}}
\qbezier(5,10)(10,10)(10,15)

\put(10,32){\line(0,1){4.5}}
\put(10,38.5){\line(0,1){4}}
\put(10,44.5){\line(0,1){3}}
\put(10,49.5){\line(0,1){3}}
\put(10,54.5){\line(0,1){3}}

\put(45,70){\line(1,0){20}}
\qbezier(65,70)(70,70)(70,65)
\put(70,15){\line(0,1){50}}

\put(45,10){\line(1,0){20}}
\qbezier(65,10)(70,10)(70,15)
\put(40,15){\line(0,1){9}}
\qbezier(40,15)(40,10)(45,10)

\put(40,26){\line(0,1){4}}
\put(40,32){\line(0,1){4.5}}
\put(40,39){\line(0,1){18}}
\put(40,60){\line(0,1){5}}
\qbezier(40,65)(40,70)(45,70)

\thinlines
\put(2,18.5){\line(1,0){7}}
\put(11.5,18.5){\line(1,0){13.5}}
\qbezier(25,18.5)(30,18.5)(30,23.5)
\put(30,23.5){\line(0,1){2.5}}
\qbezier(30,26)(30,31)(35,31)
\put(35,31){\line(1,0){6}}
\qbezier(41,31)(42.5,31)(42.5,32.5)
\qbezier(41,34)(42.5,34)(42.5,32.5)

\put(9,37.5){\line(1,0){25.5}}
\put(37.5,37.5){\line(1,0){3.5}}

\put(-3,23.5){\line(0,1){15}}
\qbezier(-3,23.5)(-3,18.5)(2,18.5)
\qbezier(-3,38.5)(-3,43.5)(2,43.5)

\qbezier(41,28)(45.75,28)(45.75,32.75)
\qbezier(41,37.5)(45.75,37.5)(45.75,32.75)
\qbezier(36,25)(36,28)(39,28)
\qbezier(36,25)(36,22)(39,22)
\qbezier(42,22)(43.5,22)(43.5,23.5)
\qbezier(42,25)(43.5,25)(43.5,23.5)
\put(41,22){\line(1,0){1}}

\put(8.5,58.5){\line(1,0){33.5}}
\qbezier(7.25,57.25)(7.25,58.5)(8.5,58.5)
\qbezier(7.25,57.25)(7.25,56)(8.5,56)

\qbezier(9,34)(7.25,34)(7.25,35.75)
\qbezier(9,37.5)(7.25,37.5)(7.25,35.75)

\put(9,31){\line(1,0){2}}
\qbezier(9,25)(6,25)(6,28)
\qbezier(6,28)(6,31)(9,31)
\qbezier(11,31)(12.5,31)(12.5,32.5)
\qbezier(11,34)(12.5,34)(12.5,32.5)

\put(11.5,25){\line(1,0){17.5}}
\put(31,25){\line(1,0){4}}
\put(37,25){\line(1,0){5}}

\put(8.5,48.5){\line(1,0){3}}
\put(8.5,53.5){\line(1,0){3}}
\qbezier(8.5,46)(7.25,46)(7.25,47.25)
\qbezier(7.25,47.25)(7.25,48.5)(8.5,48.5)
\qbezier(8.5,51)(7.25,51)(7.25,52.25)
\qbezier(7.25,52.25)(7.25,53.5)(8.5,53.5)

\put(2,43.5){\line(1,0){9.5}}
\qbezier(11.5,43.5)(12.75,43.5)(12.75,44.75)
\qbezier(11.5,46)(12.75,46)(12.75,44.75)
\qbezier(11.5,48.5)(12.75,48.5)(12.75,49.75)
\qbezier(11.5,51)(12.75,51)(12.75,49.75)
\qbezier(11.5,53.5)(12.75,53.5)(12.75,54.75)
\qbezier(11.5,56)(12.75,56)(12.75,54.75)

\put(41,56){\line(1,0){1}}
\qbezier(42,56)(43.25,56)(43.25,57.25)
\qbezier(42,58.5)(43.25,58.5)(43.25,57.25)

\put(36, 37){\line(0,1){16}}
\qbezier(36,53)(36,56)(39,56)
\qbezier(36,37)(36,34)(39,34)

\put(17.3,2.5){Figure 8.1}

\end{picture}

It is easy to find a 4-component link $L=L_a\cup{L_b}\cup{L_u}\cup{L_v}$ 
with each 2-component sublink trivial, 
and such that $L_a$ and $L_b$ represent 
(the conjugacy classes of) $A=[u,[u,v]]$ and $B=[v,[u,v]]$ in $F(u,v)$,
respectively, while $L_u$ and $L_v$ have image 1 in $F(a,b)$.
Arrange the link diagram so that $L_u$ is on the left, 
$L_v$ on the right, $L_a$ at the top and $L_b$ at the bottom. 
We may pass one bight of $L_a$ which loops around $L_u$ under a 
similar bight of $L_b$, so that $U$ now represents $[a,b]$ in $F(a,b)$.
Finally we use claspers to modify $L_u$ and $L_v$ so 
that they represent $[b,v]$ in $F(b,v)$ and $[a,u]$ in $F(a,u)$.
We obtain the link of Figure 8.2.

\setlength{\unitlength}{1mm}
\begin{picture}(90,85)(-32,-5)

\put(-6.5,9.05){$\vartriangleright$}
\put(-9,11.3){$u$}
\put(59,8){$v$}
\put(61,6.05){$\vartriangleright$}
\put(39,56){$\vartriangleright$}
\put(36.5,58.3){$a$}
\put(20.4,15){$\vartriangleright$}
\put(18,12.5){$b$}
\put(25.4,38.1){$\bullet$}

\linethickness{1pt}
\put(-15,71){\line(1,0){13}}
\qbezier(-20,66)(-20,71)(-15,71)
\put(-20,15){\line(0,1){51}}
\qbezier(-20,15)(-20,10)(-15,10)
\put(-15,10){\line(1,0){67}}
\qbezier(-2,71)(0,71)(0,69)

\put(-3,68){\line(1,0){68}}
\qbezier(-6,65)(-6,68)(-3,68)
\qbezier(-6,65)(-6,62)(-3,62)
\qbezier(65,68)(70,68)(70,63)
\put(-3,62){\line(1,0){2}}
\put(-2,66){\line(1,0){47}}
\qbezier(45,66)(50,66)(50,61)
\put(-2,64){\line(1,0){1}}

\qbezier(-3,65)(-3,64)(-2,64)
\qbezier(-3,65)(-3,66)(-2,66)
\put(1,64){\line(1,0){6}}
\put(1,62){\line(1,0){5.7}}

\put(0,61){\line(0,1){4}}
\qbezier(6.7,62)(8,62)(8,60.7)
\put(10,59){\line(0,1){2}}

\qbezier(7,64)(10,64)(10,61)
\qbezier(8,59)(8,58)(9,58)
\qbezier(9,58)(10,58)(10,59)

\put(52,7){\line(1,0){13}}
\qbezier(65,7)(70,7)(70,12)
\put(70,12){\line(0,1){51}}
\put(50,58){\line(0,1){3}}
\put(0,53){\line(0,1){6}}
\put(0,49){\line(0,1){2}}
\put(50,51){\line(0,1){5}}

\put(50,48.5){\line(0,1){1}}

\put(0,46.5){\line(0,1){1}}
\put(0,40){\line(0,1){5}}
\put(0,31){\line(0,1){7}}
\put(0,28.5){\line(0,1){1}}
\put(50,33){\line(0,1){14}}
\put(50,26.5){\line(0,1){1}}
\put(50,29){\line(0,1){2}}
\put(43,12){\line(1,0){6}}
\put(51,12){\line(1,0){1}}
\qbezier(52,12)(53,12)(53,11)
\qbezier(52,10)(53,10)(53,11)

\qbezier(40,15)(40,12)(43,12)
\qbezier(40,17)(40,18.5)(41.5,18.5)
\qbezier(41.5,18.5)(43,18.5)(43,17)
\qbezier(43,15.5)(43,14)(44.5,14)
\put(43,15.5){\line(0,1){1.5}}
\put(44.5,14){\line(1,0){4.5}}
\put(51,14){\line(1,0){4}}

\put(5,20){\line(1,0){50}}
\put(0,25){\line(0,1){2}}
\qbezier(0,25)(0,20)(5,20)
\qbezier(55,20)(58,20)(58,17)
\qbezier(55,14)(58,14)(58,17)

\put(50,21){\line(0,1){4}}
\put(50,11){\line(0,1){4}}
\put(50,17){\line(0,1){2}}
\qbezier(50,9)(50,7)(52,7)
%\put(50,7){\line(0,1){2}}

\thinlines
\put(-1.5,60){\line(1,0){10.5}}
\put(1,57){\line(1,0){50}}
\put(11,60){\line(1,0){17}}
\qbezier(28,60)(30,60)(30,58)
\qbezier(-3,58.5)(-3,57)(-1.5,57)
\qbezier(-3,58.5)(-3,60)(-1.5,60)

\put(1,54){\line(1,0){48}}
\qbezier(-2,53)(-2,54)(-1,54)
\qbezier(-2,53)(-2,52)(-1,52)
\put(-1,52){\line(1,0){50}}

\put(1,50){\line(1,0){50}}
\qbezier(-2,49)(-2,50)(-1,50)
\qbezier(-2,49)(-2,48)(-1,48)
\put(-1,48){\line(1,0){52}}

\qbezier(51,54)(52.5,54)(52.5,55.5)
\qbezier(51,57)(52.5,57)(52.5,55.5)
\qbezier(51,50)(52,50)(52,51)
\qbezier(51,52)(52,52)(52,51)

\put(-1,46){\line(1,0){7}}
\qbezier(-2,45)(-2,46)(-1,46)
\qbezier(-2,45)(-2,44)(-1,44)
\put(1,44){\line(1,0){5}}

\qbezier(-2,40)(-2,41)(-1,41)
\qbezier(-2,40)(-2,39)(-1,39)
\put(1,41){\line(1,0){25}}
\put(-1,39){\line(1,0){50}}
\qbezier(51,39)(53,39)(53,41)
\put(53,41){\line(0,1){5}}
\qbezier(51,48)(53,48)(53,46)

\put(30,45){\line(0,1){2}}
\qbezier(26,41)(30,41)(30,45)
\qbezier(6,46)(10,46)(10,42)
\qbezier(6,44)(8,44)(8,42)
\qbezier(10,38)(10,36)(12,36)
\qbezier(8,38)(8,34)(12,34)

\put(12,36){\line(1,0){37}}
\put(12,34){\line(1,0){37}}

\qbezier(51,36)(54,36)(54,33)
\put(1,32){\line(1,0){50}}

\qbezier(-2,31)(-2,32)(-1,32)
\qbezier(-2,31)(-2,30)(-1,30)
\put(-1,30){\line(1,0){50}}

\qbezier(51,32)(52,32)(52,33)
\qbezier(51,34)(52,34)(52,33)

\qbezier(51,28)(52,28)(52,29)
\qbezier(51,30)(52,30)(52,29)

\put(-1,28){\line(1,0){52}}
\qbezier(-2,27)(-2,28)(-1,28)
\qbezier(-2,27)(-2,26)(-1,26)

\put(1,26){\line(1,0){50}}

\qbezier(51,26)(52,26)(52,25)
\qbezier(51,24)(52,24)(52,25)

\put(1,24){\line(1,0){48}}

\put(-1,16){\line(1,0){43}}
\qbezier(-5,20)(-5,24)(-1,24)
\qbezier(-5,20)(-5,16)(-1,16)
\put(44,16){\line(1,0){7}}
\qbezier(51,16)(54,16)(54,19)
\put(54,21){\line(0,1){12}}

\put(17,0){Figure 8.2}

\end{picture}

This link may be partitioned into two trivial links in three distinct ways,
giving three embeddings of $M(L)$.
If the two sublinks are $L_a\cup{L_b}$ and $L_u\cup{L_v}$ 
then $A=vu^{-1}v^{-1}u^{-1}vuv^{-1}u$,
$B=vuv^{-1}u^{-1}v^{-1}uvu^{-1}$,
$U=b^{-1}aba^{-1}$ and $V=1$.
Hence $\pi_X\cong\Gamma_1$ and $\pi_Y\cong\mathbb{Z}^2$.

Each of the other partitions determine abelian embeddings, with
$\pi_X\cong\pi_Y\cong\mathbb{Z}^2$ and $\chi(X)=\chi(Y)=1$.

With a little more effort, instead of passing just one bight of $L_a$ 
under $L_b$ (as above), we may interlace the loops of $L_a$ and $L_b$ 
around each of $L_u$ and $L_v$ so that $u$ and $V$ represent $[a,[a,b]]$ 
and $[b,[a,b]]$, respectively, 
and so that each 2-component sublink of $L$ is still trivial.
If we then use claspers again we may arrange that $u$ 
represents $[a,v]$ and $v$ represents $[b,u]$,
so that we obtain a 3-manifold which has one embedding with $\pi_X\cong\pi_Y\cong\Gamma_1$ and another with
$\pi_X\cong\pi_Y\cong\mathbb{Z}^2$.
Can we refine this construction so that the third embedding has
$\pi_X\cong\Gamma_1$ and $\pi_Y\cong\mathbb{Z}^2$?

We conclude with an example which just fails to be nilpotent, 
but which illustrates the issue raised at the end of \S8.2.
This is based on the recent result of W. H. Mannan that 
there is a finite 2-complex $C$ with $\pi_1C\cong\pi_1Kb$ 
and $\chi(C)=1$ which is not homotopy equivalent to $Kb\vee{S^2}$, 
since $\pi_2C$ is not a free $\mathbb{Z}[G]$-module \cite{Ma24}.
This complex corresponds to the presentation
\[
\langle{x,y}\mid{y^{-2}xy^2x^{-1}}=x^{-3}y^{-1}xyx^2y^{-1}x^{-2}y=1\rangle.
\]
The 4-component link in Figure 8.3 determines an embedding of a 3-manifold with
complementary regions $X_L\simeq{C}$ and $Y_L\simeq\mathbb{RP}^2\vee{S^2}$.
(This embedding is virtually abelian, but not nilpotent.)

\newpage

\setlength{\unitlength}{1mm}
\begin{picture}(120,92)(-10,-5)

\put(15,79){$\vartriangleleft$}
\put(14,76){$x$}
\put(83,79){$\vartriangleleft$}
\put(80,76){$y$}
\put(50,41){$\bullet$}
\put(48,43){$s$}
\put(50,13){$\bullet$}
\put(48,15){$r$}
\put(58,13){$\vartriangleright$}
\put(42,41){$\vartriangleleft$}

\linethickness{1pt}

\put(0,13){\line(0,1){62}}
\qbezier(0,13)(0,8)(5,8)
\qbezier(0,75)(0,80)(5,80)
\put(5,8){\line(1,0){20}}
\put(5,80){\line(1,0){20}}
\qbezier(25,8)(30,8)(30,13)
\qbezier(25,80)(30,80)(30,75)
\put(30,15){\line(0,1){13}}
\put(30,30){\line(0,1){3}}

\put(30,34.7){\line(0,1){2.6}}

\put(30,38.7){\line(0,1){2.6}}
\put(30,43){\line(0,1){6}}
\put(30,51){\line(0,1){7}}
\put(30,59.7){\line(0,1){2.6}}
\put(30,64){\line(0,1){3}}
\put(30,68.7){\line(0,1){2.6}}

\put(30,73){\line(0,1){2}}

\qbezier(70,13)(70,8)(75,8)
\put(75,8){\line(1,0){20}}
\qbezier(95,8)(100,8)(100,13)
\put(100,13){\line(0,1){62}}
\qbezier(70,75)(70,80)(75,80)
\put(75,80){\line(1,0){20}}
\qbezier(95,80)(100,80)(100,75)

\put(70,13){\line(0,1){2}}
\put(70,17){\line(0,1){5}}
\put(70,24){\line(0,1){2}}
\put(70,27.7){\line(0,1){1.6}}
\put(70,31){\line(0,1){12}}
\put(70,44.7){\line(0,1){2.6}}

\put(70,48.7){\line(0,1){2.6}}

\put(70,53){\line(0,1){7}}
\put(70,62){\line(0,1){5}}
\put(70,69){\line(0,1){4}}

\thinlines
\put(31,74){\line(1,0){48}}
\qbezier(79,74)(82,74)(82,71)
\put(82,35){\line(0,1){36}}

\put(28,74){\line(1,0){1}}
\qbezier(27,73)(27,74)(28,74)
\qbezier(27,73)(27,72)(28,72)
\put(28,72){\line(1,0){4}}
\qbezier(32,72)(33,72)(33,71)
\qbezier(32,70)(33,70)(33,71)
\put(31,70){\line(1,0){1}}
\put(28,70){\line(1,0){1}}
\qbezier(27,69)(27,70)(28,70)
\qbezier(27,69)(27,68)(28,68)
\put(28,68){\line(1,0){43}}
\qbezier(71,68)(72.5,68)(72.5,66.5)
\qbezier(71,65)(72.5,65)(72.5,66.5)

\put(31,65){\line(1,0){38}}
\put(28,65){\line(1,0){1}}
\qbezier(27,64)(27,65)(28,65)
\qbezier(27,64)(27,63)(28,63)
\put(28,63){\line(1,0){4}}
\qbezier(32,63)(33,63)(33,62)
\qbezier(32,61)(33,61)(33,62)
\put(31,61){\line(1,0){1}}
\put(28,61){\line(1,0){1}}
\qbezier(27,60)(27,61)(28,61)
\qbezier(27,60)(27,59)(28,59)
\put(28,59){\line(1,0){20}}
\qbezier(48,59)(50.5,59)(50.5,56.5)
\qbezier(50.5,56.5)(50.5,54)(53,54)
\put(53,54){\line(1,0){5}}

\put(28,52){\line(1,0){1}}
\qbezier(27,51)(27,52)(28,52)
\qbezier(27,51)(27,50)(28,50)
\put(28,50){\line(1,0){20}}

\qbezier(48,50)(50.5,50)(50.5,47.5)
\put(50.5,46.5){\line(0,1){1}}
\qbezier(50.5,46.5)(50.5,44)(53,44)
\put(53,44){\line(1,0){5}}

\put(31,52){\line(1,0){27}}

\put(60,52){\line(1,0){5}}
\put(67,52){\line(1,0){5}}
\put(67,54){\line(1,0){2}}
\put(60,54){\line(1,0){5}}
\qbezier(72,52)(73,52)(73,53)
\qbezier(72,54)(73,54)(73,53)
\put(71,54){\line(1,0){1}}

\qbezier(69,50)(66,50)(66,53)
\qbezier(69,56)(66,56)(66,53)
\put(71,50){\line(1,0){1}}
\qbezier(72,50)(73,50)(73,49)
\qbezier(72,48)(73,48)(73,49)
\put(66,48){\line(1,0){6}}
\qbezier(63,45)(63,48)(66,48)

\qbezier(71,56)(73.5,56)(73.5,58.5)
\qbezier(71,61)(73.5,61)(73.5,58.5)
\put(63,61){\line(1,0){8}}
\qbezier(63,61)(59,61)(59,57)
\put(59,32){\line(0,1){25}}
\qbezier(59,32)(59,29)(56,29)
\put(28,29){\line(1,0){28}}
\qbezier(27,28)(27,29)(28,29)
\qbezier(27,28)(27,27)(28,27)
\put(28,27){\line(1,0){1}}
\put(31,27){\line(1,0){41}}

\put(60,44){\line(1,0){11}}
\qbezier(71,44)(72,44)(72,43)
\qbezier(71,42)(72,42)(72,43)
\put(60,42){\line(1,0){9}}

\put(28,42){\line(1,0){30}}
\qbezier(27,41)(27,42)(28,42)
\qbezier(27,41)(27,40)(28,40)
\put(28,40){\line(1,0){1}}
\put(31,40){\line(1,0){1}}
\qbezier(32,40)(33,40)(33,39)
\qbezier(32,38)(33,38)(33,39)
\put(28,38){\line(1,0){4}}

\qbezier(27,37)(27,38)(28,38)
\qbezier(27,37)(27,36)(28,36)

\put(28,36){\line(1,0){1}}
\put(31,36){\line(1,0){1}}
\qbezier(32,36)(33,36)(33,35)
\qbezier(32,34)(33,34)(33,35)
\put(28,34){\line(1,0){4}}
\qbezier(27,33)(27,34)(28,34)
\qbezier(27,33)(27,32)(28,32)
\put(28,32){\line(1,0){1}}
\put(31,32){\line(1,0){27}}
\put(60,32){\line(1,0){2}}
\put(63,33){\line(0,1){8}}
\qbezier(63,33)(63,30)(66,30)
\put(66,30){\line(1,0){9}}

\put(64,32){\line(1,0){5}}
\put(71,32){\line(1,0){8}}
\qbezier(79,32)(82,32)(82,35)

\qbezier(72,27)(73,27)(73,26)
\qbezier(72,25)(73,25)(73,26)
\put(68,25){\line(1,0){1}}
\qbezier(67,24)(67,25)(68,25)
\qbezier(67,24)(67,23)(68,23)
\put(68,23){\line(1,0){4}}
\qbezier(72,23)(73,23)(73,22)
\qbezier(72,21)(73,21)(73,22)
\put(71,21){\line(1,0){1}}
\put(31,21){\line(1,0){38}}

\put(27,14){\line(1,0){42}}
\put(71,14){\line(1,0){1}}
\qbezier(72,14)(73,14)(73,15)
\qbezier(72,16)(73,16)(73,15)
\put(68,16){\line(1,0){4}}
\qbezier(68,16)(67,16)(67,17)
\qbezier(67,17)(67,18)(68,18)
\put(68,18){\line(1,0){1}}
\put(71,18){\line(1,0){4}}
\qbezier(75,18)(77,18)(77,20)
\put(77,20){\line(0,1){8}}
\qbezier(75,30)(77,30)(77,28)

\put(27,21){\line(1,0){2}}
\qbezier(27,21)(25,21)(25,19)
\qbezier(27,14)(25,14)(25,16)
\put(25,16){\line(0,1){3}}

\put(0,0){Figure 8.3.\quad{$r=y^{-2}yy^{-1}xy^2x^{-1}$, 
$s=x^{-3}y^{-1}xyx^2y^{-1}x^2y$
}}

\end{picture}

%% file: eappA.tex
\chapter*{Appendix A. The linking pairings of orientable Seifert manifolds}

The linking pairings of oriented 3-manifolds which are Seifert fibred 
over non-orientable base orbifolds were computed in Chapter 3.
Here we shall consider the remaining case, 
when the base orbifold is also orientable.
Thus the Seifert fibration is induced by a fixed-point free 
$S^1$-action on the manifold.

Bryden and Deloup have used the cohomological formulation of
the linking pairing to show that every linking pairing 
on a finite abelian group of odd order 
is realised by some Seifert fibred $\mathbb{Q}$-homology sphere \cite{BD04}.
We shall work directly with the geometric definition,
giving a new proof of this result,
and shall show that there are pairings on 2-primary groups
which are not realized by any orientable Seifert fibred 3-manifold at all
(i.e., even if we allow non-orientable base orbifolds).
Our strategy shall be to localize at a prime $p$.

\section*{The torsion linking pairing}

Assume now that $M=M(g;S)$ is a Seifert manifold with Seifert data 
$S=((\alpha_1,\beta_1),\dots,(\alpha_r,\beta_r))$,
where $g\geq0$, $r\geq1$ and $\alpha_i>1$ for all $i\leq{r}$. 
Then $H_1(M;\mathbb{Z})\cong\mathbb{Z}^{2g}\oplus{H}$, 
where $H$ has a presentation
\[\langle
{q_1,\dots,q_r,h}\mid\Sigma{q_i}=0,\,\alpha_iq_i+\beta_ih=0,\,
\forall{i\geq1}\rangle.\]
The torsion subgroup $\tau_M$ is a subgroup of $H$.

We shall modify this presentation to obtain one 
with more convenient generators.
Our approach involves localizing at a prime $p$.
After reordering the Seifert data, if necessary, 
we may assume that $\alpha_{i+1}$ divides $\alpha_i$
in $\mathbb{Z}_{(p)}$, for all $i\geq1$.
(Note that $\nu=\alpha_1\varepsilon(M)$ is then in $\mathbb{Z}_{(p)}$,
while if $\varepsilon(M)=0$ then $\frac{\alpha_1}{\alpha_2}$
is invertible in $\mathbb{Z}_{(p)}$.)
Localization loses nothing, since $\ell_M$ is uniquely 
the orthogonal sum of pairings on the $p$-primary summands of $\tau_M$.
(We shall often write $\ell_M$ rather than $\mathbb{Z}_{(p)}\otimes\ell_M$,
for simplicity of notation.)

Using the relation $\Sigma{q_i}=0$ to eliminate the generator $q_1$,
we see that $\mathbb{Z}_{(p)}\otimes{H}$ has the equivalent presentation
\[\langle
{q_2,\dots,q_r,h}\mid\alpha_1\varepsilon(M){h}=0,\,\alpha_iq_i+\beta_ih=0,\,
\forall{i\geq2}\rangle.\]
If $r=1$ this group is cyclic, generated by the image of $h$.
We shall assume henceforth that $r\geq2$, since all pairings on 
finite cyclic groups are realizable by lens spaces.
Then there are integers
$m,n$ such that $m\alpha_2+n\beta_2=1$, since $(\alpha_2,\beta_2)=1$.
Let $\gamma_i=\frac{\alpha_2}{\alpha_i}\beta_i$ and
$q_i'=\gamma_2q_i-\gamma_iq_2$, for all $i$.
(Then $q_2'=0$.)
Let $s=-mh+nq_2$ and $t=\alpha_2q_2+\beta_2h$. 
Then $h=-\alpha_2s+nt$ and $q_2=\beta_2s+mt$.
Since $t=0$ in $H$ this simplifies to
\[\langle
{q_3',\dots,q_r',s}\mid\alpha_1\alpha_2\varepsilon(M)s=0,
\,\alpha_iq_i'=0,\,\forall{i\geq3}\rangle.\]
In particular, $M(0;S)$ is a $\mathbb{Q}$-homology sphere (i.e., $H=\tau_M$)
if and only if $\varepsilon(M)\not=0$. 

If exactly $r_p$ of the cone point orders $\alpha_i$ 
are divisible by $p$ and $\varepsilon(M)=0$ 
then $\tau_M$ has non-trivial $p$-torsion if and only if $r_p\geq3$, 
in which case $\mathbb{Z}_{(p)}\otimes\tau_M$ is the direct sum 
of $r_p-2$ cyclic submodules, while if $\varepsilon(M)\not=0$ then 
$\tau_M$ has non-trivial $p$-torsion if and only if $r_p\geq2$
and then $\mathbb{Z}_{(p)}\otimes\tau_M$ is the direct sum 
of $r_p-1$ cyclic submodules.
(Note however that if $r_p\leq1$ and $\varepsilon(M)\not=0$ then 
$\mathbb{Z}_{(p)}\otimes\tau_M\cong
\mathbb{Z}_{(p)}/\alpha_1\varepsilon_S\mathbb{Z}_{(p)}$, and may be non-trivial.)

The Seifert structure gives natural 2-chains relating the 1-cycles
representing the generators of $H$.

[For let $N_i$ be a torus neighborhood of the $i^{th}$ exceptional fibre,
and let $B_o$ be a section of the restriction of the
Seifert fibration to $M*=M\setminus\cup{intN_i}$.
Let $\xi_i$ and $\theta_i$ be simple closed curves 
in $\partial{N_i}$ which represent $q_i$ and $h$, respectively.
Then $\partial B_o=\Sigma\xi_i$, and there are singular 2-chains
$D_i$ in $N_i$ such that $\partial D_i=\alpha_i \xi_i+\beta_i\theta_i$,
since $\alpha_iq_i+\beta_ih=0$ in $H_1(N_i;\mathbb{Z})$.]

We may choose disjoint annuli $A_i$ in $M^*$
with $\partial{A_i}=\theta_2-\theta_i$, for $i\not=2$.
For convenience in our formulae, we shall also let $A_2=0$.
Then $C_i=\beta_2D_i-\beta_iD_2+\beta_2\beta_i{A_i}$
is a singular 2-chain with 
$\partial{C_i}=\alpha_i\beta_2\xi_i-\alpha_2\beta_i\xi_2$.

Let $\xi_i'=\gamma_2\xi_i-\gamma_i\xi_2$, for $i\geq3$,
$\sigma=-m\theta_2+n\xi_2$ and 
\[
U=\alpha_1B_o+\alpha_1\varepsilon_SnD_2
-\Sigma\frac{\alpha_1}{\alpha_i}(D_i+\beta_iA_i).
\]
Then $\xi_i'$ is a singular 1-chain 
representing $q_i'$ and $\partial{C_i}=\alpha_i\xi_i'$, for all $i\geq3$,
$\sigma$ is a singular 1-chain representing $s$ 
and $U$ is a singular 2-chain with 
$\partial{U}=\alpha_1\alpha_2\varepsilon_S\sigma$.

We may assume that $\xi_i\bullet\theta_i=1$ in $\partial{N_i}$.
In order to calculate intersections and self-intersections 
of the 1-cycles $\xi_i$ with the 2-chains $C_i$ in $M$, 
we may push each $\xi_i$ off $N_i$.
Then $\xi_i$ and $D_j$ are disjoint, for all $i,j$,
while $\xi_2\bullet{A_i}=1$, $\xi_i\bullet{A_i}=-1$
and $\xi_j\bullet{A_i}=0$, if $i,j\not=2$ and $j\not=i$.
Similarly, we may assume that $\theta_2$ is disjoint from the discs $D_j$
(for all $j$) and the annuli $A_k$ (for all $k\not=2$).
Since $B_o$ is oriented so that $\partial{B_o}=\Sigma\xi_i$,
we must have $\theta_2\bullet{B_o}=-1$.
Hence
\[\xi_i'\bullet{C_i}=-\beta_2\beta_i(\gamma_2+\gamma_i),\]
\[\xi_i'\bullet{C_j}=-\beta_2\beta_j\gamma_i,\]
and
\[\xi_i'\bullet{U}=\alpha_1\varepsilon_S\gamma_i\]
for all $i,j\geq3$ with $j\not=i$, while
\[\sigma\bullet{U}=-\frac{\alpha_1}{\alpha_2}-n\alpha_1\varepsilon_S
\]
and
\[\sigma\bullet{C_i}=n\beta_2\beta_i.\]
Then
\[\ell_M(q_i',q_i')=
[-\beta_2\beta_i\frac{\alpha_i\beta_2+\alpha_2\beta_i}{\alpha_i^2}]
\in\mathbb{Q}/\mathbb{Z}\]
and
\[\ell_M(q_i',q_j')=[-\beta_2\beta_i\beta_j\frac{\alpha_2}{\alpha_i\alpha_j}]
\in\mathbb{Q}/\mathbb{Z}.\]
If $\varepsilon(M)\not=0$ then the above calculations of $\xi_i'\bullet{U}$ and $\sigma\bullet{C_i}$ each give
\[\ell_M(s,q_i')=[\frac{\beta_i}{\alpha_i}]\in\mathbb{Q}/\mathbb{Z}\]
and
\[\ell_M(s,s)=[-\frac{\alpha_1+n\alpha_1\alpha_2\varepsilon(M)}
{\alpha_1\alpha_2^2\varepsilon(M)}]
\in\mathbb{Q}/\mathbb{Z}.
\]
In particular, the linking pairings depend only on $S$ and not on $g$.
(We could arrange that the deminators are powers of $p$,
after further rescaling the basis elements.
However that would tend to obscure the dependence on the Seifert data.)

Let $S$ and $S'$ be two systems of Seifert data, with
concatenation $S''$, and let $M''=M\#_fM'=M(0;S'')$
be the fibre-sum of $M=M(0;S)$ and $M'=M(0;S')$.
Then $\varepsilon(M'')=\varepsilon(M)+\varepsilon(M')$.
The next result is clear. 

\begin{lem}
{\bf{A1}.}
%\label{A1}
Let $M=M(g;S)$ and $M'=M(g';S')$ be Seifert manifolds such that
all the cone point orders of $S'$ are relatively prime to all the cone point orders of $S$ 
and $\varepsilon(M')=\varepsilon(M)=0$, and let $M''=M(g+g';S'')$.
Then $\varepsilon(M'')=0$ and $\ell_{M''}=\ell_M\perp\ell_{M'}$.
\qed
\end{lem}

Thus if every $p$-primary summand of a linking pairing $\ell$
can be realized by some $M(0;S)$ with all cone point orders powers of $p$
and $\varepsilon(M)=0$ then $\ell $ can also be realized by a Seifert manifold.
If one of the hypotheses fails, 
it is not clear how the linking pairings of $M,M'$ and $M''$ are related.
In order to realize pairings by Seifert manifolds with $\varepsilon(M)\not=0$
we shall need another approach.

\section*{The homogeneous case: $p$ odd}

In this section we shall show that when $p$ is odd and 
the $p$-primary component of 
$\tau_M$ is homogeneous the structure of the $p$-primary component of $\ell_M$ 
may be read off the Seifert data.
Our results shall be extended to the inhomogeneous cases in later sections.
Let  $u_i=\frac{\alpha_i}{p^k}$, for $1\leq{k}\leq{r_p}$,
and $v=\alpha_1\varepsilon(M)$.

\begin{lem}
{\bf{A2}.}
%\label{A2}
Let $M=M(g;S)$ be a Seifert manifold and $p$ a prime. 
Then $\mathbb{Z}_{(p)}\otimes\tau_M$ is homogeneous of exponent $p^k$ if and only if either
\begin{enumerate}
\item$\varepsilon(M)=0$ and $u_i$ is invertible in $\mathbb{Z}_{(p)}$, 
for $1\leq{k}\leq{r_p}$; or
\item$p^{-k}\alpha_1\varepsilon(M)$ and $u_i$ are invertible in $\mathbb{Z}_{(p)}$, 
for $2\leq{k}\leq{r_p}$; or
\item$r_p\leq2$ and $p^{-k}\alpha_1\alpha_2\varepsilon(M)$ is  invertible in 
$\mathbb{Z}_{(p)}$.
\end{enumerate}
\end{lem}

\begin{proof}
This follows immediately from the calculations in \S1, with the following observations.
If $\varepsilon(M)=0$ then $\alpha_1$ and $\alpha_2$ must have 
the same $p$-adic valuation.
If $r_p\geq2$ then $u_2=p^{-k}\alpha_2$ and $v=\alpha_1\varepsilon(M)$ 
are in $\mathbb{Z}_{(p)}$.
Hence if $u_2v=p^{-k}\alpha_1\alpha_2\varepsilon(M)$ is invertible in
$\mathbb{Z}_{(p)}$ then $u_2$ and $v$ are also invertible in $\mathbb{Z}_{(p)}$.
\end{proof}

Note that if $\mathbb{Z}_{(p)}\otimes\tau_M$ is homogeneous of exponent $p^k$
and $\varepsilon(M)\not=0$ then $u_1$ may be divisible by $p$.

\begin{thm}
{\bf{A3}.}
%\label{A3}
Let $M=M(g;S)$ be a Seifert manifold and $p$ an odd prime such that
$\mathbb{Z}_{(p)}\otimes\tau_M$ is homogeneous of exponent $p^k$.
Then
\begin{enumerate}
\item{}if $\varepsilon(M)=0$ then $d(\ell_M)=
[(-1)^{r_p-1}\frac{\alpha_1}{\alpha_2}(\Pi_{i=1}^r\beta_i)(\Pi_{j=3}^{r_p}u_j)]$;
\item{}if $\varepsilon(M)\not=0$ then $d(\ell_M)=
[(-1)^{r_p-1}\frac{\alpha_1}{\alpha_2}(\Pi_{i=1}^r\beta_i)(\Pi_{j=2}^{r_p}u_j)\nu]$.
\end{enumerate}
\end{thm}

\begin{proof}
Let $L\in\mathrm{GL}(\rho,\mathbb{Z}/p^k\mathbb{Z})$ be the matrix 
with $(i,j)$ entry $p^k\ell(e_i,e_j)$, 
where $e_1,\dots,e_\rho$ is some basis for $\mathbb{Z}_{(p)}\otimes\tau_M\cong(\mathbb{Z}/p^k\mathbb{Z})^\rho$.

Suppose first that $\varepsilon(M)=0$.
Then ${\mathbb{Z}_{(p)}\otimes\tau_M}\cong(\mathbb{Z}/p^k\mathbb{Z})^{r_p-2}$,
with basis $e_i=q_{i+2}'$, for $1\leq{i}\leq{r_p-2}$,
and $\frac{\alpha_1}{\alpha_2}=\frac{u_1}{u_2}$ 
is invertible in $\mathbb{Z}_{(p)}$.
We  apply row operations 
\[row_i\mapsto{row_i-\frac{\alpha_3\beta_{i+2}}{\alpha_{i+2}\beta_3}row_1}\] 
for $2\leq{i}\leq{r_p-2}$ and then 
\[row_1\mapsto
{row_1-\frac{\alpha_2\beta_3}{\alpha_3\beta_2}\Sigma_{i=2}^{r_p-2}row_i}\]
to $L$.
This gives a lower triangular matrix with diagonal
\[
[-\beta_2\frac{\beta_3}{u_3}(\beta_2+u_2\Sigma_{i=3}^{r_p}\frac{\beta_i}{u_i}),
-\beta_2^2\frac{\beta_4}{u_4},\dots,-\beta_2^2\frac{\beta_{r_p-2}}{u_{r_p-2}}].
\]
Therefore
\[
\det(L)=(-1)^{r_p}\beta_2^{2r_p-3}(\beta_2+u_2\Sigma_{i=3}^{r_p}\frac{\beta_i}{u_i})
\Pi_{i=3}^{r_p} \frac{\beta_j}{u_j},
\]
and so
\[
d(\ell_M)=[(-1)^{r_p-1}\frac{\alpha_2}{\alpha_1}(\pi_{i=1}^{r_p}\beta_i)\Pi_{j=3}^{r_p}u_j)].
\]

A similar argument applies if $\varepsilon(M)\not=0$.
In this case $u_2$ and $v=\alpha_1\varepsilon(M)$ are also invertible in 
$\mathbb{Z}_{(p)}$,
and $\mathbb{Z}_{(p)}\otimes\tau_M\cong(\mathbb{Z}/p^k\mathbb{Z})^{r_p-1}$,
with basis $e_i=q_{i+2}'$, for $1\leq{i}\leq{r_p-2}$, and $e_{r_p-1}=s$.
If we perform the same row operations on rows 2 to $r_p-2$, 
and then the column operation
$col_1\mapsto{col_1+\Sigma_{i=2}^{r_p-2}col_i}$ we obtain a bordered matrix
\[ 
\left(
\begin{matrix}
-\beta_2\frac{\beta_3}{u_3}(\beta_2+u_2\Sigma^*)&0&\dots&0&\frac{\beta_3}{u_3}\\
0&-\beta_2^2\frac{\beta_4}{u_4}&\dots&0&0\\
\vdots&0&\ddots&\vdots&\vdots\\
0&\dots&0&-\beta_2^2\frac{\beta_{r_p}}{u_{r_p}}&0\\
\Sigma^*&\frac{\beta_4}{u_4}&\dots&\frac{\beta_{r_p}}{u_{r_p}}^{d^*}
\end{matrix}
\right).
\]
where $\Sigma^*=\Sigma_{i=3}^{r_p}\frac{\beta_i}{u_i}$ and $d^*=-\frac{u_1+u_2v}{u_2^2v}$.
Hence
\[
\det(L)=-(\beta_2\frac{\beta_3}{u_3}(\beta_2+u_2\Sigma^*)d^*+\frac{\beta_3}{u_3}\Sigma^*)(-1)^{r_p-3}\beta_2^{2(r_p-3)}\Pi_{i=4}^{r_p}\frac{\beta_i}{u_i}
\]
\[
=(-\beta_dd^*(\beta_2+u_2\Sigma^*)+\Sigma^*)
(-1)^{r_p-2}\beta_2^{2(r_p-3)}\Pi_{i=3}^{r_p}\frac{\beta_i}{u_i}.
\]
Now
\[
(-\beta_2d^*(\beta_2+u_2\Sigma^*)+\Sigma^*)=
-(\beta_2(u_1+u_2v)(-\beta_2+u_2\Sigma^*)-u_2^2\nu\Sigma^*)/u_2^2\nu
\]
\[
=(\beta_2(u_1\beta_2+u_1u_2\Sigma^*+n\beta_2u_2v)+(n\beta_2-1)u_2^2v\Sigma^*)/
u_2^2v
\]
\[
\equiv\beta_2(u_1\beta_2+u_1u_2\Sigma^*+u_2v)\equiv-\beta_1\beta_2u_2\quad{mod}~(p),
\]
since $n\beta_2\equiv1$ {\it mod} $(p)$ and $v=-\beta_1-\frac{\beta_1u_1}{u_2}-u_1\Sigma^*$.
Therefore
\[
\det(L)\equiv(-1)^{r_p-1}\beta_2^{2(r_p-3)}(\Pi_{i=1}^{r_p}\beta_i)/
\Pi_{j=2}^{r_p}u_jv
\quad{mod}~(p),
\]
and so we now have
\[
[d(\ell_M)=
[(-1)^{r_p-1}(\Pi_{1\leq{i}\leq{r_p}}\beta_i)(\Pi_{2\leq{j}\leq{r_p}}u_j)v].
\]
\end{proof}

When all the cone point orders have the same $p$-adic valuation
(i.e.,  $u_1$ and $u_2$ are also invertible in $\mathbb{Z}_{(p)}$) 
then these formulae for $d(\ell_M)$ are  invariant under permutation 
of the indices.
For if $\varepsilon(M)=0$ then 
$[\frac{\alpha_1}{\alpha_2}]=[u_1u_2]$
in $\mathbb{F}_p^\times/(\mathbb{F}_p^\times)^2$, 
while if  $\varepsilon(M)\not=0$ then $\nu=u_1p^k\varepsilon(M)$ 
(and $p^k\varepsilon(M)$ is also invertible).

A linking pairing $\ell$ on a free $\mathbb{Z}/p^k\mathbb{Z}$-module $N$ 
is hyperbolic if and only if $\rho=rk(\ell)$ is even and $d(\ell)=(-1)^{\frac\rho2}$.
Thus  $\mathbb{Z}_{(p)}\otimes\ell_M$ 
is hyperbolic if and only if either $\varepsilon(M)=0$, 
$r_p=\rho+2$ is even and 
$[\frac{\alpha_1}{\alpha_2}
(\Pi_{1\leq{i}\leq{r_p}}\beta_i)(\Pi_{3\leq{j}\leq{r_p}}u_j)]
=[(-1)^{\frac{r_p}2-1}]$
or $\varepsilon(M)\not=0$, $r_p=\rho+1$ is odd and
$[(\Pi_{1\leq{i}\leq{r_p}}\beta_i)(\Pi_{2\leq{j}\leq{r_p}}u_j)\nu]=
[(-1)^{\frac{r_p-1}2}].$

\section*{Realization of pairings on groups of odd order}

In this section we shall show that every linking pairing on a finite group of odd order may be realized by a Seifert manifold.

Suppose first that we localize $\ell_M$ at a prime $p$.
Let $p^k$ be the exponent of $\mathbb{Z}_{(p)}\otimes\tau_M$,
and let $L$ be the matrix with entries $p^k\ell(q_i',q_j')$.
(If $\varepsilon(M)\not=0$ we need also a row and column corresponding to the generator $s$, which has the maximal order $p^k$.)
Then
\[ 
L=
\left(
\begin{matrix}
D_1&{p^{\kappa_2}B_2}&\dots&{p^{\kappa_t}B_t}\\
p^{\kappa_2}B_2^{tr}&{p^{\kappa_2}D_2}&\dots&\vdots\\
\vdots&\dots&\ddots&\vdots\\
p^{\kappa_t}B_t^{tr}&\dots&\dots&{p^{\kappa_t}D_t}
\end{matrix}
\right).
\]
where $D_i$ is a $\rho_i\times\rho_i$ block with $\det(D_i)\not\equiv0$ {\it mod} $(p)$,
for $1\leq{i}\leq{t}$,
and $0<\kappa_2<\dots<\kappa_t<k$.
We may partition $L$ more coarsely as 
$L=\left(\smallmatrix{A}&{B}\\{B^{tr}}&{D}\endsmallmatrix\right)$,
where $A=D_1$ and $B=[B_2\dots{B_t}]$.
Let $Q=\left(\smallmatrix {I-d_1^{-1}}&{p^{\kappa_2}B}\\0&{I}\endsmallmatrix\right)$
and $D'=D-p^{\kappa_2}B^{tr}A^{-1}B$.
Then $Q^trLQ=\left(\smallmatrix{A}&0\\0&{p^{\kappa_2}D'}\endsmallmatrix\right)$.
Block-diagonalizing $L$ in this fashion does not change the residues {\it mod} $(p)$ of the diagonal blocks $D_i$ or decrease the divisibility of the off-diagonal blocks.
These matrix manipulations correspond to replacing the generators $q_i'$ with $i>\rho_1$ by $\widetilde{q}_i=q_i'-p^{\kappa_2}\Sigma_{j=1}^{\rho_1}[B^{tr}A]_{ji}q_j'$.

We may iterate this process, 
and we find that $\ell_M$ is an orthogonal sum of pairings 
on homogeneous groups $(\mathbb{Z}/p^k\mathbb{Z})^{\rho_1}$,
$(\mathbb{Z}/p^{k-\kappa_2}\mathbb{Z})^{\rho_2},\dots$,
$(\mathbb{Z}/p^{k-\kappa_t}\mathbb{Z})^{\rho_t}$.
If $\varepsilon(M)=0$ or if $\alpha_1\varepsilon(M)$ is invertible in $\mathbb{Z}_{(p)}$ then
the determinantal invariants of the first summand (with the maximal exponent $p^k$)
may be computed from the block $A$ as in the previous section,
while we may read off the determinantal invariants of the other summands from the corresponding diagonal elements of the original matrix $L$.
(We shall not need to consider the possibility that $p$ divides the numerator of $\varepsilon_S$ in justifying our constructions below.)

With these reductions in mind, we may now construct Seifert manifolds realizing given pairings.

\begin{thm}
{\bf{A4}.}
%\begin{theorem}
%\label{A4}
Let $p$ be an odd prime, and let $\ell$ be a linking pairing 
on a finite abelian $p$-group.
Then there is a Seifert manifold $M=M(0;S)$ such that all 
the cone point orders $\alpha_i$ are powers of $p$, 
$\varepsilon(M)=0$ and $\ell_M\cong\ell$.
\end{thm}

\begin{proof}
The pairing $\ell_M$ is the orthogonal sum $\perp_{j=1}^t\ell_j$, 
where $\ell_j$ is a pairing on $(\mathbb{Z}/p^{k_j}\mathbb{Z})^{\rho_j}$,
with $\rho_j>0$ for $1\leq{j}\leq{t}$ and $0<k_k<k_{j-1}$ for $2\leq{j}\leq{t}$.
Let $d(\ell_j)=[w_j]$ for $1\leq{j}\leq{t}$ , and let $k=k_1$.

If $p\geq5$ we let $\alpha_i=p^k$ for $1\leq{i}\leq{m_1=\rho_1+2}$,
$\alpha_i=p^{k_2}$ for $m_1<i\leq{m_2=m_1+\rho_2},\dots,$ and
$\alpha_i=p^{k_t}$ for $m_{t-1}<i\leq{r}=(\Sigma\rho_j)+2$.
For each $1\leq{j}\leq{t}$ we let
$\beta_i=1$ for $m_j<i<m_{j+1}$ and $\beta_{m_{j+1}}=w_j$.
We must then choose $\beta_i$ for $1\leq{i}\leq{m_1}$ so that
$[\Pi_{i=1}^{m_1}\beta_i]=[w_1]$ and 
$\Sigma_{i=1}^{m_1}\beta_i=-\Sigma_{j=m_1+1}^rp^k\frac{\beta_j}{\alpha_j}$.
It is in fact sufficient to solve the equations $[\Pi_{i=1}^{m_1}\beta_i]=[w_1]$ and 
$\Sigma_{i=1}^{m_1}\beta_i=0$ with all
$\beta_i\in\mathbb{Z}/p^k\mathbb{Z}^\times$,
for subtracting $\Sigma_{j=m_1+1}^rp^k\frac{\beta_j}{\alpha_j}$ from $\beta_1$ will not change its residue {\it mod} $(p)$.

If $m_1$ is odd the equation $\Sigma\beta_i=0$
always has solutions with all $\beta_i\in\mathbb{Z}/p^k\mathbb{Z}^\times$.
If $\xi$ is a nonsquare in $\mathbb{Z}/p^k\mathbb{Z}^\times$
setting $\beta_i'=\xi\beta_i$ for all $i$ gives another solution,
and $[\Pi\beta_i']=[\xi][\Pi\beta_i]$.
(If $p\equiv3$ {\it mod} (4) we may take $\xi=-1$, 
which corresponds to a change of orientation of the 3-manifold.)

If $m_1=4t$ and $w\not\equiv1$ {\it mod} $(p)$ then there is an integer $x$
such that $2x\equiv{w-1}$ {\it mod} $(p)$.
The images of $x$ and $w-1-x$ are invertible in $\mathbb{Z}/p^k\mathbb{Z}$.
Let $\beta_1=1$, $\beta_2=-w$, $\beta_3=x$ and $\beta_4=w-1-x$,
and  $\beta_{2i+1}=1$ and $\beta_{2k+2}=-1$ for $2\leq{i}\leq2t$.
Then $\Sigma\beta_i=0$,
$\beta_4\equiv\beta_3$ {\it mod} $(p)$ and $[(-1)^{r-1}\Pi\beta_i]=[w]$.

If $\rho=4t+2$ and $w\not\equiv1$ {\it mod} $(p)$
let $y$ be an integer such that $4y\equiv{w+1}$ {\it mod} $(p)$.
Let $\beta_1=1$, $\beta_2=-w$, $\beta_3=\beta_4=\beta_5=y$ and 
$\beta_6=w-1-3y$,
and $\beta_{2i+1}=1$ and $\beta_{2i+2}=-1$ for $3\leq{i}<2t$.
Then $\Sigma\beta_i=0$,
$\beta_6\equiv\beta_3$ {\it mod} $(p)$ and $[(-1)^{r-1}\Pi\beta_i]=[w]$.

These choices work equally well for all $p\geq3$, if $[w]\not=1$.
If $\rho$ is even, $w\equiv1$ {\it mod} $(p)$ and $p>3$ 
there is an integer $n$ such that $n^2\not=0$ or 1 {\it mod} $(p)$,
and we solve as before, after replacing $w$ by $\widehat{w}=n^2w$.

However if $p=3$ and $[w]=1$ we must vary our choices.
If $\rho=4t-2$ with $t>1$ 
let $\beta_1=\beta_2=\beta_3=\beta_4=1$, $\beta_5=\beta_6=-2$
and $\beta_{2i+1}=1$ and $\beta_{2i+2}=-1$ for $3\leq{i}<2t$.
If $\rho=4t$ let
$\beta_{2i-1}=1$ and $\beta_{2i}=-1$ for $1\leq{i}\leq2t+1$.
In the remaining case (when $\rho=2$)
we find that if $\Sigma_{1\leq{i}\leq4}\beta_i=0$ then $[-\Pi\beta_i]=[-1]$.
In this case we must use instead
$S=((3^{k+1},1),(3^{k+1},5),(3^k,-1),(3^k,-1))$
to realize the pairing with ${[w]=[1]}$.
\end{proof}

The manifolds with Seifert data as above are
$\mathbb{H}^2\times\mathbb{E}^1$-manifolds,
except when $\rho=1$ and $p=3$, in which case 
they are the flat manifold $G_3$ (with its two possible orientations).

It follows immediately from Theorem {A4} and Lemma {A1} that 
every linking pairing on a finite abelian group of odd order is realized by 
some Seifert manifold $M(0;S)$ with $\varepsilon(M)=0$.

All such pairings may also be realized by Seifert manifolds
which are $\mathbb{Q}$-homology spheres.
However we must be careful to ensure that the {\it numerator\/} of $\varepsilon(M)$ does not provide unexpected torsion.

\begin{thm}
{\bf{A5}.}
%\begin{theorem}
%\label{A5}
Let $\ell$ be a linking pairing on a finite abelian group $A$ of odd order.
Then there is a Seifert manifold $M=M(0;S)$ such that $\varepsilon(M)\not=0$ 
and $\ell_M\cong\ell$.
\end{thm}

\begin{proof}
Let $P$ be the finite set of primes for which the $p$-primary summand of $A$ 
is non-trivial, 
and let $\ell=\perp_{p\in{P}}\ell^{(p)}$ be the primary decomposition of $\ell$.
We shall define Seifert data $S(p)$ for each $p\in{P}$ as follows.
Suppose that $\ell^{(p)}$ is a pairing on $\oplus_{j=1}^{t(p)}(\mathbb{Z}/p^{k_j}\mathbb{Z})^{\rho_j}$, with $\rho_j>0$ for $1\leq{j}\leq{t(p)}$
and $0<k_j<k_{j-1}$ for $2\leq{j}\leq{t(p)}$.
Then $\ell^{(p)}\cong\perp\ell_{b_i/a_i}$, where the $a_i$ are powers of $p$,
for $1\leq{i}\leq\rho(p)=\Sigma_{j=1}^{t(p)}\rho_j$ and such that
$a_i\geq{a_{i+1}}$ for $i<\rho(p)$.
Let 
$S(p)=((\alpha_1^{(p)},\beta_1^{(p)}),\dots,(\alpha_{t(p)}^{(p)}, \beta_{t(p)}^{(p)}))$,
where $(\alpha_i^{(p)}=a_i$, for all $i$, 
$\beta_1^{(p)}=(-1)^{\rho_1+1}b_1$ and $\beta_i^{(p)}=b_i$, 
for $2\leq{i}\leq\rho(p)$.
Finally let $\widetilde\alpha=e\Pi_{p\in{P}}p$, 
where $e$ is the exponent of $A$,
and let $\widetilde\beta=-1-\widetilde\alpha\Sigma_{i=1}^{(t(p)}\frac{\beta_i^{(p)}}{\alpha_i^{(p)}}$.

Let $S$ be the concatenation of $(\widetilde\alpha,\widetilde\beta)$ and the
$S(p)$s for $p\in{P}$, and let $M=M(0;S)$.
Then $\varepsilon(M)=\frac1{\widetilde\alpha}$.
For each $p\in{P}$ there are $\rho(p)+1$ cone points with order 
divisible by $p$, 
and $\mathbb{Z}_{(p)}\otimes\tau_M\cong
\oplus_{j=1}^{t(p)}\mathbb{Z}/p^{k_j}\mathbb{Z})^{\rho_j}$.
Since $\widetilde\beta\equiv-1$ {\it mod} $(p)$, for all $p\in{P}$,
the determinantal invariant of the component of 
$\mathbb{Z}_{(p)}\otimes\ell_M$ of maximal exponent $a_1=p^{k_1}$ is $[\Pi_{i=1}^{\rho_1}b_i]$.
Therefore $\mathbb{Z}_{(p)}\otimes\ell_M\cong\ell^{(p)}$, for each $p\in{P}$, 
and so $\ell_M\cong\ell$.
\end{proof}

The manifolds with Seifert data as above are $\widetilde{\mathbb{SL}}$-manifolds,
except when $A\cong\mathbb{Z}/p^k\mathbb{Z}$ and so there are just two cone points,
in which case they are lens spaces ($\mathbb{S}^3$-manifolds).

In the homogeneous $p$-primary case we may arrange that all cone points 
have order $p^k$, except when $p=3$, $\rho=2$ and $d(\ell)=[1]$. 
This case is realized by $M(0;(3^{k+1},7),(3^k,-1),(3^k,-1))$.

\section*{Realization of homogeneous $2$-primary pairings}

The situation is more complicated when $p=2$.
A linking pairing $\ell$ on $(Z/2^kZ)^\rho$ is determined by its rank $\rho$
and certain invariants $\sigma_{j}(\ell)\in{Z/8Z}\cup\{\infty\}$, 
for $\rho-2\leq{j}\leq\rho$.
(See \S3 of \cite{KK80}, and \cite{De05}.)
We shall not calculate these invariants here.
Instead, we shall take advantage of the particular form 
of the pairings given in \S2. 

If a linking pairing $\ell$ on $(\mathbb{Z}/2^k\mathbb{Z})^\rho$ is even then $\rho$ is also even. When $k=1$ all even pairings are hyperbolic.
If $k>1$ then $\ell$ is determined by the image of the matrix $L$ in 
$GL(\rho,\mathbb{Z}/4\mathbb{Z})$, 
and is either  hyperbolic (and is the orthogonal sum of $\frac\rho2$ copies
of the pairing $E_0^k$) or is the orthogonal direct sum of a hyperbolic pairing of rank 
$\rho-2$ with the pairing $E_1^k$ \cite{KK80,Wa64}.

We shall say that an element $\frac{p}q$ is {\it even} or {\it odd} if $p$ is even or odd,
respectively.
Thus $\frac{p}q$ is odd if and only if it is invertible in $\mathbb{Z}_{(2)}$.

The following result complements the criterion for homogeneity given in Lemma  {A2}.

\begin{lem}
{\bf{A6}.}
%\label{A6}
Let $M=M(g;S)$ be a Seifert manifold and assume that the Seifert data are ordered 
so that $\alpha_{i+1}$ divides $\alpha_i$ in $\mathbb{Z}_{(2)}$. 
Then
\begin{enumerate}
\item$\ell=\mathbb{Z}_{(2)}\otimes\ell_M$ is even if and only if 
$\frac{\alpha_1}{\alpha_i}$ is odd for $1\leq{i}\leq{r_2}$
and either $\varepsilon(M)=0$ or $\alpha_1\varepsilon(M)$ is odd;
\item{}if $\frac{\alpha_1}{\alpha_i}$ is odd for $1\leq{i}\leq{r_2}$ then 
$\alpha_1\varepsilon(M)\equiv{r_2}$ {\it mod} $(2)$.
\end{enumerate}
\end{lem} 

\begin{proof}
If $\ell=\mathbb{Z}_{(2)}\otimes\ell_M$ is even then $\beta_2+\frac{\alpha_2}{\alpha_i}\beta_i$ is even for all $3\leq{i}\leq{r_2}$.
Hence $\frac{\alpha_2}{\alpha_i}$ is odd, since the $\beta_i$ are all odd.
If moreover $\varepsilon(M)=0$ then $\frac{\alpha_1}{\alpha_2}$ is odd.
If $\varepsilon(M)\not=0$ then 
$\frac{\alpha_1}{\alpha_2}+n\alpha_1\varepsilon(M)$ is even.
Hence $\frac{\alpha_1}{\alpha_2}$ is again odd, 
and so $\alpha_1\varepsilon(M)$ is also odd.
In each case the converse is clear.

The second assertion holds since $\beta_i$ is odd for $1\leq{i}\leq{r_2}$ 
and $\frac{\alpha_1}{\alpha_i}$ is even for all $i>r_2$.
\end{proof}

We shall suppose for the remainder of this section that $\mathbb{Z}_{(2)}\otimes\tau_M$
is homogeneous of exponent $2^k>1$.

Suppose first that  $\mathbb{Z}_{(2)}\otimes\ell_M$ is even.
Then it is homogeneous and of even rank $\rho=2s$.
The diagonal entries of $L$ are all even and the off-diagonal entries 
are all odd.
If $k=1$ then $\mathbb{Z}_{(2)}\otimes\ell_M$ is hyperbolic, 
so we may assume that $k>1$ in the next theorem.

\begin{thm}
{\bf{A7}.}
%\begin{theorem}
%\label{A7}
Let $M=M(g;S)$ be a Seifert manifold such that the even cone point orders $\alpha_i$ all have the same $2$-adic valuation $k>1$.
Assume that either $\varepsilon(M)=0$ or $\alpha_1\varepsilon(M)$ is odd.
Let $t$ be the number of diagonal entries of $L$ which are divisible by $4$.
Then whether $\mathbb{Z}_{(2)}\otimes\ell_M$ is hyperbolic or not depends only on the images of $t$ and $\rho$ in $\mathbb{Z}/4\mathbb{Z}$.
\end{thm}

\begin{proof}
The linking pairing is even, by Lemma 6, and so $\rho$ is even.
We may reorder the basis of $\tau_M$ so that $L_{ii}\equiv0$ {\it mod} $(4)$,
for all $i\leq{t}$ and $L_{ii}\equiv2$ {\it mod} $(4)$ for $t<i\leq\rho$.
Let $t=4a+x$ and $\rho-t=4b+y$, where $0\leq{x,y}\leq3$.
Then
$E=\left(\smallmatrix{L_{11}}&{L_{12}}\\L+{21}&{L_{22}}\endsmallmatrix\right)$
is invertible.
We may partition $L$ as
$L=\left(\smallmatrix{E}&{F}\\F^{tr}&{G}\endsmallmatrix\right)$,
where $G$ is a $(\rho-2)\times(\rho-2)$-submatrix and $F$ is a 
$2\times(\rho-2)$-submatrix.
If we use
$J=\left(\smallmatrix{I_2}&{-E^{-1}F}\\0&{I_{\rho-2}}\endsmallmatrix\right)$
to obtain
$J^{tr}LJ=\left(\smallmatrix{E}&{0}\\0&{G'}\endsmallmatrix\right)$,
then $G'=G-F^{tr}E^{-1}F$.
The entries of $F$ are all odd and so the entries of $F^{tr}E^{-1}F$ are
all congruent to 2 {\it mod} $(4)$.
Therefore $G'\equiv{G}$ {\it mod} $(2)$, 
and so the off-diagonal entries of $G'$ are still odd, 
but the residues {\it mod} $(4)$ of the diagonal entries are changed.
An application of this process to $G'$ then restores the residue classes of the diagonal entries of the corresponding $(\rho-4)\times(\rho-4)$-submatrix.
Iterating this process, we find that
$\ell_M\cong(a+b)(E_0^k\perp{E_1^k})\perp\ell'$,
where $\ell'$ has rank $x+y$ and the off-diagonal entries for $\ell'$ are odd.
We also find that $\ell'\cong(x+y)E_0^k$, unless $\{x,y\}=\{1,3\}$ or $\{0,2\}$,
in which case $\ell'\cong{E_0^k}\perp{E_1^k}$ or $E_1^k$, respectively.
Since $E_0^k$ is hyperbolic, $2E_1^k\cong2E_0^k$ \cite {Wa64} and 
$E_1^k$ is not hyperbolic,
it follows that $\ell$ is hyperbolic if and only if either $a+b$ is even and 
$\{x,y\}\not=\{1,3\}$ or $\{0,2\}$, or if $a+b$ is odd and $\{x,y\}=\{1,3\}$ or $\{0,2\}$.
\end{proof}

It follows immediately from the calculations in the first section  that
\[
t=\#\{i\geq3\mid\frac{\alpha_2\beta_i+\alpha_i\beta_2}{2^k}\equiv0~mod~(4)\}+\delta,
\]
where $\delta=1$ if $\varepsilon(M)\not=0$ and $\beta_2+\alpha_2\varepsilon(M)\equiv0$
{\it mod} $(4)$, and $\delta=0$ otherwise.

In particular, if $t$ and $\rho$ are divisible by 4 then $\ell_M$ is hyperbolic.

Lemmas {A2} and {A6} also imply that if $\mathbb{Z}_{(2)}\otimes\tau_M$
is homogeneous of exponent $2^k$ then $\mathbb{Z}_{(2)}\otimes\ell_M$ is odd
if and only if either
\begin{enumerate}
\item$\varepsilon(M)=0$, $2^{-k}\alpha_i$ is even for $i=1$ and 2, 
and is odd for $2<i\leq{r_2}$; or
\item$\alpha_1\varepsilon(M)$ and $2^{-k}\alpha_i$ are odd for $1<i\leq{r_2}$,
and either $2^{-k}\alpha_1$ or $r_2$ is even.
\end{enumerate}
Odd forms on homogeneous 2-groups can be diagonalized.
In the present situation, this follows easily from the next lemma.

\begin{lem}
{\bf{A8}.}
%\label{A8}
Let $\ell$ be an odd linking pairing on $N=(\mathbb{Z}/2^k\mathbb{Z})^2$.
Then $\ell$ is diagonalizable.
\end{lem}

\begin{proof}
Let $e,f$ be the standard basis for $N$.
Since $\ell$ is odd we may assume that $\ell(e,e)=[2^{-k}a]$,
where $a$ is odd.
Let $\ell(e,f)=[2^{-k}b]$ and $\ell(f,f)=[2^{-k}d]$.
(Then $b$ is even and $d$ is odd, or vice versa, by nonsingularity of the pairing.)
Let $f'=-a^{-1}be+f$. Then $\ell(e,f')=0$ and $\ell(f',f')=[2^{-k}d']$,
where $d'\equiv{d-a^{-1}b^2}$ {\it mod} $(2^k)$.
Therefore $\ell\cong
\ell_{\frac{a}{2^k}}\perp\ell_{\frac{d'}{2^k}}$.
\end{proof}

Note that if $b\equiv0$ {\it mod} (4) then
$\ell\cong\ell_{\frac{a}{2^k}}\perp\ell_{\frac{d}{2^k}}$.

Suppose now that $\mathbb{Z}_{(2)}\otimes\ell_M$ is odd and $\varepsilon(M)=0$.
Then the diagonal entries of $L$ are odd and the off-diagonal elements are 
odd multiples of $2^{-k}\alpha_1$.
We may assume also that $r_2\geq4$, for otherwise
$\mathbb{Z}_{(2)}\otimes\tau_M$ is cyclic.
We may partition $L$ as
$L=\left(\smallmatrix{E}&{F}\\F^{tr}&{G}\endsmallmatrix\right)$,
where $E\in{GL(2,\mathbb{Z}/2^k\mathbb{Z})}$, 
$F$ is a $2\times(r_1-4)$-submatrix with even entries
and $G$ is a $(r_2-4)\times(r_2-4)$-submatrix.
Let $J=\left(\smallmatrix{I_2}&{-E^{-1}F}\\0&{I_{r_2-4}}\endsmallmatrix\right)$.
Then $\det(J)=1$ and 
$J^{tr}LJ=\left(\smallmatrix{E}&{0}\\0&{G'}\endsmallmatrix\right)$,
where $G'=G-F^{tr}E^{-1}F$.
The columns of $F$ are proportional, 
and the ratio $\frac{u_3}{u_4}$ is odd.
Since the entries of $F$ are odd multiples of $2^{-k}\alpha_1$ 
and since $E-I_2$ has even entries, $G'\equiv{G}$ {\it mod} $(8)$.
Iterating this process, we may replace $L$ by a block-diagonal matrix, 
where the blocks are all $2\times2$ or $1\times1$, 
and are congruent {\it mod} $(8)$ to the corresponding blocks of $L$.
Each such $2\times2$ block is diagonalizable, by Lemma {A8}, 
and so we may easily represent $\ell_M$ as an orthogonal sum 
of pairings of rank 1.

If $\varepsilon(M)\not=0$ then
$\ell_M(q_i',s)=[2^{-k}\frac{\beta_i}{u_i}]$ and $\ell_M(s,s)=[2^{-k}z]$,
where $\beta_i$, $u_i$ and $z$ are odd, and we first replace each $q_i'$ by $\widetilde{Q}_i=q_i'-z^{-1}u^{-1}\beta_is$.
We then see that $\mathbb{Z}_{(2)}\otimes\ell_M\cong\ell_{2^{-k}z}\perp\widetilde{\ell}$,
where the matrix for $\widetilde{\ell}$ has odd diagonal entries and even off-diagonal entries,
and we may continue as before.

\begin{thm}
{\bf{A9}.}
%\begin{theorem}
%\label{A9}
Let $\ell$ be a linking pairing on $(\mathbb{Z}/2^k\mathbb{Z})^\rho$.
Then there is a Seifert manifold $M=M(0;S)$ such that the cone point orders 
are all powers of $2$ and $\ell_M\cong\ell$.
We may have either $\varepsilon(M)=0$ or $\varepsilon(M)\not=0$.
\end{thm}

\begin{proof}
Suppose first that $\ell$ is even.
Then $\rho$ is also even, and $\ell\cong(E_0^k)^{\frac\rho2}$ or 
$(E_0^k)^{\frac\rho2-1}\perp{E_1^k}$.
We shall chose $r$ to be either $\rho+2$ or $\rho+1$, 
in order to construct examples with $\varepsilon(M)=0$ or $\not=0$.

Let $S=((2^k,\beta_1),\dots,(2^k,\beta_r))$ with $\beta_i=(-1)^i$ for $1\leq{i}\leq{r}$.
Then $\varepsilon(M)=-\frac1{2^k}$ if $r$ is odd and
$\varepsilon(M)=0$ if $r$ is even, and $\ell_M\cong(E_0^k)^r$.

If $\rho\equiv2$ {\it mod} $(4)$ let $\beta_1=-3$,
$\beta_2=\beta_3=1$ and $\beta_i=(-1)^i$ for $4\leq{i}\leq{r}$.
If $\rho\equiv0$ {\it mod} $(4)$ let $\beta_1=-5$, $\beta_2=\dots=\beta_5=1$
and $\beta_i=(-1)^i$ for $6\leq{i}\leq{r}$.
In each case, 
$\varepsilon(M)=-\frac1{2^k}$ if $r$ is odd and $\varepsilon(M)=0$ 
if $r$ is even, 
and $\ell_M\cong(E_0^k)^{\frac\rho2-1}\perp{E_1^k}$.

Now suppose that $\ell$ is odd.
Then $\ell\cong\perp_{i=1}^\rho\ell_{w_i}$ where $w_i=2^{-k}b_i$ for $1\leq{i}\leq\rho$.
To construct an example with $\varepsilon(M)=0$ we set $r=\rho+2$ and 
$S=((\alpha_1,\beta_1),\dots,(\alpha_r,\beta_r))$, where $\alpha_1=\alpha_2=2^{k+2}$, $\alpha_i=2^k$ for $3\leq{i}\leq\rho$,
$\beta_2=1,$ $\beta_i=3b_i$ for $3\leq{i}\leq{r}$ and 
$\beta_1=-1-4\Sigma_{i=3}^r\beta_i$.
Then $\varepsilon(M)=0$ and $\ell_M\cong\ell$.
(Here we may use the observation following Lemma {A8}.)

To construct an example with $\varepsilon_S\not=0$ we set $r=\rho+1$.
Here we must take into account the change of basis suggested in the paragraph before the theorem.
Suppose first that some $b_i\equiv\pm3$ {\it mod} $(8)$.
We may then arrange that $z=3$ and the matrix for $\widetilde\ell$ 
is congruent {\it mod} $(4)$ to a diagonal matrix.
After reordering the summands, and allowing for a change of orientation, 
we may assume that $b_1\equiv3$ {\it mod} $(8)$.
The we let
$S=((\alpha_1,\beta_1),\dots,(\alpha_r,\beta_r))$, where $\alpha_1=2^{k+2}$, 
$\alpha_i=2^k$ for $2\leq{i}\leq{r}$, $\beta_2=1$, $\beta_i=4-b_i$ for $3\leq{i}\leq{r}$ 
and $\beta_1=1-4\Sigma_{i=2}^r\beta_i$.
Then $\varepsilon(M)=-\frac1{2^{k+2}}$ and $\ell_M\cong\ell$.

Finally, suppose that $b_1\equiv\pm1$ {\it mod} $(8)$ for all $i$.
If $\rho=1$ let $S=((2^{k+1},1),(2^k,1))$.
If $\rho=2$ and $b_1\equiv-b_2\equiv1$ {\it mod} $(8)$,
let $S=((2^{k+1},1),(2^k,1, (2^k,-1))$.
Otherwise we may assume that $b_1\equiv{b_2}$ {\it mod} $(8)$.
But then $\ell_{w_1}\perp\ell_{w_2}\cong2\ell_{w'}$,
where $w'=2^{-k}3$, and so we may use the construction in the preceding paragraph.
\end{proof}

The manifolds constructed in this section are either
$\mathbb{H}^2\times\mathbb{E}^1$-manifolds (if $\varepsilon(M)=0$),
or $\widetilde{\mathbb{SL}}$-manifolds (if $\varepsilon(M)\not=0$),
excepting only the flat manifold $M(0;(2,-1),(2,1),(2,-1),(2,1))$
and the $\mathbb{S}^3$-manifolds $M(0;(2,1),(2,1),(2,\beta))$ 
and $M(0;(4,1),(2,1),(2,\beta))$.

\section*{The general case}

Every linking pairing $\ell$ on a finite abelian group is realized by some oriented 3-manifold \cite{KK80}.
The next theorem suffices to show that there are pairings which cannot be realized by Seifert manifolds.
We shall then show that there are no further constraints.

\begin{thm}
{\bf{A10}.}
%\begin{theorem}
%\label{A10}
Let $M=M(g;S)$ be a Seifert manifold, and let 
$S(2)=((\alpha_1,\beta_1),\dots, (\alpha_t,\beta_t))$ be the terms of the Seifert data
with even cone point orders $\alpha_i$.
Then $\mathbb{Z}_{(2)}\otimes\ell_M$ has a non-trivial even component
if and only if $e=\max\{i\leq{t}\mid\frac{\alpha_1}{\alpha_i}~is~odd\}\geq3$.
The even component has exponent $\alpha_1$, 
and all the other components are diagonalizable.
If $e$ is even and $\varepsilon(M)\not=0$ then the image of $s$ in 
$\mathbb{Z}_{(2)}\otimes\tau_M$ generates a cyclic component of maximal exponent.
If $e=2$ then $\alpha_1\alpha_2\varepsilon_S$ is divisible by $4\alpha_3$ in 
$\mathbb{Z}_{(2)}$.
\end{thm}

\begin{proof}
The block-diagonalization process of \S3 does not change the parity of the entries in the diagonal blocks $D_i$.
The first assertion then follows easily from the calculations of \S1.
(More precisely, if $e$ is odd then $\alpha_1\varepsilon(M)$ is odd and
$\mathbb{Z}_{(2)}\otimes\tau_M$ has exponent $\alpha_2$
If moreover $e=1$ then $\ell_M(s,s)$ is odd and (so) the homogeneous components 
are all odd.
If $e>1$ and is odd then $\ell_M(s,s)$ is even and the component of maximal exponent is even.
If $e$ is even and $\varepsilon(M)=0$ then the componet of maximal exponent is non-tricial (and even) if and only if $e>2$.
If $e$ is even and $\varepsilon(M)\not=0$ then $s$ has order $>\alpha_1$, 
and $\ell_M(s,s)$ is odd.
If moreover $e>2$ then the component of exponent $\alpha_1$ is even and has
rank $e-2$.
In each case, all other components are odd.)

The final assertion is clear if $\varepsilon(M)=0$ and follows by elementary arithmetic otherwise, for then $\alpha_1\varepsilon(M)$ is even.
\end{proof}

This provides  a criterion for recognizing
pairings which are not realizable by Seifert manifolds 
that is independent of the choice of generators.
Let $N\cong{N'}\oplus{N''}$ be a finite abelian group,
where the homogeneous summands of $N'$ have exponent $\geq4$ and $2N''=0$,
and let $\ell=\ell'\perp\ell''$,
where $\ell'$ is a pairing on $N'$ and 
and $\ell''$ is a hyperbolic pairing on $N''$.
Then $\ell(x,x)=0$ for all $x\in{N}$ such that $2x=0$.
In particular, if $N'$ is not cyclic then $\ell$ is 
not realizable by a Seifert manifold.

For example, if $M$ is the
$\mathbb{N}il^3$-manifold $M(0;(2,1),(2,1),(2,1),(2,-1))$
then $\ell_M\cong\ell_{\frac14}\perp{E_0^1}$,
but this pairing is not realizable by a Seifert manifold 
with $\varepsilon(M)=0$.
The pairing $E_0^2\perp{E_0^1}$ is not realized by any Seifert manifold 
with orientable base orbifold.
It is however the linking pairing of $M(-2;((2,1),(2,1),(2,1),(2,1)))$.

If $\mathbb{Z}_{(2)}\otimes\tau_M$ has exponent divisible by 16 
and a direct summand of order $2$ then
there are cone point orders $\alpha_1$ and $\alpha_m$ such that
$\alpha_1$ is divisible by $4$ and $\alpha_m=2u_m$ with $u_m$ odd,
by Lemma \ref{CH-lem3.4}.
It then follows from Theorem \ref{CH-thm3.7} that
$\mathbb{Z}_{(2)}\otimes\ell_M\cong\ell'\perp\ell_{\frac12}$,
for some pairing $\ell'$.
In particular, 
$E_0^4\perp{E_0^1}$ is not realized by any orientable Seifert fibred 3-manifold
at all.

The condition of Theorem {A10} are the only constraints on the class 
of linking pairings realized by Seifert manifolds with $\varepsilon(M)=0$.

\begin{thm}
{\bf{A11}.}
%\begin{theorem}
%\label{A11}
Let $\ell$ be a linking pairing on a finite abelian group $A$.
Then there is a Seifert manifold $M=M(0;S)$ with $\varepsilon(M)=0$ 
and such that $\ell_M\cong\ell$ if and only if the components of the 
$2$-primary summand of $\ell$ other than the component of 
maximal exponent are all odd.
\end{thm}

\begin{proof}
The condition is necessary, by Theorem {A10}.
By Lemma {A1} and Theorem {A4} it shall suffice to assume that $A$ is 2-primary,
and to realize $\ell$ by a Seifert manifold $M(0;S)$ with $\varepsilon(M)=0$ 
and such that all the cone point orders are powers of 2.

Suppose that $\ell=\perp_{j=1}^t\ell_j$, where $\ell_j$ is a pairing on
$(\mathbb{Z}/2^{k_j}\mathbb{Z})^{\rho_j}$,
for $1\leq{j}\leq{t}$,
$k_1>\dots>k_t>0$ and $\ell_j$ is odd for  $2\leq{j}\leq{t}$.
We may assume that $t>1$ and $\rho_j>0$ for $1\leq{j}\leq{t}$,
since the homogeneous case is covered by Theorem {A9}.

We must have two cone points of order at least $2^{k_1}$,
and $\rho_j$ further cone points of order $2^{k_j}$, for $1\leq{j}\leq{t}$.
We shall set $\beta_2=1$ and choose the $\beta_i$ with $i>2$ compatibly with $\ell$,
essentially by induction on $t$.
If we make thse choices then we must have $\beta_1=-\alpha_1\Sigma_{i\geq2}\frac{\beta_i}{\alpha_i}$.
(Note that the presentation of $\ell_M$ given in \S1 above does not invoke $\beta_1$ when $\varepsilon(M)=0$.)

Suppose first that $\ell_1$ is even.
Then $\rho_1$ is even and we must have $\alpha_1=\alpha_2=2^{k_1}$ also.
If $\ell_1$ is hyperbolic let $\beta_i=(-1)^i$ for $3\leq{i}\leq\rho_1+2$; 
if $\ell_1$ is even but not hyperbolic and $\rho_1\equiv2$ {\it mod} $(4)$ let $\beta_3=1$ and $\beta_i=(-1)^i$ for $4\leq{i}\leq\rho_1+2$;
and if $\ell_1$ is even but not hyperbolic and $\rho_1\equiv0$ {\it mod} $(4)$ let $\beta_3=\beta_4=\beta_5=1$ and $\beta_i=(-1)^i$ for $6\leq{i}\leq\rho_1+2$.

When $p=2$ the block-diagonalization process of \S3 does not change the parity of the entries of the blocks $B_m$ and $D_n$.
In our situation this process allows a more refined reduction.
The block $DA=D_1$ has diagonal entries $-\beta_i(1+\beta_i)$ and off-diagonal entries $-\beta_i\beta_j$, for $3\leq{i,j}\leq\rho_1+2$.
Let $\Delta=-\mathrm{diag}[\beta_3,dots,\beta_{\rho_1+2}]$ and $N=D-\Delta$.
Then $\Delta\equiv{I}$ {\it mod} $(2)$, 
since the $\beta_i$s are all odd, so $\Delta^2\equiv{I}$ {\it mod} $(4)$, and
\[
N^2\equiv-(\Sigma_{i=3}^{\rho_1+2}\beta_i^2)N\equiv\rho_1N\quad{mod}~(4).
\]
In particular $N^2\equiv0$ {\it mod} $(2)$, and so $D^{-1}\equiv{D}$  {\it mod} $(2)$,.
Since
\[
[B_m^{tr}DB_n]_{pq}=-\beta_p\beta_q((\Sigma_{i=3}^{\rho_1+2}\beta_i^2)^2+
\Sigma_{i=3}^{\rho_1+2}\beta_i^3)
\]
is even, for $3\leq{p,q}\leq2+\Sigma_{j=1}^t\rho_j$,
the first step of the reduction does not change the images of the complementary blocks 
{\it mod} $(4)$.
Truncating  a geometric series gives $D^{-1}\equiv{D}+(\rho_1-2)N$ {\it mod} $(4)$.
Hence the change {it mod} $(8)$ depends only on the residues {\it mod} $(4)$ of the $\beta_i$s, for $3\leq{i}\leq\rho_1+2$, since $-4\equiv4$ {\it mod} $(8)$.
If $k_1-k_2\geq2$ this step does not change the images {\it mod} $(8)$ at all.

The subsequent blocks have odd diagonal matrices and even off-diagonal elements.
Nevertheless a similar reduction applies,
and further changes depend only on the $\beta_i$s already determined.
(If $k_j-k_{j+1}\geq2$ for all $j$ then they depend only on the ranks 
of the homogeneous terms.)
It is clear that we may then choose the numerators $\beta_i$ to realize
all the pairings $\ell_j$.

If $\ell_1$ is odd then $\ell_1=\perp_{i=1}^{\rho_1}\ell_{w_i}$,
where $w_i=2^{-k+1}b_i$ for $1\leq{i}\leq\rho$.
Let $\alpha_1=\alpha_2=2^{k_1+2}$,
and $\beta_i=3b_i$ for $3\leq{i}\leq\rho_1+2$.
Then we may continue as before.
\end{proof}

A similar argument applies when $\varepsilon(M)\not=0$.
However we shall only sketch the argument in this case.

\begin{thm}
{\bf{A12}.}
%\begin{theorem}
%\label{A12}
Let $\ell$ be a linking pairing on a finite abelian group $A$.
Then there is a Seifert manifold $M=M(0;S)$ with $\varepsilon(M)\not=0$ 
and such that $\ell_M\cong\ell$ if and only if the $2$-primary
components of $\ell$ satisfy the conditions of Theorem {A10}.
\end{thm}

\begin{proof}
Suppose first that $A$ is 2-primary, and that $\ell=\perp_{j=1}^t\ell_j$,
where $\ell_j$ is a pairing on $(\mathbb{Z}/2^{k_j}\mathbb{Z})^{\rho_j}$,
for $1\leq{j}\leq{t}$.
Suppose also that $k_1>\dots>k_t>0$ and $\ell_j$ is odd, for $3\leq{j}\leq{t}$,
and that either $\ell_2$ is odd or $\rho_1=1$ and $\ell_2$ is even.
We may again assume that $t>1$ and $\rho_j>0$, for $1\leq{j}\leq{t}$,
since the homogeneous case is covered by Theorem {A9}.

The cone point orders are essentially determined by Theorem {A10}.
If $\ell_1$ is even we must have $\rho_1+1$ cone points of order $2^{k_1}$ 
and $\rho_j$ cone points of order $2^{k_j}$, for $2\leq{j}\leq{t}$.
If $\ell_2$ is even then $\rho_1=1$ and we must have $\rho_2+2$ cone points of order $2^{k_2}$ and $\rho_j$ cone points of order $2^{k_j}$, for $3\leq{j}\leq{t}$.
If the components $\ell_j$ are all odd then we may choose one cone point of order $\alpha_1>2^{k_1}$ and $\rho_j$ cone points of order $2^{k_j}$, for $1\leq{j}\leq{t}$.
(If $\rho_1=1$ and $k_1-k_2\geq2$ we could instead choose two cone points of order $2^{k_0}$, where $k_1>k_0>k_2$, and $\rho_j$ cone points of order $2^{k_j}$,
for  $2\leq{j}\leq{t}$.
However, we shall not use this option.)

We then have $\varepsilon(M)\frac{b}d$, 
where $b$ is odd and $d=\alpha_1$ if $\ell_2$ is odd and $d=2^{2k_2-k_1}$ 
if $\ell_2$ is even.
Let $\beta_2=1$ and choose the $\beta_i$s with $i>2$ compatibly with $\ell$ 
and our choice for $\varepsilon(M)$.
Let $\Sigma'=\Sigma_{i\geq2}\frac{\beta_i}{\alpha_i}$.
Then we must have $\beta_1=-\alpha_1(\varepsilon(M)+\Sigma')$.
(In each case this is odd.)

From here the strategy is as in Theorem {A11}, 
and we shall not give further details for this part of the construction.

In order to construct $M$ we may assume that $b_1=1$, 
and hence that $\varepsilon(M)=\frac1{2^k}$.
However, when $A$ also has odd primary summands we cannot use
Lemma {A1} to reduce to the 2-primary case.
In order to extend the argument of Theorem {A5} to pairings with 2-primary summands, 
it is convenient to allow $b$ to be the {\it inverse\/} of an odd integer.

Let $P$ be the finite set of odd primes dividing the order of $A$.
Let $\Pi=\Pi_{p\in{P}}p$ and $E$ be the product of the exponents of the odd primary summands of $A$,
and let $\varepsilon(M)=\frac1{dE\Pi}$.

If $\ell_2$ is odd we may define Seifert data $S(p)$ as in Theorem {A5}, 
for each $p\in {P}$.
We then construct Seifert data 
$S(2)=((\alpha_2,\beta_2),\dots,(\alpha_{\rho(2)},\beta_{\rho(2)}))$
with $\alpha_i$ a power of 2, for $i>1$, as above.
However we now replace $\alpha_1$ by $\widetilde\alpha=\alpha_1E\Pi$.

If $\ell_2$ is even then for each $p\in{P}$ we have $\widetilde\alpha\varepsilon(M)\equiv\frac{\alpha_1}d$ {\it mod} $(p)$,
and we must modify the choices of some of the $\beta_i^{(p)}$s for the
Seifert data $S(p)$ slightly.

Finally, let $\widetilde\beta=
-\widetilde\alpha(\varepsilon(M)+\Sigma'+
\Sigma_{p\in{P}}\Sigma_{i=1}^{t(p)}\frac{\beta_i^{(p)}}{\alpha_i^{(p)}})$, 
where $d$ and $\Sigma'$ are defined as before, 
in terms of the Seifert data of the cone points of even order.
(Note that $\widetilde\beta$ is odd and $\widetilde\beta\equiv-1$ {\it mod} $(p)$, 
for $p\in{P}$, and $\varepsilon(M)=\frac{2^m}{\widetilde\alpha}$, for some $m\geq0$.)
\end{proof}

%% file: eappB.tex
\chapter*{Appendix B.  Homologically balanced nilpotent groups}

In this appendix we shall present what we know about homologically balanced nilpotent groups.
If $G$ is such a group it can be generated by 3 elements 
\cite[Theorem 2.7]{Lub83}, 
and if $\beta_1(G;\mathbb{Q})=3$ then $G\cong\mathbb{Z}^3$ \cite{Hi22}.

\section*{Unipotent actions}

We shall extend the term ``unipotent",
to say that an action $\alpha:G\to{Aut}(A)$ is unipotent if
$\alpha(g)$ is unipotent for all $g\in{G}$.

\begin{lem}
{\bf B1.}
\label{filter}
Let $N$ be a finitely generated nilpotent group which acts unipotently
on a finitely generated abelian group $A$,
and let $\mathfrak{n}$ be the augmentation ideal of $\mathbb{Z}[N]$.
Then $A$ has a finite filtration $A=A_1>\dots>A_k=A^N>A_{k+1}=0$
by $\mathbb{Z}[N]$-submodules,
where $A^N$ is the fixed subgroup and $\mathfrak{n}A_i\leq{A_{i+1}}$, 
for $i\leq{k}$.
\end{lem}

\begin{proof}
We induct on the length of the upper central series of $N$.
The centre $\zeta{N}$ is a nontrivial abelian group which acts unipotently
on $A$, and it is easy to see that $A^{\zeta{N}}\not=0$.
The quotient $N/\zeta{N}$ acts unipotently on each of $A^{\zeta{N}}$ and 
$\overline{A}=A/A^{\zeta{N}}$, and so these each have such filtrations,
by the inductive hypothesis. The preimages of the filtration
of $\overline{A}$ in $A$ combine with the filtration of $A^{\zeta{N}}$ 
to give the required filtration.
\end{proof}

It is easy to see that the product of commuting unipotent automorphisms
is unipotent.
This observation extends to show that an action of a nilpotent group 
$N$ is unipotent if $N$  is generated by elements which act unipotently.

We shall find the following notion useful in many of our arguments.
Let  $G$ be a  group and $F$ a field.
Then an $F[G]$-module  $V$ is {\it canonically subsplit\/} if it contains a nontrivial direct sum of $F[G]$-submodules.
If $G$ acts unipotently on $V$ and $V$ is canonically subsplit then 
the subspaces of the summands fixed by $G/K$ are non-trivial,  
by Lemma B1,
%\ref{filter}, 
and so the subspace $V^G$ fixed by $G$ has dimension $>1$.

\begin{lem}
{\bf B2.}
\label{cycad}
Let $A$ be a finitely generated abelian group and $p$ a prime
such that $A$ has non-trivial $p$-torsion and $\dim_{\mathbb{F}_p}A/pA>1$.
If $p$ is odd or if $p=2$ and $A$ has no $\mathbb{Z}/2\mathbb{Z}$ 
summand then $H_2(A;\mathbb{F}_p)$ and 
$H^2(A;\mathbb{F}_p)$ are each canonically subsplit with respect to the natural action of (subgroups of) $\mathrm{Aut}(A)$.
\end{lem}

\begin{proof} 
Let $W=(A/pA)\wedge(A/pA)$ and 
$A^*=Hom (A;\mathbb{F}_p)=H^1(A;\mathbb{F}_p)$.

Then there is a natural splitting
$H_2(A;\mathbb{F}_p)=W\oplus{Tor(A,\mathbb{F}_p)}$
if $p$ is odd \cite[Chapter V.6]{Br},
or if $p=2$ and $A$ has no $\mathbb{Z}/2\mathbb{Z}$ summand \cite{IZ18}.
There is also a natural epimorphism
$\theta:H^2(A;\mathbb{F}_p)\to{Hom(W,\mathbb{F}_p)}$,
with kernel isomorphic to $Ext(A;\mathbb{F}_p)$ 
\cite[Exercises IV.3.8 and V.6.5]{Br}.

If $p$ is odd then cup product induces a monomorphism 
$c_A:A^*\wedge{A^*}\to{H^2(A;\mathbb{F}_p)}$, 
since $A$ is abelian.
If $p=2$ then cup product defines a homomorphism 
from $A^*\odot{A^*}$ to $H^2(A;\mathbb{F}_2)$.
Since $A$ has no $\mathbb{Z}/2\mathbb{Z}$ summand, 
$Sq(a)=a\cup{a}=0$ for all $a\in{A^*}$, 
and so cup product again induces a monomorphism 
$c_A:A^*\wedge{A^*}\to{H^2(A;\mathbb{F}_p)}$ \cite{Hi87}.
It is easily seen from the formulae in \cite{Br}
that $\theta\circ{c_A}$ is an isomorphism,
and so $H^2(A;\mathbb{F}_p)$ is naturally isomorphic to
$(A^*\wedge{A^*})\oplus{Ext(A;\mathbb{F}_p)}$.

The summands are all non-trivial,
since $A$ has nontrivial $p$-torsion  and $A/pA$ is not cyclic.
Thus $H_2(A;\mathbb{F}_p)$ and 
$H^2(A;\mathbb{F}_p)$ are each canonically subsplit.
\end{proof}

The case when $A$ has a summand of exponent 2 seems more complicated, 
and we consider only the cohomology.

\begin{lem}
{\bf B3.}
\label{cycad2}
Let $A$ be a finitely generated abelian group with a nontrivial summand of exponent $2$ and such that $\dim_{\mathbb{F}_2}A/2A>1$.
Suppose that a finitely generated nilpotent group $N$ acts unipotently on $A$.
Then $\dim_{\mathbb{F}_2}H^2(A;\mathbb{F}_2)^N>1$.
\end{lem}

\begin{proof}
We may assume that $A\cong{B}\oplus{E}$,
where $E\cong(\mathbb{Z}/2\mathbb{Z})^s\not=0$
and $B$ has no summand of order 2. 
The subspace $B^*$ of $A^*=Hom(A,\mathbb{F}_2)=H^1(A;\mathbb{F}_2)$ 
consisting of homorphisms which factor through homomorphisms to 
$\mathbb{Z}/4\mathbb{Z}$ is canonical.
Clearly $B^*\cong{Hom(B,\mathbb{F}_2)}$ and 
$A^*/B^*\cong{E^*}=Hom(E,\mathbb{F}_2)$.
Hence $A^*\cong{B^*}\oplus {E^*}$, but this splitting is not canonical.
Cup product induces a homomorphism 
$c_A:A^*\odot{A^*}\to{H^2(A;\mathbb{F}_2)}$,
with kernel $2A/4A\cong{B^*}$,
since $A$ is abelian  \cite{Hi87}.
There is also a natural squaring map 
$Sq:A^*\to{H^2(A;\mathbb{F}_2)}$ with kernel $B^*$.

If $B=0$ then $A$ is an elementary 2-group and $A^*=E^*$,
and $c_A$ is a monomorphism.
Let $A_1>\dots>A_{k+1}=0$ be a filtration of $A^*$ by 
$\mathbb{F}_p[N]$-submodules, as in Lemma \ref{filter}.
Then $A_k\odot{A_k}$ is fixed by $N$.
If $\dim_{\mathbb{F}_2}A_k>1$  then
$\dim_{\mathbb{F}_2}A_k\odot{A_k}\geq3$.
If $A_k$ has dimension 1,  and is generated by $b$ then
$b\odot{b}$ is fixed by $N$.
If $a\in{A_{k-1}}$ then each element of $N$ either fixes $a$ or 
sends it to $a+b$.
In either case $a\odot(a+b)$ is fixed by $N$.
Since $\dim_{\mathbb{F}_2}A_{k-1}\geq2$ the subspace
generated by $\{a\odot(a+b)\mid{a}\in{A_{k-1}}\}\cup\{b\odot{b}\}$
is fixed by $N$,
and so $\dim_{\mathbb{F}_2}H^2(A;\mathbb{F}_2)^N>1$.

The images of $B^*\odot{A^*}$ and $Sq(A^*)=Sq(E^*)$ are
canonical submodules of $H^2(A;\mathbb{F}_2)$,
with trivial intersection.
Hence they are invariant under the action 
of automorphisms of $A$,
and so if $B\not=0$ then we again have
$\dim_{\mathbb{F}_2}H^2(A;\mathbb{F}_2)^N>1$.
\end{proof}

\begin{cor*}
\label{cycboth}
Let $A$ be a finitely generated abelian group, $\psi$ be  a unipotent
automorphism of $A$, and $p$ be a prime. 
If $A$ has non-trivial $p$-torsion and $\dim_{\mathbb{F}_p}A/pA>1$ 
then $\dim_{\mathbb{F}_p}\mathrm{Ker}(H_2(\psi;\mathbb{F}_p)-I)=
\dim_{\mathbb{F}_p}\mathrm{Ker}(H^2(\psi;\mathbb{F}_p)-I)>1$.
\end{cor*}

\begin{proof}
Let $N$ be the cyclic subgroup of $Aut(A)$ generated by $\psi$.
We shall write $H_i(\psi)$ and $H^j(\psi)$ instead of 
$H_i(\psi;\mathbb{F}_p)$ and $H^j(\psi;\mathbb{F}_p)$,
for simplicity of notation.
Then $H_i(A;\mathbb{F}_p)^N=\mathrm{Ker}(H_i(\psi)-I)$
and $H^j(A;\mathbb{F}_p)^N=\mathrm{Ker}(H^j(\psi)-I)$, for any $i$.
If $\varphi$ is an endomorphism of a finite dimensional vector space $V$
then $\dim\mathrm{Cok}(\varphi)=\dim\mathrm{Ker}(\phi)$
and if $\varphi^*$ is the induced endomorphism of the dual vector space 
$V^*$ then $\varphi^*$ and $\varphi$ have the same rank.
Hence the corollary follows from Lemma B2,
%\ref{cycad},
 if $p$ is odd,
and from Lemma B3,
%\ref{cycad2}, 
if $p=2$.
\end{proof} 

It does not seem obvious that 
$\dim_{\mathbb{F}_p}H_2(A;\mathbb{F}_p)^N$ and
$\dim_{\mathbb{F}_p}H^2(A;\mathbb{F}_p)^N$ are equal
when $N$ is not cyclic.

If $\dim_{\mathbb{F}_p}A/pA\geq4$ then the restriction of 
$H_2(\psi;\mathbb{F}_p)-I$ to $(A/pA)\wedge(A/pA)$ 
has kernel of dimension $>1$, 
and so $\dim_{\mathbb{F}_p}\mathrm{Ker}(H_2(\psi;\mathbb{F}_p)-I)>1$. 
In \cite{Hi22} a related observation for free abelian groups of rank 
$\geq4$ is used to show that if $G$ is a metabelian nilpotent group
with $h(G)>4$ then $\beta_2(G;\mathbb{Q})>\beta_1(G;\mathbb{Q})$.
One of the difficulties in extending the approach of this paper 
to more general nilpotent groups is the lack of  an analogue 
to the above lemmas for non-abelian $p$-groups.

\section*{Virtually cyclic groups: $h\leq1$}

A finitely generated nilpotent group $G$ is finite if and only if 
$\beta_1(G;\mathbb{Q})=0$ if and only if $h(G)=0$.
The Sylow subgroups of a finite nilpotent group $G$ are characteristic,
and $G$ is the direct product of its Sylow subgroups \cite[5.2.4]{Rob}.
It then follows from the K\"unneth Theorem that $H_2(G)=0$ 
if and only if $H_2(P)=0$ for all such Sylow subgroups $P$.
On the other hand, it is not clear that if $H_2(G)=0$ then
$G$ must have a balanced presentation, 
even if this is so for each of its Sylow subgroups.
(The examples in \cite{HW85} of finite perfect groups with trivial multiplicator
but without balanced presentations are not nilpotent.)

If $p$ is an odd prime then every 2-generator metacyclic $p$-group 
$P$ with $H_2(P)=0$ has a balanced presentation
\[
\langle{a,b}\mid{b^{p^{r+s+t}}=a^{p^{r+s}}},~bab^{-1}=a^{1+p^r}\rangle,
\]
where $r\geq1$ and $s,t\geq0$.
(The order of such a group is $p^{3r+2s+t}$.)
There are other metacyclic 2-groups and other $p$-groups 
with 2-generator balanced presentations.
A handful of 3-generated $p$-groups (for $p=2$ and 3)
are also known to have balanced presentations. 
(See \cite{HNO'B} for a survey of what was known in the mid-1990s.)

If  $T$ is one of the 2-generator metacyclic $p$-groups of \S2 
then $H^2(T;\mathbb{F}_p)$ has no canonically split subspace, 
and such groups do arise as the torsion subgroups of homologically balanced nilpotent groups $G$ with $h(G)=1$.
(See the final paragraphs of \S5 below.)
Can we at least use such an argument to show that the torsion subgroup
must be homologically balanced?

We include the following simple lemma as some of the observations 
are not explicit in our primary reference \cite{Rob}.

\begin{lem}
{\bf B4.}
\label{2end}
Let $G$ be a finitely generated nilpotent group,
and let $T$ be its torsion subgroup. 
Then the following are equivalent
\begin{enumerate}
\item$\beta_1(G;\mathbb{Q})=1$;
\item$h(G)=1$;
\item $G/T\cong\mathbb{Z}$;
\item$G\cong{T}\rtimes_\psi\mathbb{Z}$, 
where $\psi$ is an automorphism of $T$.
\end{enumerate}
\end{lem}

\begin{proof}
In each case $G$ is clearly infinite, 
and so there is an epimorphism $f:G\to\mathbb{Z}$, 
with kernel $K$, say.
Since $G$ is finitely generated, so is $K$.
If $\beta_1(G;\mathbb{Q})=1$ then $K$ is finite, by Lemma \ref{wang app}.
If $h(G)=1$ then $h(K)=0$, so $K$ is again finite.
In each case, $K=T$ and $G/T\cong\mathbb{Z}$.
If $G/T\cong\mathbb{Z}$ and $t\in{G}$ represents a generator
of $G/T$ then conjugation by $t$ defines an automorphism $\psi$
of $T$, and $G\cong{T}\rtimes_\psi\mathbb{Z}$.
Finally, it is clear that (4) implies each of (1) and (2).
\end{proof}

We could also describe the groups considered on this lemma as
the nilpotent groups which are virtually $\mathbb{Z}$,
and as the nilpotent groups with two ends.

In the next lemma we do not assume that $G$ is nilpotent.

\begin{lem}
{\bf B5.}
\label{hbh1}
Let $G\cong{T}\rtimes_\psi\mathbb{Z}$ be a homologically balanced group, where $T$ is finite.
Then $H_2(G)$ is a finite cyclic group, 
and its order is divisible by the order of the torsion subgroup of $G^{ab}$.
\end{lem}

\begin{proof}
It is immediate from the Wang sequence for the integral homology 
of $G$ as an extension of $\mathbb{Z}$ by $T$ that 
$H_2(G)$ is finite.
It is also clear that $C=\mathrm{Cok}(\psi^{ab}-I)$ is
the torsion subgroup of $G^{ab}$.
Since $T$ is finite,
$|\mathrm{Ker}(\psi^{ab}-I)|=|\mathrm{Cok}(\psi^{ab}-I)|$,
and so $|C|$ divides the order of $H_2(G)$.

If $p$ is a prime then 
$\dim_{\mathbb{F}_p}Tor(\mathbb{F}_p,G^{ab})=
\dim_{\mathbb{F}_p}Tor(\mathbb{F}_p,C)=
\beta_1(G;\mathbb{F}_p)-1$,
since $G^{ab}\cong\mathbb{Z}\oplus{C}$.
Therefore 
\[
\dim_{\mathbb{F}_p}{Hom}(H_2(G),\mathbb{F}_p)=
\beta_2(G;\mathbb{F}_p)-\beta_1(G;\mathbb{F}_p)+1,
\]
by the Universal Coefficient exact sequences of Chapter 1.
Since $G$ is homologically balanced,
this is at most 1, for all primes $p$, 
and so $H_2(G)$ is cyclic.
\end{proof}

If $G$ is nilpotent then $H_2(\psi)-I$ is a nilpotent endomorphism of 
$H_2(T)$, 
and so $C$ and $H_2(G)$ have the same order if and only if $H_2(T)=0$.

\begin{thm}
{\bf B6.}
%\label{h=1}
Let $G\cong{T}\rtimes_\psi\mathbb{Z}$, 
where $T$ is a finite nilpotent group and $\psi$ is a unipotent automorphism of $T$.
If $G$ is  homologically balanced then
\begin{enumerate}
\item{}$G$ can be generated by $2$ elements;
\item{} if the Sylow $p$-subgroup of $T$ is abelian then it is cyclic;
\item{}if $T$ is abelian then 
$G\cong\mathbb{Z}/m\mathbb{Z}\rtimes_n\mathbb{Z}$, 
for some $m,n\not=0$ such that $m$ divides a power of $n-1$.
\end{enumerate}
\end{thm}

\begin{proof}
Let $p$ be a prime.
Then $\dim_{\mathbb{F}_p}H_1(T;\mathbb{F}_p)=
\dim_{\mathbb{F}_p}Tor(T^{ab},\mathbb{F}_p)$,
since $T$ is finite.
Moreover,
$\psi^{ab}-I$ and $Tor(\psi^{ab},\mathbb{F}_p)-I$ have the same rank.
Since $Tor(T^{ab},\mathbb{F}_p)$ is a natural quotient of 
$H_2(T;\mathbb{F}_p)$,  exactness of the Wang sequence implies
that $\beta_2(G;\mathbb{F}_p)\geq2(\beta_1(G;\mathbb{F}_p)-1)$.
Since $G$ is homologically balanced, $\beta_1(G;\mathbb{F}_p)\leq2$.
Hence the $p$-torsion of $G^{ab}$ is cyclic.
Therefore $G^{ab}\cong\mathbb{Z}\oplus{C}$ for some finite cyclic group.
Since $G$ is nilpotent and $G^{ab}$ can be generated by $2$ elements, 
so can $G$.

The Sylow subgroups of $T$ are characteristic,
and $\psi$ restricts to a unipotent automorphism of each such subgroup.
Suppose that the Sylow $p$-subgroup of $T$ is an abelian group $A$.
Since $H^2(A;\mathbb{F}_p)=Hom(H_2(A;\mathbb{F}_p),\mathbb{F}_p)$,
the endomorphisms $H^2(\psi;\mathbb{F}_p)-I$ and 
$H_2(\psi;\mathbb{F}_p)-I$ 
have the same rank.
Hence
\[
\dim(\mathrm{Ker}(H^2(\psi;\mathbb{F}_p)-I)=
\dim(\mathrm{Ker}(H_2(\psi;\mathbb{F}_p)-I)=
\dim(\mathrm{Cok}(H_2(\psi;\mathbb{F}_p)-I)\leq1.
\]
Hence $A$ must be cyclic, by the corollary to Lemma B3.
%\ref{cycboth}.

It follows immediately that if $T$ is abelian then it is a direct product 
of cyclic groups of relatively prime orders,
and so is cyclic, of order $m$, say.
Hence $G\cong\mathbb{Z}/m\mathbb{Z}\rtimes_n\mathbb{Z}$, 
for some $n$ such that $(m,n)=1$.
Such a semidirect product is nilpotent if and only if 
$m$ divides some power of $n-1$.
\end{proof}

Every semidirect product $\mathbb{Z}/m\mathbb{Z}\rtimes_n\mathbb{Z}$ 
has a balanced presentation
\[
\langle{a,t}\mid{a^m=1},~tat^{-1}=a^n\rangle.
\]
The simplest examples with $T$ non-abelian are the groups 
$Q(8k)\rtimes\mathbb{Z}$,
with the balanced presentations 
$\langle{t,x,y}\mid{x^{2k}=y^2},~tx=xt,~tyt^{-1}=xy\rangle$,
which simplify to
\[
\langle{t,y}\mid[t,y]^{2k}=y^2,~[t,[t,y]]=1\rangle.
\]
Let $m=p^s$ , where $p$ is a prime and $s\geq1$, and let $G$ be the group with presentation
\[
\langle{t,x,y}\mid{txt^{-1}=y},~tyt^{-1}=x^{-1}y^2,~yxy^{-1}=x^{m+1}\rangle.
\]
Conjugating the final relation with $t$ gives $x^{-1}yx=y^{m+1}$,
and we see that the torsion subgroup $T$
has presentation $\langle{x,y}\mid{x^m=y^m},~yxy^{-1}=x^{m+1}\rangle$,
and so is one of the metacyclic $p$-groups mentioned above.
Moreover,  $G$ is nilpotent, 
$\zeta{G}=\langle{x^m}\rangle$ and $G'=\langle{x^m,x^{-1}y}\rangle$ is abelian.
Hence $G$ is metabelian.
Each of the groups that we have described here
is 2-generated and its torsion subgroup is homologically balanced. 

\section*{Metabelian nilpotent groups with Hirsch length $>1$}

All known examples of nilpotent groups with balanced presentations 
and Hirsch length $h>1$ are torsion-free. 
We have not yet been able to show that this must be so.
However,  if such a group is also metabelian, but not $\mathbb{Z}^3$,  then 
$h(G)\leq4$ and $\beta_1(G;\mathbb{Q})=2$ \cite[Theorems 7 and 15]{Hi22}.
Our main result implies that there are just three such groups 
with $G/G'\cong\mathbb{Z}^2$.
The argument again involves finding normal subgroups with
``large enough" homology to affect the Betti numbers of the group.
We develop a number of lemmas to this end.
(The Lyndon-Hochschild-Serre spectral sequences for the homology and cohomology 
of a group which is an extension of $\mathbb{Z}^2$ by a normal subgroup 
shall largely replace the Wang sequences used earlier.)

\begin{lem}
{\bf B7.}
\label{h=2L}
Let $G$ be a finitely generated nilpotent group,
and let $T$ be its torsion subgroup. 
Then the following are equivalent
\begin{enumerate}
\item$\beta_1(G;\mathbb{Q})=2$ and $\beta_2(G;\mathbb{Q})=1$;
\item$h(G)=2$;
\item$G/T\cong\mathbb{Z}^2$.
\end{enumerate}
If these conditions hold and $G$ is homologically balanced then
$H_2(G)\cong\mathbb{Z}\oplus\mathbb{Z}/e\mathbb{Z}$,
for some $e\geq1$.
\end{lem}

\begin{proof}
If (1) holds then $h(G)\geq2$,  and so $h(G)=2$,  
by the corollary to Lemma \ref{wang app}.
It is easy to see that (2) and (3) are equivalent, and imply (1).
The final assertion follows from Corollary \ref{wang cor} and Lemma \ref{h=2}.
\end{proof}

We could also describe the groups considered in this lemma as
the nilpotent groups which are virtually $\mathbb{Z}^2$.

\begin{lem}
{\bf B8.}
\label{euler}
Let $F$ be a field and $A$ be a finite dimensional $F[\mathbb{Z}^2]$-module,
and let $b_i=\dim_FH_i(\mathbb{Z}^2;A)$, for $i\geq0$.
Then $b_2=b_0$ and $b_1=b_0+b_2=2b_0$.
\end{lem}

\begin{proof}
We may compute $H_i(\mathbb{Z}^2;A)=Tor_i^{F[\mathbb{Z}^2]}(F,A)$ 
from the complex
\[
0\to{A}\to{A^2}\to{A}\to0,
\]
in which the differentials are $\partial^1=
\left[\smallmatrix(x-1)id_A\\(y-1)id_A\endsmallmatrix\right]$,
and $\partial^2=\left[\smallmatrix(y-1)id_A,(1-x)id_A\endsmallmatrix\right]$
where $\{x,y\}$ is a basis for $\mathbb{Z}^2$.
Since the matrix for $\partial^2$ is the transpose of that for
$\partial^1$ (up to a change of sign in the second block),
they have the same rank.
Hence 
$b_2=\dim_F\mathrm{Ker}(\partial_2)=\dim_F\mathrm{Cok}(\partial_1)=b_0$.
The final assertion follows since $b_0-b_1+b_2=1-2+1=0$ is the Euler characteristic of the complex.
\end{proof}

The  modules $H_2(\mathbb{Z}^2;A)$ and $H_0(\mathbb{Z}^2;A)$ 
are the submodule of fixed points and the coinvariant quotient modules 
of the $\mathbb{Z}^2$-action, respectively.
Minor adjustments give similar results for $\dim_FH^j(\mathbb{Z}^2;A)$.
(We may also use Poincar\'e duality for $\mathbb{Z}^2$
to relate homology and cohomology.)

Recall that if $K$ is a normal subgroup of a group $G$ then conjugation in $G$ induces a natural action of $G/K$ on the homology and cohomology of $K$.

\begin{lem}
{\bf B9.}
\label{metab1}
Let $G$ be a finitely generated nilpotent group 
with a normal subgroup $K$ such that $G/K\cong\mathbb{Z}^2$,
and suppose that $\beta_2(G;F)\leq\beta_1(G;F)$ for some field $F$.
Then $\dim_FH^2(K;F)^{G/K}\leq1$.
\end{lem}

\begin{proof}
We note first that $\beta_1(G;F)=2$ or 3,
since $\beta_2(G;F)\leq\beta_1(G;F)$ 
\cite[Theorem 2.7]{Lub83}.
We may assume that  $A=H^1(K;F)\not=0$,
since $G$ is nilpotent. 
Let $N=G/K$ and $b_i=\dim_FH^i(N;A)$.
The LHS spectral sequence for cohomology with coefficients $F$
for $G$ as an extension of $N$ by $K$ 
gives two exact sequences
\begin{equation*}
\begin{CD}
0\to{H^1(N;F)}\to{H^1(G;F)}\to{A}^N@>d^{0,1}_2>>{H^2(N;F)}\to
{H^2(G;F)}\to{J}\to0
\end{CD}
\end{equation*}
and
\begin{equation*}
\begin{CD}
0\to{H^1(N;A)}\to{J}\to{H^2(K;F)^N}@>d^{0,2}_2>>{H^2(N;A)}.
\end{CD}
\end{equation*}
The first sequence gives $\dim_FJ\leq\beta_2(G;F)$.
Then  $b_1=b_0+b_2=2b_0$, by LemmaB8,
% \ref{euler}, 
and $b_0>0$,
since $A\not=0$.
Hence $b_0-b_1<0$, and so the second sequence gives   
$\dim_FH^2(K;F)^N\leq
\beta_2(G;F)+b_0-b_1\leq\beta_2(G;F)-1$.
In particular, 
$\dim_FH^2(K;F)^N\leq1$ if $\beta_2(G;F)=2$.

If $\beta_2(G;F)=3$ then
$\beta_1(G;F)=3$ also, by Corollary \ref{wang cor}.
Therefore\\
$\dim_F\mathrm{Ker}(d^{0,1}_2)=1$.
If $d^{0,1}_2\not=0$ then $b_0=2$ and so $b_1=4$,
by Lemma B8.
%\ref{euler}.
But then $\beta_2(G;F)\geq4$.
Therefore $d^{0,1}_2=0$, and so $b_0=1$.
Hence $b_1=2$ and $b_2=1$,
and $d^{0,2}_2$ is a monomorphism.
Hence we again have $\dim_FH^2(K;F)^N\leq1$.
\end{proof}

In particular, $H^2(K;F)$ is not canonically subsplit.
A parallel argument using homology shows that
$\dim_FH_0(G/K;H_2(K;F))\leq1$.

\begin{lem}
{\bf B10.}
\label{K 2end}
Let $P$ be a non-trivial finite $p$-group and $K\cong\mathbb{Z}\times{P}$.
Then $H^2(K;\mathbb{F}_p)$ is canonically subsplit.
\end{lem}

\begin{proof}
We shall use the Universal Coefficient exact sequence for cohomology.
The projection of $K$ onto $K/P\cong\mathbb{Z}$ determines (up to sign)
a class $\eta\in{H^1(K;\mathbb{F}_p)}=Hom(K,\mathbb{F}_p)$,
and cup product with $\eta$ maps $H^1(K;\mathbb{F}_p)$ non-trivially to 
$H^2(K;\mathbb{F}_p)$, by the K\"unneth Theorem.
The restriction from $Ext(K^{ab},\mathbb{F}_p)$ to 
$Ext(P^{ab},\mathbb{F}_p)$ is an isomorphism, 
and so  $Ext(K^{ab},\mathbb{F}_p)$ and 
$\eta\cup{H^1(K;\mathbb{F}_p)}$ have trivial intersection.
Hence 
$Ext(K^{ab},\mathbb{F}_p)\oplus(\eta\cup{H^1(K;\mathbb{F}_p)})$ 
is a subspace of $H^2(K;\mathbb{F}_p)$,
and the summands  are invariant under 
the action of automorphisms of $K$,
by the naturality of the Universal Coefficient Theorem.
The summands are non-trivial, since $P\not=1$.
\end{proof}

The next four lemmas (leading up to Theorem B15)
consider nilpotent groups which are extensions of $\mathbb{Z}^2$ 
by finite normal subgroups.

\begin{lem}
{\bf B11.}
\label{P factor}
Let $G$ be a finitely generated nilpotent group,
and let $T$ be its torsion subgroup.
Let $P$ be a non-trivial Sylow $p$-subgroup of $T$ and let 
$\gamma_p:G\to{Out}(P)$
be the homomorphism determined by conjugation in $G$.
If $G/T\cong\mathbb{Z}^2$ and the image of $\gamma_p$ is cyclic then 
$\beta_2(G;\mathbb{F}_p)>\beta_1(G;\mathbb{F}_p)$.
\end{lem}

\begin{proof}
We may write $G\cong{K}\rtimes_\psi\mathbb{Z}$, 
where $\psi$ is a unipotent automorphism of $K$,
and $K$ is in turn an extension of $\mathbb{Z}$ by $T$.
Let $P$ be the Sylow $p$-subgroup of $T$,
and let $N$ be the product of the other Sylow subgroups of $T$.
Since the Sylow subgroups of $T$ are characteristic, 
conjugation in $G$ determines a homomorphism $\gamma_p:G\to{Out}(P)$.
Moreover, $N$ is normal in $G$,  and the projection of $G$ onto $G/N$ 
induces isomorphisms on homology and cohomology 
with coefficients $\mathbb{F}_p$.
Hence we may assume that $N=1$ and so $T=P$ is a non-trivial $p$-group.

If the image of $\gamma_P$ is cyclic then $\gamma_P$ 
factors through an epimorphism $f:G\to\mathbb{Z}$, 
with kernel $K\cong\mathbb{Z}\times{P}$.
Since $H^2(K;\mathbb{F}_p)$ has a subspace which is the direct sum of non-trivial canonical summands, by Lemma B10,
%\ref{K 2end},
 $\dim_{\mathbb{F}_p}\mathrm{Ker}(H^2(\psi;\mathbb{F}_p)-I)>1$
(as in Lemma B2).
%\ref{cycad}).
The result now follows from Lemma \ref{wang app}.
\end{proof}

Thus the group with presentation 
$\langle{x,y}\mid[x,[x,y]]=[y,[x,y]]=[x,y]^p=1\rangle$
mentioned  above does not have a balanced presentation.
Similarly, no nilpotent extension of $\mathbb{Z}^2$ by $Q(8)$ can have a balanced presentation,
since the abelian subgroups of $Out(Q(8))\cong\mathcal{S}_3$ are cyclic.

If $p$ is an odd prime and $C$ is a cyclic $p$-group then 
$\mathrm{Aut}(C)$ is cyclic, and so Lemma B11
%\ref{P factor} 
may apply.
However, dealing with 2-torsion again requires more effort.

\begin{lem}
{\bf B12.}
\label{partial3}
Let $G$ be a finitely generated nilpotent group,
and let $T$ be its torsion subgroup.
If $G/T\cong\mathbb{Z}^2$ and the Sylow $2$-subgroup of $T$
is a nontrivial cyclic group then 
$\beta_2(G;\mathbb{F}_2)>\beta_1(G;\mathbb{F}_2)$. 
\end{lem}

\begin{proof}
We may factor out the maximal odd-order subgroup of $T$ without changing
the $\mathbb{F}_2$-homology, 
and so we may assume that $T\cong\mathbb{Z}/k\mathbb{Z}$,
where $k=2^n$,
for some $n\geq1$.
We may also assume that the action of $G$ on $T$ 
by conjugation does not factor through a cyclic group, 
by Lemma B10,
%\ref{P factor}, 
and so $k\geq8$.
Let $U$ be the subgroup of $(\mathbb{Z}/k\mathbb{Z})^\times$
represented by integers $\equiv1~mod~(4)$.
Then $Aut(\mathbb{Z}/k\mathbb{Z})\cong\{\pm1\}\times{U}$.
It is easily verified that noncyclic subgroups of $Aut(\mathbb{Z}/k\mathbb{Z})$
have $\{\pm1\}$ as a direct factor, and so $G$ has a presentation
\[
\langle{x,y,z}\mid[x,y]=z^f,~z^k=1,~xzx^{-1}=z^{-1},~yzy^{-1}=z^\ell\rangle,
\]
where $f$ divides $k$, 
$1<\ell<k$ and $\ell\equiv1~mod~(4)$.
Let $m$ be a mutiplicative inverse for $\ell~mod~(k)$,
so that $1<m<k$ and $m\ell=wk+1$ for some $w\in\mathbb{Z}$.
Note that $\beta_1(G;\mathbb{F}_2)=2$ if $f=1$ and is 3 if $f>1$.

The ring $\mathbb{Z}[G]$ is a twisted polynomial extension 
of $\mathbb{Z}[\mathbb{Z}/k\mathbb{Z}]=\mathbb{Z}[z]/(z^k-1)$, 
and so is  noetherian.
We may assume each monomial is normalized in alphabetical order: 
$x^hy^iz^j$, for exponents $h,i\in\mathbb{Z}$ and $0\leq{j}<k$.
Let $\nu=\Sigma_{i=0}^{k-1}z^i$ be the norm element for 
$\mathbb{Z}[\mathbb{Z}/k\mathbb{Z}]$.
Then $z\nu=\nu$, so $\nu^2=k\nu$, and $\nu$ is central in $\mathbb{Z}[G]$.
We shall use the fact that if $\gamma,\delta\in\mathbb{Z}[G]$
are such that $\gamma\nu=0$ and $\delta(z-1)=0$
then $\gamma=\gamma'(z-1)$ and $\delta=\delta'\nu$, 
for some $\gamma',\delta'\in\mathbb{Z}[G]$.
On the other hand, non-zero terms not involving $z$
are not zero-divisors in $\mathbb{Z}[G]$.

The augmentation module $\mathbb{Z}$ has a Fox-Lyndon partial resolution
\begin{equation*}
\begin{CD}
C_2@>\partial_2>>{C_1}@>\partial_1>>C_0=\mathbb{Z}[G]@>\varepsilon>>\mathbb{Z}\to0,
\end{CD}
\end{equation*}
where $\varepsilon:\mathbb{Z}[G]\to\mathbb{Z}$ is the augmentation
homomorphism,
$C_1\cong\mathbb{Z}[G]^3$ has basis $\{e_x,e_y,e_z\}$
corresponding to the generators and 
 $C_2\cong\mathbb{Z}[G]^4$ has basis $\{r,s,t,u\}$
corresponding to the relators $r=z^fyxy^{-1}x^{-1}$, 
$s=z^k$, $t=zxzx^{-1}$ and $u=z^\ell{y}z^{-1}y^{-1}$.
The differentials are given by
\[
\partial_1(e_x)=x-1, \quad \partial_1(e_y)=y-1\quad\mathrm{ and}\quad 
\partial(e_z)=z-1;\quad\mathrm{ and}
\]
\[
\partial_2(r)=(z^fy-1)e_x+(z^f-x)e_y+(\Sigma_{i=0}^{f-1}z^i )e_z,\quad
\partial_2(s)=\nu{e_z},
\]
\[
\partial_2(t)=(z-1)e_x+(1+zx)e_z\quad\mathrm{and}\quad
\partial_2(u)=(z^\ell-1)e_y+(\Sigma_{j=0}^{\ell-1}z^j-y)e_z.
\]
We may choose a homomorphism $\partial_3:C_3\to{C_2}$ with domain 
$C_3$ a  free $\mathbb{Z}[G]$-module and image $\mathrm{Ker}(\partial_2)$, 
which extends the resolution one step to the left.
(We may assume that $C_3$ is finitely generated, 
since $\mathbb{Z}[G]$ is noetherian.)
It is clear from the Fox-Lyndon partial resolution that 
$\dim_{\mathbb{F}_2}\mathrm{Ker}(\mathbb{F}_2\otimes_{\mathbb{Z}[G]}\partial_2)
=\beta_1(G;\mathbb{F}_2)+1$.
We shall show that $\mathbb{F}_2\otimes_{\mathbb{Z}[G]}\partial_3=0$,
and so $\beta_2(G;\mathbb{F}_2)=\beta_1(G;\mathbb{F}_2)+1$.

Let $\varepsilon_2:\mathbb{Z}[G]\to\mathbb{F}_2$ be the $mod~(2)$
reduction of $\varepsilon$.
Since $\mathrm{Im}(\partial_3)=\mathrm{Ker}(\partial_2)$,
it shall suffice to show that if
\[
\partial_2(ar+bs+ct+du)=0,
\]
for some $a,b,c,d\in\mathbb{Z}[G]$ then 
$\varepsilon_2(a)=\varepsilon_2(b)=
\varepsilon_2(c)=\varepsilon_2(d)=0$.

The coefficients $a,b,c,d$ must satisfy the three equations
\[
a(z^fy-1)+c(z-1)=0,
\]
\[
a(z^f-x)+d(z^\ell-1)=0
\]
and
\[
a(\Sigma_{i=0}^{f-1}z^i )+b\nu+c(zx+1)+d(\Sigma_{j=0}^{\ell-1}z^j-y)=0.
\]
Multiplying the first of these equations by $\nu$ gives $af\nu(y-1)=0$.
Hence $a\nu=0$ and so $a=A(z-1)$, 
for some $A\in\mathbb{Z}[G]$ not involving $z$.
The first equation becomes
\[
A(z-1)(z^fy-1)+c(z-1)=[A(yz^{fm}(\Sigma_{j=0}^{m-1}z^j)-1)+c](z-1)=0,
\]
and so $c=-A(yz^{fm}(\Sigma_{j=0}^{m-1}z^j)-1)+C\nu$,
for some $C\in\mathbb{Z}[G]$ not involving $z$.
Similarly, the second equation becomes
\[
A(z-1)(z^f-x)+d(z^\ell-1)=A(zx+z^f)(z^{m\ell}-1)+d(z^\ell-1)=0,
\]
and so $d=-A(zx+z^f)(\Sigma_{j=0}^{m-1}z^{j\ell})+D\nu$, 
for some $D\in\mathbb{Z}[G]$ not involving $z$.
At this point it  is already clear that $\varepsilon_2(a)=
\varepsilon_2(c)=\varepsilon_2(d)=0$.

Multiplying the third equation by $\nu$ gives
\[
kb\nu+c\nu(x+1)+d\nu(\ell-y)=0.
\]
Rearranged and written out in full, this becomes
\[
kb\nu=(A(ym-1)-Ck)(x+1)\nu+(A(x+1)m-Dk)(\ell-y)\nu.
\]
Since $yx=z^{-f}xy=xzy=xyz^{fm}$ we have $yx\nu=xy\nu$
and so this simplifies to 
\[
kb\nu=(A(m\ell-1)(x+1)-kC(x+1)-kD(\ell-y))\nu.
\]
Write $b=b_1+B(z-1)$, where $b_1$ does not involve $z$.
Then $b\nu=b_1\nu$.
Since the terms $b_1, A,C$ and $D$ do not involve $z$,
and since $m\ell-1=wk$, we get
\[
kb_1=k(Aw(x+1)-C(x+1)-D(\ell-y)).
\]
We may solve for $b_1$,  and so
\[
b=b_1+B(z-1)=wA(x-1)+B(z-1)-C(x+1)-D(\ell-y).
\]
Hence $\varepsilon_2(b)=0$ also, 
so $\mathbb{F}_2\otimes_{\mathbb{Z}[G]}\partial_3=0$ and thus
$\beta_2(G;\mathbb{F}_2)=\beta_1(G;\mathbb{F}_2)+1$.
\end{proof}

\begin{lem}
{\bf B13.}
\label{metab}
Let $G$ be a finitely generated nilpotent group,
and let $T$ be its torsion subgroup. 
If $G/T\cong\mathbb{Z}^2$ and the Sylow $p$-subgroup of $T$ 
is abelian and non-trivial then 
$\beta_2(G;\mathbb{F}_p)>\beta_1(G;\mathbb{F}_p)$. 
\end{lem}

\begin{proof}
Let $N$ be the product of all the Sylow $p'$-subgroups of $T$ with $p'\not=p$,
and let $A$ be the image of $T$ in $\overline{G}=G/N$.
Then $\beta_i(\overline{G};\mathbb{F}_p)=\beta_i(G;\mathbb{F}_p)$, 
for all $i$.
The Sylow $p$-subgroup of $T$ projects isomorphically onto $A$, 
and $\overline{G}/A\cong\mathbb{Z}^2$.
If $\beta_2(G;\mathbb{F}_p)\leq\beta_1(G;\mathbb{F}_p)$
then $\dim_{\mathbb{F}_p}H^2(A;\mathbb{F}_p)^{G/A}\leq1$,
by Lemma B9.
%\ref{metab1}.
Hence $A$ is cyclic, by Lemmas B2 and B3.
% \ref{cycad} and \ref{cycad2}.
If $p=2$ the result follows from Lemma B12
%\ref{partial3},
while if $p$ is odd it follows from Lemma B11
%\ref{P factor},
since the automorphism group of a cyclic group of odd $p$-power order is cyclic.
\end{proof}

For the next result we need an analogue of Lemma B10.
% \ref{K 2end}.

\begin{lem}
{\bf B14.}
\label{K 1end}
Let $K\cong{T}\rtimes\mathbb{Z}^2$, 
where $T$ is a finite $p$-group,
and $T\not\leq{K'}$.
Then $H^2(K;\mathbb{F}_p)$ is canonically subsplit.
\end{lem}

\begin{proof}
Let $\alpha:K\to\mathbb{Z}^2$ be the canonical epimorphism.
Since $\alpha$ splits,
$H^2(\alpha;\mathbb{F}_p)$ is a monomorphism.
The other hypotheses imply $Ext(K^{ab},\mathbb{F}_p)\not=0$.
Hence
$Ext(K^{ab},\mathbb{F}_p)\oplus\mathrm{Im}(H^2(\alpha;\mathbb{F}_p))$ 
is a subspace of $H^2(K;\mathbb{F}_p)$ with the desired properties.
\end{proof}

If $G$ is a homologically balanced, 
metabelian nilpotent group then either $G\cong\mathbb{Z}^3$ 
or $\beta_1(G;\mathbb{Q})\leq2$ and $h(G)\leq4$.
In the latter case the torsion-free quotient $G/T$ is either free abelian 
of rank $\leq2$, 
or is the $\mathbb{N}il^3$-group $\Gamma_q$, for some $q\geq1$ 
\cite[Corollary 8 and Theorems 10 and 15]{Hi22}.

\begin{thm}
{\bf B15.}
%\label{Z^2}
Let $G$ be a homologically balanced nilpotent group with\\ 
$\beta_1(G;\mathbb{Q})=2$,
and let $T$ be its torsion subgroup.
Then
\begin{enumerate}
\item{}if $h(G)=2$ and $T$ is abelian then $G\cong\mathbb{Z}^2$; 
\item{}if $h(G)=3$  and the outer action $:G\to{Out(T)}$ determined by
conjugation in $G$ factors through $\mathbb{Z}^2$ then $G\cong\Gamma_q$,
for some $q\geq1$; 
\item{}if $h(G)=4$ and $G$ has an abelian normal subgroup $A$ with
$G/AT\cong\mathbb{Z}^2$ then $G\cong\Omega$.
\end{enumerate}
\end{thm}

\begin{proof}
Let $K=I(G)$.
Then $TG'\leq{K}$, $G/K\cong\mathbb{Z}^2$, and $K/G'$ is (finite) cyclic, 
since $\beta_1(G;\mathbb{F}_p)\leq3$ for all primes $p$.

If $h(G)=2$ then $K=T$ and $G/T\cong\mathbb{Z}^2$.
Hence if $T$ is abelian then $T=1$, by Lemma B13,
%\ref{metab},
and so $G\cong\mathbb{Z}^2$.

If $h(G)=3$ and the outer action $:G\to{Out(T)}$ determined by
conjugation in $G$ factors through $\mathbb{Z}^2$
then $K\cong{T}\times\mathbb{Z}$.
Thus if $T\not=1$ then 
$\beta_2(G;\mathbb{F}_p)>\beta_1(G;\mathbb{F}_p)$,
for any prime $p$ dividing $|T|$,
by Lemmas B9 and B10.
%\ref{metab1} and \ref{K 2end}.
This contradicts the hypothesis that $G$ has a balanced presentation.

If $h(G)=4$ and $G$ has an abelian normal subgroup $A$ with
such that $G/AT\cong\mathbb{Z}^2$ then $K=AT$,
and $\overline{A}=A/A\cap{T}\cong\mathbb{Z}^2$.
Since $K$ is nilpotent, 
the action of the finite group $T/A\cap{T}$ on $\overline{A}$ is trivial.
Hence $K/A\cap{T}\cong\overline{A}\times(T/A\cap{T})$,
and so $K\cong{T}\rtimes\mathbb{Z}^2$.
Moreover,  if $T\not=1$ then $T\not\leq{K'}$.
Lemmas B9 and B14 (together with Lemmas B2 and B3) 
% \ref{metab1} and \ref{K 1end} (together with Lemmas \ref{cycad} and \ref{cycad2}) 
then give a similar contradiction.

In parts (2) and (3) the group $G$ is torsion-free,
and so must be one of the known examples given above.
\end{proof}

Imposing a stronger constraint gives a clearer statement.

\begin{cor*}
\label{Z2cor}
If $G$ is a homologically balanced nilpotent group with 
an abelian normal subgroup $A$ such that $G/A\cong\mathbb{Z}^2$
then $G\cong\mathbb{Z}^2$,  $\Gamma_q$ (for $q\geq1$)
or $\Omega$. 
\qed
\end{cor*}

Note that the second hypotheses in parts (2) and (3) of the theorem 
are not by themselves equivalent to assuming that $G$ is metabelian,
while the hypothesis in the corollary is somewhat stronger.
(On the other hand, it includes all 2-generated metabelian nilpotent groups $G$
with $h(G)>1$.)

The above work leaves open the following questions, for $G$ a nilpotent group with torsion subgroup $T$ and a balanced presentation.
\begin{enumerate}
\item{}If $h(G)=1$ is $H_2(T)=0$?
\item If $h(G)>1$ is $T=1$?
\end{enumerate}

%% file: eappC.tex
\chapter*{Appendix C. Some questions}

This list includes questions from \cite{BB22, Do15,IM20} and \cite{Liv05}.
Many of the following questions have obvious analogues for smooth embeddings.

As in the text, $M$ is a closed orientable 3-manifold, 
and if $j:M\to{S^4}$ is a locally flat embedding then the complementary regions 
$X$ and $Y$ are labelled so that $\chi(X)\leq\chi(Y)$.
The fundamental groups are $\pi=\pi_1M$, $\pi_X=\pi_1X$ and $\pi_Y=\pi_1Y$.

\begin{enumerate}
\item{}Is there an effective way of deciding whether a framed link in $S^3$ is equivalent under Kirby moves to a 0-framed bipartedly slice link?

\item{}If each  complementary region for an embedding $j$
may be obtained from the 4-ball by adding 1- and 2-handles,
must $j=j_L$ for some 0-framed bipartedly trivial link $L$?

\item{}[BB20] Is there a 3-manifold $M$ which embeds smoothl;iy in $S^4$, but which has no smooth embedding deriving from a 0-framed link presentation?

\item{}(Attributed to M. Freedman in \cite{BB22}.)
If $M\#(S^2\times{S^1})$ embeds, does $M$ embed?
\end{enumerate}

If $L$ is a slice link then the 3-manifold $M(L)$ obtained by 0-framed surgery on $L$ embeds, as we may use a set of slice discs in the upper half of 
$S^4=D^4_+\cup{D^4_-}$ to attach 2-handles to $D^4_-$ in side $S^4$. 
To what extent is the converse true?

\begin{enumerate}
\addtocounter{enumi}{4}
\item{}If $K$ is a knot such that $M(K)$ embeds must $K$ be (TOP) slice?
\end{enumerate}

Most of the algebraic arguments obstructing  embeddings of a given  3-manifold 
into $S^4$ that we shall use obstruct Poincar\'e embeddings in homology 4-spheres.

\begin{enumerate}
\addtocounter{enumi}{5}
\item{}Is there a 3-manifold $M$ which has a Poincar\'e embedding in a homology 
4-sphere but which does not embed in $S^4$?
\end{enumerate}

If $M$ embeds and has non-trivial homology then then $\pi_X\not=1$, 
and we may use 2-knot surgery to construct infinitely many embeddings of $M$, distinguished by the groups $\pi_X$.
On the other hand,   if $M=S^3$ then $\pi_X=\pi_Y=1$, and all
embeddings of $S^3$ are essentially equivalent, by the Schoenflies Theorem.
The question remains open for homology 3-spheres.
\begin{enumerate}
\addtocounter{enumi}{6}
\item{}Does every homology 3-sphere bound a non-simply connected acyclic 4-manifold? have an embedding with one complementary region not 1-connected?
have an embedding with neither complementary region 1-connected?
\end{enumerate}
Similarly, if $M=S^2\times{S^1}$ then $\pi_Y=1$.

\begin{enumerate}
\addtocounter{enumi}{7}
\item{}Does every homology handle other than $S^2\times{S^1}$ have embeddings with neither complementary region 1-connected?
\end{enumerate}

There are uniform non-embedding results for Seifert manifolds, 
but the known constructions apply only when there is no 2-torsion.

\begin{enumerate}
\addtocounter{enumi}{8}

\item{}What are the possible values of $\chi(X)$ for embeddings of the Seifert manifold $M(k;S)$?

\item{}Let $\ell$ be a hyperbolic linking pairing with homogeneous 2-primary summand 
$\ell_{(2)}$ (as defined in Appendix A).
Is $\ell$ realized by some Seifert manifold $M(0;S)$ which embeds in $S^4$?

\item{}Are there any restrictions related to 2-torsion in the cone point orders 
of the base orbifolds of Seifert manifolds which embed in $S^4$?

\item{}Do the results and conjectures of Donald and Issa and McCoy 
cited in Chapters 3 and 4 hold also for TOP locally flat embeddings?
\end{enumerate}
Donald and McCoy each use Donaldson's diagonalization theorem,
and so their arguments only apply for smooth embeddings.

Graph manifolds have natural parametrizations in terms of weighted plumbing graphs.
See \cite{Ne81} and \cite[Chapter 2]{Orl}.

\begin{enumerate}
\addtocounter{enumi}{12}
\item{}What can be said about embeddings of graph manifolds?
\end{enumerate}

In the remaining questions the emphasis is on the complementary regions.
\begin{enumerate}
\addtocounter{enumi}{13}
\item{}Is there a 3-manifold with at least two abelian embeddings
which are inequivalent? 
\item{}Is there an embedding with $X$ aspherical and $c.d.\pi_X=3$? 
\end{enumerate}

If $c.d.\pi_X=3$ then $\pi_X$ cannot be a $PD_3$-group, since $H_3(X)=0$.
See also \cite{DH23}.

\begin{enumerate}
\addtocounter{enumi}{15}
\item{}Is there a 3-manifold which embeds, 
but has no embedding with $\chi(X)=\chi(Y)=1$ (if $\beta$ is even) 
or $\chi(X)=0$ and $\chi(Y)=2$ (if $\beta$ is odd)?
In particular, 
is there an embedding of the Seifert manifold $M(-3;(1,6))$ with $\chi(X)=\chi(Y)=1$?

\item{}Does the $\mathbb{S}ol^3$-manifold $M_{2,4}$ have an abelian embedding?
\end{enumerate}
Whether there is an abelian embedding has been decided for all other
3-manifolds with virtually solvable fundamental groups.
(See Chapter 7.)

\begin{enumerate}
\addtocounter{enumi}{17}

\item{}What can be said about complementary regions of hypersurfaces 
with more than one component?

\item{}\cite{Liv05} Is a group $G$ the fundamental group of a homology 4-sphere if and only if it is the fundamental group of the complement of a contractible  
submanifold of $S^4$?

\end{enumerate}
One implication is clear. 
If $W$ is the closure of the complement of a contractible codimension-0 submanifold 
of $S^4$ then $W$ is acyclic,  the double $DW$ is a homology 4-sphere, 
and $\pi_1DW\cong\pi_1W$.

There is a smooth embedding $j$ such that one of the groups $\pi_X$ or $\pi_Y$ has an unsolvable word problem \cite{DR93}. 

\begin{enumerate}
\addtocounter{enumi}{19}

\item{}\cite{BB22} If $M$ embeds in $S^4$ does it have an embedding such that $\pi_X$ and $\pi_Y$ each have solvable word problem?

\end{enumerate}

%% file: ebib.tex
\bibliographystyle{amsalpha}

%% file: eind.tex
\begin{theindex}

{\sl Expressions beginning with Greek 

characters and non-alphabetic symbols are
listed at the end of this index}.
\bigskip

\item{}abelian embedding, 14,84

\item{}algebraically slice (knot), 5

\item{}aspherical embedding, 77

\item{}Aitchison's Theorem, 69

\item{}balanced, 2

\item{}base orbifold, 27

\item{}bi-epic, 19

\item{}bipartedly (ribbon, slice, trivial), 15

\item{}Blanchfield pairing, 6

\item{}$Bo=6^3_2$ (Borromean rings), 20

\item{}Borromean rings $Bo$,  20

\item{}Brown's Collaring Theorem, 10

\item{}$BS(1,m)$, 1

\item{}cohomological dimension, 76

\item{}complementary region, 109

\item{}connected sum (of embeddings), 19

\item{}$D_\infty$, 1

\item$d(\ell)$ (determinantal invariant of linking pairing), 3

\item{}doubly slice, 5,18

\item{}embedding, 9

\item{}equivalent (embedding), 9

\item{}equivariant (cohomology), 3

\item{}essentially unique, 15

\item{}Euler number, 28

\item{}even (linking pairing), 3

\item{}expansion, 29

\item{}fibred sum, 28

\item{}$F(r)$ (free group), 1

\item{}generalized Euler number, 28

\item{}Generalized Schoenflies Theorem, 9, 68

\item$G$-Signature Theorem,7

\item$h(G)$ (Hirsch length), 1

\item{}Hall Lemma, 101

\item{}2-handlebody, 4

\item{}$H$-cobordism (over $R$), 4

\item{} Hirsch length ($h(G)$), 1

\item{}$Ho=2^2_1$ (Hopf link), 15, 53

\item{}homologically balanced, 2

\item{}homology handle, 5,88

\item{}homology sphere, 10,87

\item{}homotopically ribbon (knot), 5

\item{}Hopf link $Ho$, 15, 53

\item{}hyperbolic (linking pairing), 3

\item{}$I(G)=\gamma_2^\mathbb{Q}G$, 1

\item{}isolator ($I(G)$), 1

\item$j_L$, 15

\item{}$j_X$, $j_Y$, 11

\item{}Kawauchi Theorem, 14

\item{}Kawauchi-Kojima Lemma, 13

\item{}Kirby calculus, 16

\item{}$Kb$ (Klein bottle), 4

\item{}2-knot surgery, 21

\item{}$\ell_M$, 5

\item{}$\ell_w$, 3

\item{}Lickorish-Quinn Theorem, 16

\item{}linking pairing, 3

\item{}locally flat, 4

\item{}locally flat embedding, 9

\item$M_o$ (punctured manifold), 4

\item$M(g;S)$ (Seifert manifold), 27

\item$M(L)$ (surgery on link), 5

\item$M_{m,n}$, 55

\item{}Massey product, 4, 69

\item$N$ (mapping cylinder), 52

\item{}neutral (Blanchfield pairing), 6

\item{}nilpotent embedding, 14, 105

\item{}odd (linking pairing), 3

\item{}Poincar\'e embedding, 11

\item{}proper 2-knot surgery, 21

\item$R$-homology isomorphism, 4

\item{}$rk(\ell)$, 3

\item$P_\ell$, 4

\item{}$\#^c\mathbb{RP}^2$, 4

\item $R\Lambda$ (Laurent polynomial ring), 2

\item{}rank (of linking pairing), 3

\item{}restrained (embedding), 14, 83

\item{}restrained (group), 1

\item{}$s$-concordant, 14

\item{}Seifert data, 27

\item{}Seifert manifold, 27

\item{}$\mathrm{sign}(g,M)$, 8

\item{}skew-symmetric (Seifert data), 29

\item{}slice (link), 5

\item{} smooth embedding, 23

\item{}smoothable (embedding), 11

\item{}strict (Seifert data), 28

\item{}surgery,  6

\item{}$T$ and $T_g$, 4

\item{}twisted double, 5

\item$U$ (the unknot), 5

\item{}unipotent (automorphism), 101

\item{}$val_p$, 35

\item{}Van Kampen Theorem, 13

\item{}virtually solvable, 1

\item{}Wang sequence, 4

\item$Wh=5^2_1$ (Whitehead link), 53, 86

\item$W_+(\mathbb{Q}(t),\mathbb{Q}\Lambda)$ (Witt group), 6

\item$X,Y$ (complementary regions), 11

\item{}$X(L)$ (link complement), 5

\item{}$\mathbb{Z}$-homology cobordism, 4

\bigskip
{\bf Other symbols}

\item$\beta=\beta_1(M)$, 11 

\item{}$\gamma_kG$, 1

\item{}$\gamma_k^\mathbb{Q}G$, 1

\item$\Delta_0(L)$ (order), 2

\item$\varepsilon(M), \varepsilon_S$, 28

\item$\zeta{G}$, 1

\item$\eta$ ($=\Sigma\alpha_i\beta_i$), 43

\item$\Lambda=\mathbb{Z}[t,t^{-1}]$, 2

\item$\mu_M$, 12

\item$\pi=\pi_1M$, 11

\item$\pi{L}=\pi_1X(L)$, 5

\item{}$\tau_M$, 5

\item{}$\pi_X=\pi_1X, \pi_Y=\pi_1Y$, 11

\item$_pA$, 2

\item{}$[g,h]$, 1

\item{}$G'$ (commutator subgroup), 1

\item{}$G^{ab}=G/G'$ (abelianization), 1

\item$\langle{a,b,c}\rangle$ (Massey product), 4

\item{}$j\#^\pm{j'}$ (sum of embeddings), 19

\item$M\#_fM'$ (fibred sum), 28

\item$\widetilde{W}$ (universal cover of $W$), 6

\end{theindex}